\documentclass[a4paper,reqno]{amsart}
\usepackage{amsfonts}
\usepackage{amsmath,amssymb,amsthm,amsxtra}
\usepackage{float}
\usepackage[colorlinks, linkcolor=blue!72,anchorcolor=orange,
    citecolor=red,urlcolor=Emerald, bookmarksopen,bookmarksdepth=2]{hyperref}
\usepackage[usenames,dvipsnames]{xcolor}
\usepackage{enumitem}
\usepackage{geometry,array}
\usepackage{graphicx}
\usepackage{subfigure}
\usepackage{bookmark}
\usepackage{tikz}
\usepackage{url}

\usetikzlibrary{matrix,cd,positioning,decorations.markings,arrows,decorations.pathmorphing,	
    backgrounds,fit,positioning,shapes.symbols,chains,shadings,fadings,calc}
\tikzset{->-/.style={decoration={  markings,  mark=at position #1 with
    {\arrow{>}}},postaction={decorate}}}
\tikzset{-<-/.style={decoration={  markings,  mark=at position #1 with
    {\arrow{<}}},postaction={decorate}}}
\usepackage[all]{xy}

\usepackage{pdflscape}

\theoremstyle{plain}
\newtheorem{theorem}{Theorem}[section]
\newtheorem{lemma}[theorem]{Lemma}
\newtheorem{corollary}[theorem]{Corollary}
\newtheorem{proposition}[theorem]{Proposition}
\theoremstyle{definition}
\newtheorem{definition}[theorem]{Definition}
\newtheorem{construction}[theorem]{Construction}
\newtheorem{example}[theorem]{Example}
\newtheorem{remark}[theorem]{Remark}
\newtheorem{notations}[theorem]{Notations}
\newtheorem{convention}[theorem]{Convention}
\numberwithin{equation}{section}

\numberwithin{equation}{section}

\dedicatory{Dedicated to Jie Xiao on the occasion of his sixtieth birthday}

\newcommand\Old{\bgroup\markoverwith{\textcolor{red}{\rule[0.5ex]{2pt}{0.4pt}}}\ULon}

\def\datum{(\overrightarrow{m},\simeq,\mathbf{gr})}
\def\<{\langle}
\def\>{\rangle}
\def\NN{\mathbb{N}}
\def\ZZ{\mathbb{Z}}

\def\CC{\mathbb{C}}

\newcommand{\add}{\operatorname{add}\hspace{.01in}}
\newcommand{\rep}{\mathsf{rep}}
\newcommand{\repb}{\mathsf{rep}^b}
\newcommand{\Tri}{\operatorname{Tri}^\bullet}

\renewcommand{\mod}{\mathsf{mod}\hspace{.01in}}
\newcommand{\Int}{\operatorname{Int}}

\newcommand{\oInt}{\overrightarrow{\operatorname{Int}}}
\newcommand{\oIntd}{\oInt^{\rho}}
\newcommand{\Hom}{\operatorname{Hom}\nolimits}
\newcommand{\huaHom}{\mathcal{H}om}
\newcommand{\End}{\operatorname{End}\nolimits}
\newcommand{\Ext}{\operatorname{Ext}\nolimits}

\newcommand{\I}{\operatorname{I}\nolimits}

\newcommand{\id}{\operatorname{id}\nolimits}
\newcommand{\ind}{\operatorname{ind}\nolimits}
\newcommand{\Ind}{\operatorname{Ind}\nolimits}
\newcommand{\D}{\operatorname{\mathcal{D}}}
\newcommand{\per}{\operatorname{per}}
\newcommand{\near}[2]{{#1}|_{#2\to}}
\newcommand{\ori}[2]{{#1}|_{#2\to 1-#2}}
\newcommand{\uex}[2]{u^{#1}_{#2}}
\newcommand{\vex}[2]{v^{#1}_{#2}}
\newcommand{\aex}[2]{\alpha^{#1}_{#2}}
\def\numbers{\begin{enumerate}[label=\arabic*{$^\circ$}.]}
\def\ends{\end{enumerate}}

\def\wc{\widetilde{c}}
\def\wg{\widetilde{\gamma}}
\def\we{\widetilde{\eta}}
\def\ws{\widetilde{\sigma}}
\def\wt{\widetilde{\tau}}
\def\wa{\widetilde{\alpha}}
\def\wb{\widetilde{\beta}}

\def\wsx{\ws^\times}
\def\wtx{\wt^\times}
\def\wax{\wa^\times}
\def\wbx{\wb^\times}
\def\wex{\we^\times}

\def\OO{\mathcal{W}_{\operatorname{l.s.}}}

\def\fF{\mathbb{F}^-}
\def\fG{\mathbb{F}^+}
\def\fM{\mathbb{F}}
\def\fE{\mathbb{E}}
\def\fK{\mathbb{K}}
\def\fMM{\mathbb{M}}

\def\arrow{red}

\newcommand{\MCG}{\operatorname{MCG}}

\newcommand\Dehn[1]{\mathrm{D}_{#1}}

\def\T{\mathbf{T}}

\def\M{\mathbf{M}}
\def\Y{\mathbf{Y}}
\def\MP{\mathbf{M}_\mathbf{P}}
\def\YP{\mathbf{Y}_\mathbf{P}}

\def\P{\mathbf{P}}

\def\uas{\operatorname{UA}}	
\def\as{\operatorname{AS}}	
\def\ac{\mathbf{A}}
\def\dac{\ac^\ast}
\def\arc{\mathcal{a}}

\def\lu{\mathfrak{p}}
\renewcommand{\k}{\mathbf{k}}

\def\C{\mathcal{C}}
\def\H{\mathcal{H}}
\def\surf{\mathbf{S}} 
\def\surfi{\surf^\circ}
\newcommand\Sp{\operatorname{Sp}}

\usetikzlibrary{decorations.markings}
\tikzset{->-/.style={decoration={  markings,  mark=at position #1 with
    {\arrow{>}}},postaction={decorate}}}
\tikzset{-<-/.style={decoration={  markings,  mark=at position #1 with
    {\arrow{<}}},postaction={decorate}}}

\def\arc{\mathfrak{a}}

\def\bads{\overline{\operatorname{Ad}}(\gms)}
\def\wads{\widetilde{\operatorname{Ad}}(\gms)}

\def\nn{node{$\bullet$}}

\def\ww{node[white]{$\bullet$}node[red]{$\circ$}}
\def\grad{\lambda}
\def\OA{\operatorname{OA}}
\def\wOA{\widetilde{\OA}}

\def\gmsx{\gms\x}         
\def\x{_\vot}
\def\y{^\vot}
\def\sg{\Lambda}

\newcommand\coho[1]{\operatorname{H}^{#1}}
\newcommand\ho[1]{\operatorname{H}_{#1}}
\newcommand{\PP}{\mathrm{D}}
\newcommand{\Pp}{\mathrm{E}}
\newcommand{\pt}{\star}
\newcommand{\gms}{\surf^\grad}

\def\Amm{\wOA_{\M}^{\M}(\gms)}
\def\Amp{\wOA^{\M}_{\P}(\gms)}
\def\App{\wOA_{\P}^{\P}(\gms)}

\def\CC{\operatorname{CC}}
\def\wCC{\widetilde{\CC}}
\def\ccc{\wCC(\gms)}
\def\Amm{\wOA_{\M}^{\M}(\gms)}
\def\Amp{\wOA^{\M}_{\P}(\gms)}
\def\App{\wOA_{\P}^{\P}(\gms)}
\def\ccc{\wCC(\gms)}
\def\bione{\mathfrak{w}}

\newcommand\fc[1]{\operatorname{\Upsilon}(#1)}

\def\DD{\mathbb{D}}
\def\DBX{\DD(\dac)}
\def\DBY{\DD^\circ(\dac)}
\def\DBZ{\DD^\times(\dac)}
\def\DBXx{\DD(\dac\x)}
\def\DBYx{\DD^\circ(\dac\x)}
\def\DBZx{\DD^\times(\dac\x)}

\def\gUC{\widetilde{\operatorname{UA}}}
\def\care{\widetilde{\operatorname{OC}}_{\text{l.s.}}(\gms)}
\def\period{\aleph}
\def\SW{\mathcal{SFW}}
\def\AW{\mathcal{AFW}}
\def\SPW{\mathcal{SPW}}
\def\APW{\mathcal{APW}}

\def\Dfang{\mathrm{D}^2_\vot}
\usepackage{marvosym}
\def\vot{\text{\Biohazard}}
\def\Vot{\text{\Cancer}}

\def\was{\widetilde{a}^{\ws}}
\def\has{\widehat{a}^{\ws}}
\def\wbs{\widetilde{b}^{\ws}}
\def\hbs{\widehat{b}^{\ws}}

\def\uas{a^{\ws}}
\def\ubs{b^{\ws}}

\def\ubt{b^{\wt}}

\def\goodind{\per^@\sg}
\def\careII{\widetilde{\operatorname{TA}}(\gms)}

\def\?{B}
\def\rad{\operatorname{rad}}

\newcommand{\cpy}[1]{{\mathcal{X}}^{\bullet}_{#1}}
\newcommand{\cply}[2]{{\mathcal{X}}^{#2}_{#1}}

\newcommand{\pmo}[2]{\xi_{#1}(#2)}
\newcommand{\omo}[2]{y_{#1}(#2)}
\newcommand{\mdf}[1]{\Xi(#1)}

\def\wks{\widetilde{k}^{\ws}}	
\def\wkt{\widetilde{k}^{\wt}}	

\def\xis{{\xi}^{\ws}}
\def\xit{{\xi}^{\wt}}

\def\hks{\widehat{k}^{\ws}}	
\def\hkt{\widehat{k}^{\wt}}	

\newcommand{\fout}[1]{f^{out}_{#1}}
\newcommand{\fin}[1]{f^{in}_{#1}}
\newcommand{\fen}[2]{\frac{#1}{#2}}

\def\ks{k^{\ws}}	
\def\kt{k^{\wt}}	
\def\Vs{V^{\ws}}	
\def\Vt{V^{\wt}}	
\def\wsi{\varsigma^{\ws}}	
\def\wti{\varsigma^{\wt}}	
\def\ints{q}	
\def\cias{\operatorname{AS}^{\times}}	
\def\icias{\operatorname{AS}^{\times}_{\operatorname{in}}}	
\def\wicias{\widetilde{\operatorname{AS}}^{\times}_{\operatorname{in}}}	
\def\owicias{\underrightarrow{\widetilde{\operatorname{AS}}}^{\times}_{\operatorname{in}}}	
\def\upas{\operatorname{AS}^{\circ}}	
\def\wupas{\widetilde{\operatorname{AS}}^{\circ}} 
\def\owupas{\underrightarrow{\widetilde{\operatorname{AS}}}^{\circ}}
\def\Was{\widetilde{\operatorname{AS}}} 
\def\oWas{\underrightarrow{\widetilde{\operatorname{AS}}}} 

\def\exi{l}

\newcommand{\asu}[3]{{#1}_{#2 \sim #3}}
\newcommand{\aso}[3]{{#1}_{#2 \to #3}}

\def\wmu{\widetilde{\mu}}
\def\wnu{\widetilde{\nu}}

\def\huaRad{\mathcal{R}ad}
\newcommand{\ec}[1]{\overline{#1}}
\def\Gr{ \overrightarrow{\mbox{gr}} }
\newcommand{\inter}[1]{#1 _{\mathrm{in}}}
\newcommand{\tou}[1]{#1 _{\mathrm{en}}}

\def\bush{\k\overline{S}}
\def\bushz{\k\overline{S^+}}
\def\bushf{\k\overline{S^-}}
\def\hull{\add\bush}

\def\OME{\Omega}
\def\OMEs{\OME^{\sim}}        
\def\OMEpm{\OME^{\pm}}      
\def\OMEx{\OME^{\times}}      
\def\OMEepm{\OME^{\emptyset,\pm}} 
\def\OMEsx{\OME^{\sim,\times}} 
\def\vs{\mathbb{V}}
\def\bvs{\overline{\vs}}

\def\qst{Q^{\ws,\wt}}

\def\xian{arc\;}

\def\gvec{\underline{\operatorname{vect}}}
\begin{document}
\title{Two geometric models for graded skew-gentle algebras}
\author{Yu Qiu}
\address{Qy:
	Yau Mathematical Sciences Center and Department of Mathematical Sciences,
	Tsinghua University,
    100084 Beijing,
    China.
    \&
    Beijing Institute of Mathematical Sciences and Applications, Yanqi Lake, Beijing, China}
\email{yu.qiu@bath.edu}
\author{Chao Zhang}
\address{Zc:
Department of Mathematics,
School of Mathematics and Statistics,
Guizhou University,
550025, Guiyang,
China.}
\email{zhangc@amss.ac.cn}
\author{Yu Zhou}
\address{Zy:
Yau Mathematical Sciences Center,
Tsinghua University,
100084 Beijing,
China}
\email{yuzhoumath@gmail.com}
%


\begin{abstract}
In Part~\ref{part:1}, we classify (indecomposable) objects in the perfect derived category $\mathrm{per}\Lambda$ of a graded skew-gentle algebra $\Lambda$, generalizing technique/results of Burban-Drozd and Deng to the graded setting. We also use the usual punctured marked surface $\mathbf{S}^\lambda$ with grading (and a full formal arc system) to give a geometric model for this classification.

In Part~\ref{part:2}, we introduce a new surface $\mathbf{S}^\lambda_\text{\Biohazard}$ with binary from $\mathbf{S}^\lambda$ by replacing each puncture $P$ by a boundary component $\text{\Biohazard}_P$ (called a binary) with one marked point, and composing an equivalent relation $D_{\text{\Biohazard}_P}^2=\mathrm{id}$, where $D_{\text{\Biohazard}_p}$ is the Dehn twist along $\text{\Biohazard}_P$. Certain indecomposable objects in $\mathrm{per}\Lambda$ can be also classified by graded unknotted arcs on $\mathbf{S}^\lambda_\text{\Biohazard}$. Moreover, using this new geometric model,
we show that the intersections between any two unknotted arcs provide a basis of the morphisms between the corresponding arc objects, i.e. formula $\mathrm{Int}=\mathrm{dim}\mathrm{Hom}$ holds.
\end{abstract}

\maketitle
\tableofcontents \addtocontents{toc}{\setcounter{tocdepth}{1}}

\setlength\parindent{0pt}
\setlength{\parskip}{5pt}

\section*{Introduction}
\subsection{Overall}
Skew-gentle algebras were introduced by Gei{\ss}-de la Pe\~{n}a \cite{GP},
as a generalization of gentle algebras. They are important classes of representation-tame finite-dimensional algebras, where typical examples are type $A$/$\tilde{A}$ (for gentle case) and type $D$/$\tilde{D}$ (for skew-gentle case).

The indecomposable modules of a skew-gentle algebra (or, more generally, a clannish algebra) were classified by Crawley-Boevey \cite{CB89}, Deng \cite{Deng} and Bondarenko \cite{B}. A basis of the space of morphisms between (certain) indecomposable modules was given by Gei{\ss} \cite{G}. The indecomposable objects in the (perfect) derived category of a skew-gentle algebra were classified by Bekkert-Marcos-Merklen \cite{BMM} and Burban-Drozd \cite{BD}.

In the first part of this paper, we classify indecomposable objects in the perfect derived category $\per\Lambda$ of a graded skew-gentle algebra $\Lambda$, generalizing technique/results of Burban-Drozd and Deng to the graded setting and providing the corresponding geometric model via the usual punctured marked surface $\surf$ with grading $\grad$ (and with full formal arc system $\ac$). In the second part of this paper, we introduce a new surface $\gmsx$ by modifying punctures to binaries and show that the intersections between arcs provide a basis for the morphism space between the corresponding \xian objects.

\subsection{cluster theory}
One of our motivations comes from Fomin-Zelevinsky's theory of cluster algebras \cite{FZ}, which has become very popular during the last 20 years as cluster phenomenon appears in various areas in mathematics, as well as in physics. In particular, we are interested in the geometric models for cluster algebras and the additive categorification of cluster algebras via quiver representations. The former one, a marked surface (with punctures) $\surf$, was introduced by Fomin-Shapiro-Thurston (=FST) \cite{FST}, where simple tagged open arcs, i.e. simple arcs connecting marked points/punctures with tagging, correspond to cluster variables. The tricky part in their story is the tagging, which provides the extra $\ZZ^2$ symmetry at each puncture. The latter one, the (Calabi-Yau-2) cluster category (for this geometric model)
$$\C(\surf)=\per\Gamma_\T/\D_{fd}(\Gamma_\T)$$
was developed by Buan-Marsh-Reineke-Reiten-Todorov \cite{BMRRT}, Derksen-Weyman-Zelevinsky \cite{DWZ}, Amoit \cite{A}, Keller \cite{Ke11}, Cerulli-Irelli-Labardini-Fragoso \cite{CI-LF}, and so on. Here, $\T$ is a tagged triangulation, $\Gamma_\T=(Q_\T, W_\T)$ is the Ginzburg dg algebra of the quiver with potential $(Q_\T, W_\T)$ associated to $\T$ and $\D_{fd}(\Gamma_\T)\subset\per\Gamma_\T$ are finite-dimensional/perfect derived categories of $\Gamma_\T$. In our previous work \cite{QZ1}, we prove that when $\surf$ has non-empty boundary, there is a bijection
\[
    \underline{M}\colon \mathrm{OA}^\times(\surf) \to \C^@(\surf),
\]
from the set $\mathrm{OA}^\times(\surf)$ of tagged open arcs in $\surf$
to the set $\C^@(\surf)$ of \xian (indecomposable) objects in the category $\C(\surf)$
such that
\[
    \Int(\gamma_1^{\kappa_1},\gamma_2^{\kappa_2})=
    \dim\Ext_{\C(\surf)}^1(\underline{M}(\gamma_1^{\kappa_1}),\underline{M}(\gamma_2^{\kappa_2})).
\]
See also \cite{CCS} for disk case (i.e. type $A$), \cite{S} for once-punctured disk case (i.e. type $D$), \cite{ABCP,BZ,ZZZ,CS} for unpunctured case (i.e. gentle type), and \cite{AP} for punctured case via orbifolds. Note that the result above is based on the following facts. When choosing an admissible tagged triangulation $\T$ of $\surf$, the corresponding Jacobian algebra $J(Q_\T,W_\T)$ is a skew-gentle algebra \cite{GL-FS,QZ1} and we have an equivalence \cite{KR,A}
\[
    \C(\surf)/T_\T\simeq\mod J(Q_\T,W_\T)
\]
for the canonical cluster tilting object $T_\T$ associated to $\T$.
Thus, the geometric model $\surf$ in fact also provides a model for the module category of a skew-gentle algebra, see \cite{CS,BS,HZZ}.

Moreover, from the point of view of Ikeda-Qiu \cite{IQ1,IQ2}, the perfect derived categories of finite-dimensional algebras can be realized as cluster-$\mathbb{X}$ categories. This motivates us to give the geometric model for perfect derived categories. In fact, for the gentle algebras case, there are many developments in this direction, as well as some developments in the skew-gentle case (see the next subsection).

\subsection{Topological Fukaya categories}

The surface models for categories arising from the representation of algebras share a lot of similarities with Fukaya categories in symplectic geometry. In fact, Haiden-Katzarkov-Kontsevich (=HKK) \cite{HKK} constructed the perfect derived category of
a graded gentle algebra $\sg$ as the topological Fukaya category of a graded marked surface $\gms$. They classify all indecomposable objects by showing that there is a bijection
\[
    M \colon \mathrm{OC}^{l.s.}(\gms) \to\Ind\per\gms\colon=\Ind\per\sg,
\]
between the set $\mathrm{OC}^{l.s.}(\gms)$ of open curves with local system and the set of indecomposable objects in $\per\sg$. Later, this geometric model was used to study the derived invariants of graded gentle algebras, see \cite{OPS18,LP1,APS19}, where a basis of the space of morphisms is interpreted via intersections of curves.

The key extra information in HKK's model $\gms$, compared to FST's model $\surf$, is that one needs to add a grading (a cohomology class in the projectivization of the tangent space of $\surf$) to realize the shift in the derived category.
The grading information is implicitly contained in the triangulation
after decorating $\surf$ with a set of decorations, introduced in \cite{QQ} (cf. also \cite{QQ2,QZ2}).

As the results of \cite{FST,QZ1,S} work for punctured case, it is natural to generalize things above to the punctured case for any graded skew-gentle algebra $\sg$. The main difficulty lies on interpreting the $\ZZ_2$ symmetry at punctures, since FST's tagging trick needs some modification. Different from \cite{AB,L-FSV,A21} which uses the technique of orbifold, our solution, for generalizing HKK's model to the punctured case, is to introduce binaries, where we replace each puncture $P$ with a boundary component with a marked point (called a binary) $\Vot_P$, and compose the condition that the square of the Dehn twist along $\Vot_P$ equals the identity (cf. Figure~\ref{fig:000} below).
\begin{figure}[htpb]
	\begin{tikzpicture}[scale=.3]
    \clip (-20,-3) rectangle (20,9);
    \begin{scope}[shift={(-15,0)}]
	\draw[ultra thick,fill=gray!10]
        (7,2)node{\Huge{$\rightsquigarrow$}};
    \draw[blue] (90:8) .. controls +(-45:4) and +(-30:7) .. (90:.5) (90:.25)node[below]{$P$}
        (1,0)node[rotate=20]{$+$};
    \draw[cyan, ultra thick] (90:8) .. controls +(225:4) and +(-150:7) .. (90:.5);
    \draw[Green] (90:8)\nn;
	\draw[blue, thick] (90:.5)\nn;
    \end{scope}
	\draw[ultra thick,fill=gray!10] (90:2) ellipse (1.5) node {$\Vot_P$}
        (6,2)node{\Huge{=}};
    \draw[blue] (90:8) .. controls +(-45:4) and +(-30:7) .. (90:.5) (90:.25) node[below]{$m_P$};
    \draw[cyan, ultra thick] (90:8) .. controls +(225:4) and +(-150:7) .. (90:.5);
    \draw[Green] (90:8)\nn;
	\draw[blue, thick] (90:.5)\nn (0,3.5)node[white,above]{$y_P$}\ww;
    \begin{scope}[shift={(13,0)}]
	\draw[ultra thick,fill=gray!10] (90:2) ellipse (1.5) node {$\Vot_P$};
    \draw[blue] (90:8) .. controls +(-45:5) and +(-30:7) .. (90:.5);
    \draw[cyan, ultra thick] (90:8)
        .. controls +(-30:7) and +(-30:9) .. (0,-1)
        .. controls +(150:5) and +(160:3) .. (0,4.5)
        .. controls +(-20:4) and +(-15:2) .. (0,.5);
    \draw[Green] (90:8)\nn;
	\draw[very thick,blue] (90:.5)\nn  (0,3.5)\ww;
    \end{scope}	
    \end{tikzpicture}
\caption{$\Dehn{\Vot_P}^2$-action}\label{fig:000}
\end{figure}
On the level of classification of objects (via arcs/curves) in $\per\sg$,
this is equivalent to the tagging model. However, when studying the morphisms between objects, this new model provides a better correspondence between morphisms and (clockwise angles from one arc to the other at) intersections as shown in Figure~\ref{fig:001}. Although numerically, one can still use the tagging model by ad-hoc definition of tagged intersection, cf. Figure~\ref{fig:oriented int} and Lemma~\ref{lem:int}.

\begin{figure}[htpb]
	\begin{tikzpicture}[scale=.5]
	\draw[thick, Green,->-=.5,>=stealth](-1,1)to[bend left=30](1,1);
	\draw[thick, Green,-<-=.5,>=stealth](-1,-1)to[bend left=-30](1,-1);
	\draw[blue,very thick](2,2)to(-2,-2)(-2,2)to(2,-2);
    \begin{scope}[shift={(8,-2)}]
	\draw[thick, Green,->-=.5,>=stealth](-1.7,1.7)to[bend left=30](1.7,1.7);
    \draw[ultra thick](3,0)to(-3,0);
	\draw[blue,very thick](4,4)to(0,0)\nn to(-4,4);
    \end{scope}	
    \end{tikzpicture}
\caption{Clockwise angles from one arc to the other at intersections}\label{fig:001}
\end{figure}

Recently, there were also some works concerning geometric models for derived categories of skew-gentle algebras \cite{A21,AB,L-FSV}. One feature of our work is to consider graded skew-gentle algebras, which will be useful for silting theory, and another is to give a basis for morphism spaces (of certain objects), which was not constructed before even in an algebraic or combinatorial way.

\subsection{Mirror symmetry}
Our final motivation comes from homological mirror symmetry (HMS), where the Fukaya(-type) categories originally come from. For instance, Lekili-Polishchuk (\cite{LP2}) shows the following version of HMS
\[\D_{fd}(\gms)\cong \D^b\mathrm{Coh}(C)\]
for certain graded marked surface $\gms$ and nodal curve $C$, where $\D^b\mathrm{Coh}(C)$ is the derived category of coherent sheaves on $C$.

Furthermore, a geometric correspondence arising from a (general) HMS
\[\D\mathrm{Fuk}(Y)=\D^b\mathrm{Coh}(X)\]
is the conjectural isomorphism between complex manifolds:
$$\mathrm{Stab}\D\mathrm{Fuk}(Y)=\mathcal{M}(X),$$
where $(Y,X)$ is a mirror pair, $\mathrm{Stab} \D\mathrm{Fuk}(Y)$ is the space of Bridgeland stability conditions and
$\mathcal{M}(X)$ is the moduli space of certain Kahler structure on $X$.
While the general case is extremely hard,
there are several works that have been done during the past couple of years.
Namely, for the categories $\D(\surf^?_?)$ being the ones above
with the corresponding surface model $\surf^?_?$ (due to \cite{HKK,BS,KQ2,IQ2}),
one has
\[\ \mathrm{Stab}\D(\surf^?_?)=\mathrm{FQuad}(\surf^?_?), \]
where $\mathrm{FQuad}$ is the moduli space of quadratic differentials on $\surf^?_?$ with pre-described singularities and framing. We plan to use our new model to generalize HKK's result to the punctured case in the follow-up works.

\subsection{Contents}
In Part~\ref{part:1}, we focus on a classification of objects
and have the following:
\begin{itemize}
\item In \S~\ref{sec:GSGA}, we review basics of graded skew-gentle algebras and graded punctured marked surfaces.
\item In \S~\ref{sec:bush}, we associate a bush to a full formal closed arc system and apply results in \cite{Deng} to get a classification of representations of the bush.
\item In \S~\ref{sec:class}, we review works of Burban-Drozd on triple categories, and establish a relationship between bushes and triple categories. Then we classify indecomposable objects in the perfect derived category $\per\sg$ of a graded skew-gentle algebra $\sg$ via admissible curves with local system on graded punctured marked surfaces (Theorem~\ref{thm1}).
\end{itemize}
In Part~\ref{part:2}, we will focus on morphisms between arc objects
(the indecomposable objects that correspond to arcs in the geometric model)
and prove Theorem~\ref{thm:part2}.
More precisely:
\begin{itemize}
\item In \S~\ref{sec:AGM}, we introduce the new geometric model--graded marked surface with binary and
graded unknotted arcs on it (which will correspond to arc objects).
\item In \S~\ref{sec:DG-a}, we associate dg models to
graded unknotted arcs and show that such a construction is compatible with
the geometric model in Part~\ref{part:1} (Theorem~\ref{thm2}).
\item In \S~\ref{sec:mor}, we prove the $\Int=\dim\Hom$ formula (Theorem~\ref{thm:int=dim}), in the way that we construct a morphism explicitly from each intersection between any two graded unknotted arcs, which are shown to form a basis of the morphism space between the corresponding arc objects.
\end{itemize}
In Appendix~\ref{app:period}, the relative notations of period words are given. In Appendix~\ref{app:pfBD}, a proof of Theorem~\ref{thm:BD} is given. In Appendix~\ref{app:pfs}, technical proofs of certain results in Section~\ref{sec:mor} are given.


\subsection*{Conventions}

In the paper, $\k$ denotes a field. All algebras are finite-dimensional algebras over $\k$. For an algebra $\sg$, a $\sg$-module is a finitely generated right module of $\sg$, and we denote by $\mod \sg$ the category of $\sg$-modules. For a graded algebra $\sg$, the degree of a homogeneous element $a\in \sg$ is denoted by $|a|$. For any two arrows $\alpha$ and $\beta$ in a quiver $Q$, we denote by $\alpha\beta$ the path in $Q$ first $\alpha$ then $\beta$. For any two morphisms $f,g$, the composition $gf=g\circ f$ stands for first $f$ then $g$.

For any equivalence relation (or a symmetric relation that can be uniquely extended to an equivalence relation) on a set $X$ and an element $x\in X$, we use $\ec{x}$ to denote the equivalence class containing $x$. To simplify the notations, when $\ec{x}$ contains exactly one element, we also use $x$ to denote the equivalence class.


\subsection*{Acknowledgement}

We would like to thank Xiao-Wu Chen, Bangming Deng, Dong Yang, and Bin Zhu for their helpful discussions. Parts of the results in this paper were presented at conferences International Workshop on
Cluster Algebras and Related Topics at
Morningside Center of Mathematics, CAS, 2021, and New developments in representation theory arising from cluster algebras at Isaac Newton Institute for Mathematical Sciences, 2021. The authors would like to thank the organizers of these conferences.
We would also like to thank Ping He for assisting on tikz-drawing of some of the figures and for proofreading.
This work was supported by National Natural Science Foundation of China (Grant Nos. 12271279, 12031007 and 11961007) and National Key R\&D Program of China (NO. 2020YFA0713000).

\part{Classification of objects and the usual geometric model}\label{part:1}
\section{Graded skew-gentle algebras and punctured marked surfaces}\label{sec:GSGA}
\subsection{Graded skew-gentle algebras}\label{subsec:clannish}

Skew-gentle algebras, modeled on gentle algebras \cite{AS87}, were introduced in \cite{GP} as a special class of clannish algebras defined in \cite{CB89}. We consider its graded version in this paper.

A quiver $Q$ is a tuple $(Q_0,Q_1,s,t)$, where $Q_0$ is the set of \emph{vertices}, $Q_1$ is the set of \emph{arrows}, and $s,t:Q_1\rightarrow Q_0$ are the \emph{start} and \emph{terminal} functions of arrows, respectively. We call an arrow $\alpha\in Q_1$ a \emph{loop} if $s(\alpha)=t(\alpha)$. A \emph{path} of \emph{length} $n$ in $Q$ is a sequence $a_1a_2\cdots a_n$ of arrows $a_i,1\leq i\leq n$ with $t(a_i)=s(a_{i+1})$ for any $1\leq i\leq n-1$. For each vertex $i\in Q_0$, we denote by $e_i$ the trivial path of length 0 associated to it.
We denote by $\k Q$ the path algebra of $Q$.
A \emph{relation set} of $Q$ is a finite set of linear combinations of paths in $Q$.

A \emph{graded quiver} is a quiver $Q$ with a grading, i.e., a function $|\cdot|:Q_1\to \ZZ$. For any arrow $\alpha\in Q_1$, we call $|\alpha|$ the \emph{degree} of $\alpha$. For a path $p=\alpha_1\cdots\alpha_n$ of length $n\geq 1$, the degree of $p$ is defined to be $|p|:=\sum_{i=1}^n|\alpha_i|$. For a trivial path $e_i$, $i\in Q_0$, the degree $|e_i|$ is defined to be 0. The path algebra of a graded quiver is naturally a $\ZZ$-graded algebra.

\begin{definition}\label{def:gentle}
A pair $(Q,I)$ of a graded quiver $Q$ and a relation set $I$ is called a \emph{gentle pair} provided the following hold.
\begin{enumerate}
\item[(G1)] Any element in $I$ is a path of length 2.
\item[(G2)] For each vertex in $Q_0$, there are at most two arrows in and at most two arrows out.
\item[(G3)] For any arrow $\alpha$, there is at most one arrow $\beta$ (resp. $\gamma$) with
$\alpha\beta\in I$ (resp. $\gamma\alpha\in I$).
\item[(G4)] For any arrow $\alpha$, there is at most one arrow $\beta$ (resp. $\gamma$) with
$\alpha\beta\notin I$ (resp. $\gamma\alpha\notin I$).
\end{enumerate}
\end{definition}

\begin{definition}\label{def:sg1}
	A triple $(Q,\Sp,I)$, with $Q$ a graded quiver, $\Sp$ a subset of $Q_0$ and $I$ a relation set, is called \emph{graded skew-gentle} if the pair $(Q^{\mathrm{sp}},I^{\mathrm{sp}})$ is gentle, where
	\begin{itemize}
		\item $Q_0^{\mathrm{sp}}=Q_0$,
		\item $Q_1^{\mathrm{sp}}=Q_1\cup\{\varepsilon_i\mid i\in \Sp\}$ with $s(\varepsilon_i)=t(\varepsilon_i)=i$ and $|\varepsilon_i|=0$, and
		\item $I^{\mathrm{sp}}=I\cup\{\varepsilon_i^2\mid i\in \Sp \}$.
	\end{itemize}
A finite-dimensional graded algebra $\Lambda$ is called \emph{skew-gentle} if $\Lambda$ is Morita equivalent to
$$\boxed{\sg(Q,\Sp,I):=\k Q^{\mathrm{sp}}/\<I\cup\{\varepsilon_i^2-\varepsilon_i\mid i\in \Sp\}\>}$$
for a graded skew-gentle triple $(Q,\Sp,I)$. When $\operatorname{Sp}=\emptyset$, we call the graded algebra $\sg$ \emph{gentle}.
\end{definition}

\begin{example}\label{exm:clan}
	Let $Q$ be the following quiver
	\[
	\xymatrix@C=3pc@R=1pc{
		&1\ar[ld]_b&& 6\ar[ld]_g\\
		2\ar[rd]_c&&4\ar[lu]_a\ar[rd]_e\\
		&3\ar[ru]_d&&5 \ar[uu]_f\\
	}\]
	with $|a|=|d|=-1,|b|=|c|=|e|=|g|=|f|=0$, $\Sp=\{1,2\}$ and $I=\{ab,bc,da,ef, fg, ge\}$. The quiver $Q^{\mathrm{sp}}$ is the following quiver
	\[
	\xymatrix@C=3pc@R=1pc{
		&1\ar@{->}@(ur,ul)[]_{\varepsilon_1}\ar[ld]_b&& 6\ar[ld]_g\\
		2\ar@{->}@(lu,ld)[]_{\varepsilon_2}\ar[rd]_c&&4\ar[lu]_a\ar[rd]_e\\
		&3\ar[ru]_d&&5 \ar[uu]^f\\
	}\]
	with $|\varepsilon_1|=|\varepsilon_2|=0$, and the relation set $I^{\mathrm{sp}}=I\cup\{\varepsilon_1^2,\varepsilon_2^2\}$. By definition, $(Q^{\mathrm{sp}},I^{\mathrm{sp}})$ is a gentle pair. So we get the following graded skew-gentle algebra:
	$$Q^{\mathrm{sp}}/\<ab,bc,da,ef,fg,ge,\varepsilon_1^2-\varepsilon_1,\varepsilon_2^2-\varepsilon_2\>.$$
\end{example}

\subsection{Alternative description}\label{subsec:clannish2}

In what follows, we will use an equivalent definition of graded skew-gentle algebras, which is the graded version of that in \cite[Section~2]{BD}.

For any subset $\widetilde{\Sigma}\subseteq\{1,\cdots,m\}$, let $M_{m,\widetilde{\Sigma}}(\k)$ be the $\k$-algebra of $(m+|\widetilde{\Sigma}|)\times(m+|\widetilde{\Sigma}|)$ matrices whose entries are in $\k$. The rows and columns of a matrix in $M_{m,\widetilde{\Sigma}}(\k)$ are both indexed by the set $\{1,\cdots,m\}\cup\{j^+,j^-\mid j\in\widetilde{\Sigma}\}\setminus\widetilde{\Sigma}$ whose elements are presented by $j^\kappa$ with $1\leq j\leq m$, $\kappa=\emptyset$ for $j\notin \widetilde{\Sigma}$, and $\kappa\in\{+,-\}$ for $j\in \widetilde{\Sigma}$. Denote by $E_{j_1^{\kappa_1},j_2^{\kappa_2}}$ the matrix in $M_{m,\widetilde{\Sigma}}(\k)$ whose $(j_1^{\kappa_1},j_2^{\kappa_2})$-entry is $1$ and the other entries are 0. Let $|\cdot|$ be a grading on $M_{m,\widetilde{\Sigma}}(\k)$ such that each $E_{j_1^{\kappa_1},j_2^{\kappa_2}}$ is homogeneous and $|E_{j_1^{\kappa_1},j_2^{\kappa_2}}|=|E_{i_1^{\xi_1},i_2^{\xi_2}}|$ if $j_1=i_1$ and $j_2=i_2$. Set $\chi_{j_1,j_2}=|E_{j_1^{\kappa_1},j_2^{\kappa_2}}|$. By the compatibility between grading and multiplication, we have
\begin{itemize}
    \item $\chi_{j,j}=0$ for any $1\leq j\leq m$,
    \item $\chi_{j_2,j_1}=-\chi_{j_1,j_2}$ for any $1\leq j_1,j_2\leq m$, and
    \item $\chi_{j_1,j_2}=\sum_{k=j_1}^{j_2-1}\chi_{k,k+1}$ for any $1\leq j_1<j_2\leq m$.
\end{itemize}
It follows that the grading is determined by the values $\chi_{j,j+1}$, $1\leq j<m$.

For any matrix $X\in M_{m,\widetilde{\Sigma}}(\k)$, denote by $X_{j_1^{\kappa_1},j_2^{\kappa_2}}$ the $(j_1^{\kappa_1},j_2^{\kappa_2})$-entry of $X$. Let $T_{m,\widetilde{\Sigma}}$ be the graded subalgebra of $M_{m,\widetilde{\Sigma}}(\k)$ as follows:
$$T_{m,\widetilde{\Sigma}}=\{X\in M_{m,\widetilde{\Sigma}}(\k)\mid X_{j_1^{\kappa_1}j_2^{\kappa_2}}=0\text{ for $j_1>j_2$}\}.$$
Then $\{E_{j_1^{\kappa_1},j_2^{\kappa_2}}\mid 1\leq j_1\leq j_2\leq m\}$ is a basis of $T_{m,\widetilde{\Sigma}}$.

\begin{definition}\label{def:sg}
    A \emph{graded skew-gentle datum} is a datum $(\overrightarrow{m},\simeq,\mathbf{gr})$ consisting of the following:
    \begin{itemize}
	\item $\overrightarrow{m}=(m_1,\cdots,m_t)\in\NN^t_{\geq 1}$,
	\item $\simeq$ is a symmetric (but not necessarily reflexive) relation on the set
	$$\OME=\OME(\overrightarrow{m}):=\{(i,j)\mid 1\leq i\leq t,1\leq j\leq m_i \}$$
	such that for any $(i_1,j_1)\in\OME$, there is exactly one $(i_2,j_2)\in\OME$ with $(i_1,j_1)\simeq(i_2,j_2)$.
	\item $\mathbf{gr}=(\Gr_1,\cdots,\Gr_t)$, where $$\Gr_i=(\chi_{1,2}^{(i)},\cdots,\chi_{m_i-1,m_i}^{(i)})\in\ZZ^{m_i-1}, 1\leq i\leq t.$$
    \end{itemize}
    We denote $\fc{\OME}=\{(i,j)\in\OME\mid (i,j)\simeq(i,j) \}$. The corresponding \emph{graded skew-gentle} algebra $\sg=\sg\datum$ is defined as follows:
    \begin{itemize}
	\item $\widetilde{\Sigma}_i=\{1\leq j\leq m_i\mid (i,j)\in\fc{\OME} \}\subset\{1,2,\cdots,m_i\}$,
	\item $H_i=T_{m_i,\widetilde{\Sigma}_i}$ whose grading is determined by $\chi^{(i)}_{j,j+1}$, $1\leq j< m_i$,
	\item $H=H\datum=H_1\times\cdots\times H_t$, and
	\item the subalgebra $\Lambda\subseteq H$ is
        \[\left\{(X(i))_{1\leq i\leq t}\in H\,\middle\vert\, \begin{matrix} X(i_1)_{j_1j_1}=X(i_2)_{j_2j_2}&\text{for $(i_1,j_1)\simeq (i_2,j_2)\notin\fc{\OME}$},\\ X(i)_{j^+j^-}=0=X(i)_{j^-j^+}& \text{for $(i,j)\in\fc{\OME}$}
        \end{matrix} \right\}.\]
    \end{itemize}
\end{definition}

\begin{remark}\label{rmk:new}
    For any graded skew-gentle datum $(\overrightarrow{m},\simeq,\mathbf{gr})$, reordering $m_1,\cdots,m_t$ if necessary, we may assume that there is an integer $1\leq N\leq t$ such that $m_i=1$ if and only if $i> N$. We define $\overrightarrow{m}'=\{1,\cdots,m_{N}\}$ and a symmetric relation $\simeq'$ on the set $\Omega(\overrightarrow{m}')$ by $(i_1,j_1)\simeq'(i_2,j_2)$ if and only if $(i_1,j_1)\simeq(i_2,j_2)$. Then we get a datum $(\overrightarrow{m}',\simeq')$ in the sense of \cite[Definition~2.3]{BD}, whose corresponding skew-gentle algebra coincides the ungraded version of $\sg\datum$.
\end{remark}

The relation $\simeq$ on $\OME$ can be extended uniquely to an equivalence relation (by adding the reflexive property). So we have the quotient set $\OMEs:=\OME/\simeq$, i.e., $$\OMEs=\{\{(i_1,j_1),(i_2,j_2)\}\mid (i_1,j_1)\simeq(i_2,j_2)\notin\fc{\OME}\}\cup\fc{\OME}.$$
For any $(i,j)\in\OME$, set
$$\OMEpm(i,j):=\begin{cases}
\{(i,j^+),(i,j^-)\}&\text{if $(i,j)\in\fc{\OME}$,}\\
\{(i,j)\}&\text{if $(i,j)\notin\fc{\OME}$.}
\end{cases}$$
Then for any $1\leq i\leq t$, the union $\bigcup_{1\leq j\leq m_i}\OMEpm(i,j)$ is the index set of matrices in $T_{m_i,\widetilde{\Sigma}_i}$. Set
$$\OMEpm=\bigcup_{(i,j)\in\OME}\OMEpm(i,j).$$
There is an induced symmetric relation $\simeq$ on $\OMEpm$ that $(i_1,j_1)\simeq(i_2,j_2)$ if $(i_1,j_1)\simeq(i_2,j_2)\notin\fc{\OME}$, and that $(i,j^+)\simeq(i,j^+)$ and $(i,j^-)\simeq(i,j^-)$ for any $(i,j)\in\fc{\OME}$. Denote
$$\fc{\OMEpm}=\{(i,j^+),(i,j^-)\mid (i,j)\in\fc{\OME} \}.$$
Similarly as $\OME$, we have the quotient set $\OMEx:=\OMEpm/\simeq$,
$$\OMEx=\{\{(i_1,j_1),(i_2,j_2)\}\mid (i_1,j_1)\simeq(i_2,j_2)\notin\fc{\OMEpm}\}\cup\fc{\OMEpm}.$$

\begin{remark}\label{rmk:datum to triple}
We give a brief explanation of the equivalence between Definition~\ref{def:sg1} and Definition~\ref{def:sg}. Let $(\overrightarrow{m},\simeq,\mathbf{gr})$ be a graded skew-gentle datum. We associate a triple $(Q=(Q_0,Q_1,s,t,|\cdot|),\operatorname{Sp},I)$, where
\begin{itemize}
\item $Q_0=\OMEs$,
\item $Q_1=\{\alpha_{(i,j)}:\overline{(i,j-1)}\to\overline{(i,j)}\mid(i,j)\in\OME\}$ with the degree $|\alpha_{(i,j)}|=\chi^{(i)}_{j-1,j}$,
\item $\operatorname{Sp}=\fc{\OME}$, and
\item $I=\{\alpha_{(i_1,j_1)}\alpha_{(i_2,j_2+1)}\mid (i_1,j_1)\simeq(i_2,j_2)\}$.
\end{itemize}
Note that $(i_1,j_1)$ may coincide with $(i_2,j_2)$ in the construction of $I$. We have an isomorphism of algebras
$$\boxed{
\sg\datum\cong\sg(Q,\operatorname{Sp},I).
}$$
\end{remark}

\begin{example}\label{ex:1}
Let $\datum$ be a datum, where
\begin{itemize}
    \item $\overrightarrow{m}=(7,2,1)\in\mathbb{N}^3_{\geq 1}$,
    \item the symmetric relation $\simeq$ on
    $$\OME=\{(1,1),(1,2),\cdots,(1,7),(2,1),(2,2),(3,1)\}$$
    is given by $(1,1)\simeq (2,2)$, $(1,2)\simeq (1,6)$, $(1,3)\simeq(1,3)$, $(1,4)\simeq(1,4)$, $(1,5)\simeq(3,1)$, and $(1,7)\simeq(2,1)$,
    \item $\mathbf{gr}=(\Gr_1,\Gr_2,\emptyset)$, where
    $$\Gr_1=(0,-1,0,0,-1,0), \Gr_2=(0).$$
\end{itemize}
The sets $\OMEpm(1,3)=\{(1,3^+),(1,3^-)\}$, $\OMEpm(1,4)=\{(1,4^+),(1,4^-)\}$ and $\OMEpm(i,j)=\{(i,j)\}$ for any $(i,j)\neq (1,3),(1,4)$.

Then the associated graded quiver $Q$ is
\[
\xymatrix@C=3pc@R=1pc{
	&\{(1,3),(1,3)\}\ar[ld]_{\alpha_{(1,4)}}&& \{(1,1),(2,2)\}\ar[ld]_{\alpha_{(1,2)}}\\
	\{(1,4),(1,4)\}\ar[rd]_{\alpha_{(1,5)}}&&\{(1,2),(1,6)\}\ar[lu]_{\alpha_{(1,3)}}\ar[rd]_{\alpha_{(1,7)}}\\
	&\{(1,5),(3,1)\}\ar[ru]_{\alpha_{(1,6)}}&&\{(1,7),(2,1)\} \ar[uu]_{\alpha_{(2,2)}}\\
}\]
with $|\alpha_{(1,3)}|=|\alpha_{(1,6)}|=-1$, $|\alpha_{(1,2)}|=|\alpha_{(1,4)}|=|\alpha_{(1,5)}|=|\alpha_{(1,7)}|=|\alpha_{(2,2)}|=0$, $\operatorname{Sp}=\{\{(1,3)\},\{(1,4)\}\}$ and $$I=\{\alpha_{(1,3)}\alpha_{(1,4)}, \alpha_{(1,4)}\alpha_{(1,5)}, \alpha_{(1,6)}\alpha_{(1,3)},\alpha_{(1,2)}\alpha_{(1,7)},\alpha_{(1,7)}\alpha_{(2,2)},\alpha_{(2,2)}\alpha_{(1,2)}\}.$$
So we get the graded skew-gentle triple in Example~\ref{exm:clan}.
\end{example}

\begin{remark}\label{rmk:id}
    One can identify $\OMEpm$ with a complete set of pairwise orthogonal primitive idempotents of $H$ by sending $(i,j^\kappa)\in\OMEpm$ to $e_{(i,j^\kappa)}=(X(i_0))_{1\leq i_0\leq t}\in H$ with $X(i_0)=0$ for $i_0\neq i$ and $X(i)=E_{j^\kappa,j^\kappa}$.

    One can identify $\OMEx$ with a complete set of pairwise orthogonal primitive idempotents of $\sg$ by sending $\{(i_1,j_1),(i_2,j_2)\}$ with $(i_1,j_1)\simeq(i_2,j_2)\notin\fc{\OME}$ to $e_{(i_1,j_1)}+e_{(i_2,j_2)}$, and sending $(i,j^+)$ and $(i,j^-)$ to $e_{(i,j^+)}$ and $e_{(i,j^-)}$ respectively for $(i,j)\in\fc{\OME}$.
\end{remark}

\begin{notations}\label{not:path}
For any $(i,j_1^{\kappa_1}),(i,j_2^{\kappa_2})\in\OMEpm$ with $1\leq j_1\leq j_2\leq m_i$, we set
$$\lu((i,j_1^{\kappa_1}),(i,j_2^{\kappa_2})):=(X(i_0))_{1\leq i_0\leq t}\in H,$$
where $X(i_0)=0$ for any $i_0\neq i$ and $X(i)=E_{j_1^{\kappa_1},j_2^{\kappa_2}}$.
\end{notations}

\begin{remark}\label{rmk:BR}
Let
$$\mathcal{B}=\{((i,j_1^{\kappa_1}),(i,j_2^{\kappa_2}))\in\OMEpm\times\OMEpm\mid j_1\leq j_2 \}.$$
The set $\{\lu(x_1,x_2)\mid (x_1,x_2)\in\mathcal{B}\}$ is a basis of $H$, and for any $(x_1,x_2),(x_2',x_3)\in\mathcal{B}$, we have
\begin{equation}\label{eq:comp}
\lu(x_1,x_2)\lu(x_2',x_3)=\begin{cases}\lu(x_1,x_3)&\text{if $x_2'=x_2$,}\\0&\text{otherwise.}\end{cases}
\end{equation}
Let
$$\mathcal{R}=\{((i,j_1^{\kappa_1}),(i,j_2^{\kappa_2}))\in\OMEpm\times\OMEpm\mid j_1<j_2\}\subset\mathcal{B}.$$
The set $\{\lu(x_1,x_2)\mid (x_1,x_2)\in\mathcal{R}\}$ is a basis of $\rad(H)=\rad(\sg)$. The algebra $\sg$ has a basis $\{\lu(x_1,x_2)\mid(x_1,x_2)\in\mathcal{R}\}\cup\OMEx$ (via the second identification in Remark~\ref{rmk:id}).
\end{remark}

\subsection{Graded punctured marked surfaces and full formal arc systems}\label{subsec:surface}

\begin{definition}\label{def:msp}
A {\it marked surface with punctures} $\surf=(\surf,\M,\Y,\P)$ is a compact connected oriented surface $\surf$ with
\begin{itemize}
\item a non-empty boundary $\partial \surf$;
\item a finite set $\M\subset\partial \surf$ of \emph{open marked points} and a finite set $\Y\subset\surf$ of \emph{closed marked points}, satisfying
\begin{itemize}
\item each component of $\partial \surf$ contains marked points in both $\M$ and $\Y$, and
\item marked points in $\M$ and $\Y\cap\partial\surf$ are alternative in any component of $\partial \surf$;
\end{itemize}
\item a finite set $\P$ of \emph{punctures} in $\surf\setminus\partial \surf$.
\end{itemize}
Denote $\Y^\circ=\Y\cap(\surf\setminus\partial \surf)$ the set of interior closed marked points, and $\surfi=\surf\setminus(\partial\surf\cup\P\cup\Y^\circ)$ the interior of $\surf$.
\end{definition}

\begin{convention}
In the figures,
\begin{itemize}
\item open objects (e.g. marked points/arcs/curves) are drawn in blue;
\item closed objects (e.g. marked points/arcs/curves) are drawn in red;
\item boundaries and punctures are drawn in black.
\end{itemize}
\end{convention}

\begin{definition}\label{def:grading}
A \emph{grading} $\grad$ on $\surf$ is a class in $\coho{1}\left(\mathbb{P}T(\surfi),\ZZ\right)$, where $\mathbb{P}T(\surfi)$ is a projectivized tangent bundle,
with
\begin{itemize}
    \item value 1 on each clockwise loop $\{p\}\times\mathbb{R}\mathbb{P}^1$ on $\mathbb{P}T_p(\surfi)$, and
    \item value $-1$
on each clockwise loop
$l_p\times\{x\}$ on $\surf$ around any $p \in \P$ and $x\in\mathbb{RP}^1$,
or equivalently,
\begin{gather}\label{eq:grad=1}
    \grad[ l_p\times\{x\} ]=1.
\end{gather}
\end{itemize}
Note that each grading corresponds to a section of $\mathbb{P}T(\surfi)$.
A {\it graded marked surface with punctures} (=GMSp) $\gms$ is a pair of a marked surface $\surf$ with punctures and a grading $\grad$ on it, cf. \cite{HKK,LP1,IQZ}.
\end{definition}

The new input here is \eqref{eq:grad=1}, which corresponds to the condition that $|\varepsilon_i|=0$ in Definition~\ref{def:sg1}.

\begin{remark}
The \emph{winding number} $w(\alpha)$ of a loop $\alpha$ in $\surf^\circ$ is defined to be $\grad([\dot{\alpha}])$, where $[\dot{\alpha}]\in\ho{1}\left(\mathbb{P}T(\surfi),\ZZ\right)$
is the tangent of $\alpha$.
For a boundary component $\partial$ (or a singularity $p\in\P\cup\Y^\circ$), one can also define the winding number to be $w(\alpha)$ for any $\alpha$ that is isotopy to $\partial$ (or $p$).
So \eqref{eq:grad=1} says that the winding number of any puncture is one.
\end{remark}

\begin{definition}\label{def:curve}
A \emph{curve} on $\surf$ is an immersion $c:I\to \surf$ where $I=[0,1]$ or $S^1$. A \emph{grading} $\wc$ on $c$ is given by a homotopy class of paths in $\mathbb{P}T_{c(t)}(\surfi)$ from $\grad(c(t))$ to $\dot{c}(t)$, varying continuously with $t\in I$. The pair $(c,\wc)$ (or $\wc$ for short) is called a {\it graded curve}.

For any graded curve $\wc$ and any $\rho\in\ZZ$, denote by $\wc[\rho]$ the graded curve whose underlying curve is the same as $\wc$ and whose grading is the composition of $\wc(t):\grad(c(t))\to\dot{c}(t)$ and the path from $\dot{c}(t)$ to itself given by clockwise rotation by $\rho\pi$.
\end{definition}

For any graded curves $(c_1,\wc_1), (c_2,\wc_2)$ in $\surf$.
Let $q=c_1(t_1)=c_2(t_2)\in \surfi$ be a point where $c_1$ and $c_2$ intersect transversely.
The \emph{intersection index} of $\wc_1$ and $\wc_2$ at $q$ is defined to be
\[\ind_q(\wc_1,\wc_2)=\wc_1(t_1)\cdot\kappa\cdot\wc_2^{-1}(t_2)\ \in\pi_1(\mathbb{P}T_{q}(\surf^\circ))\cong \ZZ\]
where $\kappa$ is (the homotopy class of) the path in $\mathbb{P}T_q(\surfi)$ from $\dot{c}_1(t_1)$ to $\dot{c}_2(t_2)$
given by clockwise rotation by an angle smaller than $\pi$. Sometimes, we omit $q$ if there is no confusion arising.

The notion of intersection indices can be generalized to the case $q\in\partial \surf\cup\P\cup\Y^\circ$ as follows.
Fix a small circle for $q\in\P\cup\Y^\circ$ or a small half-circle for $q\in\partial \surf$ around $q$,
denoted by $l\subset\surfi$.
Let $\alpha:[0,1]\to l$ be an embedded arc which moves clockwise around $q$,
such that $\alpha$ intersect $c_1$ and $c_2$ at $\alpha(0)$ and $\alpha(1)$, respectively.
Then $\alpha$ is unique up to a change of parametrization.
Fixing an arbitrary grading $\wa$ on $\alpha$,
the intersection index $\ind_q(\wc_1,\wc_2)$ is defined to be
\begin{equation}\label{eq:bs}
\ind_q(\wc_1,\wc_2):=\ind_{\alpha(0)}(\wc_1,\wa)-\ind_{\alpha(1)}(\wc_2,\wa).
\end{equation}
For any $\rho\in\ZZ$, we have
\begin{equation}\label{eq:shift}
    \ind_q(\wc_1,\wc_2[\rho])=\ind_q(\wc_1,\wc_2)-\rho.
\end{equation}

\begin{definition}\label{defn:int}
Let $\wc_1,\wc_2$ be two graded curves on $\surf$ in a minimal position. An intersection $q$ between $\wc_1$ and $\wc_2$ is called an \emph{oriented intersection from $\wc_1$ to $\wc_2$} if there is a small arc in $\surfi$ around $q$ from a point in $\wc_1$ to a point in $\wc_2$ clockwise. Denoted by
\begin{enumerate}
    \item $\overrightarrow{\cap}^{\rho}(\wc_1,\wc_2)$ the set of oriented intersections from $\wc_1$ to $\wc_2$ with index $\rho$;
    \item $\overrightarrow{\cap}_{\surfi}^{\rho}(\wc_1,\wc_2)$ the subset of $\overrightarrow{\cap}^{\rho}(\wc_1,\wc_2)$ consisting of oriented intersections in $\surfi$.
\end{enumerate}
\end{definition}

\begin{definition}\label{def:arc}
	An \emph{open} (resp. \emph{closed}) \emph{curve} on $\surf$ is a curve $\gamma:I\rightarrow \surf$, satisfying
	\begin{itemize}
		\item $\gamma(\partial I)\subset \M\cup\P$ (resp. $\gamma(\partial I)\subset \Y\cup\P$), and $\gamma(I\setminus\partial I)\subset\surfi$;
		\item $\gamma$ is not homotopic to a point.
	\end{itemize}
	An open/closed curve is called an open/closed \emph{arc} if $I=[0,1]$. Open/closed curves are considered up to homotopy relative to $\partial I$.
\end{definition}

\begin{definition}[Full formal arc systems]\label{def:system}
Let $\gms$ be a GMSp.
\begin{itemize}
\item[(1)] An \emph{open} (resp. \emph{closed}) \emph{arc system} $\ac$ in $\surf$ is a collection of graded open (resp. closed) arcs with $|\overrightarrow{\cap}^\rho_{\surfi}(\wc_1,\wc_2)|=0$ for any $\rho\in\ZZ$ and any $\wc_1,\wc_2\in\ac$.
\item[(2)] A \emph{full formal} \emph{open} (resp. \emph{closed}) {arc system} $\ac$ is an open (resp. closed) arc system which cuts out $\surf$ into polygons, called \emph{$\ac$-polygons}, such that each $\ac$-polygon $\PP$ contains exactly one point in $\P\cup\Y$ (resp. $\P\cup\M$), denoted by $\pt_\PP$, satisfying the following.
\begin{itemize}
\item[(I)] If $\pt_\PP\in\P$, then $\PP$ is a monogon, and $\pt_\PP$ is in the interior of $\PP$. In this case, $\PP$ is called a once-punctured monogon.
\item[(II)] If $\pt_\PP\in\partial \surf$, then $\PP$ is an $n$-gon with $n\geq 2$ (and $\pt_\PP$ is on an edge of $\PP$).
\item[(III)] If $\pt_\PP\notin (\P\cup \partial \surf)$, then $\PP$ is an $n$-gon with $n\geq 1$ and $\pt_\PP$ is in the interior of $\PP$.
\end{itemize}
An $\ac$-polygon of type (I) is called \emph{punctured} while an $\ac$-polygon of type (II) or (III) is called \emph{unpunctured}.
\end{itemize}
\end{definition}

Note that since we require $\M\subset\partial \surf$,
$\pt_\PP\notin (\P\cup \partial \surf)$ implies $\pt_\PP\in\Y^\circ$, i.e.
the type (III) can only occur in a full formal open arc system $\ac$. See Figure~\ref{fig:monogon} with the first row for $\ac$-polygons of a full formal open arc system and the second row for $\ac$-polygons of a full formal closed arc system.

\begin{figure}[htpb]\centering
	\begin{tikzpicture}[xscale=2,yscale=.3,rotate=90]
	\draw[ultra thick,white]plot [smooth,tension=1] coordinates {(170:4.5) (180:4) (190:4.5)};
	\draw[blue,thick]plot [smooth,tension=1] coordinates {(-4,0) (2,.5) (5,0) (2,-.5) (-4,0)};
	\draw[thick](0:3)node{$\bullet$}node[below]{$\pt_\PP$};
	\draw[thick,blue](180:4)node{$\bullet$};
	\end{tikzpicture}\qquad
	\begin{tikzpicture}[scale=-.4]
	\foreach \j in {1,2,0,3}{
		\draw[blue,thick](90*\j+20:4)\nn to(90*\j-20:4)\nn ;
		\draw[dashed,blue,thin](90*\j+20:4)to[bend left=-15](90*\j-20+90:4);}
	\draw[very thick](90+20:4) node[blue]{$\bullet$} to(90-20:4) node[blue]{$\bullet$};
	\draw[red] (0,3.8)\ww node[above]{$\pt_\PP$};
	\end{tikzpicture}\qquad
	\begin{tikzpicture}[scale=-.4]
	\foreach \j in {1,2,0,3}{
		\draw[blue,thick](90*\j+20:4)\nn to(90*\j-20:4)\nn ;
		\draw[dashed,blue,thin](90*\j+20:4)to[bend left=-15](90*\j-20+90:4);}
	\draw[red] (0,0)\ww(0,0)node[below]{$\pt_\PP$};
	\draw[thick,blue](90+20:4) node[blue]{$\bullet$} to(90-20:4)node[blue]{$\bullet$};
	\end{tikzpicture}
	
	\begin{tikzpicture}[xscale=2,yscale=.3,rotate=90]
	\draw[ultra thick,white]plot [smooth,tension=1] coordinates {(170:4.5) (180:4) (190:4.5)};
	\draw[red,thick]plot [smooth,tension=1] coordinates {(-4,0) (2,.5) (5,0) (2,-.5) (-4,0)};
	\draw[thick](0:3)node{$\bullet$}node[below]{$\pt_\PP$};
	\draw[thick,red](180:4)\ww;
	\end{tikzpicture}\qquad
	\begin{tikzpicture}[scale=-.4]
	\foreach \j in {1,2,0,3}{
		\draw[dashed,red,thin](90*\j+20:4)to[bend left=-15](90*\j-20+90:4);}
	\foreach \j in {1,2,0,3}{	
		\draw[red,thick](90*\j+20:4)\ww to(90*\j-20:4)\ww ;}
	\draw[blue] (0,3.8)node{$\bullet$}node[above]{$\pt_\PP$};
	\draw[very thick](90+20:4) \ww to(90-20:4) \ww;
	\end{tikzpicture}\qquad
	\caption{Types of $\ac$-polygons}\label{fig:monogon}
\end{figure}

\begin{remark}
Let $\ac$ be a full formal open/closed arc system of a GMSp $\gms$. Then
\[|\ac|=2g+b+|\Y|+|\P|-2,\]
where $g$ denotes the genus of $\surf$ and $b$ is the number of components of $\partial \surf$.
\end{remark}

\begin{lemma}\label{lem: exist-system}
    Let $\gms$ be a GMSp. Then it admits a full formal open (resp. closed) arc system.
\end{lemma}

\begin{proof}
    Since $\partial \surf\neq\emptyset$, we have $\M\neq\emptyset$. Let $M\in \M$ be an open marked point. First, we take arcs whose endpoints are $M$ such that each of them encloses exactly one puncture. Then these arcs divide $\surf$ into several once-punctured monogons and one marked surface without punctures. Then by \cite[Lem.~3.3]{HKK}, the lemma follows.
\end{proof}

For each full formal open arc system $\ac=\{\wg_1,\wg_2,\cdots,\wg_n\}$ of $\surf$,
denote by $\ac^\ast=\{\we_1,\we_2,\cdots,\we_n\}$ the \emph{dual} of $\ac$,
which is a full formal closed arc system consisting of the graded closed arcs $\we_i$ satisfying
$$|\overrightarrow{\cap}^{\rho}_{\surfi}(\wg_j,\we_i)|=\begin{cases}
0&\text{if $j\neq i$ or $\rho\neq 0$,}\\
1&\text{if $j=i$, $\rho=0$ and $\wg_i$ does not enclose a puncture,}\\
2&\text{if $j=i$, $\rho=0$ and $\wg_i$ does enclose a puncture.}
\end{cases}$$
Indeed, the map $\ac\mapsto\dac$ gives a bijection between full formal open arc systems and full formal closed arc systems of $\gms$. See Figure~\ref{fig:full. formal arc system} for an example of a full formal open arc system and its dual.

\begin{figure}[htpb]\centering
		\begin{tikzpicture}[scale=.8]
		\draw[ultra thick](0,0) circle (3.5);
		\draw[ultra thick](0,0) circle (1);
		
		\draw(-1.25,2.25)\nn (1.25,2.25)\nn (-1,0)\ww (1,0)\ww (0,-1.5)\ww (0,3.5)\ww (0,1)[blue]\nn (0,-1)[blue]\nn (0,-3.5)[blue]\nn;
		
		\draw[blue,thick](0,1)to[out=160,in=-135](-1.5,2.5)to[out=45,in=110](0,1)
		(0,1)to[out=20,in=-45](1.5,2.5)to[out=135,in=70](0,1);
		
		\draw[blue,thick](0,-3.5)to[out=140,in=-90](-2.7,0)to[out=90,in=160](0,1);
		
		\draw[blue,thick](0,-1)to[out=-10,in=-90](1.5,0)to[out=90,in=10](0,1)
		(0,1)to[out=170,in=90](-1.5,0)to[out=-90,in=-170](0,-1);
		
		\draw[blue,thick](0,1)to[out=15,in=90](2,0)to[out=-90,in=0](0,-2)to[out=180,in=-90](-2,0)to[out=90,in=165](0,1);
		\begin{scope}[shift={(8,0)}]
			\draw[ultra thick](0,0) circle (3.5);
			\draw[ultra thick](0,0) circle (1);			
			\draw[red,thick](0,3.5)to[out=-70,in=-135](1.5,2)to[out=45,in=-20](0,3.5) (0,3.5)to[out=-110,in=-45](-1.5,2)to[out=135,in=-160](0,3.5);
			\draw[red,thick](-1,0)to[out=-100,in=160](0,-2)to[out=20,in=-80](1,0);
			\draw[red,thick](0,-2)to[out=0,in=-90](2.5,.75)to[out=90,in=-10](0,3.5);
			\draw[red,thick](0,3.5)to[out=-5,in=90](3,.25)to[out=-90,in=0](0,-3)to[out=180,in=-90](-3,.25)to[out=90,in=-175](0,3.5);
		\draw(-1.25,2.25)\nn (1.25,2.25)\nn (-1,0)\ww (1,0)\ww (0,-2)\ww (0,3.5)\ww (0,1)[blue]\nn (0,-1)[blue]\nn (0,-3.5)[blue]\nn;
		\end{scope}
	\end{tikzpicture}
	\caption{A full formal open arc system $\ac$ and its dual $\ac^\ast$.}
	\label{fig:full. formal arc system}
\end{figure}

\subsection{Graded skew-gentle algebras from full formal arc systems}\label{subsec:sys-to-alg}
Let $\gms$ be a GMSp and $\dac=\{\we_1,\we_2,\cdots,\we_n\}$ a full formal closed arc system. Denote by
\begin{itemize}
    \item $\DBX$ the set of all $\dac$-polygons,
    \item  $\DBY=\{\PP_1,\cdots,\PP_t\}$ the set of unpunctured $\dac$-polygons (i.e. of type (II)),
    \item $\DBZ=\{\Pp_P\mid P\in\P\}$ the set of punctured $\dac$-polygons (i.e. of type (I)).
\end{itemize}

\begin{construction}\label{cons:datum}
We associate a datum $\datum$ to $\dac$ as follows.
\begin{itemize}
\item $\overrightarrow{m}=(m_1,\cdots,m_t)$, where $m_i+1$ is the number of edges of $\PP_i\in\DBY$ (where exactly one edge of $\PP_i$ is not a closed arc in $\dac$).
\item Label the edges of $\PP_i\in\DBY$ by $(i,0),\cdots,(i,m_i)$ respectively in the anticlockwise order such that $(i,0)$ is a segment of $\partial\surf$, or equivalently, $(i,0)$ contains the open marked point $\pt_{\PP_i}$, cf. Figure~\ref{fig:order2}.
\item The symmetric relation $\simeq$ on $\OME$ is given by
\begin{itemize}
\item $(i_1,j_1)\simeq (i_2,j_2)$ for $(i_1,j_1)\neq(i_2,j_2)$,
if the edge $(i_1,j_1)$ of $\PP_{i_1}$ and the edge $(i_2,j_2)$ of $\PP_{i_2}$ coincide in $\ac^\ast$
(note that $i_1$ may coincide with $i_2$);
\item $(i,j)\simeq(i,j)$, if the edge $(i,j)$ coincides with the edge of a punctured $\dac$-polygon in $\DBZ$;
\end{itemize}
\item $\mathbf{gr}=(\Gr_1,\Gr_2,\cdots,\Gr_t)$,  $\Gr_i=(\chi_{1,2}^{(i)},\cdots,\chi_{m_i-1,m_i}^{(i)})\in\ZZ^{m_i-1}, 1\leq i\leq t,$ with
$$\chi^{(i)}_{j,j+1}=1-\ind((i,j),(i,j+1)),$$
for any $1\leq j<m_i$.
\end{itemize}
By definition, this datum $\datum$ is a graded skew-gentle datum. The set $\OMEs$ can be identified with $\ac^\ast$. The set $\fc{\OME}$, when regarded as a subset of $\OME$, can be identified with the set of edges of punctured $\dac$-polygons; and when regarded as a subset of $\OMEs$, can be identified with the set of self-folded loops in $\dac$ (i.e. the loops in $\dac$ which encloses a puncture).
\end{construction}

\begin{example}\label{ex:datum}
	The graded skew-gentle datum associated to the full formal closed arc system shown in Figure~\ref{fig:full. formal arc system} (cf. Figure~\ref{fig:grading} for the local grading) is the one in Example~\ref{ex:1}, see Figure~\ref{fig:datum}, where $\DBY=\{\PP_1,\PP_2,\PP_3\}$ and $\DBZ=\{E_{p_1},E_{p_2}\}$.
\begin{figure}[htpb]
\begin{tikzpicture}[scale=1,font=\tiny]
    \draw[ultra thick](0,0) circle (3.5);
    \draw[ultra thick](0,0) circle (1);    		
    \draw[red,thick](0,3.5)to[out=-70,in=-135](1.5,2)to[out=45,in=-20](0,3.5)
        (0,3.5)to[out=-110,in=-45](-1.5,2)to[out=135,in=-160](0,3.5);	
    \draw[red,thick](-1,0)to[out=-100,in=160](0,-2)to[out=20,in=-80](1,0);
    \draw[red,thick](0,-2)to[out=0,in=-90](2.5,.75)to[out=90,in=-10](0,3.5);
    \draw[red,thick](0,3.5)to[out=-5,in=90](3,.25)to[out=-90,in=0](0,-3)to[out=180,in=-90](-3,.25)to[out=90,in=-175](0,3.5);
    \draw(1,2.2)node{$p_1$}(-1,2.2)node{$p_2$}(.8,2.7)node{$E_{p_1}$}(-.8,2.7)node{$E_{p_2}$}
        (0,.7)node{$(1,0)$}(0,-.7)node{$(2,0)$}(0,-1.5)node{$D_2$}(-2,0)node{$D_1$}(0,-3.3)node{$D_3$}(0,-3.7)node{$(3,0)$};
    \draw[font=\tiny]
        (1.2,1.6)node[red]{$(1,3)\simeq(1,3)$}
        (-1.2,1.6)node[red]{$(1,4)\simeq(1,4)$}
        (-1.1,-1.2)node[rotate=-68][red]{$(1,7)\simeq(2,1)$}
        (1.1,-1.2)node[red,rotate=68]{$(2,2)\simeq(1,1)$}
        (2.2,.5)node[red,rotate=83]{$(1,2)\simeq(1,6)$}
        (0,-2.7)node[red]{$(1,5)\simeq(3,1)$};
    \draw(-1.25,2.25)\nn (1.25,2.25)\nn (-1,0)\ww (1,0)\ww (0,-2)\ww (0,3.5)\ww (0,1)[blue]\nn (0,-1)[blue]\nn (0,-3.5)[blue]\nn;
\end{tikzpicture}
	\caption{A full formal closed arc system $\dac$ and the corresponding datum}\label{fig:datum}
	\end{figure}
\end{example}

\begin{theorem}\label{thm:A}
    Any graded skew-gentle datum $\datum$ is obtained in the way of Construction~\ref{cons:datum}.
\end{theorem}

\begin{proof}
Let $\datum$ be a graded skew-gentle datum with $\overrightarrow{m}=(m_1,\cdots,m_t)$ and  $\mathbf{gr}=(\Gr_1,\Gr_2,\cdots,\Gr_t)$, where $\Gr_i=(\chi_{1,2}^{(i)},\cdots,\chi_{m_i-1,m_i}^{(i)})\in\ZZ^{m_i-1}, 1\leq i\leq t$. We take $t$ oriented polygons $\PP_1, \PP_2, \cdots, \PP_t$,  where $\PP_i$ is an $(m_i+1)$-gon. Then we add one open marked point on one edge, labeled by $(m_i,0)$, of $\PP_i$, and label the rest edges of $\PP_i$ by $(i,1), (i,2), \cdots, (i, m_i)$ anticlockwise. Moreover, we define a line field on $D_i$ according to the degree $\Gr_i=(\chi_{1,2}^{(i)},\cdots,\chi_{m_i-1,m_i}^{(i)})$, such that the intersection index of the edges $(i, j-1)$ and $(i, j)$ is precisely $1-\chi_{j-1, j}^{(i)}$.
This can be done (for $\PP=\PP_i$) as follows:
\begin{itemize}
    \item First take a line field induced by a foliation that connects $\pt_\PP$ to every point in $\partial \PP$,
    such that it is orthogonal to the edges (but not at the vertices), cf. the left picture of Figure~\ref{fig:grading}.
    \item Then one can expend a line segment connecting $\pt_\PP$ and a closed marked point
    into many copies of the folded half plane, as shown in the right picture of Figure~\ref{fig:grading}, where each half plane changes the intersection index by $\pm 1$.
    This adjusts the intersection indices between the corresponding edges into the desired ones.
\end{itemize}

\begin{figure}[ht]\centering
\begin{tikzpicture}[yscale=1.4,xscale=1.6,scale=.9]
\draw[Emerald,thick]
  (1,1.7)to[bend left=25](2.5,1)
  (1,2.8).. controls +(0:.6) and +(130:.3) ..(2.5,1)
  (4,1.7)to[bend left=-25](2.5,1)
  (4,2.8).. controls +(0:-.6) and +(70:.3) ..(2.5,1)
  (1,4)to[bend left=10](2.5,1)(4,4)to[bend left=-10](2.5,1)
  (2,4).. controls +(-90:.5) and +(95:.5) ..(2.5,1)
  (3,4).. controls +(-90:.5) and +(85:.5) ..(2.5,1);
\draw[thick](1,1)to(4,1);
\draw[red,thick](1,1)\ww rectangle(4,4)\ww (1,4)\ww (4,1)\ww;
\draw[blue](2.5,1)\nn;
\draw[blue](2.5,1)node[below]{$\pt_D$};
\end{tikzpicture}
\begin{tikzpicture}[yscale=2,xscale=1]
\draw[ultra thick, Emerald] plot [smooth,tension=1] coordinates
    {(0,-1) (180:.8) (90:.8) (0:.8)  (0,-1)};
\foreach \j in {.5,.6,.7}
{\draw[Emerald,thick] plot [smooth,tension=1] coordinates
    {(0,-1) (180:\j) (90:\j) (0:\j)  (0,-1)};}
\draw[Emerald,thick] (0,-1)
    .. controls +(55:2.3) and +(125:2.3) .. (0,-1)
    .. controls +(63:1.9) and +(117:1.9) .. (0,-1)
    .. controls +(73:1.6) and +(107:1.6) .. (0,-1);
\draw[red,thick]
    (0,-1).. controls +(18:2.4) and +(-15:.5) .. (0,1.2)
    (0,-1).. controls +(162:2.4) and +(195:.5) .. (0,1.2);
\draw[](0,-1)\nn;\draw(0,1.2)\nn;
\end{tikzpicture}
\begin{tikzpicture}[xscale=1.8,yscale=2]
\clip plot [smooth,tension=1] coordinates
    {(0,-1.03) (180:.8) (90:.8) (0:.8)  (0,-1.03)};
\begin{scope}[shift={(0,0)}]
\foreach \j in {1,...,36}
    {\draw[thick,Emerald] (10*\j:1)to(10*\j:-1);}
\end{scope}
\draw[white,fill=white](2,0)rectangle(-2,-1);
\foreach \j in {0,...,-10}
    {\draw[thick,Emerald] (1,.09*\j)to(-1,.09*\j);}
\draw[ultra thick, red] plot [smooth,tension=1] coordinates
    {(0,-1) (180:.8) (90:.8) (0:.8)  (0,-1)};
\draw[](0,0)\nn(0,-1)\ww;
\end{tikzpicture}
\caption{Local grading}\label{fig:grading}
\end{figure}

Now we glue the edges according to the equivalence relation $\simeq$ in the graded skew-gentle datum. More precisely, for the case
$(i,j)\simeq (i',j')\notin\fc{\OME}$, we glue two edges $(i,j)$ and $(i',j')$ together; while for the case $(i,j)\simeq (i,j)$, we glue two endpoints of the edge $(i, j)$ to get a monogon and add a puncture $P$ in its interior, where the line field is given as shown in the right picture of Figure~\ref{fig:grading}. Note that such a grading ensures \eqref{eq:grad=1}. Thus we get a GMSp $\gms$ with an initial full formal closed arc system $\dac$, whose associated graded skew-gentle datum given in Construction \ref{cons:datum} is exactly $\datum$.
\end{proof}

\begin{remark}
Let $\dac$ be a full formal closed arc system of a GMSp and $\datum$ be the associated graded skew-gentle datum in Construction~\ref{cons:datum}. Then the graded skew-gentle triple $(Q,\operatorname{Sp},I)$ associated to $\datum$ in Remark~\ref{rmk:datum to triple} can be obtained from the dual $\ac$ of $\dac$ in a direct way.
\begin{itemize}
\item $Q_0=\{1,2,\cdots,n\}$ indexed by the open arcs $\wg_1,\cdots,\wg_n$ in $\ac$.
\item There is an arrow $i\to j$ in $Q_1$ with degree $i_{p}(\wg_j,\wg_i)$ (cf. Figure~\ref{fig:p-s} and the implication there)
whenever there is an unpunctured $\ac$-polygon having $\wg_i$ and $\wg_j$ as consecutive edges meeting at $p$, with $\wg_j$ following $\wg_i$ in the clockwise order.
\item $\operatorname{Sp}$ consists of vertices $i$ of $Q$ such that $\wg_i$ is the edge of a punctured $\ac$-polygon.
\item $I$ consists of $a_1a_2$ for $a_1:i\to j, a_2:j\to l$, such that
$i,j,l$ are consecutive edges in an unpunctured $\ac$-polygon.
\end{itemize}
It is worth mentioning that, the grading near a puncture as shown in the right picture of Figure~\ref{fig:grading}
ensures that the two ends of the red loop there have intersection index 0
at the closed marked point,
which is also the degree of the corresponding loop in $\operatorname{Sp}$.
\end{remark}

\begin{figure}[htpb]
	\begin{tikzpicture}[scale=.4]
	\draw[blue,thick] (-4,-2)to(0,2)\nn to(4,-2);
	\draw[red,thick] (-4,2)to(0,-2)\ww to(4,2);
	\draw[blue] (-1.25,.75)node[above]{$\wg_i$} (1.25,.75)node[above]{$\wg_j$};
	\draw[red] (-1.25,-.75)node[below]{$\we_i$} (1.25,-.75)node[below]{$\we_j$};
	\draw (0,2)node[above]{$a$} (0,-2)node[below]{$b$} (-2,0)node[left]{$c$} (2,0)node[right]{$d$};
	\draw (0,-3.5)node{$i_c(\wg_i,\we_i)=0=i_d(\wg_j,\we_j)\implies i_b(\we_i,\we_j)=1-i_a(\wg_j,\wg_i)$}; 
	\end{tikzpicture}
	\caption{Intersection indices}\label{fig:p-s}
\end{figure}

\subsection{Arc Segments}\label{sec:AS}

For later use (e.g., Section~\ref{sec:bush}),
we describe some notions and notations concerning arc segments.
\begin{definition}\label{def:seg}
    Let $\dac$ be a full formal closed arc system of a GMSp $\gms$. An \emph{arc segment} of $\dac$ is a curve on $\surf$ satisfying that
    \begin{itemize}
        \item it lives in an $\dac$-polygon, i.e. its interior does not meet any arc in $\dac$;
        \item it has no self-intersections;
        \item its endpoints are in $\M\cup\P$ or interiors of arcs in $\dac$;
        \item it is not homotopic to a point relative to $\M\cup\P$ and the interiors of arcs in $\dac$.
    \end{itemize}
    Moreover,
    \begin{enumerate}
        \item An arc segment of $\dac$ is called \emph{interior} if its both endpoints are in the interiors of arcs in $\dac$.
        \item An orientation of an interior arc segment $\mu$ of $\dac$ is called \emph{positive} (resp. \emph{negative}) if $\pt_\PP$ is on its left (resp. right) hand side, where $\PP$ is the $\dac$-polygon containing $\mu$. See Figure~\ref{fig:st}. An orientation of a non-interior arc segment of $\dac$ is called \emph{positive} (resp. \emph{negative}) if it starts (resp. ends) at $\M\cup\P$.
        \item An arc segment of $\dac$ is called \emph{unpunctured} (resp. \emph{punctured}) if it lives in an unpunctured (resp. punctured) $\dac$-polygon.
	\end{enumerate}
\end{definition}

\begin{figure}[htpb]\centering
	\begin{tikzpicture}[xscale=2.5,yscale=.333,rotate=-90]
		\draw[thick, blue,-<-=.5,>=stealth](-2,-.25)tonode[below]{$+$}(-2,.25);
		\draw[ultra thick,white]plot [smooth,tension=1] coordinates {(170:4.5) (180:4) (190:4.5)};
		\draw[red,thick]plot [smooth,tension=1] coordinates {(-4,0) (2,.5) (5,0) (2,-.5) (-4,0)};
		\draw[thick](0:3)node{$\bullet$}node[below]{$\pt_\PP$};
		\draw[thick,red](180:4)node{$\bullet$};
	\end{tikzpicture}\qquad
	\begin{tikzpicture}[scale=-.4]
		\draw[thick, blue,-<-=.5,>=stealth](-5:3.8)to[bend left=45]node[below]{$+$}(185:3.8);
		\foreach \j in {1,2,0,3}{
			\draw[red,thick](90*\j+20:4)\nn to(90*\j-20:4)\nn ;
			\draw[dashed,red,thin](90*\j+20:4)to[bend left=-15](90*\j-20+90:4);}
		\draw[very thick](90+20:4) node[red]{$\bullet$} to(90-20:4) node[red]{$\bullet$};
		\draw[blue] (0,3.8)\nn (0,3.2)node[above]{$\pt_\PP$};
	\end{tikzpicture}\qquad
	\caption{Positive arc segments w.r.t. a full formal closed arc system}\label{fig:st}
\end{figure}

\linespread{1.5}

\begin{notations}\label{not:set}
    We shall use the following notations for certain sets of arc segments of $\dac$:
    \begin{center}
    \begin{tabular}{|c|l|}
    \hline $\as(\dac)$ & the set of all arc segments of $\dac$ \\\hline
    $\upas(\dac)$ & the set of unpunctured arc segments of $\dac$\\\hline
    $\wupas(\dac)$ & the set of graded unpunctured arc segments of $\dac$\\\hline
    $\owupas(\dac)$ & the set of oriented graded unpunctured arc segments of $\dac$\\\hline
    $\cias(\dac)$ & the set of punctured arc segments of $\dac$\\\hline
    $\icias(\dac)$ & the set of interior punctured arc segments of $\dac$\\\hline
    $\wicias(\dac)$ & the set of graded interior punctured arc segments of $\dac$\\\hline
    $\owicias(\dac)$ & the set of oriented graded interior punctured arc segments of $\dac$\\\hline
    $\as(\PP)$ & the set of arc segments living in an $\dac$-polygon $\PP$ \\\hline
    $\Was(\PP)$ & the set of graded arc segments living in an $\dac$-polygon $\PP$ \\\hline
    $\oWas(\PP)$ & the set of oriented graded arc segments living in an $\dac$-polygon $\PP$ \\\hline
    \end{tabular}
    \end{center}
\end{notations}

Recall that each $\PP_i\in\DBY$ has $m_i+1$ edges labeled by $(i,0),(i,1),\cdots,(i,m_i)$ in the anticlockwise order such that $(i,0)$ is the edge containing the open marked point $\pt_{\PP_i}$. We fix an arbitrary grading for each $(i,0)$. Since $\M=\{\pt_{\PP_i}\mid \PP_i\in\DBY \}$, one can identify $\M$ with $\{(i,0)\mid 1\leq i\leq t\}$. Using this, we have
\begin{equation}\label{eq:omem}
    \Sigma:=\OME\cup\M=\{(i,j)\mid 1\leq i\leq t,\ 0\leq j\leq m_i \}.
\end{equation}

\begin{notations}\label{not:as}
    For any unpunctured $\dac$-polygon $\PP_i\in\DBY$, any arc segment in $\as(\PP_i)$ is determined by the edges of $\PP_i$ where its endpoints are. We denote by
    \begin{itemize}
        \item $(i,j_1)-(i,j_2)$ the arc segment in $\as(\PP_i)$ connecting $(i,j_1)$ and $(i,j_2)$, and
	    \item $(i,j_1)\rightarrow(i,j_2)$ the orientation of $(i,j_1)-(i,j_2)$ from $(i,j_1)$ to $(i,j_2)$.
    \end{itemize}
\end{notations}

By definition, an unpunctured arc segment $(i,j_1)-(i,j_2)$ is interior if and only if both $j_1$ and $j_2$ are not $0$; an orientation $(i,j_1)\rightarrow(i,j_2)$ is positive if and only if $j_1<j_2$.

\begin{lemma}\label{lem:r1r2}
For any arc segment $\mu=(i,j_1)-(i,j_2)\in\as(\PP_i)$ with a grading,
let $r_1=\ind(\mu,(i,j_1))$ and $r_2=\ind(\mu,(i,j_2))$. If $0<j_1<j_2$, we have $$r_1-r_2=\sum_{j=j_1}^{j_2-1}\chi^{(i)}_{j,j+1}-1.$$
\end{lemma}

\begin{proof}
By Construction~\ref{cons:datum}, $\chi^{(i)}_{j,j+1}=1-\ind((i,j),(i,j+1))$. For $j_2=j_1+1$, the formula follows directly from the definition of intersection index, see \eqref{eq:bs}. Then the general case follows from the fact that the arc segment
$(i,j_1)-(i,j_2)$ and the edges $(i,j_1),(i,j_1+1),\cdots,(i,j_2)$ of $\PP_i$ form a contractible polygon.
\end{proof}

\begin{notations}\label{not:grad}
    A grading of an arc segment $\mu=(i,j_1)-(i,j_2)$ is denoted by $$\wmu=(i,j_1){}^{r_1}-{}^{r_2}(i,j_2),$$
    where $r_1=\ind(\wmu,(i,j_1))$ and $r_2=\ind(\wmu,(i,j_2))$.
Since $r_1$ and $r_2$ determine each other (e.g. see  Lemma~\ref{lem:r1r2}),
we also use $(i,j_1){}^{r_1}-(i,j_2)$ or $(i,j_1)-{}^{r_2}(i,j_2)$ to denote it. For any integer $r\in\ZZ$, denote by
$$\wmu[r]=(i,j_1){}^{r_1+r}-{}^{r_2+r}(i,j_2).$$
\end{notations}

\begin{example}\label{ex:as}
We continue Example~\ref{ex:datum}. Let $\mu_1=(2,0)-^{0}(2,2)$, $\mu_2=(1,1)^{0}-^{2}(1,3)$,  $\mu_3=(1,3)^{2}-^{3}(1,4)$, $\mu_4=(1,4)^{3}-^{5}(1,7)$, $\mu_5=(2,0)-^{5}(2,1)$, $\mu_6=(1,2)^{0}-^{3}(1,6)$, $\mu_7=(1,1)^{2}-^{3}(1,2)$, $\mu_8=(2,1)^{1}-^{2}(2,2)$ and $\mu_9=(1,6)^{0}-^{1}(1,7)$, see Figure~\ref{fig:as}. We denote by $\overrightarrow{\mu_i}$ the positive orientation of $\mu_i$ and $\overleftarrow{\mu_i}$ the negative one, e.g. $\overrightarrow{\mu_2}=(1,1)^{0}\to^{2}(1,3)$ and $\overleftarrow{\mu_2}=(1,3)^{2}\to^{0}(1,1)$.

\begin{figure}[htpb]
\begin{tikzpicture}[scale=1,font=\tiny]
\draw[ultra thick](0,0) circle (3.5);
\draw[ultra thick](0,0) circle (1);
\draw[red,thick](0,3.5)to[out=-70,in=-135](1.5,2)to[out=45,in=-20](0,3.5)
    (0,3.5)to[out=-110,in=-45](-1.5,2)to[out=135,in=-160](0,3.5);
\draw[red,thick](-1,0)to[out=-100,in=160](0,-2)to[out=20,in=-80](1,0);
\draw[red,thick](0,-2)to[out=0,in=-90](2.5,.75)to[out=90,in=-10](0,3.5);
\draw[red,thick](0,3.5)to[out=-5,in=90](3,.25)to[out=-90,in=0](0,-3)to[out=180,in=-90](-3,.25)to[out=90,in=-175](0,3.5);
\draw(-1.2,2.15)node[red]{$(1,4)$}(1.2,2.15)node[red]{$(1,3)$};
\draw[blue,thick,bend right=20](-.7,2.2)to(.7,2.2) (-1.5,2)to(-.95,-1) (.95,-1)to(1.5,2) (-.7,-1.5)to(0,-1)to(.7,-1.5) (.4,-1.8)to(-.4,-1.8) (1,-1.8)to(.8,-1.3);
\draw[blue,thick](1.9,-1)to[out=-100,in=0](0,-2.5)to[out=180,in=-90](-2.5,-.3)to[out=90,in=180](0,1.7)to[out=0,in=130](2.5,.5);
\draw[blue,thick](-.3,-1.88)to[out=-90,in=160](-.1,-2.2)to[out=-20,in=-90](.6,-1.93);
\draw[blue](0,2.3)node{$\mu_3$}(0,-1.6)node{$\mu_8$}(1.6,0)node{$\mu_2$}(-1.6,0)node{$\mu_4$}(-.5,-1.2)node{$\mu_5$}(.5,-1.2)node{$\mu_1$}(1.15,-1.45)node{$\mu_7$}(-2.7,-.3)node{$\mu_6$}(-.5,-2.05)node{$\mu_9$};
\draw(-1.15,2.45)\nn (1.15,2.45)\nn (-1,0)\ww (1,0)\ww (0,-2)\ww (0,3.5)\ww
    (0,1)[blue]\nn (0,-1)[blue]\nn (0,-3.5)[blue]\nn;
\end{tikzpicture}
	\caption{Arc segments}\label{fig:as}
	\end{figure}
\end{example}

\section{A bush and its representations}\label{sec:bush}

Let $\dac$ be a full formal closed arc system of a GMSp $\gms$. In this section, we associate a bush $S=S(\dac)$ to the full formal closed arc system $\dac$ and classify its representations by applying results in \cite{Deng}.

\subsection{A bush from a full formal closed arc system}\label{subsec:segment}

\begin{construction}\label{cons:bush}
    The bush $S=S(\dac)$ (see \cite[\S~1.3]{Deng} for the definition of a bush) is defined as follows:
	\begin{itemize}
		\item The index set $\widetilde{\Sigma}=\Sigma\times\ZZ=\{((i,j),r)\mid 1\leq i\leq t,0\leq j\leq m_i,r\in\ZZ\}$ (see \eqref{eq:omem} for the notation $\Sigma=\M\cup\OME$).
		\item For any $((i,j),r)\in\widetilde{\Sigma}$, the rod (i.e. a finite ordered set)
		$$S_{(i,j),r}^+=\begin{cases}
		\OMEpm(i,j)\times\{r\}&\text{if $(i,j)\in\OME$, i.e. $j\neq 0$,}\\
		\emptyset&\text{if $(i,j)\in\M$, i.e. $j=0$,}
		\end{cases}$$
		with a trivial order. The union $$S^+:=\bigcup_{((i,j),r)\in\widetilde{\Sigma}}S_{(i,j),r}^+=\OMEpm\times\ZZ$$ admits an equivalence relation $\sim$ induced by $((i_1,j_1),r)\sim((i_2,j_2),r)$ for any  $(i_1,j_1)\simeq(i_2,j_2)$ and $r\in\ZZ$.
		\item For any $((i,j),r)\in\widetilde{\Sigma}$, the rod
		$$S_{(i,j),r}^-=\{(i,j){}^{r}\rightarrow (i,j')\mid 0\leq j'\leq m_i,j'\neq j\}\subset\oWas(\PP_i)$$
		with an order given by
		$$(i,j){}^{r}\rightarrow (i,j')<(i,j){}^{r}\rightarrow (i,j'')$$
		if and only if one of $j''<j'<j$, $j<j''<j'$ and $j'<j<j''$ holds. That is, $S_{(i,j),r}^-$ consists of the oriented graded arc segments living in $\PP_i$ and starts at (a point in) the edge $(i,j)$ with intersection index $r$,  with the order given by left smaller than right, see Figure~\ref{fig:order2}. The union $$S^-:=\bigcup_{((i,j),r)\in\widetilde{\Sigma}}S_{((i,j),r)}^-=\owupas(\dac)$$
		admits an equivalence relation $\sim$ given by $(i,j){}^{r}\rightarrow (i,j')\sim(i,j')\rightarrow {}^{r}(i,j)$ (cf. Notations~\ref{not:seg}), i.e. two oriented graded arc segments with the same underlying graded arc segment are equivalent.
		\item $S=S^+\cup S^-$.
	\end{itemize}
\end{construction}

\begin{figure}[hptb]
	\begin{tikzpicture}
    \clip (-4,-3.5) rectangle (4,3.5);
	\draw[blue,thick,bend right,->-=.5,>=stealth] (-.4,-2.9)to(-150:2.9);
	\draw[blue,thick,bend right,->-=.5,>=stealth] (-.2,-2.9)to(-210:2.9);
	\draw[blue,thick,->-=.5,>=stealth] (0,-2.9)to(0,2.9);
	\draw[blue,thick,bend left,->-=.5,>=stealth] (.2,-2.9)to(30:2.9);
	\draw[blue,thick,bend left,->-=.5,>=stealth] (.4,-2.9)to(-30:2.9);
	\draw (-1,-1.5)node{$<$} (1,-1.5)node{$<$} (-.3,-.5)node{$<$} (.3,-.5)node{$<$};
	\draw[red] (0,-3)node[below]{$(i,j)$};

	\foreach \j in {0,...,5}
			{\draw[red,dashed] (60*\j-15:3)to(60*\j+15:3);
			\draw[red] (60*\j-45:3)to(60*\j-15:3);}
	\draw[black,ultra thick] (105:3)to(75:3);
	\draw[blue] (90:2.88)node{$\bullet$} (90:3.2)node{$(i,0)$};
	\draw[red] (60:3.3)node{$(i,m_i)$} (120:3.3)node{$(i,1)$};
	\foreach \j in {0,...,11}
	{\draw[red] (30*\j+15:3)\ww;}
	\end{tikzpicture}
	\caption{The set $S^-_{(i,j),r}$ with order: left $<$ right}\label{fig:order2}
\end{figure}

The bush $S=S(\dac)$ is the only bush we will consider in this paper. We shall need the following notations  from \cite[\S~2.1]{Deng}.

\begin{notations}\label{not:wedge}
    For any $s\in S$, we set
    $$s^\sim   =  \begin{cases}
	(i,j')\rightarrow {}^{r}(i,j)&\text{if $s=(i,j){}^{r}\rightarrow (i,j')\in S^-$}\\
	((i_2,j_2),r)&\text{if $s=((i_1,j_1),r)\in S^+$ with $(i_1,j_1)\simeq(i_2,j_2)\notin\fc{\OME}$,}\\
	s&\text{otherwise,}
	\end{cases} $$$$
      s^\wedge  =  \begin{cases}
	((i,j^-),r)&\text{if $s=((i,j^+),r)\in S^+$,}\\
	((i,j^+),r)&\text{if $s=((i,j^-),r)\in S^+$,}\\
	s&\text{otherwise,}
	\end{cases} $$$$
	s^\ast=(s^\sim)^\wedge  =  \begin{cases}
	((i_2,j_2),r)&\text{if $s=((i_1,j_1),r)\in S^+$ with $(i_1,j_1)\simeq (i_2,j_2)\notin\fc{\OME}$,}\\
	((i,j^{-}),r)&\text{if $s=((i,j^+),r)\in S^+$,}\\
	((i,j^{+}),r)&\text{if $s=((i,j^-),r)\in S^+$,}\\
	(i,j')\rightarrow{}^{r} (i,j)&\text{if $s=(i,j){}^{r}\rightarrow (i,j')\in S^-$.}
	\end{cases}
    $$
	In particular, the bush $S$ is complete in the sense that $s^\ast\neq s$ for any $s\in S$.
\end{notations}

\begin{example}
    In Example~\ref{ex:as}, we have $\overrightarrow{\mu_2}^\ast=\overleftarrow{\mu_2}$.
\end{example}

\subsection{Representations of the bush}

The category $\k S$ \cite[\S~1.3]{Deng} associated to $S$ is a $\k$-linear category with
\begin{itemize}
	\item $\mbox{Obj}(\k S)=S$,
	\item $\k S(s_1,s_2)=\begin{cases}\k(s_2\ge s_1)& \text{if $s_2\geq s_1$,}\\ 0&\text{otherwise,}\end{cases}$
    \item the composition is such that
    $$(s_3\geq s_2')\circ(s_2\ge s_1)=\begin{cases}(s_3\geq s_1)&\text{if $s_2=s_2'$,}\\0&\text{otherwise.}\end{cases}$$
\end{itemize}
For any rod $R$ of $S$, there is a module $\vs_R$ over $\k S$ given by
\begin{itemize}
	\item $\vs_R(s)=0$ if $s\notin R$, and
	\item $\vs_R(s_1)=\vs_R(s_2)=\k$, $\vs_R(s_2\ge s_1)=\id_\k$ if $s_1,s_2\in R$ and $s_2\geq s_1$.
\end{itemize}
To simplify the notations, for any $({(i,j),r})\in\widetilde{\Sigma}$, we denote $$\vs^-_{(i,j),r}=\vs_{S_{(i,j),r}^-},\ \vs^+_{(i,j),r}=\vs_{S_{(i,j),r}^+}.$$

Denote by $\overline{S^+}$ and $\overline{S^-}$ the sets of equivalence classes of $S^+$ and $S^-$, respectively. Then we have the identifications:
\begin{equation}\label{eq:id}
    \overline{S^+}=\OMEx\times\ZZ,\ \overline{S^-}=\wupas(\dac).
\end{equation}

Let $\overline{S}=\overline{S^+}\cup\overline{S^-}$. The category $\bush$ \cite{Deng} associated to $\overline{S}$ is a $\k$-linear category with
\begin{itemize}
	\item $\mbox{Obj}(\bush)=\overline{S}=\left(\OMEx\times\ZZ\right)\cup\wupas(\dac)$,
	\item for any $a_1,a_2\in\overline{S}$, the space of radical morphisms is $$\operatorname{rad}_{\bush}(a_1,a_2)=\bigoplus_{s_1\in a_1,s_2\in a_2,s_2>s_1}\k(s_2\ge s_1),$$
	\item the composition is such that
	$$(s_3\ge s_2')\circ(s_2\ge s_1)=\begin{cases}
	(s_3\ge  s_1)&\text{if }s_2'=s_2,\\
	0&\text{otherwise.}
	\end{cases}$$
\end{itemize}

Let $\bushz$ and $\bushf$ be the categories associated to $\overline{S}^+$ and $\overline{S}^-$, respectively. Then $\bushz$ and $\bushf$ are full subcategories of $\bush$ and there are no nonzero morphisms between $\bushz$ and $\bushf$.

Denote by $\hull$ the additive hull of $\bush$. For any rod $R$ of $S$, there is a module $\bvs_R$ over $\hull$ with
\begin{itemize}
\item $\bvs_R(a)=\bigoplus_{s\in a}\vs_R(s)$ for each $a\in\overline{S}$;
\item for a morphism $(s_2\ge s_1)$ in $\bush$, the action of $\bvs_R(s_2\ge s_1)$ on the direct summand $\vs_R(s)$ of $\bvs_R(a)$ coincides with that of $\vs_R(s_2\ge s_1)$ if $s=s_1$ or else is 0.
\end{itemize}
To simplify the notations, for any $((i,j),r)\in\widetilde{\Sigma}$, we denote $$\bvs^-_{(i,j),r}=\bvs_{S_{(i,j),r}^-},\ \bvs^+_{(i,j),r}=\bvs_{S_{(i,j),r}^+}.$$
For any $((i,0),r)\in\M\times\ZZ\subset\widetilde{\Sigma}$, since $S^+_{(i,0),r}=\emptyset$, we have $\bvs^+_{(i,0),r}=0$ and hence there is no nonzero map from $\bvs^-_{(i,0),r}(X)$ to $\bvs^+_{(i,0),r}(X)$ for any $X\in\add\bush$. So we only care about $\bvs^-_{(i,j),r}$ and $\bvs^+_{(i,j),r}$ for $((i,j),r)\in\OME\times\ZZ$, and set
$$\bvs^-=\bigoplus_{((i,j),r)\in\OME\times\ZZ}\bvs^-_{(i,j),r},\ \bvs^+=\bigoplus_{((i,j),r)\in\OME\times\ZZ}\bvs^+_{(i,j),r}.$$

\begin{definition}[{\cite[\S~1.3, 1.1]{Deng}}]
A \emph{representation} of $S$ is a sequence $(X;f)$,
where $X\in\hull$ and $f\in\Hom_\k(\bvs^-(X),\bvs^+(X))$ is a diagonal map
$$f=\operatorname{diag}\left(f_{(i,j),r},\ ((i,j),r)\in\OME\times\ZZ\right)\colon\bigoplus_{((i,j),r)\in\OME\times\ZZ}\bvs_{(i,j),r}^-(X)\to\bigoplus_{((i,j),r)\in\OME\times\ZZ}\bvs_{(i,j),r}^+(X).$$
A \emph{morphism} of representations from $(X;f)$ to $(X';f')$ is given by a morphism $\alpha\in\hull(X,X')$ such that the following diagram commutes
\begin{equation}\label{eq:mor1}
\xymatrix{
\bvs^-(X)\ar[r]^{f}\ar[d]_{\bvs^-(\alpha)} & \bvs^+(X)\ar[d]^{\bvs^+(\alpha)}\\
\bvs^-(X')\ar[r]^{f'} & \bvs^+(X')
}
\end{equation}
Denote by $\rep(S)$ the category of representations of $S$ and by $\repb(S)$ the full subcategory of $\rep(S)$ formed by the representations $(X;f)$ such that $f$ is a bijection.
\end{definition}

\begin{example}
    Consider the graded arc segment $\mu_6=(1,2)^{0}-^{3}(1,6)$ from Example~\ref{ex:as} and see Figures~\ref{fig:datum} and \ref{fig:as}. Let $X=\mu_6\oplus\{((1,2),0),((1,6),0)\}$. Then
    \[\bvs^-_{(i,j),r}(X)=\begin{cases}
        \k & \text{if $((i,j),r)=((1,2),0)$ or $((1,6),3)$,}\\
        0 & \text{otherwise,}
    \end{cases}\]
    and
    \[\bvs^+_{(i,j),r}(X)=\begin{cases}
        \k & \text{if $((i,j),r)=((1,2),0)$ or $((1,6),0)$,}\\
        0 & \text{otherwise.}
    \end{cases}\]
    One can take $f$ to be the diagonal map with only one nonzero component $\id:\bvs^-_{(1,2),0}\to \bvs^+_{(1,2),0}$ such that $(X;f)\in\rep(S)$. However, this representation $(X;f)$ is not in $\repb(S)$. Indeed, for the $X$ in this example, there is no linear map $f$ such that $(X;f)\in\repb(S)$.
\end{example}

\begin{remark}\label{rmk:rep and mor}
    We regard $\bvs^-$ and $\bvs^+$ as $(\OME\times\ZZ)$-graded modules of $\k\overline{S^-}$ and $\k\overline{S^+}$, respectively. In other words, $\bvs^-$ and $\bvs^+$ are functors from $\k\overline{S^-}$ and $\k\overline{S^+}$ respectively to the category $(\OME\times\ZZ)$-$\gvec$ of finite-dimensional $(\OME\times\ZZ)$-graded vector spaces. Note that both $\bvs^-$ and $\bvs^+$ are faithful.

    For any $X\in\add\overline{S}$, by definition, there are $X_-\in\k\overline{S^-}$ and $X_+\in\k\overline{S^+}$ such that $X=X_-\oplus X_+$. Moreover, we have $\bvs^-_{(i,j),r}(X)=\bvs^-_{(i,j),r}(X^-)$ and $\bvs^+_{(i,j),r}(X)=\bvs^+_{(i,j),r}(X^+)$. Therefore, a representation of $S$ can be realized as a triple $(X_-,X_+,f)$, where
    \begin{itemize}
        \item $X_-\in\k\overline{S^-}$,
        \item $X_+\in\k\overline{S^+}$, and
        \item $f:\bvs^-(X^-)\to\bvs^+(X_+)$ a morphism in $(\OME\times\ZZ)$-$\gvec$.
    \end{itemize}
    We have that $(X_-,X_+,f)\in\repb(S)$ if and only if $f$ is an isomorphism. In this way, a morphism between representations of $S$ from $(X_-,X_+,f)$ to $(X'_-,X'_+,f)$ is a pair $(\alpha_-,\alpha_+)$ of $\alpha_-:X_-\to X'_-$ and $\alpha_+:X_+\to X'_+$ such that the following diagram commutes
    \begin{equation}\label{eq:mor2}
    \xymatrix{
    \bvs^-(X_-)\ar[r]^{f}\ar[d]_{\bvs^-(\alpha_-)} & \bvs^+(X_+)\ar[d]^{\bvs^+(\alpha_+)}\\
    \bvs^-(X'_-)\ar[r]^{f'} & \bvs^+(X'_+)
}
    \end{equation}
    One can compare this with the definition of triple categories introduced in Section~\ref{subsec:triple to bush}.

    In the case that $f$ and $f'$ are bijections, it follows from the commutative diagram~\eqref{eq:mor2} that $\bvs^-(\alpha_-)$ and $\bvs^+(\alpha_+)$ determine each other. Hence, by the faifulness of $\bvs^-$ and $\bvs^+$, the morphisms $\alpha_-$ and $\alpha_+$ determine each other, too.
\end{remark}

\subsection{Words with local system}

A sequence $w=w_1 w_2\cdots w_m$ of elements in the bush $S$ is called a \emph{finite word} if $w_k^\ast|w_{k+1}$ for all $1\leq k<m$, where $w_k^\ast|w_{k+1}$ stands for that $$(w_k^\ast,w_{k+1})\in\bigcup_{((i,j),r)\in\widetilde{\Sigma}}((S_{((i,j),r)}^-\times S_{((i,j),r)}^+)\cup(S_{((i,j),r)}^+\times S_{((i,j),r)}^-)).$$
Two finite words $w=w_1 w_2\cdots w_m$ and $w'=w'_1 w'_2\cdots w'_{m'}$ are called \emph{equivalent} if $m=m'$ and $w_k'=w_k$ or $w_k^\wedge$ for all $1\leq k\leq m$. The \emph{inverse} of a finite word $w=w_1 w_2\cdots w_m$ is $w^\ast=w_m^\ast\cdots w_2^\ast w_1^\ast$. A finite word $w$ is called \emph{symmetric} if $\ec{w}=\ec{w^\ast}$,
and \emph{asymmetric} otherwise, where recall that $\ec{w}$ denotes the equivalence class containing $w$.

\begin{remark}\label{rmk:3.1}
By Notations~\ref{not:wedge}, two finite words $w=w_1 w_2\cdots w_m$ and $w'=w'_1 w'_2\cdots w'_m$ are equivalent if and only if for any $1\leq k\leq m$, if $w_k\neq w_k'$ then $\{w_k,w_k'\}=S_{(i,j),r}^+$ for some $(i,j)\in\fc{\OME}$ and $r\in\ZZ$.
\end{remark}

An order on the set of equivalence classes of finite words is given by
$$\ec{v_1\cdots v_n}<\ec{w_1\cdots w_m}$$
if and only if one of the following three cases occurs
\begin{enumerate}
    \item $n<m$, $\ec{v_1\cdots v_n}=\ec{w_1\cdots w_n}$ and $w_{n+1}\in S^-$;
    \item $m<n$, $\ec{v_1\cdots v_m}=\ec{w_1\cdots w_m}$ and $v_{m+1}\in S^+$;
    \item there exists $1\leq k\leq \min\{m,n\}$ such that $\ec{v_1\cdots v_{k-1}}=\ec{w_1\cdots w_{k-1}}$ and $v_k<w_k$ in $S$ (see Construction~\ref{cons:bush} and cf. Figure~\ref{fig:order2} for the order in $S$).
\end{enumerate}
When $v_1,w_1$ are in the same rod, $\ec{v}$ and $\ec{w}$ are always comparable. In particular, for any finite word $w=w_1w_2\cdots w_m$ and any $1\leq k\leq m$, if $w_k\in S_{(i,j),r}^+$ for some $(i,j)\in\fc{\OME}$, then $\ec{w_{k+1}\cdots w_m}$ and $\ec{w_{k-1}^\ast\cdots w_1^\ast}$ are comparable, because $w_{k+1}$ and $w_{k-1}^\ast$ are in the same rod.

Let $\AW$ (resp. $\SW$) be a set of asymmetric (resp. symmetric) finite words, which contains one representative of each class $\ec{w}\cup\ec{w^\ast}$ (resp. $\ec{w}$) such that for any element $w_1w_2\cdots w_m$ in it and for any $1\leq k\leq m$, if $w_k\in S_{(i,j),r}^+$ for some $(i,j)\in\fc{\OME}$ and $r\in\ZZ$, then
$$w_k=\begin{cases}
	((i,j^+),r)&\text{if $\ec{w_{k+1}\cdots w_m}\geq\ec{w_{k-1}^\ast\cdots w_1^\ast}$,}\\
	((i,j^-),r)&\text{if $\ec{w_{k+1}\cdots w_m}<\ec{w_{k-1}^\ast\cdots w_1^\ast}$.}
\end{cases}$$

\begin{definition}
A finite word $w=w_1w_2\cdots w_m$ is called \emph{inextensible} if there do not exist $w_0$ and $w_{m+1}$ in $S$ such that $w_0w$ or $ww_{m+1}$ is a finite word.
\end{definition}

\begin{lemma}\label{lem:inex}
    A finite word $w_1 w_2\cdots w_m$ is inextensible if and only if $$w_1,w_m^\ast\in\bigcup_{((i,j),r)\in\M\times\ZZ}S^-_{(i,j),r}.$$
\end{lemma}

\begin{proof}
	By construction, for any $((i,j),r)\in\widetilde{\Sigma}$, $S^-_{(i,j),r}$ is always not empty, and $S^+_{(i,j),r}$ is empty if and only if $(i,j)\in\M$. So for any $s\in S$, there is no $s'\in S$ such that $s|s'$ if and only if $s\in\bigcup_{((i,j),r)\in\M\times\ZZ}S^-_{(i,j),r}$. Then the lemma follows.
\end{proof}

Denote by $\AW^b$ (resp. $\SW^b$) the subset of $\AW$ (resp. $\SW$) consisting of inextensible finite words.

A \emph{periodic word} is a sequence $w=(w_k)_{k\in\ZZ}$ satisfying $w_k^\ast|w_{k+1}$ for all $k\in\ZZ$ and there exists a natural number $\period\geq 1$ such that $w_{k+\period}=w_k$ for all $k\in\ZZ$. The smallest $\period$ satisfying these conditions is called the \emph{period} of $w$. The \emph{rotation} $w\{p\}$ of a periodic word with $p\in\ZZ$ is a periodic word with $w\{p\}_k=w_{p+k},k\in\ZZ$. Two periodic words $w=(w_k)_{k\in \ZZ}$ and $w'=(w'_k)_{k\in \ZZ}$ are called \emph{equivalent} if there is $p\in\ZZ$ such that $w'\{p\}_k=w_k$ or $w_k^\wedge$ for all $k\in\ZZ$. The \emph{inverse} of a periodic word $w=(w_k)_{k\in\ZZ}$ is $w^\ast=(w_{-i}^\ast)_{i\in\ZZ}$. A periodic word $w$ is called \emph{symmetric} if $\ec{w^\ast}=\ec{w\{p\}}$ for some $p\in\ZZ$, and \emph{asymmetric} otherwise.

Let $\APW$ (resp. $\SPW$) be a set of asymmetric (resp. symmetric) periodic words which contains one representation of each class $\bigcup_{p\in\ZZ}(\ec{w\{p\}}\cup\ec{w^\ast\{p\}})$ (resp. $\bigcup_{p\in\ZZ}\ec{w\{p\}}$) such that for any element $(w_k)_{k\in\ZZ}\in\APW\cup\SPW$ in it and any $1\leq k\leq \period$, if $w_k\in S^+_{(i,j),r}$ for some $(i,j)\in\fc{\OME}$ and $r\in\ZZ$, then
$$w_k=\begin{cases}
	((i,j^+),r)&\text{if $\ec{w_{k+1}\cdots w_{\period+k}}\geq\ec{w_{k-1}^\ast\cdots w_{-\period+k}^\ast}$,}\\
	((i,j^-),r)&\text{if $\ec{w_{k+1}\cdots w_{\period+k}}<\ec{w_{k-1}^\ast\cdots w_{-\period+k}^\ast}$.}
\end{cases}$$
Note that if $\ec{w_{k+1}\cdots w_{\period+k}}=\ec{w_{k-1}^\ast\cdots w_{-\period+k}^\ast}$, then $\ec{w_{k+1}\cdots w_{k+1+l}}=\ec{w_{k-1}^\ast\cdots w_{k-1-l}^\ast}$ for any $l\in\NN$, since $\period$ is the period of $w$.

\begin{definition}\label{def:wwls}
    A word with \emph{local system} is a pair $(w,N)$, where $$w\in\AW^b\cup\SW^b\cup\APW\cup\SPW$$ and $N$ is an isoclass of indecomposable $A_w$-modules, where
    $$A_w=\begin{cases}
    \k&\text{if $w\in\AW^b$},\\
    \k[x]/(x^2-x)&\text{if $w\in\SW^b$},\\
    \k[x,x^{-1}]&\text{if $w\in\APW$},\\
    \k\langle x,y\rangle/(x^2-x,y^2-y)&\text{if $w\in\SPW$.}
    \end{cases}$$
    We denote by $\OO$ the set of words with local system.
\end{definition}

\begin{example}\label{ex:words}
    We continue Example~\ref{ex:as}. Let $$w'=\overrightarrow{\mu_1}((2,2),0)\overrightarrow{\mu_2}((1,3^+),2)\overrightarrow{\mu_3}((1,4^+),3)\overleftarrow{\mu_3}((1,3^+),2)\overrightarrow{\mu_3}((1,4^+),3)\overrightarrow{\mu_4}((1,7),5)\overleftarrow{\mu_5},$$
    $$w''=\overrightarrow{\mu_1}((2,2),0)\overrightarrow{\mu_2}((1,3^+),2)\overrightarrow{\mu_3}((1,4^+),3)\overleftarrow{\mu_3}((1,3^-),2)\overleftarrow{\mu_2}((1,1),0)\overleftarrow{\mu_1},$$
    $w'''=(w'''_i)_{i\in\ZZ}$ with period 8 and
    $$w'''_1w'''_2\cdots w'''_8=\overrightarrow{\mu_6}((1,6),3)\overleftarrow{\mu_7}((1,1),2)\overleftarrow{\mu_8}((2,1),1)\overleftarrow{\mu_9}((1,6),0),$$
    $w''''=(w''''_i)_{i\in\ZZ}$ with period 4 and
    $$w''''_1\cdots w''''_4=\overrightarrow{\mu_3}((1,4^+),1)\overleftarrow{\mu_3}((1,3^+),0).$$
    Then we have $w'\in\AW^b,w''\in\SW^b,w'''\in\APW$ and $w''''\in\SPW$.
\end{example}

\subsection{Representations attached to words with local system}\label{sec:Rep_b}

For any $s\in S$, we define an object $\widehat{s}$ in $\hull$ as follows.
$$\widehat{s}=\begin{cases}
s\oplus s^\wedge&\text{if $s\neq s^\wedge$,}\\
\overline{s}&\text{otherwise,}
\end{cases}$$
where $\overline{s}$ denotes the equivalence class containing $s$.

\begin{construction}\label{cons:fw}
Recall from \cite[Section~2.2]{Deng}, that to each finite word $w=w_1w_2\cdots w_m$,
there is an associated representation $R(w):=(X_w;f_{w})$ as follows. Set
$$X_w=\widehat{w_1}\oplus\widehat{w_2}\oplus\cdots\oplus\widehat{w_m}.$$
Thus each term $x$ of the sequence $w_1w_1^\ast w_2w_2^\ast\cdots w_m w_m^\ast$ which is in $S^{\varepsilon}_{(i,j),r}$ contributes a 1-dimensional direct summand $\k\underline{x}$ to the space $\bvs_{(i,j),r}^\varepsilon (X_w)$.

For any $\underline{x}\in\{\underline{w_1},\underline{w_1^\ast},\underline{w_2},\underline{w_2^\ast},\cdots,\underline{w_m},\underline{w_m^\ast}\}$ with $x\in \bigcup_{((i,j),r)\in\OME\times\ZZ}S_{(i,j),r}^-$, define $f_w(\underline{x})\in S^+$ as follows.
\begin{enumerate}
\item Case $x=w_{k+1},1\leq k< m$:
\begin{enumerate}
\item if $w_{k}^{\ast}\in S^+_{(i,j),r}$ for some $(i,j)\notin\fc{\OME}$ and $r\in\ZZ$, see the first picture in Figure~\ref{fig:fw},
$$f_w(\underline{x})=\underline{w_{k}^\ast};$$
\item if $w_{k}^{\ast}\in S^+_{(i,j),r}$ for some $(i,j)\in\fc{\OME}$ and $r\in\ZZ$,
$$f_w(\underline{x})=\begin{cases}
\underline{w_{k}^\ast}, &\text{if $w_k^\ast=((i,j^+),r)$, see the second picture in Figure~\ref{fig:fw},}\\
\underline{w_{k}^\ast}+\underline{w_{k}}, &\text{if $w_k^\ast=((i,j^-),r)$, see the third picture in Figure~\ref{fig:fw}.}
\end{cases}$$
\end{enumerate}
\item Case $x=w_{k-1}^\ast,1< k\leq m$:
\begin{enumerate}
\item if $w_{k}\in S^+_{(i,j),r}$ for some $(i,j)\notin\fc{\OME}$ and $r\in\ZZ$, see the first picture in Figure~\ref{fig:fw},
$$f_w(\underline{x})=\underline{w_{k}};$$
\item if $w_{k}\in S^+_{(i,j),r}$ for some $(i,j)\in\fc{\OME}$,
$$f_w(\underline{x})=\begin{cases}
\underline{w_{k}}, &\text{if $w_k=((i,j^+),r)$, see the third picture in Figure~\ref{fig:fw},}\\
\underline{w_{k}}+\underline{w_{k}^\ast}, &\text{if $w_k=((i,j^-),r)$, see the second picture in Figure~\ref{fig:fw}.}
\end{cases}$$
\end{enumerate}
\item Case $x=w_1$ or $w_m^\ast$:
$$f_w(\underline{x})=0.$$
\end{enumerate}

\begin{figure}[htpb]
	\begin{tikzpicture}
	\draw[red,thick] (0,-2)\ww to(0,2)\ww ;
	\draw[blue,thick] (-2,0)to(0,0)to(2,0);
	\draw[blue] (-1.5,0)node[below]{$w_{k-1}^\ast$} (1.5,0)node[below]{$w_{k+1}$};
	\draw[blue] (-.25,0)node[below]{$w_k$} (.25,0)node[below]{$w_k^\ast$};
	\draw[Green,thick,bend right,->-=1,>=stealth] (-1.5,-.55)to(-.3,-.4);
	\draw[Green] (-.9,-.55)node[below]{$f_{w}$};
	\draw[Green,thick,bend left,->-=1,>=stealth] (1.5,-.4)to(.3,-.55);
	\draw[Green] (.9,-.55)node[below]{$f_{w}$};
	\end{tikzpicture}\qquad\qquad
	
	\begin{tikzpicture}[xscale=1.5,yscale=2]
	\draw[thick] (0,-.8)node{$\bullet$};
	\draw[red, thick] (0,0) .. controls +(-55:1) and +(0:.5) .. (0,-1) .. controls +(180:.5) and +(-125:1) .. (0,0);
	\draw[blue,thick]plot [smooth, tension=1] coordinates {(-2.2,-1) (0,-.3) (2.2,-1)};
	\draw[blue] (-1.5,-.45)node{$w_{k-1}^\ast$} (1.5,-.45)node{$w_{k+1}$};
	\draw[blue] (-.33,-.2)node{$w_k$} (.3,-.2)node{$w_k^\ast$};
	\draw[Green,thick,bend left,->-=.999,>=stealth] (-1.5,-.3)to(-.35,-.1);
	\draw[Green,thick,bend right,->-=.999,>=stealth] (1.5,-.3)to(.35,-.05);
	\draw[Green,thick,bend left=50,->-=.999,>=stealth] (-1.5,-.3)to(.3,-.1);
	\draw[Green] (-.95,-.2)node{$f_{w}$} (.95,.15)node{$f_{w}$} (-.5,.35)node{$f_{w}$};
	\draw[red] (0,0)\ww ;
	\end{tikzpicture}\qquad
	\begin{tikzpicture}[xscale=1.5,yscale=2]
	\draw[thick] (0,-.8)node{$\bullet$};
	\draw[red, thick] (0,0) .. controls +(-55:1) and +(0:.5) .. (0,-1) .. controls +(180:.5) and +(-125:1) .. (0,0);
	\draw[blue,thick]plot [smooth, tension=1] coordinates {(-2.2,-1) (0,-.3) (2.2,-1)};
	\draw[blue] (-1.5,-.45)node{$w_{k+1}$} (1.5,-.45)node{$w_{k-1}^\ast$};
	\draw[blue] (-.35,-.2)node{$w_k^\ast$} (.3,-.2)node{$w_k$};
	\draw[Green,thick,bend left,->-=.999,>=stealth] (-1.5,-.3)to(-.35,-.1);
	\draw[Green,thick,bend right,->-=.999,>=stealth] (1.5,-.3)to(.35,-.05);
	\draw[Green,thick,bend left=50,->-=.999,>=stealth] (-1.5,-.3)to(.3,-.1);
	\draw[Green] (-.95,-.2)node{$f_{w}$} (.95,.15)node{$f_{w}$} (-.5,.35)node{$f_{w}$};
	\draw[red] (0,0)\ww ;
	\end{tikzpicture}
	\caption{Construction of $f_{w}$}\label{fig:fw}
\end{figure}

\end{construction}

\begin{lemma}\label{lem:bw}
	For any finite word $w$, $R(w)\in\repb(S)$ if and only if $w$ is inextensible.
\end{lemma}

\begin{proof}
	By the construction, restricting to the subspace with basis $\underline{w_1^\ast},\underline{w_2},\underline{w_2^\ast},\cdots,\underline{w_m}$, $f_w$ is a bijection. So $R(w)\in\repb(S)$ if and only if $w_1,w_m^\ast\in \bigcup_{((i,j),r)\in\M\times\ZZ}S^-_{(i,j),r}$ (that is, they are neither in the domain nor the codomain of $f_w$), which is equivalent to that $w$ is inextensible by Lemma~\ref{lem:inex}.
\end{proof}

For any $w=w_1w_2\cdots w_m\in\SW$, we have $m=2m'+1$ for some integer $m'$, and $w_{m'+1}\in S^+_{(i,j),r}$ for some $(i,j)\in\fc{\OME}$. Then the representation $$R(w)\cong R(w,+)\oplus R(w,-),$$ where $R(w,\kappa)$, $\kappa\in\{+,-\}$, is given by the restriction of $R(w)$ on the direct summand $\widehat{w_1}\oplus\widehat{w_2}\oplus\cdots\oplus\widehat{w_{m'}}\oplus((i,j^\kappa),r)$ of $X_w$.


\begin{lemma}\label{lem:bw2}
	For any $w\in\SW$ and any $\kappa\in\{+,-\}$, $R(w,\kappa)$ is in $\repb(S)$ if and only if $w\in\SW^b$.
\end{lemma}

\begin{proof}
    By the construction, $R(w)$ is in $\repb(S)$ if and only if so is any of $R(w,+)$ and $R(w,-)$. Hence this lemma follows from Lemma~\ref{lem:bw}.
\end{proof}

\begin{example}
    For the word $w''$ in Example~\ref{ex:words}, the representation $R(w'',+)=(X,f)$ is given by
    $X=X_-\oplus X_+$ with $$X_-=\mu_1\oplus\mu_2\oplus\mu_3,$$
    $$\begin{array}{rl}
        X_+= & \{((2,2),0),((1,1),0)\}\oplus\{((1,3^+),2)\}\oplus\{((1,3^-),2)\}\oplus\{((1,4^+),3)\},
    \end{array}$$
    and
    $$f=\begin{pmatrix}
    1\\
    &1\\
    &&1&1\\
    &&&1\\
    &&&&1
    \end{pmatrix}$$
    from $\k\overleftarrow{\mu_1}\oplus\k\overrightarrow{\mu_2}\oplus\k\overleftarrow{\mu_2}\oplus\k\overrightarrow{\mu_3}\oplus\k\overleftarrow{\mu_3}$ to $\k((2,2),0)\oplus\k((1,1),0)\oplus\k((1,3^+),2)\oplus\k((1,3^-),2)\oplus\k((1,4^+),3).$
\end{example}

For a period word $w\in\APW\cup\SPW$ and an isoclass of indecomposable $A_w$-modules $N$, there is an associated invertible matrix $P(N)$ and a representation $R(w,P(N))$ whose construction is based on $R(w)$ but
\begin{enumerate}
    \item identify $w_k$ with $w_{k+\period}$ for any $k\in\ZZ$,
    \item each 1-dimensional direct summand $\k\underline{x}$ is tensored with a vector space of dimension $\dim N$, and
    \item almost every component of the map $f_w$ is tensored with the corresponding identities between the tensored vector spaces unless, at one point, tensored with the matrix $P(N)$.
\end{enumerate}
We refer to Appendix~\ref{app:period} for a more explicit construction. Since $w$ is period and $P(N)$ is invertible, we have the following.

\begin{lemma}\label{lem:bw3}
    For any period word $w\in\APW\cup\SPW$, the representation $R(w,P(N))$ is in $\repb(S)$.
\end{lemma}

\begin{construction}\label{cons:rep}
    We define a map
    \begin{equation}\label{eq:rep}
    \boxed{
        R\colon \OO \to \repb(S)
    }\end{equation}
    as follows. Let $(w,N)$ be a word with local system.
    \begin{enumerate}
        \item If $w\in\AW^b$, since $A_w=\k$, there is a unique indecomposable module $N$ up to isomorphism. We define $R(w,N)=R(w)$.
        \item If $w\in\SW^b$, since $A_w=\k[x]/(x^2-x)$, there are two indecomposable modules $N$ up to isomorphism, one with $N(x)=1$ and the other with $N(x)=0$. We define $R(w,N)=R(w,+)$ if $N(x)=1$ and $R(w,N)=R(w,-)$ if $N(x)=0$.
        \item If $w\in\APW\cup\SPW$, we define $R(w,N):=R(w,P(N))$.
    \end{enumerate}
\end{construction}

By the main result in \cite{Deng}, together with Lemmas~\ref{lem:bw}, \ref{lem:bw2} and \ref{lem:bw3}, we get a classification of indecomposable objects in $\repb(S)$.

\begin{theorem}\label{thm:Deng}
The representations
$$\{R(w,N)\;|\;(w,N)\in\OO\}$$
form a complete set of pairwise non-isomorphic indecomposable objects in $\repb(S)$.
\end{theorem}

\begin{remark}\label{rmk:1-dim}
    In Part~\ref{part:2}, we shall care about (the objects in $\per\sg$ associated to) words with local system $(w,N)$ such that $w\in\AW^b\cup\SW^b\cup\SPW$ and $\dim_\k N=1$. Note that when $w\in\AW^b\cup\SW^b$, any indecomposable $A_w$-module is 1-dimensional, and the associated representation $R(w,N)$ has been explicitly described. Now, let us describe the representation $R(w,N)$ for $w\in\SPW$ and $\dim_\k N=1$ more explicitly (cf. Appendix~\ref{subapp:spw}). Up to rotation, we may write $w$ as a repeat of
    $$w_0w_1\cdots w_{\frac{\period}{2}-1}w_{\frac{\period}{2}}w_{\frac{\period}{2}-1}^\ast\cdots w_1^\ast,$$
    with $w_0\in S_{(i_1,j_1),r_1}^+$ and $w_{\frac{\period}{2}}\in S_{(i_2,j_2),r_2}^+$ for some $(i_1,j_1), (i_2,j_2)\in\fc{\OME}$ and $r_1,r_2\in\ZZ$. The indeterminate $x,y$ in the algebra $A_w=\k\langle x,y\rangle/(x^2-x,y^2-y)$ are associated to $w_0$ and $w_{\frac{\period}{2}}$, respectively. Then the associated representation $R(w,N)=(X_{w,N},f_{w,N})$ is constructed as follows. Set
    $$X_{w,N}=((i_1,j_1^{\kappa_1}),r_1)\oplus\widehat{w_1}\oplus\widehat{w_2}\oplus\cdots\oplus\widehat{w_{\frac{\period}{2}-1}}\oplus ((i_2,j_2^{\kappa_2}),r_2),$$
    where $\kappa_1=+/-$ if $N(x)=1/0$ and $\kappa_2=+/-$ if $N(y)=1/0$. Thus, each term $x$ of the sequence $((i_1,j_1^{\kappa_1}),r_1)w_1w_1^\ast w_2w_2^\ast\cdots w_{\frac{\period}{2}-1} w_{\frac{\period}{2}-1}^\ast((i_2,j_2^{\kappa_2}),r_2)$ which is in $S^{\varepsilon}_{(i,j),r}$ contributes a 1-dimensional direct summand $\k\underline{x}$ to the space $\bvs_{(i,j),r}^\varepsilon (X_w)$. Define $f_{w,N}$ to be the same as $f_w$ in Construction~\ref{cons:fw} except $$f_{w,N}(\underline{w_{1}})=((i_1,j_1^{\kappa_1}),r_1) \text{ and } f_{w,N}(\underline{w_{\frac{\period}{2}-1}^\ast})=((i_2,j_2^{\kappa_2}),r_2).$$

\end{remark}

\subsection{Admissible curves with local system}

\begin{definition}
Let $\gamma$ be an open curve. The \emph{completion} $\overline{\gamma}$ of $\gamma$ is defined as shown in Figure~\ref{fig:completion} if $\gamma\cap\P\neq\emptyset$, or $\overline{\gamma}=\gamma$ otherwise. For any grading $\wg$ of $\gamma$, denote by $\widetilde{\overline{\gamma}}$ the induced grading of $\overline{\gamma}$ ensured by \eqref{eq:grad=1}.
\end{definition}

\begin{figure}[htpb]\centering
	\begin{tikzpicture}[xscale=1.5,yscale=.3,rotate=90,scale=.8]
		\draw[ultra thick]plot [smooth,tension=1] coordinates {(170:4.5) (180:4) (190:4.5)};
		\draw[blue,thick](180:4)to(0:4);
		\draw[thick] (1,-1.5)node{$\Longrightarrow$};
		\draw[thick](0:4)node{$\bullet$};
		\draw[thick,blue](180:4)\nn;
	\end{tikzpicture}\quad
	\begin{tikzpicture}[xscale=1.5,yscale=.3,rotate=90,scale=.8]
		\draw[ultra thick]plot [smooth,tension=1] coordinates {(170:4.5) (180:4) (190:4.5)};
		\draw[blue,thick]plot [smooth,tension=1] coordinates {(-4,0) (2,.5) (5,0) (2,-.5) (-4,0)};
		\draw[blue,thick](5,0)to(5,0.001);
		\draw[blue,dotted,thick](180:4)to(0:3);
		\draw[thick](0:3)node{$\bullet$};
		\draw[thick,blue](180:4)\nn;
	\end{tikzpicture}\qquad\qquad
	\qquad\begin{tikzpicture}[xscale=.7,yscale=.3,scale=.8]
		\draw[blue,thick](0,4)to(0,-4) (0,-4.5)node{ }(0,7)node{ };\draw(3,0)node{$\Longrightarrow$};
		\draw[thick](0,4)node{$\bullet$}(0,-4)node{$\bullet$};
	\end{tikzpicture}\qquad
	\begin{tikzpicture}[xscale=.7,yscale=.3,scale=.8]
		\draw[blue,thick](0,0) ellipse (1 and 5);
		\draw[blue,dotted,thick](0,3)to(0,-3);
		\draw[thick](0,3)node{$\bullet$}(0,-3)node{$\bullet$};
	\end{tikzpicture}
	\caption{Completions of curves}
	\label{fig:completion}
\end{figure}

\begin{definition}\label{def:ad}
	An \emph{admissible} curve on $\surf$ is an open curve $\gamma:I\to\surf$ such that
	\begin{enumerate}[label=(A\arabic*),ref=(A\arabic*)]
	\item\label{itm:A1} $\gamma$ does not cut out a once-punctured monogon by a self-intersection (see Figure~\ref{fig:monogon cut});
	\item\label{itm:A2} if $I=[0,1]$ and $\gamma(\{0,1\})\subset\P$, then the completion $\overline{\gamma}$ is not a proper power (in the quotient group of the fundamental group of $\surf$ by the squares of the loops enclosing a puncture) of another open curve $\gamma':S^1\to \surf$;
	\item\label{itm:A3} if $I=S^1$, then $\gamma$ is neither a proper power (in the quotient group of the fundamental group of $\surf$ by the squares of the loops enclosing a puncture) of another open curve $\gamma':S^1\to \surf$, nor the completion of an open arc $\gamma'':[0,1]\to\surf$ with $\gamma''(\{0,1\})\subset\P$.
	\end{enumerate}
	Denote by $\wads$ the set of graded admissible curves.
\end{definition}

\begin{figure}[htpb]
	\begin{tikzpicture}[xscale=2,yscale=.3,rotate=90]
		\draw[ultra thick,white]plot [smooth,tension=1] coordinates {(170:4.5) (180:4) (190:4.5)};
		\draw[blue,thick]plot [smooth,tension=1] coordinates {(-3,-.5) (2,.5) (5,0) (2,-.5) (-3,.5)};
		\draw[thick](0:3)node[below]{$\bullet$};
	\end{tikzpicture}\qquad
	\begin{tikzpicture}[xscale=2,yscale=.3,rotate=90]
	    \draw[ultra thick,white]plot [smooth,tension=1] coordinates {(170:4.5) (180:4) (190:4.5)};
		\draw[ultra thick]plot [smooth,tension=1] coordinates {(170:3.5) (180:3) (190:3.5)};
		\draw[blue,thick]plot [smooth,tension=1] coordinates {(-3,0) (2,.5) (5,0) (2,-.5) (-3,0)};
		\draw[thick](0:3)node[below]{$\bullet$};
		\draw[blue] (-3,0)\nn;
	\end{tikzpicture}\qquad
	\begin{tikzpicture}[xscale=2,yscale=.3,rotate=90]
	    \draw[ultra thick,white]plot [smooth,tension=1] coordinates {(170:4.5) (180:4) (190:4.5)};
		\draw[blue,thick]plot [smooth,tension=1] coordinates {(-3,0) (2,.5) (5,0) (2,-.5) (-3,0)};
		\draw[thick](0:3)node[below]{$\bullet$};
		\draw (-3,0)node{$\bullet$};
	\end{tikzpicture}
    \caption{Curves cutting out a once-punctured monogon by a self-intersection}\label{fig:monogon cut}
\end{figure}

\begin{example}\label{ex:non-ad}
    The curve $\gamma_1$ shown in Figure~\ref{fig:non-ad cur} is admissible. The curves $\gamma_2$ and $\gamma_3$ are not admissible, because $\gamma_2$ is the completion of $\gamma_1$ and the completion $\gamma_3$ is the third power of $\gamma_2$.

\begin{figure}[htpb]
    \begin{tikzpicture}[scale=1.3]
    \clip(-4,1.3) rectangle (4,3.6);
   	\draw[ultra thick](0,0) circle (3.5);
   	\draw[ultra thick](0,0) circle (1);    	
   	\draw[red,thick](0,3.5)to[out=-70,in=-135](1.5,2)to[out=45,in=-20](0,3.5)
        (0,3.5)to[out=-110,in=-45](-1.5,2)to[out=135,in=-160](0,3.5);
    	
    \draw[red,thick](-1,0)to[out=-100,in=160](0,-2)to[out=20,in=-80](1,0);
    \draw[red,thick](0,-2)to[out=0,in=-90](2.5,.75)to[out=90,in=-10](0,3.5);	\draw[red,thick](0,3.5)to[out=-5,in=90](3,.25)to[out=-90,in=0](0,-3)to[out=180,in=-90](-3,.25)to[out=90,in=-175](0,3.5);
   	\draw[blue,thick,bend left=5](-1.25,2.25)to(1.25,2.25);
   	\draw[blue,thick](0,2.25) ellipse (1.8 and .8);
    \draw[cyan,thick](1.25,2.25)to[out=-180,in=0](-1.25,2)to[out=180,in=-90](-1.5,2.25)to[out=90,in=-175](-1.25,2.5)to[out=5,in=175](1.25,2.5)to[out=-5,in=90](1.5,2.25)to[out=-90,in=0](1.25,2)to[out=180,in=0](-1.25,2.25);
    \draw(0,2.4)node[blue]{$\gamma_1$}(1.25,1.45)node[blue]{$\gamma_2$}(0,1.9)node[cyan]{$\gamma_3$};
    \draw(-1.25,2.25)\nn (1.25,2.25)\nn (-1,0)\ww (1,0)\ww (0,-2)\ww (0,3.5)\ww (0,1)[blue]\nn (0,-1)[blue]\nn (0,-3.5)[blue]\nn;
\end{tikzpicture}
    \caption{Examples of admissible/non-admissible curves}
    \label{fig:non-ad cur}
\end{figure}
\end{example}

Let $\bads$ be the set of graded open curves whose underlying curve $\gamma:I\to\surf$ satisfies
\begin{enumerate}[label=($\operatorname{\overline{A}}$\arabic*),ref=($\operatorname{\overline{A}}$\arabic*)]
\item\label{itm:barA1} $\gamma$ does not cut out a once-punctured monogon by an interior self-intersection (see the first picture in Figure~\ref{fig:monogon cut});
\item\label{itm:barA2} if $I=[0,1]$, then $\gamma(\{0,1\})\subset\M$;
\item\label{itm:barA3} if $I=S^1$, then $\gamma$ is not a proper power of another open curve $\gamma':S^1\to \surf$.
\end{enumerate}

\begin{lemma}\label{lem:comp}
    There is a bijection
    $$\wads\to\bads$$
    sending a graded admissible curve to its completion.
\end{lemma}

\begin{proof}
The map is clearly injective and surjective provided it is well-defined. Thus, we only need to show that the completion of an admissible curve satisfies the conditions \ref{itm:barA1}, \ref{itm:barA2} and \ref{itm:barA3}. In fact, \ref{itm:barA2} is due to the definition of completion of curves, and \ref{itm:barA3} is due to \ref{itm:A2} and \ref{itm:A3}. So we only prove \ref{itm:barA1}.
To see this one, consider the completion $\overline{\gamma}$ of an admissible curve $\gamma$.
We assume that both endpoints of $\gamma$ are in $\P$; the case that exactly one endpoint of $\gamma$ is in $\P$ can be proved similarly. One can parameterize $\overline{\gamma}\colon[0,2]\to S^1\to \surf^\circ$ (in a minimal position),
for $[0,2]\to S^1:t\mapsto e^{t\mathbf{i}\pi}$ such that there exists $0<\epsilon<1$ and:
\begin{itemize}
  \item $l_i=\overline{\gamma}(i-\epsilon,i+\epsilon)$ is in a neighborhood $N_i$ of $\gamma(i)$
  without self-intersection, for $i=0,1$.
  \item When shrinking $N_i$ to $\gamma(i)$, the curves $\gamma_0\colon=\overline{\gamma}(\epsilon,1-\epsilon)$ and $\gamma_1\colon=\overline{\gamma}(1+\epsilon,2-\epsilon)$ are homotopic to $\gamma$.
\end{itemize}
Since $\gamma$ does not cut out a once-punctured monogon by a self-intersection,
if $\overline{\gamma}$ cuts out a once-punctured monogon $L$ by a self-intersection,
such a monogon can not be contained in $\gamma_0$ nor in $\gamma_1$. Since $L$ can not contain both $l_i$,
$L$ must contain exactly one of $l_i$, say $l_0$.
Furthermore, suppose that $L$ is parameterized as $\overline{\gamma}[-s,t]$,
for $-1+\epsilon<-s<-\epsilon<\epsilon<t<1-\epsilon$, with $\overline{\gamma}(-s)=\overline{\gamma}(t)$.
If $s=t$, then such an intersection can be resolved, which is not possible.
Hence, we can assume that $s<t$.
As $\overline{\gamma}[s,t]=\gamma[s,t]$ is a non-trivial loop $L'$,
the loop $L$ is the sum of $L'$ and the loop $L''=\overline{\gamma}[-s,s]$ around $\gamma(0)$ in the fundamental group.
But $L$ and $L''$ are (homotopy to) once-punctured monogons, which implies $L'$ must enclose exactly two punctures, one of which is $\gamma(0)$, as shown in Figure~\ref{fig:0-st}.
Consider the first intersection of $\gamma(t,1-\epsilon]$ with $L'$,
since it does not cut out a monogon (cf. red dashed lines in Figure~\ref{fig:0-st}), it must stop at the other puncture in $L'$ (cf. blue lines in Figure~\ref{fig:0-st}),
or continue to wind the two punctures (cf. green lines in Figure~\ref{fig:0-st}).
In the latter case,
it can only wind finitely many times before stopping at the other puncture. Hence \ref{itm:A2} fails (cf. Example~\ref{ex:non-ad} and Figure~\ref{fig:non-ad cur}), a contradiction.

\begin{figure}[htpb]	
\begin{tikzpicture}[xscale=1,yscale=1,rotate=0]
\draw[thick,->,>=stealth](0,0)node[below]{$\gamma(0)$} to (1,0);
\draw[thick](1,0)arc(-90:90:1.5)(1,3)to(0,3)arc(90:270:2.5).. controls +(0:1) and +(-90:1) ..(1,0);
\draw(0,0)\nn(0,1)\nn (1,0)node[below right]{$t$}node[above right]{$s$};
\draw[blue,dashed,thick](0,1) to[bend left=45] (1,0);
\draw[red,dashed,thick](-2.5,.5) to[bend left=35] (1,0)     .. controls +(150:1) and +(90:1.5) ..(-1,-.5)
    .. controls +(-90:.5) and +(-90:1.5) ..(.5,0);
\draw[Green,dashed,thick](1,0)to(1,1)
    .. controls +(90:2) and +(90:1.5) ..(-1.5,-.5)
    .. controls +(-90:1.5) and +(-90:2) ..(.5,0);
\end{tikzpicture}
\caption{}\label{fig:0-st}
\end{figure}

\end{proof}

\begin{notations}\label{nota:four}
We divide $\wads$ into the following four subsets:
\begin{itemize}
\item $\Amm$ the set of graded admissible arcs $\wg$ with both endpoints in $\M$,
\item $\Amp$ the set of graded admissible arcs $\wg$ with one endpoint in $\M$ and the other in $\P$,
\item $\App$ the set of graded admissible arcs $\wg$ with both endpoints in $\P$, and
\item $\ccc$ the set of graded admissible curves $\wg$ without endpoints.
\end{itemize}
Denote by $\wOA(\gms)$ the set of graded admissible arcs, i.e.
$$\wOA(\gms)=\Amm\cup\Amp\cup\App.$$
\end{notations}

\begin{definition}\label{def:cwls}
    A graded admissible curve with \emph{local system} is a pair $(\wg,N)$, where $\wg$ is a graded admissible curve and $N$ is the isoclass of an indecomposable $A_{\wg}$-module, where
    $$A_{\wg}=\begin{cases}
    \k\langle x_t\mid \gamma(t)\in\P\rangle/(x_t^2-x_t\mid \gamma(t)\in\P)&\text{if $\wg\in\wOA(\gms)$},\\
    \k[x,x^{-1}]&\text{if $\wg\in\ccc$.}
    \end{cases}$$
    We denote by $\care$ the set of graded admissible curves with local system.
\end{definition}

Let $w=w_1 w_2\cdots w_m\in\AW^b\cup\SW^b$. Since $w$ is inextensible, we have $w_1,w_m\in S^-$ by Lemma~\ref{lem:inex}. Then $m$ is odd and for any $1\leq k\leq m$, $\omega_k\in S^{-}$ for $k$ odd and $\omega_k\in S^+$ for $k$ even. Moreover, for any $((i,j),r)\in\widetilde{\Sigma}$ and any even $k$, we have (cf. Figure~\ref{fig:fw})
$$w_{k-1}^\ast\in S^-_{(i,j),r}\Leftrightarrow w_{k}\in S^+_{(i,j),r}\Leftrightarrow w_{k+1}\in S^-_{(i',j'),r},$$
where $(i',j')\simeq(i,j)$. That is, the oriented graded arc segment $w_{k-1}$ ends at the same arc $\{(i,j),(i',j')\}\in\dac$ with the same intersection index $r$ as $w_{k+1}$ starts (cf. Figure~\ref{fig:fw}).

\begin{notations}
Recall that $\fc{\OME}\subset\OMEs$ can be identified with the set of edges of punctured $\dac$-polygons. For any $(i,j)\in\fc{\OME}$, we denote by $\stackrel{\frown}{(i,j)}\in\icias(\dac)$ the unique interior arc segment in the punctured $\dac$-polygon whose edge is $(i,j)$. Then we have
$$\icias(\dac)=\{\stackrel{\frown}{(i,j)}\mid (i,j)\in\fc{\OME} \}.$$

Denote by $\stackrel{\frown}{{}^{r_1}(i,j){}^{r_2}}$ the grading of $\stackrel{\frown}{(i,j)}$ such that $r_1$ and $r_2$ are the intersection indexes of the two endpoints of $\stackrel{\frown}{{}^{r_1}(i,j){}^{r_2}}$, respectively, meeting $(i,j)$. By \eqref{eq:grad=1}, we have $r_1=r_2$. Then we have
$$\wicias(\dac)=\{\stackrel{\frown}{{}^{r}(i,j){}^{r}}\mid (i,j)\in\fc{\OME},r\in\ZZ \}.$$

For any $\stackrel{\frown}{(i,j)}\in\icias(\dac)$, we use $\stackrel{\curvearrowleft}{(i,j)}$ (resp. $\stackrel{\curvearrowright}{(i,j)}$) to denote its positive (resp. negative) orientation, see Figure~\ref{fig:pn}. Then we have
$$\owicias(\dac)=\{\stackrel{\curvearrowleft}{{}^{r}(i,j){}^{r}},\stackrel{\curvearrowright}{{}^{r}(i,j){}^{r}}\mid (i,j)\in\fc{\OME} \}.$$
\end{notations}

\begin{figure}[htpb]\centering
	\begin{tikzpicture}[xscale=2.5,yscale=.333,rotate=-90]
		\draw[thick, blue,-<-=.5,>=stealth](-1,-.35)tonode[above]{$\stackrel{\curvearrowleft}{(i,j)}$}(-1,.35);
		\draw[thick, blue,->-=.5,>=stealth](-.5,-.37)tonode[below]{$\stackrel{\curvearrowright}{(i,j)}$}(-.5,.37);
		\draw[ultra thick,white]plot [smooth,tension=1] coordinates {(170:4.5) (180:4) (190:4.5)};
		\draw[red,thick]plot [smooth,tension=1] coordinates {(-4,0) (2,.5) (5,0) (2,-.5) (-4,0)} (5,0)node[below]{$(i,j)$};
		\draw[thick](0:3.5)node{$\bullet$};
		\draw[thick,red](180:4)\ww;
	\end{tikzpicture}
	\caption{Oriented interior punctured arc segments}\label{fig:pn}
\end{figure}

\begin{construction}\label{cons:wordtocurve}
	Let $w=w_1 w_2\cdots w_m\in\AW^b\cup\SW^b$. We construct a graded open arc $\arc(w)$ from $w$ in the following way. Locally, for any even $k$ (i.e. $w_k\in S^+$),
	\begin{enumerate}
	\item if $w_k=((i,j^+),r)$, we define $\arc(w_{k-1}w_kw_{k+1})$ to be the gluing (in order) of $w_{k-1}$, $\stackrel{\curvearrowleft}{{}^r(i,j){}^r}$ and $w_{k+1}$, see the first picture in Figure~\ref{fig:2};
	\item if $w_k=((i,j^-),r)$, we define $\arc(w_{k-1}w_kw_{k+1})$ to be the gluing (in order) of $w_{k-1}$, $\stackrel{\curvearrowright}{{}^r(i,j){}^r}$ and $w_{k+1}$, see the second picture in Figure~\ref{fig:2};
	\item otherwise, we define $\arc(w_{k-1}w_kw_{k+1})$ to be the gluing (in order) of $w_{k-1}$ and $w_{k+1}$, see the third picture in Figure~\ref{fig:2}.
	\end{enumerate}
	\begin{figure}[htpb]
	\begin{tikzpicture}[scale=1]
	\draw[ultra thick] (0,0)node{$\bullet$};
	\draw[red,thick] (0,2)to[out=-150,in=180](0,-1.5)to[out=0,in=-30](0,2);
	\draw[blue,thick,-<-=.5,>=stealth] (-2,1)to(-.75,1);
	\draw[blue,thick,-<-=.5,>=stealth] (.75,1)to(2,1);
	\draw[blue] (-1.5,1)node[below]{$w_{k+1}$} (1.5,1)node[below]{$w_{k-1}$} (0,2.5)node{$w_k=((i,j^+),r)$};
	\draw[red] (0,-1.5)node[below]{$(i,j)\simeq(i,j)$};
	\draw[blue,thick,-<-=.5,>=stealth] (-.75,1)to node[below]{$\stackrel{\curvearrowleft}{{}^r(i,j){}^r}$}(.75,1);
	\draw[] (0,2)\ww;
\end{tikzpicture}\qquad
\begin{tikzpicture}[scale=1]
	\draw[ultra thick] (0,0)node{$\bullet$};
	\draw[red,thick] (0,2)to[out=-150,in=180](0,-1.5)to[out=0,in=-30](0,2);
	\draw[blue,thick,->-=.5,>=stealth] (-2,1)to(-.75,1);
	\draw[blue,thick,->-=.5,>=stealth] (.75,1)to(2,1);
	\draw[blue] (-1.5,1)node[below]{$w_{k-1}$} (1.5,1)node[below]{$w_{k+1}$} (0,2.5)node{$w_k=((i,j^-),r)$};
	\draw[red] (0,-1.5)node[below]{$(i,j)\simeq(i,j)$};
	\draw[blue,thick,->-=.5,>=stealth] (-.75,1)to node[below]{$\stackrel{\curvearrowright}{{}^r(i,j){}^r}$}(.75,1);
	\draw[] (0,2)\ww;
\end{tikzpicture}\qquad
\begin{tikzpicture}[scale=1]
	\draw[blue,thick] (-2,1)to(2,1);
	\draw[blue] (-1.5,1)node[below]{$w_{k-1}$} (1.5,1)node[below]{$w_{k+1}$} (0,2.5)node{$w_k\ w_k^\ast$};
	\draw[red] (0,-1.5)node[below]{$(i,j)$};
	\draw[red,thick] (0,2.5)to(0,-1.5)\ww;
\end{tikzpicture}
	\caption{Gluing arc segments}\label{fig:2}
\end{figure}
Let $w=(w_k)_{k\in\ZZ}\in\APW\cup\SPW$ with period $\period$. We construct a graded open curve $\arc(w):S^1\to\surf$ from $w$ as above unless we additionally identify $w_k$ with $w_{k+\period}$ for any $k\in\ZZ$.
\end{construction}

\begin{lemma}\label{lem:bi2}
There is a bijection
    \begin{equation}\label{eq:l.s.}
    \boxed{
        \bione\colon\care\to\OO
    }\end{equation}
sending a graded admissible curve with local system $(\wg,N)$ to $(w,N)$, such that $\arc(w)=\overline{\wg}$. Moreover, this bijection restricts to bijections
$$\Amm\to\AW^b,\ \Amp\to\SW^b,\ \App\to\SPW,\ \ccc\to\APW.$$
\end{lemma}

\begin{proof}
    By a proof similar to the proofs of \cite[Lemmas~4.11 and 4.12]{QZ1}, the map $w\mapsto\arc(w)$ gives bijections
    $$\AW^b\to\overline{\Amm},\ \SW^b\to\overline{\Amp},\ \SPW\to\overline{\App},\ \APW\to\overline{\ccc},$$
    where $\overline{?}$ denotes the set of completions of graded admissible curves in $?$. Then the lemma follows by the bijection between $\wads$ and $\bads$ and by the compatibility between local systems in Definition~\ref{def:wwls} and Definition~\ref{def:cwls}.
\end{proof}

\begin{remark}
    For the four bijections in Lemma~\ref{lem:bi2}, the first and the fourth correspond $w$ to $\arc(w)$, while the second corresponds $w=w_1\cdots w_m$ to the gluing of $\arc(w_1\cdots w_{\frac{m}{2}-1})$ with a non-interior punctured arc segment, and the third corresponds $w=(w_k)_{k\in\ZZ}$ to the gluing $\arc(w_1\cdots w_{\frac{\period}{2}-1})$ with two non-interior punctured arc segments.
\end{remark}

\begin{example}\label{ex:ad cur}
    The admissible curves $\gamma',\gamma'',\gamma''',\gamma''''$ (with certain gradings) shown in Figure~\ref{fig:ad cur} correspond to the words $w',w'',w''',w''''$ in Example~\ref{ex:words}, respectively.

\begin{figure}[htpb]
\begin{tikzpicture}[scale=1,font=\tiny]
\draw[ultra thick](0,0) circle (3.5);
\draw[ultra thick](0,0) circle (1);
\draw[red,thick](0,3.5)to[out=-70,in=-135](1.5,2)to[out=45,in=-20](0,3.5)
    (0,3.5)to[out=-110,in=-45](-1.5,2)to[out=135,in=-160](0,3.5);
\draw[red,thick](-1,0)to[out=-100,in=160](0,-2)to[out=20,in=-80](1,0);
\draw[red,thick](0,-2)to[out=0,in=-90](2.5,.75)to[out=90,in=-10](0,3.5);
\draw[red,thick](0,3.5)to[out=-5,in=90](3,.25)to[out=-90,in=0](0,-3)to[out=180,in=-90](-3,.25)to[out=90,in=-175](0,3.5);
\draw[purple,thick,bend right=10](-1.25,2.25)to(1.25,2.25);
\draw[blue,thick](0,-1)to[out=-10,in=-90](2.2,1)to[out=90,in=-15](1.1,2.6)to[out=165,in=15](-1.1,2.6)to[out=-165,in=90](-1.5,2.25)to[out=-90,in=165](-1.1,1.9)to[out=-15,in=-165](1.1,1.9)to[out=15,in=-110](1.5,2.25)to[out=70,in=-15](1.1,2.8)to[out=165,in=15](-1.1,2.8)to[out=-165,in=90](-2.4,.9)to[out=-90,in=-170](0,-1);
\draw[cyan,thick](0,-1)to[out=0,in=-90](2,1)to[out=90,in=-15](1.1,2.5)to[out=165,in=20](-1.25,2.25);
\draw(-2.2,1)node[blue]{$\gamma'$}(1.5,0)node[cyan]{$\gamma''$}(0,2.3)node[purple]{$\gamma''''$}(1.7,-2.2)node[orange]{$\gamma'''$};
\draw[orange,thick](0,-1.5)to[out=180,in=90](-.5,-2)to[out=-90,in=170](0,-2.5)to[out=-10,in=-90](2,-.5)to[out=90,in=0](0,1.5)to[out=180,in=90](-2,-.5)to[out=-90,in=-170](0,-2.5)to[out=10,in=-90](.5,-2)to[out=90,in=0](0,-1.5);
\draw(-1.25,2.25)\nn (1.25,2.25)\nn (-1,0)\ww (1,0)\ww
    (0,-2)\ww (0,3.5)\ww (0,1)[blue]\nn (0,-1)[blue]\nn (0,-3.5)[blue]\nn;
\end{tikzpicture}
\caption{Examples of admissible curves}
\label{fig:ad cur}
\end{figure}
\end{example}

Combining Lemma~\ref{lem:bi2} and Theorem~\ref{thm:Deng}, we have the following classification of indecomposable representations in $\repb(S)$ via graded admissible curves with local system.

\begin{corollary}\label{cor:Deng}
    There is a bijection
    $$\boxed{R\circ\bione\colon\care\to\Ind\repb(S)},$$
    where $\Ind\repb(S)$ is the set of isoclasses of indecomposable representations in $\repb(S)$.
\end{corollary}

\section{Indecomposables in the perfect derived category of a GSGA}\label{sec:class}

Let $\dac$ be a full formal closed arc system of a GMSp $\gms$
and $\sg$ the associated graded skew-gentle algebra. In this section, we will classify all indecomposable objects in the perfect derived category $\per\sg$ of $\sg$ using this geometric model when $\sg$ is non-positive. When $\sg$ is not necessarily non-positive, we classify a certain class of indecomposable objects.

\subsection{Perfect derived categories}

We follow \cite{Ke94}. For a differential graded $\k$-space $M$, we denote by $|M|=\oplus_{i\in\ZZ}M^i$ the underlying graded $\k$-space and $d_M$ the differential.

Let $C=(|C|,d_C)$ be a differential graded $\k$-algebra (or dg algebra for short). For any two dg $C$-modules $M=(|M|,d_M)$ and $N=(|N|,d_N)$, the differential graded $\k$-space $$\huaHom_C(M,N)=(\bigoplus_{i\in\ZZ}\huaHom^i(M,N),d),$$
where $\huaHom^i(M,N)$ is the subspace of $\prod_{j\in\ZZ}\Hom_\k(M^j,N^{j+i})$ consisting of morphisms $f$ such that $f(mc)=f(m)c$, for all $m\in M$ and all $c\in C$, and the differential $d$ on $\huaHom_C(M,N)$ is given by
$$d(f)=f\circ d_M-(-1)^{|f|}d_N\circ f,$$
for a homogeneous morphism $f$ of degree $|f|$.
We denote by
\begin{itemize}
    \item $\C(C)$ the category of dg $C$-modules, whose objects are the dg $C$-modules, and whose morphism spaces are $Z^0\huaHom_C(M,N)$;
    \item $\H(C)$ the homotopy category of dg $C$-modules, whose objects are the dg $C$-modules, and whose morphism spaces are $H^0\huaHom_C(M,N)$;
    \item $\mathcal{D}(C)$ the derived category of dg $C$-modules, which is the localization of $\H(C)$ at the full subcategory of acyclic dg $C$-modules;
    \item $\operatorname{per}(C)$ the perfect derived category of dg $C$-modules, which is the smallest full subcategory of $\mathcal{D}(C)$ containing $C$ and closed under taking shifts, extensions and direct summands.
\end{itemize}

Suppose that the dg algebra $C$ is finite-dimensional (i.e. $\dim_\k|C|<\infty$), and have zero differential (i.e., $d_C=0$). Let $1=e_1+\cdots+e_n$ be a decomposition of the unity into a sum of primitive orthogonal idempotents. A dg $C$-module $M=(|M|,d_M)$ is called \emph{strictly perfect}, provided that there is a decomposition
\begin{itemize}
    \item $|M|=\bigoplus_{l=1}^t R_l$ for some $t\in\NN$, where $R_l=e_{a_l}C[b_l]$ for some $1\leq a_l\leq n$ and for some $b_l\in\ZZ$, and
    \item $d_M=(f_{l,l'})_{1\leq l,l'\leq t}$ is a $t\times t$ strictly upper triangular matrix, where $f_{l,l'}:R_{l'}\to R_{l}$ is of degree $1$.
\end{itemize}
Such a strictly perfect dg $C$-module $M$ is called  \emph{minimal} if each entry $f_{l,l'}:R_{l'}\to R_l$ of $d_M$ is in the radical. By definition, strictly perfect dg $C$-modules have property (P) (see \cite[Section~3.1]{Ke94} for the definition).

The dg algebra $C$ is called \emph{non-positive} if $C^i=0$ for any $i>0$. Note that $\sg$ is non-positive if and only if so is $H$.

\begin{lemma}[{\cite[Lemma~4.2]{KY}}]\label{lem:KY}
    Suppose that $C$ is non-positive. Then any object in $\per C$ is isomorphic to a minimal strictly perfect dg $C$-module.
\end{lemma}

Let $D$ be another basic finite-dimensional dg algebra with zero differential, and $f: C\to D$ a homomorphism of dg algebras. Then $D$ can be regarded as a dg $C$-module via the action given by $f$. There is an induced left derived tensor functor
$$-\otimes^{\mathbf{L}}_{C}D:\per C\to\per D,$$
which acts on strictly perfect dg $C$-modules and morphisms between them as the usual tensor functor $-\otimes_{C}D$ (cf. \cite[Section~6.4]{Ke94}).

\subsection{Triple categories}\label{subsec:triple to bush}

The main idea in this subsection is from \cite{BD}. Denote by $\overline{H}=H/\mbox{rad}(H)$ and $\overline{\sg}=\sg/\mbox{rad}(\sg)$. The category of triples $\Tri\sg$ is defined as follows.
\begin{itemize}
	\item An object is a triple $(Y^\bullet,V^\bullet,\theta)$, where
	\begin{itemize}
		\item $Y^\bullet\in \per H$,
		\item $V^\bullet\in \per\overline{\sg}$, and
		\item $\theta:Y^\bullet\otimes^{\mathbf{L}}_{H}\overline{H}\to V^\bullet\otimes^{\mathbf{L}}_{\overline{\sg}}\overline{H}$ is an isomorphism in $\per\overline{H}$.
	\end{itemize}
	\item A morphism from $(Y_1^\bullet,V_1^\bullet,\theta_1)$ to $(Y_2^\bullet,V_2^\bullet,\theta_2)$ is given by a pair $(g,h)$, where
	\begin{itemize}
		\item $g:Y^\bullet_1\to Y^\bullet_2$ is a morphism in $\per H$.
		\item $h:V_1^\bullet\to V_2^\bullet$ is a morphism in $\per\overline{\sg}$.
	\end{itemize}
	such that the following diagram is commutative in $\per\overline{H}$:
	\begin{equation}\label{eq:comm}
	\xymatrix{
	Y^\bullet_1\otimes^{\mathbf{L}}_H\overline{H}\ar[r]^{\theta_1}\ar[d]_{g\otimes^{\mathbf{L}}_H {\overline{H}}}&V^\bullet_1\otimes^{\mathbf{L}}_{\overline{\sg}}\overline{H}\ar[d]^{h\otimes^{\mathbf{L}}_{\overline{\sg}}{\overline{H}}}\\
	Y_2^\bullet\otimes^{\mathbf{L}}_H \overline{H}\ar[r]_{\theta_2}&V^\bullet_2\otimes^{\mathbf{L}}_{\overline{\sg}}\overline{H}
    }
	\end{equation}
\end{itemize}

We shall establish a relationship between the triple category $\Tri\sg$ and the category $\repb(S)$.

First, we investigate $\per H$. Recall from Remark~\ref{rmk:id} that $\OMEpm$ is identified with a complete set of pairwise orthogonal primitive idempotents of $H$, and from Remark~\ref{rmk:BR} that
we have a basis $\{\lu(x_1,x_2)\mid (x_1,x_2)\in\mathcal{B} \}$ of $H$, where $$\mathcal{B}=\{((i,j_1^{\kappa_1}),(i,j_2^{\kappa_2}))\in\OMEpm\times\OMEpm\mid 1\leq i\leq t,1\leq j_1\leq j_2\leq m_i \}.$$
For any $(x_1,x_2)\in\mathcal{B}$, we denote
$$\begin{array}{rccc}
	f^H_{x_1,x_2}:&x_2H&\to&x_1H\\
	&h&\mapsto&\lu(x_1,x_2)h
\end{array}.$$
the morphism in $\huaHom_H(x_2H,x_1H)$ between indecomposable direct summands of $H$ induced by $\lu(x_1,x_2)$. In particular, we have $f^H_{x,x}=\id_{xH}$ for any $x\in\OMEpm$.

\begin{remark}\label{rmk:fH}
Note that $f^H_{x_1,x_2}$ is indeed a basis of $\huaHom_H(x_2H,x_1H)$. Moreover, since $d_H=0$, the differential of $\huaHom_H(x_2H,x_1H)$ is also zero. So we have $$\huaHom_H(x_2H,x_1H)=H^\ast\huaHom_H(x_2H,x_1H)=\bigoplus_{\rho\in\ZZ}\Hom_{\per H}(x_2H,x_1H[\rho]).$$
Hence, $f^H_{x_1,x_2}$ is also a basis of $\Ext^\ast_{\per H}(x_2 H,x_1 H):=\bigoplus_{\rho\in\ZZ}\Hom_{\per H}(x_2 H,x_1 H[\rho])$.
\end{remark}

Since $(i,j^+)H\cong (i,j^-)H$ for any $(i,j)\in\fc{\OME}$, the set of isoclasses of indecomposable direct summands of $H$ is indexed by $\OME$.

\begin{notations}
For any $(i,j)\in\OME$, we denote
$$H_{(i,j)}=\begin{cases}
(i,j)H&\text{if $(i,j)\notin\fc{\OME}$,}\\
(i,j^+)H&\text{if $(i,j)\in\fc{\OME}$,}
\end{cases}$$
and for any $(i,j_1),(i,j_2)\in\OME$ with $j_1\leq j_2$, denote
$$f^H_{(i,j_1),(i,j_2)}:=f^H_{(i,j_1^{\kappa_1}),(i,j_2^{\kappa_2})}:H_{(i,j_2)}\to H_{(i,j_1)},$$
where $(i,j_1^{\kappa_1}),(i,j_2^{\kappa_2})\in\OMEpm$ with $\kappa_1,\kappa_2\in\{\emptyset,+\}$. For the convenience, we take $H_{(i,0)}=0$ and $f^H_{(i,0),(i,j)}=0$; we also take $f^H_{(i,j_1),(i,j_2)}=0$ for the case $j_1>j_2$. Sometimes, we also use $f^H_{(i,j_1),(i,j_2)}$ to denote its induced morphism of certain degree between the shifts of $H_{(i,j_2)}$ and $H_{(i,j_1)}$, if there is no confusion arising.
\end{notations}

Similarly, the set of isoclasses of indecomposable direct summands of $\overline{H}$ is also indexed by $\OME$. For any $(i,j)\in\OME$, denote
$$\overline{H}_{(i,j)}=H_{(i,j)}\otimes_H \overline{H}.$$

\begin{remark}
    Since $\overline{H}$ is semisimple, the perfect category $\per\overline{H}$ is semisimple and whose simples, up to isomorphism, are $\overline{H}_{(i,j)}[r], (i,j)\in\OME,r\in\ZZ$. Hence there is a natural equivalence
    \begin{equation}\label{eq:simeq}
        \fMM\colon\per\overline{H}\xrightarrow{\simeq}(\OME\times\ZZ)\text{-}\gvec,
    \end{equation}
    sending $\overline{H}_{(i,j)}[r]$ to a homogeneous 1-dimensional vector space $\k_{((i,j),r)}$ of degree $((i,j),r)$.
    Recall from Remark~\ref{rmk:rep and mor} that $(\OME\times\ZZ)\text{-}\gvec$ is the category of finite-dimensional $(\OME\times\ZZ)$-graded vector spaces.
\end{remark}

Now we study the relationship between $\per H$ and $\add\k\overline{S^-}$.

\begin{construction}\label{cons:W}
For any graded unpunctured arc segment $\wmu=(i,j_1){}^{r_1}-{}^{r_2}(i,j_2)\in\wupas(\dac)$, with $j_1<j_2$, we define a dg $H_i$-module (and hence becomes a dg $H$-module) $$W^\bullet(\wmu)=\left(H_{(i,j_1)}[r_1]\oplus H_{(i,j_2)}[r_2],\begin{pmatrix}
    0&f^H_{(i,j_1),(i,j_2)}\\0&0
\end{pmatrix}\right).$$
\end{construction}

By the construction, $W^\bullet(\wmu)$ is a minimal strictly perfect dg $H$-module.

\begin{remark}\label{rmk:geo}
    Each category $\per H_i$ has a geometric model by the unpunctured $\dac$-polygon $D_i$ (cf. e.g. \cite{HKK}) as follows.
    \begin{itemize}
        \item The indecomposable objects are $W^\bullet(\wmu)$,  $\wmu\in\Was(D_i)$.
        \item For any $\wmu,\wnu\in\Was(D_i)$, using Notations~\ref{not:as} and \ref{not:grad}, we write $\wmu=(i,j_1){}^{r_1}-{}^{r_2}(i,j_2)$ and $\wnu=(i,j_3){}^{r_3}-{}^{r_4}(i,j_4)\in$ with $j_1<j_2$ and $j_3<j_4$. The morphism space $\Hom_{\per H_i}(W^\bullet(\wmu), W^\bullet(\wnu))$ is nonzero if and only if it is 1-dimensional if and only if one of the following cases occurs.
    \begin{enumerate}[label={(\alph*)}]
        \item\label{item:a} [\textbf{Non-crossing case}] There exist $l\in\{1,2\}$, $h\in\{3,4\}$ such that $j_l=j_h$, $r_l=r_h$ and $\wmu$ is to the left of $\wnu$ when starting at $(i,j_l)=(i,j_h)$. This case is divided into the following subcases.
        \begin{enumerate}[label={(a\arabic*)}]
            \item\label{item:a1} $l=1$, $h=3$ and $j_2\geq j_4$, see the left picture of Figure~\ref{fig:typeA}. A basis is
            \[\begin{pmatrix}
            \id&0\\0&f^H_{(i,j_4),(i,j_2)}
            \end{pmatrix}:H_{(i,j_1)}[r_1]\oplus H_{(i,j_2)}[r_2]\to H_{(i,j_3)}[r_3]\oplus H_{(i,j_4)}[r_4],\]
            where if $j_1=j_3=0$, we remove the identity.
            \item\label{item:a2} $l=2$, $h=3$ and $j_1< j_4$, see the middle picture of Figure~\ref{fig:typeA}. A basis is
            \[\begin{pmatrix}
            0&\id\\ 0&0
            \end{pmatrix}:H_{(i,j_1)}[r_1]\oplus H_{(i,j_2)}[r_2]\to H_{(i,j_3)}[r_3]\oplus H_{(i,j_4)}[r_4].\]
            \item\label{item:a3} $l=2$, $h=4$ and $j_1\geq j_3$, see the right picture of Figure~\ref{fig:typeA}. A basis is
            \[\begin{pmatrix}
            f^H_{(i,j_3),(i,j_1)}&0\\ 0&\id
            \end{pmatrix}:H_{(i,j_1)}[r_1]\oplus H_{(i,j_2)}[r_2]\to H_{(i,j_3)}[r_3]\oplus H_{(i,j_4)}[r_4].\]
        \end{enumerate}
        \begin{figure}[htpb]
             \begin{tikzpicture}[xscale=1.6,yscale=1.4]
            	\begin{scope}[shift={(-3,0)}]
            		\draw[ultra thick](-1,1)to(1,1);
            		\draw[red,thick,dotted](-.4,1)to(-.8,.6)  (-.4,-1)to(.4,-1) (.8,-.6)to(1,0) (.8,.6)to(.4,1);
            		\draw[red,thick](-.8,.6)to(-.4,-1) (.4,-1)to(.8,-.6) (1,0)to(.8,.6);
            		\draw[blue](0,1)\nn;
            		\draw[red](-.65,0)node[left]{$j_1$}(-.55,-.4)node[left]{$j_3$}(1.1,.3)node{$j_2$}(.7,-.9)node{$j_4$};
            		
            		\draw[blue,thick,bend right=20] (.6,-.8)tonode[below]{$\wnu$}(-.55,-.4) (.9,.3)tonode[above]{$\wmu$}(-.65,0);
            		\draw[orange,thick,bend left=50,->-=1,>=stealth](-.6,0)tonode[right]{$\id$}(-.5,-.4);
                    \draw[orange,thick,bend left=20,-<-=.1,>=stealth](.6,-.8)to(.9,.3);
            	\end{scope}
            	\draw[ultra thick](-1,1)to(1,1);
            	\draw[red,thick,dotted](-.4,1)to(-.8,.6) (-1,0)to(-.8,-.6) (.8,-.6)to(1,0) (.8,.6)to(.4,1);
            	\draw[red,thick](-.8,.6)to(-1,0) (-.8,-.6)to(.8,-.6) (1,0)to(.8,.6);
            	\draw[blue](0,1)\nn;
            	\draw[red](-1.1,.3)node{$j_1$}(1.1,.3)node{$j_4$}(0,-.8)node{$j_2=j_3$};
                \draw[blue,thick,bend right=20](-.2,-.6)tonode[above]{$\wmu$}(-.9,.3);
                \draw[blue,thick,bend left=20](.2,-.6)tonode[above]{$\wnu$}(.9,.3);
            	\draw[orange,thick,->-=1,>=stealth,bend left=50](-.2,-.6)tonode[above]{$\id$}(.2,-.6);
            	\begin{scope}[shift={(3,0)}]
            		\draw[ultra thick](1,1)to(-1,1);
            		\draw[red,thick,dotted](.4,1)to(.8,.6)  (.4,-1)to(-.4,-1) (-.8,-.6)to(-1,0) (-.8,.6)to(-.4,1);
            		\draw[red,thick](.8,.6)to(.4,-1) (-.4,-1)to(-.8,-.6) (-1,0)to(-.8,.6);
            		\draw[blue](0,1)\nn;
            		\draw[red](.65,0)node[right]{$j_4$}(.55,-.4)node[right]{$j_2$}(-1.1,.3)node{$j_3$}(-.7,-.9)node{$j_1$};
            		
            		\draw[blue,thick,bend left=20] (-.6,-.8)tonode[below]{$\wmu$}(.55,-.4) (-.9,.3)tonode[above]{$\wnu$}(.65,0);
            		\draw[orange,thick,bend right=50,->-=1,>=stealth](.6,0)tonode[left]{$\id$}(.5,-.4);
                    \draw[orange,thick,bend right=20,-<-=.1,>=stealth](-.6,-.8)to(-.9,.3);
            	\end{scope}
            \end{tikzpicture}
            \caption{Subcases \ref{item:a1}, \ref{item:a2} and \ref{item:a3} in Remark~\ref{rmk:geo}}
            \label{fig:typeA}
        \end{figure}
        \item\label{item:b} [\textbf{Crossing case}] $j_1,j_2,j_3,j_4$ are different from each other, and there is an oriented intersection from $\wmu$ to $\wnu$ of index $0$, cf. Figure~\ref{fig:rad2}. This case is divided into the following subcases.
        \begin{enumerate}[label={(b\arabic*)}]
            \item\label{item:b1} If $0\leq j_1<j_3<j_2<j_4$, a basis is
            \[\begin{pmatrix}
            0&f^H_{(i,j_3),(i,j_2)}\\0&0
            \end{pmatrix}:H_{(i,j_1)}[r_1]\oplus H_{(i,j_2)}[r_2]\to H_{(i,j_3)}[r_3]\oplus H_{(i,j_4)}[r_4].\]
            \item\label{item:b2} If $0\leq j_3<j_1<j_4<j_2$, a basis is
            \[\begin{pmatrix}
            f^H_{(i,j_3),(i,j_1)}&0\\0&f^H_{(i,j_4),(i,j_2)}
             \end{pmatrix}:H_{(i,j_1)}[r_1]\oplus H_{(i,j_2)}[r_2]\to H_{(i,j_3)}[r_3]\oplus H_{(i,j_4)}[r_4].\]
        \end{enumerate}
        \begin{figure}[htpb]
             \begin{tikzpicture}[xscale=1.6,yscale=1.4]
            	\begin{scope}[shift={(-3,0)}]
            		\draw[ultra thick](-1,1)to(1,1);
            		\draw[red,thick,dotted](-.4,1)to(-.8,.6) (-1,0)to(-.8,-.6) (-.4,-1)to(.4,-1) (.8,-.6)to(1,0) (.8,.6)to(.4,1);
            		\draw[red,thick](-.8,.6)to(-1,0) (-.8,-.6)to(-.4,-1) (.4,-1)to(.8,-.6) (1,0)to(.8,.6);
            		\draw[blue](0,1)\nn;
            		\draw[red](-1.2,.3)node{$(i,j_1)$}(-.9,-1)node{$(i,j_3)$}(1.2,.3)node{$(i,j_4)$}(.9,-1)node{$(i,j_2)$};
            		
            		\draw[blue,thick,bend right=20](.9,.3)to(-.6,-.8) (.6,-.8)to(-.9,.3) (.5,.3)node{$\tilde{\nu}$}(-.5,.3)node{$\tilde{\mu}$};
            		\draw[orange,thick,bend left=20,-<-=.1,>=stealth](-.6,-.8)to(.6,-.8);
            	\end{scope}
            	\draw[ultra thick](-1,1)to(1,1);
            	\draw[red,thick,dotted](-.4,1)to(-.8,.6) (-1,0)to(-.8,-.6) (.8,-.6)to(1,0) (.8,.6)to(.4,1);
            	
            	\draw[red,thick](-.8,.6)to(-1,0) (-.8,-.6)to(.8,-.6) (1,0)to(.8,.6);
            	
            	\draw[blue](0,1)\nn;
            	\draw[red](-1.2,.3)node{$(i,j_1)$}(1.2,.3)node{$(i,j_2)$}(0,1.2)node{$(i,j_3=0)$}(0,-.8)node{$(i,j_4)$};
            	
            	\draw[blue,thick](0,1)to(0,-.6);
            	\draw[blue,thick,bend right=20](.9,.3)to(-.9,.3) (-.2,.1)node{$\tilde{\nu}$}(.5,.6)node{$\tilde{\mu}$};
            	\draw[orange,thick,->-=1,>=stealth,bend left=10](.9,.3)to(0,-.6);
            	\begin{scope}[shift={(3,0)}]
            		\draw[ultra thick](-1,1)to(1,1);
            		\draw[red,thick,dotted](-.4,1)to(-.8,.6) (-1,0)to(-.8,-.6) (-.4,-1)to(.4,-1) (.8,-.6)to(1,0) (.8,.6)to(.4,1);
            		\draw[red,thick](-.8,.6)to(-1,0) (-.8,-.6)to(-.4,-1) (.4,-1)to(.8,-.6) (1,0)to(.8,.6);
            		\draw[blue](0,1)\nn;
            		\draw[red](-1.2,.3)node{$(i,j_3)$}(-.9,-1)node{$(i,j_1)$}(1.2,.3)node{$(i,j_2)$}(.9,-1)node{$(i,j_4)$};
            		
            		\draw[blue,thick,bend right=20](.9,.3)to(-.6,-.8) (.6,-.8)to(-.9,.3) (.5,.3)node{$\tilde{\mu}$}(-.5,.3)node{$\tilde{\nu}$};
            		\draw[orange,thick,bend right=10,->-=1,>=stealth](.9,.3)to(.6,-.8);
            		\draw[orange,thick,bend right=10,->-=1,>=stealth](-.6,-.8)to(-.9,.3);
            	\end{scope}
            \end{tikzpicture}
            \caption{Cases for \ref{item:b} for crossing $\wmu$ and $\wnu$}
            \label{fig:rad2}
        \end{figure}
    \end{enumerate}
    \end{itemize}

    It is straightforward to see that the morphisms in case (b) linearly generate an ideal of $\per H_i$.
\end{remark}

\begin{lemma}\label{lem:full}
There is a full and dense functor
\[\boxed{
    \fF:\per H\to{\add\bushf},
}\]
sending $W^\bullet(\wmu)$ to $\wmu$ for any $\wmu\in\wupas(\dac)$, such that the following properties hold.
\begin{enumerate}
    \item The functor $\fF$ gives rise to a bijection between isoclasses of indecomposable objects in these two categories.
    \item The following commutative diagram commutes
    \begin{equation}\label{eq:zm}
    \begin{tikzcd}[column sep=35,row sep=30]
      \per H \ar[r,"\fF"] \ar[d,"-\otimes_{H}^{\mathbf{L}}{\id_{\overline{H}}}"']    
        & \add\k\overline{S^-} \ar[d,"\bvs^-"]\\
      \per \overline{H} \ar[r,"\fMM","\eqref{eq:simeq}"']& (\OME\times\ZZ)\text{-}\gvec
    \end{tikzcd}
    \end{equation}
\end{enumerate}
\end{lemma}

\begin{proof}
    Recall from \eqref{eq:id} that the indecomposable objects in $\add\bushf$ are identified with $\wupas(\dac)$. Denote by ${\bushf}_i$ the full subcategory of $\bushf$ whose object set is $\Was(D_i)$. Since arc segments in different polygons are incomparable, the category $\add\bushf$ is the direct sum of $\add\bushf_i$. On the other hand, $\per H$ is also the direct sum of $\per H_i$. So we only need to construct a functor $\fF_i:\per H_i\to\add{\bushf}_i$ for each $1\leq i\leq t$, satisfying the properties in the lemma.
	
    By the geometric model of $\per H_i$ in Remark~\ref{rmk:geo}, there is a full and dense functor $\fF_i:\per H_i\to \add{\bushf}_i$ satisfying property (1) and whose kernel is linearly generated by the morphisms in case (b).
	
	For property (2), for any $\wmu=(i,j_1){}^{r_1}-{}^{r_2}(i,j_2)\in\Was(D_i)$, we have
    $$W^\bullet(\wmu)\otimes_H \overline{H}=\overline{H}_{(i,j_1)}[r_1]\oplus\overline{H}_{(i,j_2)}[r_2],$$
    and
    $$\dim\bvs_{(i,j),r}^-(\wmu)=\begin{cases}
    1&\text{if $((i,j),r)=((i,j_1),r_1)$ or $((i,j_2),r_2)$,}\\0&\text{otherwise.}
    \end{cases}$$
    So we have
    $$W^\bullet(\wmu)\otimes_{H}\overline{H}=\bigoplus_{((i,j),r)\in\OME\times\ZZ}(\overline{H}_{(i,j)}[r])^{\dim\bvs_{(i,j),r}^-(\wmu)},$$
    which implies the commutative diagram~\eqref{eq:zm} on the object level.

    On the morphism level, we only need to consider the basis, denoted by $g$, given in Remark~\ref{rmk:geo}.
    \begin{itemize}
    \item In case \ref{item:a1}, there are the following subcases.
    \begin{itemize}
    \item $j_2=j_4$ (then $r_2=r_4$). Then $g$ is the identity. So both $g\otimes_{H}\id_{\overline{H}}$ and $\bvs^-(\fF_i(g))$ are the identity.
    \item If $j_1=j_3=0$ and $j_2\neq j_4$, we have $g\otimes_{H}\id_{\overline{H}}=0$. On the other hand, $\fF_i(g)=\left(\left((i,j_3){}^{r_3}\to{}^{r_4}(i,j_4)\right)\geq\left((i,j_1){}^{r_1}\to{}^{r_2}(i,j_2)\right)\right)$ belongs to $S^-_{(i,0),r_1}$. Since $((i,0),r_1)\notin\Omega\times\ZZ$, we have $\bvs^-(\fF_i(g))=0$.
    \item If $j_1=j_3\neq 0$ and $j_2\neq j_4$, we have
    \[g\otimes_{H}\id_{\overline{H}}=\begin{pmatrix}
        \id &0\\0&0            \end{pmatrix}
    :\overline{H}_{(i,j_1)}[r_1]\oplus\overline{H}_{(i,j_2)}[r_2]\to \overline{H}_{(i,j_3)}[r_3]\oplus\overline{H}_{(i,j_4)}[r_4].\]
    On the other hand, $\fF_i(g)=\left(\left((i,j_3){}^{r_3}\to{}^{r_4}(i,j_4)\right)\geq\left((i,j_1){}^{r_1}\to{}^{r_2}(i,j_2)\right)\right)$. So
    \[\bvs^-(\fF_i(g))=\begin{pmatrix}
        \id &0\\0&0            \end{pmatrix}
    :\k_{((i,j_1),r_1)}\oplus\k_{((i,j_2),r_2)}\to \k_{((i,j_3),r_3)}\oplus\k_{((i,j_4),r_4)}.\]
    \end{itemize}
    \item In case \ref{item:a2}, we have
    \[g\otimes_{H}\id_{\overline{H}}=\begin{pmatrix}
        0 &\id\\0&0            \end{pmatrix}
    :\overline{H}_{(i,j_1)}[r_1]\oplus\overline{H}_{(i,j_2)}[r_2]\to \overline{H}_{(i,j_3)}[r_3]\oplus\overline{H}_{(i,j_4)}[r_4].\]
    On the other hand, $\fF_i(g)=\left(\left((i,j_3){}^{r_3}-{}^{r_4}(i,j_4)\right)\geq\left((i,j_2){}^{r_2}\to{}^{r_1}(i,j_1)\right)\right)$. So
    \[\bvs^-(\fF_i(g))=\begin{pmatrix}
        0 & \id\\0&0            \end{pmatrix}
    :\k_{((i,j_1),r_1)}\oplus\k_{((i,j_2),r_2)}\to \k_{((i,j_3),r_3)}\oplus\k_{((i,j_4),r_4)}.\]
    \item Case \ref{item:a3} is similar with case \ref{item:a1}. So we omit the details.
    \item In case \ref{item:b}, since $g$ is in the kernel of $\fF_i$, we have $\fF_i(g)=0$. On the other hand, since each component of $g$ is in the radical, we have $g\otimes_{H}\id_{\overline{H}}=0$.
    \end{itemize}
    In each case, we have $\bvs^-(\fF_i(g))=\fMM(g\otimes_{H}\id_{\overline{H}})$. Thus, the proof is complete.
\end{proof}


Next, we investigate the relationship between $\per\overline{\sg}$ and $\add\k\overline{S^+}$.

\begin{lemma}\label{lem:equiv}
    There is an equivalence
    \[\boxed{
    \fG:\per\overline{\sg}\to\add\bushz,
    }\]
    sending $z\overline{\sg}[r]$ to $(z,r)$ for any $(z,r)\in\OMEx\times\ZZ$, such that the following diagram commutes
    \begin{equation}\label{eq:fm}
    \begin{tikzcd}[column sep=35,row sep=30]
      \per \overline{\sg} \ar[r,"\fG"] \ar[d,"-\otimes_{H}^{\mathbf{L}}{\id_{\overline{H}}}"']
        & \add\k\overline{S^+} \ar[d,"\bvs^+"]\\
      \per \overline{H} \ar[r,"\fMM","\eqref{eq:simeq}"']& (\OME\times\ZZ)\text{-}\gvec
    \end{tikzcd}    
    \end{equation}
\end{lemma}

\begin{proof}
    Since $\overline{\sg}=\sg/\rad(\sg)$ is semisimple and has a complete set of orthogonal primitive idempotents identified with $\OMEx$, the isoclasses of the indecomposable objects in $\per\overline{\sg}$ are indexed by $\OMEx\times\ZZ$ and there are no radical morphisms in $\per\overline{\sg}$. Since $S^+$ has trivial order and $\overline{S^+}=\OMEx\times\ZZ$ (see \eqref{eq:id}), the category $\add\bushz$ has the same property as $\per\overline{\sg}$. Hence we have the required equivalence, where \eqref{eq:fm} follows from that for any $z\in\OMEx$, we have $$z\overline{\sg}\otimes_{\overline{\sg}} \overline{H}= \begin{cases} \overline{H}_{(i_1,j_1)}\oplus\overline{H}_{(i_2,j_2)}&\text{if }z=\{(i_1,j_1),(i_2,j_2) \} \\ \overline{H}_{(i,j)}&\text{if $z=(i,j^+)$ or $(i,j^-)$.}\end{cases}$$.
\end{proof}


Finally, we establish a functor from
$\Tri\sg$ to $\repb(S)$.

\begin{theorem}\label{thm:tribu}
    There is a full and dense functor
    \begin{gather}\label{eq:tribu}\boxed{
	\fM:\Tri\sg\to \repb(S),
	}\end{gather}
    which sends an object $(Y^\bullet,V^\bullet,\theta)$ to $(\fF(Y^\bullet),\fG(V^\bullet),\fMM(\theta))$, and which sends a morphism $(g,h)$ to $(\fF(g),\fG(h))$. Moreover, the following hold.
	\begin{enumerate}
		\item For any objects $T_1,T_2$ in $\Tri\sg$, we have $\fM(T_1)\cong \fM(T_2)$ if and only if $T_1\cong T_2$.
		\item For any object $T$ in $\Tri\sg$, $\fM(T)$ is indecomposable if and only if $T$ is indecomposable.
	\end{enumerate}
\end{theorem}

\begin{proof}
    Fisrt, $(\fF(g),\fG(h))$ is a morphism in $\repb(S)$ due to the commutative diagrams~\eqref{eq:comm}, \eqref{eq:zm} and \eqref{eq:fm}. Then it is straightforward to show that $\fM$ is a functor. Since by Lemmas~\ref{lem:full} and \ref{lem:equiv}, the functors $\fF$ and $\fG$ are full and dense, so is $\fM$ by the construction.
	
    To show (1), for any object $T$ in $\Tri\sg$ and any morphism $(g,h)\in\End_{\Tri\sg}(T)$, if $\fM((g,h))=(\fF(g),\fG(h))=0$, then $h=0$ since $\fG$ is an equivalence, and $g$ is nilpotent since the kernel of $\fF$ is the ideal spanned by the morphisms induced by the interior oriented intersections. So $(g,h)$ is nilpotent. Assume that $\fM(T_1)\cong \fM(T_2)$ in $\repb(S)$. Then we have mutually inverse isomorphism $\varphi: \fM(T_1)\to \fM(T_2)$ and $\psi: \fM(T_2)\to \fM(T_1)$. Since $\fM$ is full, there exist $\widetilde{\varphi}:T_1\to T_2$ and $\widetilde{\psi}:T_2\to T_1$ such that $\varphi=\fM(\widetilde{\varphi})$ and $\psi=\fM(\widetilde{\psi})$. Since $\fM(\id_{T_1}-\widetilde{\psi}\widetilde{\varphi})=0$, the endomorphism $\id_{T_1}-\widetilde{\psi}\widetilde{\varphi}$ is nilpotent. So $\widetilde{\psi}\widetilde{\varphi}$ is an isomorphism. Similarly, we have $\widetilde{\varphi}\widetilde{\psi}$ is also an isomorphism. Hence we have $T_1\cong T_2$. (2) follows from (1) and the denseness of $\fM$.
\end{proof}

\subsection{Classification of indecomposable objects in the perfect derived category}\label{subsec:class}\

\def\tw{\per^s}
\def\glue{\operatorname{Tri}^s}
\def\cares{\widetilde{\operatorname{OC}}_{\text{l.s.}}^s(\gms)}

Let $(Y^\bullet,V^\bullet,\theta)$ be an object of $\Tri\sg$ such that $Y^\bullet$ and $V^\bullet$ are minimal strictly perfect (dg $H$-module and dg $\overline{\sg}$-module, respectively). Then we have $$Y^\bullet\otimes^{\mathbf{L}}_H\overline{H}=Y^\bullet\otimes_H\overline{H}\ \text{and}\ V^\bullet\otimes^{\mathbf{L}}_{\overline{\sg}}\overline{H}=V^\bullet\otimes_{\overline{\sg}}\overline{H}.$$ Since $\overline{\sg}$ and $\overline{H}$ are semisimple, the differentials of $V^\bullet$, $Y^\bullet\otimes_H \overline{H}$ and $V^\bullet\otimes_{\overline{\sg}} \overline{H}$ are zero. Hence the isomorphism $\theta:Y^\bullet\otimes_H \overline{H}\to V^\bullet\otimes_{\overline{\sg}} \overline{H}$ in $\per\overline{H}$ is indeed an isomorphism of graded $\overline{H}$-modules. Let $\theta^{-1}$ be the inverse of $\theta$ and take
$$\widetilde{\theta}=\theta^{-1}\circ \iota_{\overline{\sg}}^{\overline{H}}: V^\bullet \to Y^\bullet\otimes_H \overline{H},$$
where $\iota_{\overline{\sg}}^{\overline{H}}: V^\bullet \to V^\bullet\otimes_{\overline{\sg}}{\overline{H}}$ stands for the canonical map given by $-\otimes_{\overline{\sg}} 1_{\overline{H}}$ (similar for below). We define an object $\fK(Y^\bullet,V^\bullet,\theta)$ in $\C(\sg)$ associate with $(Y^\bullet,V^\bullet,\theta)$ to be
\begin{equation}\label{eq:defK}
    \fK(Y^\bullet,V^\bullet,\theta)=\operatorname{ker}\left(Y^\bullet\oplus V^\bullet\xrightarrow{\begin{pmatrix}
\iota_H^{\overline{H}}&-\widetilde{\theta}
\end{pmatrix}} Y^\bullet\otimes_H \overline{H}\right).
\end{equation}

Let $\tw\sg$ be the full subcategory of $\per\sg$ consisting of minimal strictly perfect dg $\sg$-modules, and let $\glue\sg$ be the full subcategory of $\Tri(\sg)$ consisting of $(Y^\bullet,V^\bullet,\theta)$ satisfying that $Y^\bullet$ and $V^\bullet$ are minimal strictly perfect and $\fK(Y^\bullet,V^\bullet,\theta)$ is strictly perfect.

\begin{theorem}[{Graded version of \cite[Theorem 4.2]{BD}}]\label{thm:BD}
The functor
\begin{gather}\label{eq:BD}\boxed{
\begin{array}{rccc}
\fE:&\tw \sg&\to&\glue\sg\\
&X^\bullet&\mapsto&(X^\bullet\otimes^{\mathbf{L}}_\sg H,X^\bullet\otimes^{\mathbf{L}}_\sg \overline{\sg},\theta_{X^\bullet})
\end{array}
}\end{gather}
where $\theta_{X^\bullet}\colon(X^\bullet\otimes^{\mathbf{L}}_\sg H)\otimes^{\mathbf{L}}_{H}\overline{H}\to (X^\bullet\otimes^{\mathbf{L}}_\sg\overline{\sg})\otimes^{\mathbf{L}}_{\overline{\sg}}\overline{H}$ is the canonical isomorphism in the category $\per\overline{H}$, has the following properties.
\begin{enumerate}
    \item $\fE$ is dense.
	\item $\fE(X^\bullet_1)\cong \fE(X^\bullet_2)$ if and only if $X^\bullet_1\cong X^\bullet_2$.
	\item $\fE(X^\bullet)$ is indecomposable if and only if $X^\bullet$ is indecomposable.
\end{enumerate}
In the case that $\sg$ is non-positive, $\fE$ is a functor from $\per\sg$ to $\Tri\sg$.
\end{theorem}

The proof of \cite[Theorem 4.2]{BD} also works for our setting. We contain a proof in Appendix~\ref{app:pfBD} for the completeness of the paper. We also refer to \cite{CL} for another approach to \cite[Theorem 4.2]{BD} in a more general framework.

Recall from Definition~\ref{def:cwls} that $\care$ is the set of graded admissible curves with local system. Let $\cares$ be the subset of $\care$ consisting of $(\wg,N)$ satisfying $R(\bione(\wg,N))\in\fM(\glue\sg)$. Denote by $\Ind\mathcal{C}$ the set of isoclasses of indecomposable objects in a category $\mathcal{C}$. Combining Corollary~\ref{cor:Deng}, Theorem~\ref{thm:tribu} and Theorem~\ref{thm:BD}, we have the following main result of Part~\ref{part:1}.

\begin{theorem}\label{thm1}
There is a bijection
    \begin{equation}\label{eq:glue}\boxed{
    \widetilde{X}:\cares\to\Ind\tw\sg},
    \end{equation}
    which fits into the following commutative diagram
    \begin{equation}\label{eq:thmcomm2}
    \xymatrix{
    &\cares\ar[ld]_{\widetilde{X}}\ar[rr]^{\bione}_{\eqref{eq:l.s.}}  && \OO\ar[rd]^{R}_{\eqref{eq:rep}}\\
    \tw\sg\ar[rr]^{\fE}_{\eqref{eq:BD}} &&\glue\sg\subseteq\Tri\sg\ar[rr]^{\fM}_{\eqref{eq:tribu}} &&\repb(S)
    }
\end{equation}
In the case that $\sg$ is non-positive, \eqref{eq:glue} becomes a bijection
\begin{equation}\label{eq:care}
    \widetilde{X}:\care\to\Ind\per\sg,
\end{equation}
and \eqref{eq:thmcomm2} becomes a commutative diagram
\begin{equation}\label{eq:thmcomm}
   \xymatrix{
&\care\ar[ld]_{\widetilde{X}}\ar[rr]^{\bione}&&\OO\ar[rd]^{R}\\
\per\sg\ar[rr]^{\fE}&&\Tri\sg\ar[rr]^{\fM}&&\repb(S)
}
\end{equation}
\end{theorem}

Although we do not have a topological criterion for a graded admissible curve with local system to be in $\cares$ (and hence we do not give a name for them), there is an example (below) showing that there is some graded admissible curve with local system not in $\cares$ even for the unpunctured case.

\begin{example}
    Let $\sg$ be the graded gentle algebra given by the graded quiver
    \[\xymatrix{
    1\ar@/^/[r]^{a}&2\ar@/^/[l]^{b}
    },\ |a|=|b|=1,\]
    with relations $ab,ba$. Then $\sg$ can be realized, in the way of Section~\ref{subsec:clannish2}, as the subalgebra of $H=H_1\times H_2$, where
    \[H_1=\k\left((1,1)\xrightarrow{a}(1,2)\right)\text{ and }H_2=\k\left((2,1)\xrightarrow{b}(2,2)\right),\]
    associated to the idempotents $e_1=(1,1)+(2,2)$ and $e_2=(1,2)+(2,1)$. The corresponding $\gms$ with full formal closed arc system $\dac=\{\we_1,\we_2\}$ is shown in Figure~\ref{fig:non-grable}. The circle $\gamma$ on $\gms$ surrounding the interior closed marked point (which is gradable), with any grading and any local system $N$, is not in $\cares$.

\begin{figure}[htpb]
    \begin{tikzpicture}[scale=.5]
    \draw[ultra thick](0,0) circle (3.5);
    \draw[thick, red](0,3.5) to (0,0) to (0,-3.5);
    \draw (0,0)\ww;
    \draw[blue,thick](0,0) circle (1.75);
    \draw (0,3.5)\ww (0,-3.5)\ww (3.5,0)[blue]\nn (-3.5,0)[blue]\nn;
    \draw[red] (0,2.5)[right]node{$\we_1$} (0,-2.5)[right]node{$\we_2$};
    \draw[blue] (1.75,0)[right]node{$\wg$};
    \end{tikzpicture}
    \caption{An example of admissible curves not in $\cares$}\label{fig:non-grable}
    \end{figure}

    For instance, if we take the grading $\wg$ of $\gamma$ such that the intersection index between $\wg$ and $\we_1$ is 0, and take $N$ to be the 1-dimensional module given by the irreducible polynomial $x-1$ (cf. Appendix~\ref{subapp:apw}), then the triple $(Y^\bullet,V^\bullet,\theta)=\fM^{-1}\circ R\circ\bione (\wg,N)$ is given by
    \begin{itemize}
        \item $Y^\bullet=\left((1,1)H_1\oplus(1,2)H_1,d_1\right)\oplus \left((2,1)H_2\oplus(2,2)H_2,d_2\right)$, where $d_1$ and $d_2$ are induced by $a$ and $b$, respectively;
        \item $V^\bullet=e_1\overline{\sg}\oplus e_2\overline{\sg}$;
        \item $\theta$ is the identity from $(1,1)\overline{H_1}\oplus(1,2)\overline{H_1}\oplus(2,1)\overline{H_2}\oplus(2,2)\overline{H_2}$ to itself.
    \end{itemize}
    Hence the dg $\sg$-module $\fK(Y^\bullet,V^\bullet,\theta)$ is $$\left(e_1\sg\oplus e_2\sg,\begin{pmatrix}
        0&d_1\\d_2&0
    \end{pmatrix}\right),$$
    which is not strictly perfect. Moreover, $\fK(Y^\bullet,V^\bullet,\theta)$ is null-homotopic and hence is isomorphic to zero in $\per\sg$.

\end{example}

The following class of indecomposable objects and morphisms between them will be studied in Part~\ref{part:2}, using another geometric model.

\begin{definition}\label{def:xian}
    An \emph{\xian} object in $\per\sg$ is of the form $\widetilde{X}(\wg,N)$ for $\wg\in\wOA(\gms)$ and $\dim N=1$. Denote by $\goodind$ the set of \xian objects.
\end{definition}

We will show that the arc objects are indeed in $\tw\sg$ (see Theorem~\ref{thm2} with the identify~\eqref{eq:tag to ls}), i.e. $\goodind\subset\tw\sg$.

\begin{remark}
    For $(\wg,N)\in\care$, if $\wg\in\Amm\cup\Amp$, then $\dim N=1$. In particular, we have
    $$\begin{array}{cl}
        \goodind= & \{\widetilde{X}(\wg,N)\mid \wg\in\Amm\cup\Amp\} \\
         & \cup\{\widetilde{X}(\wg,N)\mid \wg\in\App,\ \dim N=1\}.
    \end{array}$$
\end{remark}

\newpage

\part{Classification of morphisms via surfaces with binary}\label{part:2}

\section{Surfaces with binary}\label{sec:AGM}
In this section, we introduce binaries as an alternative realization of $\ZZ_2$-symmetry at punctures,
comparing to tagging.

As in Part~\ref{part:1}, we fix a GMSp $\gms=(\surf,\M,\Y,\P,\lambda)$ with a full formal closed arc system $\dac$.
Let $\sg$ be the associated graded skew-gentle algebra. Recall that
\begin{itemize}
    \item $\OME$ parameters the non-boundary edges of unpunctured $\dac$-polygons,
    \item $\OMEs$ parameters the arcs in $\dac$, and
    \item $\fc{\OME}\subset \OMEs$ parameters the edges of punctured $\dac$-polygons.
\end{itemize}

The geometric model using binaries will interpret $\OMEpm$ and $\OMEx$, and classify arc objects (see Definition~\ref{def:xian}) in $\per\sg$ by graded unknotted arcs (compared with Theorem~\ref{thm1}).

\subsection{Graded marked surfaces with binary}\label{subsec:binary}

\begin{construction}
We construct a \emph{graded marked surface with binary} (=GMSb) $\gmsx$,
together with an inherited full formal closed arc system $\ac^\ast\x$ as follows:
\begin{itemize}
\item replace each puncture $P\in\P$ by a boundary component $\Vot_P$,
called a \emph{binary}, with one open marked point $m_P$ and one closed marked point $y_P$ on it, and
\item decompose any graded arc $(i,j)$ in $\fc{\OME}$ into two graded arcs $(i,j^+),(i,j^-)\in\OMEx$ (which inherit the grading from $(i,j)$) as shown in Figure~\ref{fig:split}.
\end{itemize}
Denote by $\vot$\footnote{ \footnotesize{We will call this notation bios.} } the set of all binaries $\Vot_P$, $P\in\P$.
\end{construction}

\begin{figure}[htpb]
	\begin{tikzpicture}[scale=.4]\clip ((6,-5)rectangle(-20,6.5);
	\draw[ultra thick]plot [smooth,tension=1] coordinates {(3,-4.5) (-90:4) (-3,-4.5)};
	\draw[ultra thick,fill=gray!10] (90:2) ellipse (1.5) node {$\Vot_P$};
	\draw[very thick,blue] (90:.5)\nn node[below]{$m_P$};
	\draw[red](3,0)node[right]{$(i,j^-)$}(-3,0)node[left]{$(i,j^+)$};
    \draw[red,very thick] (90:-4) .. controls +(45:5) and +(30:8) .. (90:3.5);
    \draw[red,very thick] (90:-4) .. controls +(135:5) and +(150:8) .. (90:3.5);
    \draw[] (0,3.5)node[red,above]{$y_P$}\ww;
    \draw[] (-90:4)\ww;
    \draw[] (-9,0) node {\Large{$\Longrightarrow$}};
    \begin{scope}[shift={(-15,1)}]
       	\draw[ultra thick]plot [smooth,tension=1] coordinates {(3,-4.5) (-90:4) (-3,-4.5)};
    	\draw[ultra thick] (90:1) \nn;
    	\draw[red] (0,3.5)node[above]{$(i,j)$};
        \draw[red,very thick] (90:-4) .. controls +(45:5) and +(0:4) .. (0,3.5);
        \draw[red,very thick] (90:-4) .. controls +(135:5) and +(180:4) .. (0,3.5);
        \draw[] (90:1)node[below]{$P$};
       	\draw[] (-90:4)\ww;
    \end{scope}	
    \end{tikzpicture}
	\caption{From $\dac$ to $\dac\x$}\label{fig:split}
\end{figure}

\begin{remark}
By the construction, $\gmsx$ can be regarded as a graded marked surface without punctures $(\surf\x,\M\x,\Y\x,\lambda)$, with $\M\x=\M\cup\MP$ and $\Y\x=\Y\cup\YP$, where $\MP=\{m_P\mid P\in\P \}$ and $\YP=\{y_P\mid P\in\P \}.$ The set $\dac\x$ is a full formal closed arc system of $\gmsx$ in the sense of Definition~\ref{def:system}. So there is a skew-gentle datum associated to $\dac\x$, which gives a gentle algebra. However, we are only interested in how to realize $\sg$ from $\dac\x$.
\end{remark}

There is a natural bijection from
the set $\DBX$ of $\dac$-polygons to the set $\DBXx$ of $\dac\x$-polygons. We denote by $\PP\y$ the $\dac\x$-polygon corresponding to $\PP\in\DBX$. Recall that (see Section~\ref{subsec:sys-to-alg}) $$\DBX=\DBY\cup\DBZ,$$
where $\DBY=\{\PP_1,\cdots,\PP_t\}$ consists of the unpunctured $\dac$-polygons and $\DBZ=\{\Pp_P\mid P\in\P\}$ consists of the punctured $\dac$-polygons.

\begin{definition}\label{def:poly}
We denote
$$\DBYx=\{\PP_1\y,\cdots,\PP_t\y\},\ \DBZx=\{\Pp_P\y\mid P\in\PP\},$$
and call $\PP_i\y$ \emph{unbinaried $\dac\x$-polygons} and call $\Pp_P\y$ \emph{binaried $\dac\x$-polygons}. Arc segments in unbinaried (resp. binaried) $\dac\x$-polygons are called \emph{unbinaried} (resp. \emph{binaried}) arc segments.
\end{definition}

Note that the set of edges of $\PP_i$ is $\{(i,0),(i,1),\cdots,(i,m_i)\}$ while the set of edges of $\PP_i\y$ is $\cup_{1\leq j\leq m_i}\OMEpm(i,j)\cup\{(i,0)\}$; $\Pp_P$ is a once-punctured monogon while $\Pp_P\y$ is a triangle with binary, whose edges in the clockwise order are $(i,j^-)$, $(i,j^+)$ and the loop $\Vot_P$ based at $y_P$ (cf. Figure~\ref{fig:split}).

Hence, we have the following identifications:
\begin{itemize}
\item $\OMEpm$ parameters non-boundary edges of unbinaried $\dac\x$-polygons,
\item $\OMEx$ parameters arcs in $\ac^\ast\x$, and
\item $\fc{\OMEpm}$ parameters non-boundary edges of binaried $\dac\x$-polygons.
\end{itemize}

\begin{example}

The GMSb model for the Example~\ref{ex:datum} is shown in Figure~\ref{fig:gmsb}.

\begin{figure}[htpb]
\begin{tikzpicture}[scale=1,font=\tiny]
\draw[ultra thick](0,0) circle (3.5);
\draw[ultra thick](0,0) circle (1);
\draw[red,thick](-1,0)to[out=-100,in=160](0,-2)to[out=20,in=-80](1,0);
\draw[red,thick](0,-2)to[out=0,in=-90](2.8,.8)to[out=90,in=-10](0,3.5);
\draw[red,thick](0,3.5)to[out=-5,in=90](3.3,.25)to[out=-90,in=0](0,-3)to[out=180,in=-90](-3.3,.25)to[out=90,in=-175](0,3.5);
\draw[ultra thick](1.25,2.25) circle (.35);
\draw[ultra thick](-1.25,2.25) circle (.35);
\draw[blue](-1,2.5)\nn(-.75,2.6)node{$m_{p_2}$}(1,2.5)\nn(.75,2.65)node{$m_{p_1}$};\draw[red,thick](0,3.5)to[out=-150,in=150](-1.9,2)to[out=-30,in=180](-1.5,2) (0,3.5)to[out=-120,in=-60](-1.5,1.6)to[out=120,in=-90](-1.5,2);
\draw[red,thick](0,3.5)to[out=-30,in=30](1.9,2)to[out=-150,in=0](1.5,2) (0,3.5)to[out=-60,in=-120](1.5,1.6)to[out=60,in=-90](1.5,2);
\draw[font=\tiny](0,.8)node{$(1,0)$}(0,-.8)node{$(2,0)$}(0,-1.5)node{$D^\ast_2$}
    (-2,0)node{$D^\ast_1$}(0,-3.3)node{$D^\ast_3$}
    (-.5,3)node{$E^\ast_{p_2}$}(.5,3)node{$E^\ast_{p_1}$};
\draw[font=\tiny]
    (-1.1,-1.2)node[rotate=-68][red]{$(1,7)\simeq(2,1)$}
    (1.1,-1.2)node[red,rotate=68]{$(2,2)\simeq(1,1)$}
    (2.4,.5)node[red,rotate=83]{$(1,2)\simeq(1,6)$}
    (0,-2.7)node[red]{$(1,5)\simeq(3,1)$}
    (0,-3.7)node{$(3,0)$}(1.25,1.2)node[red]{$(1,3^-)$}(2.1,1.8)node[red]{$(1,3^+)$}
    (-1.25,1.2)node[red]{$(1,4^+)$}(-2.2,1.8)node[red]{$(1,4^-)$};
\draw[red](-1.5,2)\ww(-1.25,2.1)node{$y_{p_2}$}(1.5,2)\ww(1.25,2.1)node{$y_{p_1}$};
\draw(-1,0)\ww (1,0)\ww (0,-2)\ww (0,3.5)\ww (0,1)[blue]\nn (0,-1)[blue]\nn (0,-3.5)[blue]\nn;	
\end{tikzpicture}
\caption{An example of GMSb model}
\label{fig:gmsb}
\end{figure}
\end{example}

\def\shk{F\x}

\subsection{Unknotted arcs and \texorpdfstring{$\ZZ_2$-}\ actions via Dehn twists}

For any graded arc $\ws$ on $\gmsx$, we denote by $\shk(\ws)$ the graded arc on $\gms$ obtained from $\ws$ by shrinking each binary $\Vot_P$ back to the puncture $P$. For any graded open arc $\ws$ on $\gmsx$, the endpoints of $\shk(\ws)$ are in $\M\cup\P$. So $\shk(\ws)$ is a graded open arc on $\gms$.

\begin{definition}
A graded open arc $\ws$ on $\gmsx$ is called \emph{admissible} provided that $\shk(\ws)$ is admissible on $\gms$ (i.e. $\shk(\ws)\in\wOA(\gms)$). Denote by $\wOA(\gmsx)$ the set of graded admissible arcs on $\gmsx$.
\end{definition}

Figure~\ref{fig:knot} shows some non-admissible cases for a GMSb, where in the first three cases, $\shk(\ws)$ is homotopic to a point, while in the last three cases, $\shk(\ws)$ cuts out a once-punctured monogon by one of its self-intersections (i.e., does not satisfy condition \ref{itm:A1} in Definition~\ref{def:ad}).

\begin{figure}[htpb]
\begin{tikzpicture}[xscale=1.2,yscale=1.6,rotate=180]
	\draw[blue,thick] (0,-.76)\nn[out=-10,in=90]to(.8,-1.2)[out=-90,in=0]to(0,-2)node[above]{$\ws$};
	\draw[blue,thick] (0,-.76)[out=190,in=90]to(-.8,-1.2)[out=-90,in=180]to(0,-2);
\end{tikzpicture}\quad
\begin{tikzpicture}[xscale=1.2,yscale=1.6,rotate=180]
	\draw[ultra thick,fill=gray!11] (0,-.93) ellipse (.27 and 0.17);
	\draw[red] (0,-1.1)\ww;
	\draw[thick,blue] (0,-.76)\nn;
	\draw[blue,thick] (0,-.76)[out=20,in=90]to(.8,-1.2)[out=-90,in=0]to(0,-2)node[above]{$\ws$};
	\draw[blue,thick] (0,-.76)[out=160,in=90]to(-.8,-1.2)[out=-90,in=180]to(0,-2);
\end{tikzpicture}\quad
\begin{tikzpicture}[xscale=1.2,yscale=1.6,rotate=180]
	\draw[ultra thick,fill=gray!11] (0,-.93) ellipse (.27 and 0.17);
	\draw[red] (0,-1.1)\ww;
	\draw[thick,blue] (0,-.76)\nn;
	\draw[blue,thick] (0,-.5)[out=0,in=90]to(.8,-1.2)[out=-90,in=0]to(0,-2)node[above]{$\ws$};
	\draw[blue,thick] (0,-.76)[out=160,in=90]to(-.8,-1.2)[out=-90,in=180]to(0,-2);
	\draw[blue,thick] (0,-.5)[out=180,in=90]to(-.9,-1.2)[out=-90,in=180]to(0,-1.8);
	\draw[blue,thick] (0,-.76)[out=20,in=90]to(.6,-1.2)[out=-90,in=0]to(0,-1.8);
	\draw[blue,thick] (0,-1.95)node{$\cdot$} (0,-1.9)node{$\cdot$} (0,-1.85)node{$\cdot$};
\end{tikzpicture}\quad
\begin{tikzpicture}[xscale=1.2,yscale=1.6,rotate=180]
	\draw[ultra thick,fill=gray!11] (0,-.93) ellipse (.27 and 0.17);
	\draw[thick,blue] (0,-.76)\nn;
	\draw[blue,thick] (0,-.5)[out=-10,in=90]to(.8,-1.2)[out=-90,in=0]to(0,-2)node[above]{$\ws$};
	\draw[blue,thick] (0,-.5)[out=190,in=90]to(-.8,-1.2)[out=-90,in=180]to(0,-2);
	\draw[red] (0,-1.1)\ww;
	\draw[thick,blue] (0,-.5)\nn;
\end{tikzpicture}\quad
\begin{tikzpicture}[xscale=1.2,yscale=1.6,rotate=180]
	\draw[ultra thick,fill=gray!11] (0,-.93) ellipse (.27 and 0.17);
	\draw[thick,blue] (0,-.76)\nn;
	\draw[blue,thick] (-.5,-.5)[out=10,in=90]to(.8,-1.2)[out=-90,in=0]to(0,-2)node[above]{$\ws$};
	\draw[blue,thick] (.5,-.5)[out=170,in=90]to(-.8,-1.2)[out=-90,in=180]to(0,-2);
	\draw[red] (0,-1.1)\ww;
\end{tikzpicture}\quad
\begin{tikzpicture}[xscale=1.2,yscale=1.6,rotate=180]
	\draw[ultra thick,fill=gray!11] (0,-1.43) ellipse (.27 and 0.17);
	\draw[red] (0,-1.6)\ww;
	\draw[thick,blue] (0,-1.26)\nn;
	\draw[ultra thick,fill=gray!11] (0,-.93) ellipse (.27 and 0.17);
	\draw[red] (0,-1.1)\ww;
	\draw[thick,blue] (0,-.76)\nn;
	\draw[blue,thick] (0,-.76)[out=20,in=90]to(.8,-1.2)[out=-90,in=0]to(0,-2)node[above]{$\ws$};
	\draw[blue,thick] (0,-.76)[out=160,in=90]to(-.8,-1.2)[out=-90,in=180]to(0,-2);
\end{tikzpicture}
\caption{Some cases of non-admissible arcs on $\gmsx$}\label{fig:knot}
\end{figure}

Let $\Dfang$ be the subgroup of the usual mapping class group of $\gmsx$ generated by the squares of Dehn twists $\mathrm{D}_{\Vot}$ along any binary $\Vot\in\vot$.
The binary mapping class group $\MCG(\gmsx)$ of $\gmsx$ is the quotient of the usual one by $\Dfang$.

\begin{definition}\label{defn:orbit}
    The \emph{$\Dfang$-orbit} $\Dfang\cdot\ws$ of a graded admissible arc $\ws\in\wOA(\gmsx)$ consists of the graded admissible arcs which are obtained from $\ws$ by actions of $\Dfang$ on the ends (which are in binaries) separately.
\end{definition}

See the cyan arcs in the two pictures of Figure~\ref{fig:nomaps}, which are in the same $\Dfang$-orbit.

\begin{figure}[htpb]
	\begin{tikzpicture}[scale=.35]
	\draw[ultra thick]plot [smooth,tension=1] coordinates {(3,-4.5) (-90:4) (-3,-4.5)};
	\draw[ultra thick,fill=gray!10] (90:2) ellipse (1.5) node {$\Vot_P$};
    \draw[red] (90:-4) .. controls +(45:5) and +(30:4) .. (90:3.5);
    \draw[red] (90:-4) .. controls +(135:5) and +(150:4) .. (90:3.5);
    \draw[blue, thick] (90:8) .. controls +(-45:4) and +(-30:7) .. (90:.5);
    \draw[cyan, ultra thick] (90:8) .. controls +(225:4) and +(-150:7) .. (90:.5);
    \draw[Green] (90:8)\nn;
	\draw[blue, thick] (90:.5)\nn (90:.25) node[below]{$m_P$} (0,3.5)node[white,above]{$y_P$}\ww;
	\draw[red] (-90:4)\ww ;
    \begin{scope}[shift={(13,0)}]
	\draw[ultra thick]plot [smooth,tension=1] coordinates {(3,-4.5) (-90:4) (-3,-4.5)};
	\draw[ultra thick,fill=gray!10] (90:2) ellipse (1.5) node {$\Vot_P$};
    \draw[red] (90:-4) .. controls +(45:5) and +(30:4) .. (90:3.5);
    \draw[red] (90:-4) .. controls +(135:5) and +(150:4) .. (90:3.5);
    \draw[blue, thick] (90:8) .. controls +(-45:5) and +(-30:7) .. (90:.5);
    \draw[cyan, ultra thick] (90:8)
        .. controls +(-30:7) and +(-30:9) .. (0,-1)
        .. controls +(150:5) and +(160:3) .. (0,4.5)
        .. controls +(-20:4) and +(-15:2) .. (0,.5);
    \draw[Green] (90:8)\nn;
	\draw[very thick,blue] (90:.5)\nn  (0,3.5)\ww;
	\draw[red] (-90:4)\ww ;
    \end{scope}	
    \end{tikzpicture}
\caption{$\Dehn{\Vot_P}^2$-action on the cyan arc}\label{fig:nomaps}
\end{figure}

\begin{lemma}\label{lem:minimal}
    For any graded admissible arc $\ws\in\wOA(\gmsx)$, the following are equivalent:
	\begin{enumerate}
		\item $|\overrightarrow{\cap}(\ws,\we)|\leq|\overrightarrow{\cap}(\ws',\we)|$ for any $\ws'\in\Dfang\cdot\ws$ and any $\we\in\dac\x$;
		\item $\ws$ does not have a segment in the situation shown in Figure~\ref{fig:non-per}.
		\begin{figure}[htpb]
			\begin{tikzpicture}[scale=2]
			\draw[ultra thick,fill=gray!11] (0,1.2) circle (.3) node {$\Vot_P$};
			\draw[ultra thick] (-1,2)to(1,2);
			\draw[thick,blue] (0,1.5)node{$\bullet$};
			\draw[red, thick]plot [smooth, tension=1] coordinates {(0,2) (-.4,1.2) (0,0.9)};
			\draw[red, thick]plot [smooth, tension=1] coordinates {(0,2) (.4,1.2) (0,0.9)};
			\draw[blue,thick]plot [smooth, tension=1] coordinates {(-.28,1.6) (-.6,1.1) (0,.7) (.6,1.1) (.28,1.6)};
			\draw[blue, thick] (0,.7)node[below]{$\ws$};
			\draw[red] (0,0.9)\ww  (0,2)\ww ;
			\end{tikzpicture}
			\caption{non-minimal segment around a binary}
			\label{fig:non-per}
		\end{figure}
	\end{enumerate}
    Moreover, each $\Dfang$-orbit of graded admissible arcs contains exactly one arc satisfying the above conditions.
\end{lemma}

\begin{proof}
	This lemma can be checked locally, as the Dehn twist of a binary only involves a neighborhood of the boundary, cf. Figure~\ref{fig:nomaps}.
\end{proof}

\begin{definition}
	We call $\ws\in\wOA(\gmsx)$ \emph{unknotted} if it satisfies (one of) the conditions in Lemma~\ref{lem:minimal}. Denote by $\gUC(\gmsx)$ the set of graded unknotted arcs on $\gmsx$.
\end{definition}

By Lemma~\ref{lem:minimal}, there is a natural identification
$$\gUC(\gmsx)\longrightarrow \wOA(\gmsx)/\Dfang,$$
sending a graded unknotted arc to its $\Dfang$-orbit.

\begin{definition}[Oriented intersection on GMSb]\label{def:int of unknotted}
    Let $\ws$ and $\wt$ be two graded unknotted arcs in $\gUC(\gmsx)$. We define the intersection number from $\ws$ to $\wt$ of index $\rho$ is
    $$\oIntd(\ws,\wt):=\min\{|\overrightarrow{\cap}^\rho(\ws',\wt')|\mid\ws'\in\Dfang\cdot\ws,\ \wt'\in\Dfang\cdot\wt\}.$$
    Here, recall from Definition~\ref{defn:int} that $\overrightarrow{\cap}^\rho(\ws',\wt')$ consists of the clockwise angles at intersections from $\ws'$ to $\wt'$ with index $\rho$. More precisely,
in the left picture of Figure~\ref{fig:001},
the pair of clockwise angles counts one and in the right picture there, the clockwise angle counts one.
\end{definition}

\begin{example}\label{ex:unknotted}
   The arcs $\sigma$ and $\tau$ shown in the left picture of Figure~\ref{fig:unk cur} are unknotted. We take the gradings of $\ws$ and $\wt$ such that the intersection indices of $\ws$ and $\wt$ with the arc $\overline{(1,1)}\in\dac\x$ are $0$. Then we have $$\oIntd(\ws,\wt)=\begin{cases}2&\text{if $\rho=0$,}\\1&\text{if $\rho=-5$}\\0&\text{otherwise.}\end{cases}.$$

\begin{figure}[htpb]
\begin{tikzpicture}[scale=1]
    \draw[ultra thick](0,0) circle (3.5);
    \draw[ultra thick](0,0) circle (1);
    \draw[red,thin](-1,0)to[out=-100,in=160](0,-2)to[out=20,in=-80](1,0);
    \draw[red,thin](0,-2)to[out=0,in=-90](2.8,.8)to[out=90,in=-10](0,3.5);
    \draw[red,thin](0,3.5)to[out=-5,in=90](3.3,.25)to[out=-90,in=0](0,-3)to[out=180,in=-90](-3.3,.25)to[out=90,in=-175](0,3.5);
    \draw[ultra thick](1.35,2.15) circle (.25);
    \draw[ultra thick](-1.35,2.15) circle (.25);
    \draw[red,thin](0,3.5)to[out=-150,in=150](-1.9,2)to[out=-30,in=180](-1.53,1.97)
        (0,3.5)to[out=-120,in=-60](-1.5,1.6)to[out=120,in=-90](-1.53,1.97);
    \draw[red,thin](0,3.5)to[out=-30,in=30](1.9,2)to[out=-150,in=0](1.53,1.97)
        (0,3.5)to[out=-60,in=-120](1.5,1.6)to[out=60,in=-90](1.53,1.97);
    \draw[blue,thick]
        (0,-1)to[out=-10,in=-90](2.2,1)to[out=90,in=-15]
        (1.2,2.7)to[out=165,in=15](-1.2,2.7)to[out=-165,in=90]
        (-1.8,2.15)to[out=-90,in=180]
        (-1.5,1.2)to(1.3,1.2) to[out=0,in=-110]
        (1.9,2)to[out=70,in=-15](1.2,2.8)to[out=165,in=15](-1.2,2.8)to[out=-165,in=90]
        (-2.4,.9)to[out=-90,in=-170](0,-1);
    \draw[cyan,thick]
        (0,-1)to[out=0,in=-90](1.8,1.4)
              to[out=90,in=0](1.35,2.5)
              to[out=180,in=35](0,1.8)
        .. controls +(215:3) and +(120:1) ..(-1.2,2.33);
    \draw(1.8,-.5)node[blue]{$\tau$}(1.4,.2)node[cyan]{$\sigma$};
\draw[blue](-1.2,2.33)\nn(1.2,2.33)\nn;
\draw[red](-1.53,1.97)\ww(1.53,1.97)\ww;
\draw(-1,0)\ww (1,0)\ww (0,-2)\ww (0,3.5)\ww (0,1)[blue]\nn (0,-1)[blue]\nn (0,-3.5)[blue]\nn;	
\end{tikzpicture}
\begin{tikzpicture}[scale=1]
        \draw[ultra thick](0,0) circle (3.5);
        \draw[ultra thick](0,0) circle (1);
        \draw[red,thin](-1,0)to[out=-100,in=160](0,-2)to[out=20,in=-80](1,0);
        \draw[red,thin](0,-2)to[out=0,in=-90](2.8,.8)to[out=90,in=-10](0,3.5); \draw[red,thin](0,3.5)to[out=-5,in=90](3.3,.25)to[out=-90,in=0](0,-3)to[out=180,in=-90](-3.3,.25)to[out=90,in=-175](0,3.5);
        \draw[red,thin](0,3.5)to[out=-160,in=135](-1.6,1.9)to[out=-45,in=-110](0,3.5);
        \draw[red,thin](0,3.5)to[out=-20,in=45](1.6,1.9)to[out=-135,in=-70](0,3.5); \draw[blue,thick](0,-1)to[out=-10,in=-90](2.2,1)to[out=90,in=-15](1.2,2.7)to[out=165,in=15](-1.2,2.7)to[out=-165,in=90](-1.8,2.15)to[out=-90,in=165](-1.2,1.5)to[out=-15,in=-165](1.2,1.5)to[out=15,in=-110](1.8,2)to[out=70,in=-15](1.2,2.8)to[out=165,in=15](-1.2,2.8)to[out=-165,in=90](-2.4,.9)to[out=-90,in=-170](0,-1);
        \draw[cyan,thick](0,-1)to[out=0,in=-90](2,1.2)to[out=90,in=-10](1.1,2.5)to[out=170,in=20](-1.35,2.15);
        \draw(-1.8,-.7)node[blue]{$\tilde{\tau}^\times$}(1.6,.5)node[cyan]{$\tilde{\sigma}^\times$}(-1.2,2.3)node[cyan]{$-$};
        \draw(-1,0)\ww (1,0)\ww (0,-2)\ww (0,3.5)\ww  (1.35,2.15)\nn (-1.35,2.15)\nn (0,1)[blue]\nn (0,-1)[blue]\nn (0,-3.5)[blue]\nn;
\end{tikzpicture}
\caption{Example of unknotted arcs v.s. tagged arcs}
\label{fig:unk cur}
\end{figure}
\end{example}

The rest of the paper is devoted to showing the following main result in Part~\ref{part:2}, which is divided into Theorem~\ref{thm2} and Theorem~\ref{thm:int=dim}. Recall that $\goodind$ is the set of (isoclasses of) \xian objects in $\per\sg$ (see Definition~\ref{def:xian}).

\begin{theorem}\label{thm:part2}
There is a bijection
\begin{equation}\label{eq:bi}
\boxed{    \cpy{?}:\gUC(\gmsx)\to\goodind   }
\end{equation}
such that for any graded unknotted arcs $\ws$ and $\wt$, we have
\begin{equation}\label{eq:=}
\boxed{
\oIntd(\ws,\wt)=\dim\Hom_{\per\sg}(\cpy{\ws},\cpy{\wt}[\rho]).}
\end{equation}
\end{theorem}

\section{String model/dg modules for unknotted arcs}\label{sec:DG-a}
In this section, we will associate dg $\sg$-modules to graded unknotted arcs and prove that this construction is compatible with the construction (of objects) in Part~\ref{part:1}.

\subsection{Unknotted arcs via tagged arcs}
To link graded unknotted arcs on $\gmsx$ with certain graded admissible arcs with local system on $\gms$, we need the following notion, which was essentially introduced in \cite{FST}, cf. also \cite{QZ1}.

\begin{definition}
A \emph{tagged arc} on $\gms$ is a pair $(\wg,\kappa)$ of a graded admissible arc $\wg\in\wOA(\gms)$ and a function
$$\kappa:\{t\in \{0,1\}\mid \gamma(t)\in\P\}\to\{+,-\}.$$
Denote by $\careII$ the set of graded tagged arcs on $\gms$.
\end{definition}

\begin{construction}\label{cons:tagtoloc}
We identify each graded tagged arc $(\wg,\kappa)$ with a graded admissible arc with local system $(\wg,N_\kappa)$ with $\dim_\k N_\kappa=1$ as follows. Since $A_{\wg}=\k\langle x_t\mid \gamma(t)\in\P\rangle/(x_t^2-x_t\mid\gamma(t)\in\P)$, we can define $N_{\kappa}$ by $N_{\kappa}(x_{t})=\kappa(t).$ Indeed, all 1-dimensional $A_{\wg}$-modules are obtained in this way (compared with Remark~\ref{rmk:1-dim}).  So there is a bijection
\begin{equation}\label{eq:tag to ls}
    \careII\to\{(\wg,N)\in\care\mid \wg\in\wOA(\gms),\ \dim_\k N=1\},
\end{equation}
sending $(\wg,\kappa)$ to $(\wg, N_{\kappa})$.
\end{construction}

\begin{construction}\label{cons:tag}
To any $\ws\in\gUC(\gmsx)$, we associate a graded tagged arc $\ws^\times$ on $\gms$ given by $$\ws^\times=(\shk(\ws),\kappa_{\ws}),$$
where for any $t\in\{0,1\}$ with $\shk(\ws)(t)\in\P$ (or equivalently, $\ws(t)\in\M_\P$),
$\kappa_{\ws}(t)=\varepsilon$, if the segment of $\ws$ near $t$ crosses $(i,j^\varepsilon)$. Here, $(i,j)$ denotes the edge of the once-punctured monogon enclosing $\ws(t)$.
\end{construction}

\begin{lemma}\label{lem:bij1}
The map
\begin{equation}\label{eq:bij}
\begin{array}{ccc}
\gUC(\gmsx)&\to&\careII\\
\ws&\to&\ws^\times
\end{array}
\end{equation}
is a bijection.
\end{lemma}

\begin{proof}
This is a local problem. Figure~\ref{fig:bij1} shows that at each punctured endpoint of a graded tagged arc $(\wg,\kappa)$ in $\careII$, there are two natural graded arcs $\wg_{\pm}$ in $\gUC(\gmsx)$, such that their images under $\shk$ intersect $\dac$ in the same way as $\wg$. More precisely, when regarding $(i,j^+)\cup(i,j^-)$ as the arc $(i,j)$, $m_P=P$ and forgetting $\Vot_P$,
these two arcs $\shk(\wg_{\pm})$ are in minimal position with respect to $\dac$ and hence becomes $\wg$.
We will choose the convention of signs as: at $m_P$, the starting arc segments of the two arcs can be ordered anticlockwise as $\wg_+,\wg_-$ and they will map to $\wg\times\{+\}$ and $\wg\times\{-\}$ respectively as in Figure~\ref{fig:bij1}.
\end{proof}

\begin{figure}[htpb]
	\begin{tikzpicture}[scale=.35]
	\draw[ultra thick]plot [smooth,tension=1] coordinates {(3,-4.5) (-90:4) (-3,-4.5)};
	\draw[ultra thick,fill=gray!10] (90:2) ellipse (1.5) node {$\Vot_P$};
    \draw[red] (90:-4) .. controls +(45:5) and +(30:4) .. (90:3.5);
    \draw[red] (90:-4) .. controls +(135:5) and +(150:4) .. (90:3.5);
    \draw[red] (-3.2,-2)node{$(i,j^+)$} (3.2,-2)node{$(i,j^-)$};
    \draw[blue, very thick] (90:8) .. controls +(-45:4) and +(-30:7) .. (90:.5);
    \draw[blue, very thick] (90:8) .. controls +(225:4) and +(-150:7) .. (90:.5)
        (5,2)node{$\wg_-$}(-5,2)node{$\wg_+$};
    \draw[Green] (90:8)\nn;
	\draw[very thick,blue] (90:.5)\nn (90:.25)node[below]{$m_P$} (0,3.5)node[white,above]{$y_P$}\ww;
	\draw[red] (-90:4)\ww ;
    \begin{scope}[shift={(13,0)}]
       	\draw[ultra thick]plot [smooth,tension=1] coordinates {(3,-4.5) (-90:4) (-3,-4.5)};
    	\draw[ultra thick] (90:1) \nn;
        \draw[red] (90:-4) .. controls +(45:5) and +(0:4) .. (0,3.5);
        \draw[red] (90:-4) .. controls +(135:5) and +(180:4) .. (0,3.5);
        \draw[red] (-3.2,-2)node{$(i,j)$};
        \draw[blue, very thick,bend left] (90:8)to(90:1);
        \draw[blue, very thick,bend right] (90:8)to(90:1);
        \draw (90:1)node[below,black]{$P$}
            (90:1.5)node[right]{$-$} (90:1.5)node[left]{$+$}
            (0,5)node{$\wg$};;
        \draw[Green] (90:8)\nn;
    	\draw[red] (-90:4)\ww ;
    \end{scope}	
    \end{tikzpicture}
	\caption{A bijection between $\gUC(\gmsx)$ and $\careII$}\label{fig:bij1}
\end{figure}

\begin{example}\label{ex:tag}
    The tagged arcs $\wsx$ and $\wtx$ corresponding to the unknotted curves $\ws$ and $\wt$ in Example~\ref{ex:unknotted} are shown in the right picture of Figure~\ref{fig:unk cur}, where the underlying admissible arcs of $\wsx$ and $\wtx$ are $\gamma''$ and $\gamma'$ in Example~\ref{ex:ad cur}, and the tagging of $\wsx$ at the puncture $p_2$ is $-$.
\end{example}

We also need to introduce the notion of tagged oriented intersections between tagged arcs to investigate oriented intersections between unknotted arcs. For any admissible arc $\gamma$ and any $t\in\{0,1\}$, we use $\near{\gamma}{t}$ to denote the segment of $\gamma$ near $t$ and starting from $t$, and use $\ori{\gamma}{t}$ to denote the orientation of $\gamma$ from $t$ to $1-t$.

\begin{definition}[Tagged oriented intersection (TOI)]\label{def:int}
Let $(\wg_1,\kappa_1)$ and $(\wg_2,\kappa_2)$ be two tagged arcs on $\gms$ in a minimal position. Define $\overrightarrow{\cap}^\rho((\wg_1,\kappa_1),(\wg_2,\kappa_2))$ to consist of the following six types of oriented intersections:
let $q=\gamma_1(t_1)=\gamma_2(t_2)$ be an intersection from $\gamma_1$ to $\gamma_2$ of index $\rho$ and
we have:
\begin{enumerate}
	\item[(I)] $q\notin\M\cup\P$ and the pair of clockwise angles at $q$ as in picture (I) in Figure~\ref{fig:oriented int} counts one TOI.
	\item[(II)] $q\in\M$ and $\near{\gamma_1}{t_1}$ is to the left of $\near{\gamma_2}{t_2}$.
	Moreover, we need to exclude case (ii) in Figure~\ref{fig:oriented int2}
	(i.e. $\gamma_1|_{t_1\to 1-t_1}\sim\gamma_2|_{t_2\to 1-t_2}$ and $\kappa_1\neq\kappa_2$).
	Then the clockwise angle at $q$ as in (II)  Figure~\ref{fig:oriented int} counts one TOI.
	\item[(III)] $q\in\P$, $\kappa_1(t_1)=\kappa_2(t_2)$, $\near{\gamma_1}{t_1}$ is to the left of $\near{\gamma_2}{t_2}$  (relative to the closed marked point of the $\dac$-polygon containing $q$, same as below) and  $\ori{\gamma_1}{t_1}\nsim\ori{\gamma_2}{t_2}$.
	Then the clockwise angle at $q$ facing away from the nearby closed marked point as in the picture (III) in Figure~\ref{fig:oriented int} counts one TOI.
	\item[(IV)] $q\in\P$, $\kappa_1(t_1)\neq\kappa_2(t_2)$, $\near{\gamma_2}{t_2}$ is to the left of $\near{\gamma_1}{t_1}$ and $\ori{\gamma_1}{t_1}\nsim\ori{\gamma_2}{t_2}$.
	Then the clockwise angle at $q$ facing towards the nearby closed marked point as in picture (IV) in Figure~\ref{fig:oriented int} counts one TOI.
	\item[(V)] $q\in\P$, $\kappa_1   (t_1)=\kappa_2(t_2)$, $\near{\gamma_1}{t_1}$ is to the left of $\near{\gamma_2}{t_2}$, $\gamma_1|_{t_1\to 1-t_1}\sim\gamma_2|_{t_2\to 1-t_2}$, and $\kappa_1(1-t_1)=\kappa_2(1-t_2)$.
	Then the clockwise angle at $q$ facing away from the nearby closed marked point as in the picture (V.1)/(V.2) in Figure~\ref{fig:oriented int} (depending on where the other endpoint of $\gamma_i$ is in $\P$ or in $\M$) counts one TOI.

	\item[(VI)] $q\in\P$, $\kappa_1(t_1)\neq\kappa_2(t_2)$, $\near{\gamma_2}{t_2}$ is to the left of $\near{\gamma_1}{t_1}$, $\gamma_1|_{t_1\to 1-t_1}\sim\gamma_2|_{t_2\to 1-t_2}$,
	$\gamma_1(1-t_1)=\gamma_2(1-t_2)\in\P$ and
	$\kappa_1(1-t_1)\neq\kappa_2(1-t_2)$. Then the clockwise angle at $q$ facing toward the nearby closed marked point as in picture (VI) in Figure~\ref{fig:oriented int} counts one TOI.
	
	\end{enumerate}
We denote
\begin{gather}\label{eq:tag.int.num.}
\oIntd((\wg_1,\kappa_1),(\wg_2,\kappa_2))=|\overrightarrow{\cap}^\rho((\wg_1,\kappa_1),(\wg_2,\kappa_2))|.
\end{gather}
\end{definition}
Note that $\oIntd=\oIntd_{\surfi}+\oIntd_{\partial\surf}$,
where the first term consists of interior intersections (i.e. type (I)) and the second consists of endpoint ones (all other types).

\begin{figure}[htpb]
		\begin{tikzpicture}[scale=.5]
	    \begin{scope}[shift={(0,0)}]
		\draw[thick, Green,->-=.99,>=stealth](-1,1)to[bend left=30](1,1);
		\draw[thick, Green,-<-=.10,>=stealth](-0.90,-1)to[bend left=-30](1,-1);
		\draw[blue,very thick](-2,-2)to(2,2)node[above]{$\gamma_2$}(2,-2)to(-2,2)node[above]{$\gamma_1$};
		\draw (0,0)[above]node{$q$};
		\draw (0,-3)node{(I)};
	   \end{scope}
		\begin{scope}[shift={(7,-2)}]
			\draw[thick, Green,->-=.99,>=stealth](-0.9,1.7)to[bend left=30](0.9,1.7);
			\draw[ultra thick](2,0)to(-2,0);
			\draw[blue,very thick](2,4)node[above]{$\gamma_2$}to(0,0)\nn to(-2,4)node[above]{$\gamma_1$};
				\draw (0,0.2)[above]node{$q$};
				\draw (0,-1)node{(II)};
		\end{scope}
	 \begin{scope}[shift={(14,0)}, scale=.5]
		\draw[ultra thick]plot [smooth,tension=1] coordinates {(3,-4.5) (-90:4) (-3,-4.5)};
		\draw[ultra thick] (90:0.1) \nn;
		\draw[red] (90:-4) .. controls +(45:5) and +(0:4) .. (0,3.5);
		\draw[red] (90:-4) .. controls +(135:5) and +(180:4) .. (0,3.5);
		\draw[blue, very thick] (90:0.1)to(60:6)node[right]{$\gamma_2$};
		\draw[blue, very thick] (90:0.1)to(120:6)node[left]{$\gamma_1$};
		\draw (90:0.1)node[below,black]{$q$}(0,5)node{$\nsim$}(0,1)node{$=$};
		\draw[thick, Green,->-=.99,>=stealth](120:2)to[bend left=30](60:2);
		\draw[red] (-90:4)\ww ;
		\draw (0,-6)node{(III)};
	\end{scope}	
\begin{scope}[shift={(10.5,-10.5)}, scale=.5]
	\draw[ultra thick]plot [smooth,tension=1] coordinates {(3,-3.5) (-90:3) (-3,-3.5)};
	\draw[ultra thick, blue] (90:0.1) \nn (90:8) \nn;
	\draw[red] (90:-3) .. controls +(45:4) and +(0:3) .. (0,2.5);
	\draw[red] (90:-3) .. controls +(135:4) and +(180:3) .. (0,2.5);
	\draw[blue, very thick,bend left] (90:8)to(90:0.1);
	\draw[blue, very thick,bend right] (90:8)to(90:0.1);
	\draw (90:0.1)node[below,black]{$q$}(0,4)node{$\sim$}(75:4)node[right,blue]{$\gamma_2$}(105:4)node[left,blue]{$\gamma_1$}(0,1)node{$=$}(0,6.5)node{$=$};
	\draw[thick, Green,->-=.99,>=stealth](115:2)to[bend left=30](65:2);
	\draw[red] (-90:3)\ww ;
	\begin{scope}[shift={(0,8)}, rotate=180]
    \draw[red] (90:-3) .. controls +(45:4) and +(0:3) .. (0,2.5);
    \draw[red] (90:-3) .. controls +(135:4) and +(180:3) .. (0,2.5);
    	\draw[ultra thick]plot [smooth,tension=1] coordinates {(3,-3.5) (-90:3) (-3,-3.5)};
    	\draw[red] (-90:3)\ww ;	
	\end{scope}
	\draw (0,-5)node{(V.2)};
\end{scope}	
\begin{scope}[shift={(10.5,-6.5)}, scale=-.5]
	\draw[thick,dashed, Green,->-=.99,>=stealth](115:2)to[bend left=30](65:2);
\end{scope}
\begin{scope}[shift={(3.5,-10)}, scale=.5]
	\draw[ultra thick]plot [smooth,tension=1] coordinates {(3,-4.5) (-90:4) (-3,-4.5)};
	\draw[ultra thick, blue] (90:0.1) \nn (90:8) \nn;
	\draw[red] (90:-4) .. controls +(45:5) and +(0:4) .. (0,3.5);
	\draw[red] (90:-4) .. controls +(135:5) and +(180:4) .. (0,3.5);
	\draw[blue, very thick,bend left] (90:8)to(90:0.1);
	\draw[blue, very thick,bend right] (90:8)to(90:0.1);
	\draw (90:0.1)node[below,black]{$q$}(0,4)node{$\sim$}(75:4)node[right,blue]{$\gamma_2$}(105:4)node[left,blue]{$\gamma_1$}(0,1)node{$=$};
	\draw[thick, Green,->-=.99,>=stealth](115:2)to[bend left=30](65:2);
	\draw[ultra thick](4,8)to(-4,8);
	\draw[red] (-90:4)\ww ;
	\draw (0,-6)node{(V.1)};
    \draw[Green,dashed,->-=.99, thick,>=stealth](.8,6.3)to[bend left=15](-.8,6.3);
\end{scope}	
	 \begin{scope}[shift={(21,0)}, scale=.5]
	\draw[ultra thick]plot [smooth,tension=1] coordinates {(3,-4.5) (-90:4) (-3,-4.5)};
	\draw[ultra thick] (90:0.1) \nn;
	\draw[red] (90:-4) .. controls +(45:5) and +(0:4) .. (0,3.5);
	\draw[red] (90:-4) .. controls +(135:5) and +(180:4) .. (0,3.5);
	\draw[blue, very thick] (90:0.1)to(60:7)node[above]{$\gamma_1$};
	\draw[blue, very thick] (90:0.1)to(120:7)node[above]{$\gamma_2$};
	\draw(90:0.1)node[below,black]{$q$}(0,5)node{$\nsim$}(0,1)node{$\neq$};
	\draw[Green,->-=.99, thick,>=stealth]plot [smooth,tension=1] coordinates {(60:1.7) (0:1.7) (-90:1.7) (-180:1.7) (-240:1.7)};
	\draw[red] (-90:4)\ww ;
	\draw (0,-6)node{(IV)};
\end{scope}	
\begin{scope}[shift={(17.5,-6.5)}, scale=-.5]
\draw[Green,dashed,->-=.99, thick,>=stealth]plot [smooth,tension=1] coordinates
    {(60:1.6) (0:1.6) (-90:1.6) (-180:1.6) (-240:1.6)};
\end{scope}
\begin{scope}[shift={(17.5,-10.5)}, scale=.5]
	\draw[ultra thick]plot [smooth,tension=1] coordinates {(3,-3.5) (-90:3) (-3,-3.5)};
	\draw[ultra thick, blue] (90:0.1) \nn (90:8) \nn;
	\draw[red] (90:-3) .. controls +(45:4) and +(0:3) .. (0,2.5);
	\draw[red] (90:-3) .. controls +(135:4) and +(180:3) .. (0,2.5);
	\draw[blue, very thick,bend left] (90:8)to(90:0.1);
	\draw[blue, very thick,bend right] (90:8)to(90:0.1);
	\draw (90:0.1)node[below,black]{$q$}(0,4)node{$\sim$}(75:4)node[right,blue]{$\gamma_1$}(105:4)node[left,blue]{$\gamma_2$}(0,1)node{$\neq$}(0,6.5)node{$\neq$};
\draw[Green,->-=.99, thick,>=stealth]plot [smooth,tension=1] coordinates
    {(60:1.6) (0:1.6) (-90:1.6) (-180:1.6) (-240:1.6)};
	\draw[red] (-90:3)\ww ;
	\begin{scope}[shift={(0,8)}, rotate=180]
		\draw[red] (90:-3) .. controls +(45:4) and +(0:3) .. (0,2.5);
		\draw[red] (90:-3) .. controls +(135:4) and +(180:3) .. (0,2.5);
		\draw[ultra thick]plot [smooth,tension=1] coordinates {(3,-3.5) (-90:3) (-3,-3.5)};
		\draw[red] (-90:3)\ww ;	
	\end{scope}
	\draw (0,-5)node{(VI)};
\end{scope}	
	\end{tikzpicture}
	\caption{Types of TOIs from $(\wg_1,\kappa_1)$ to $(\wg_2,\kappa_2)$}
	\label{fig:oriented int}
\end{figure}

\begin{remark}\label{rmk:int}
Note that if $\gamma_1\nsim\gamma_2$,
the TOIs consist of the usual clockwise angles from $\gamma_1$ to $\gamma_2$ at intersections not at punctures, and at punctures, the clockwise angles will count according to the tagging.
When $\gamma_1\sim\gamma_2$, certain modification is required.
For completeness, here we list all the cases where the usual clockwise angle does not count as TOI in the GMSp.
Let $q=\gamma_1(t_1)=\gamma_2(t_2)$ be an intersection from $(\gamma_1,\kappa_1)$ to $(\gamma_2,\kappa_2)$ and
we have:
\begin{itemize}
\item[(ii)] $q\in\M$, $\gamma_1|_{t_1\to 1-t_1}\sim\gamma_2|_{t_2\to 1-t_2}$,
    $\kappa_1\neq\kappa_2$ with the other endpoints of $\gamma_i$ are in $\P$ and
    $\near{\gamma_1}{t_1}$ is to the left of $\near{\gamma_2}{t_2}$.
	Then the clockwise angle at $q$ as in picture (ii) in Figure~\ref{fig:oriented int2} does not count a TOI.
\item[(iii)/(iv)] In type (III) or (IV) in Definition~\ref{def:int},
the clockwise angle from $\gamma_2$ to $\gamma_1$ at $q$ (cf. the angle that is not drawn in the picture (III) or (IV) in Figure~\ref{fig:oriented int}) does not count.
\item[(v)] In type (V) in Definition~\ref{def:int},
the clockwise angle from $\gamma_2$ to $\gamma_{1}$ at $q$ facing towards the nearby closed marked point (cf. the angle that not drawn in the picture (V.1)/(V.2) in Figure~\ref{fig:oriented int}) do not count.
\item[(vi)] In type (VI) in Definition~\ref{def:int},
the clockwise angle from $\gamma_2$ to $\gamma_1$ at $q$ facing away from the nearby closed marked point (cf. the angle that is not drawn in the picture (VI) in Figure~\ref{fig:oriented int}) do not count.
\item[(vii)] $q\in \P$,
	   $\kappa_1(t_1)=\kappa_2(t_2)$, $\near{\gamma_1}{t_1}$ is to the left of $\near{\gamma_2}{t_2}$, $\gamma_1|_{t_1\to 1-t_1}\sim\gamma_2|_{t_2\to 1-t_2}$, and $\gamma_1(1-t_1)=\gamma_2(1-t_2)\in\P$, $\kappa_1(1-t_1)\neq\kappa_2(1-t_2)$.
	   Then the clockwise angle at $q$ facing away from the nearby closed marked point as in picture (vii) in Figure~\ref{fig:oriented int2} does not count.
	\end{itemize}
\end{remark}

\begin{figure}[htpb]
\begin{tikzpicture}[scale=.5]
\begin{scope}[shift={(0,0)}, scale=.5]
	\draw[ultra thick]plot [smooth,tension=1] coordinates {(3,-4.5) (-90:4) (-3,-4.5)};
	\draw[ultra thick, blue] (90:0.1) \nn (90:8) \nn node[above]{$q$};
	\draw[red] (90:-4) .. controls +(45:5) and +(0:4) .. (0,3.5);
	\draw[red] (90:-4) .. controls +(135:5) and +(180:4) .. (0,3.5);
	\draw[blue, very thick,bend left] (90:8)to(90:0.1);
	\draw[blue, very thick,bend right] (90:8)to(90:0.1);
	\draw(0,4)node{$\sim$}(75:4)node[right,blue]{$\gamma_1$}(105:4)node[left,blue]{$\gamma_2$}(0,1.2)node{$\neq$};
	\draw[ultra thick](4,8)to(-4,8);
	\draw[red] (-90:4)\ww ;
	\draw (0,-6)node{(ii)};
	\draw[Green,->-=.99, thick,>=stealth](.8,6.3)to[bend left=15](-.8,6.3);
\end{scope}	
\begin{scope}[shift={(8,-0.5)}, scale=.5]
	\draw[ultra thick]plot [smooth,tension=1] coordinates {(3,-3.5) (-90:3) (-3,-3.5)};
	\draw[ultra thick, blue] (90:0.1) \nn (90:8) \nn;
	\draw[red] (90:-3) .. controls +(45:4) and +(0:3) .. (0,2.5);
	\draw[red] (90:-3) .. controls +(135:4) and +(180:3) .. (0,2.5);
	\draw[blue, very thick,bend left] (90:8)to(90:0.1);
	\draw[blue, very thick,bend right] (90:8)to(90:0.1);
	\draw (90:0.1)node[below,black]{$q$}(0,4)node{$\sim$}(75:4)node[right,blue]{$\gamma_2$}(105:4)node[left,blue]{$\gamma_1$}(0,1)node{$=$}(0,6.5)node{$\neq$};
	\draw[red] (-90:3)\ww ;
	\begin{scope}[shift={(0,8)}, rotate=180]
		\draw[red] (90:-3) .. controls +(45:4) and +(0:3) .. (0,2.5);
		\draw[red] (90:-3) .. controls +(135:4) and +(180:3) .. (0,2.5);
		\draw[ultra thick]plot [smooth,tension=1] coordinates {(3,-3.5) (-90:3) (-3,-3.5)};
		\draw[red] (-90:3)\ww ;	
	\end{scope}
	\draw (0,-5)node{(vii)};
	\draw[thick, Green,->-=.99,>=stealth](115:2)to[bend left=30](65:2);
\end{scope}	
	\end{tikzpicture}
	\caption{Cases of oriented intersections that do not count TOI}
	\label{fig:oriented int2}
\end{figure}

\begin{lemma}\label{lem:int}
Let $\ws$ and $\wt$ be graded unknotted arcs in $\gUC(\gmsx)$. Then we have
\begin{gather}\label{eq:intb=intx}
\oIntd(\ws,\wt)=\oIntd(\ws^\times,\wt^\times),
\end{gather}
for any $\rho\in\ZZ$.
\end{lemma}

\begin{proof}
Since the winding number of a puncture/binary is always one,
the intersection index will match in \eqref{eq:intb=intx} when passing from GMSp to GMSb (with any representatives).
For instance, the indexes of the intersections shown in the left and right figures of any of the Figures~\ref{fig:case(1)}, \ref{fig:case(2)}, \ref{fig:case(3)} or \ref{fig:case(4)} are the same. Thus, we only need to prove the equality by counting intersections without mentioning their indexes.

We say an oriented intersection $x\in \overrightarrow{\cap}(\ws,\wt)$ (in GMSb) is a nearby one (with respect to a binary $\Vot_p$) if
$\ws$ and $\wt$ end at $m_p$ and the arc segments from $x$ to $m_p$ in $\ws$ and $\wt$ are isotopy
after passing from GMSb to GMSp.
Then, if $x$ is an interior nearby intersection,
it must look like the top point in the left picture of Figure~\ref{fig:nomaps} (but may with extra self-intersections).
Such an intersection can be eliminated or moved to the near binary by choosing another representative of $\ws$ or $\wt$.
Thus we may assume that there is no interior nearby intersection.
Moreover, when passing from GMSb to GMSp, the non-nearby interior intersections in $\overrightarrow{\cap}(\ws,\wt)$
will remain as interior intersections in the tagged case  (since no digon will be produced).
This implies that there is a natural bijection between the interior intersections in GMSb and GMSp, i.e. they contribute the same in \eqref{eq:intb=intx}.
Thus, we only need to do a case-by-case study to match
the endpoint intersections in GMSb with endpoint intersections in GMSp.

First, let $\wsx(t_1)=\wtx(t_2)=m\in\M$ for some $t_1,t_2\in\{0,1\}$.
There are two cases, i.e.
case (ii) in Remark~\ref{rmk:int} (where $m$ does not count in the right hand side of \eqref{eq:intb=intx})
and type (II) in Definition~\ref{def:int}
(where it counts).
It is straightforward to check that
the same holds for GMSb model. More precisely,  Figure~\ref{fig:case(ii)} shows that the oriented intersections at $m$ do not count in case (ii). On the other hand, in type (II), either $\wsx\nsim\wtx$ or $\wsx\sim\wtx$ and their taggings (if there are any) at the other end coincide. Either way, the change of the other end (if there is any) does not affect the local relative position at $m$. Therefore, the intersection at $m$ counts too.
 \begin{figure}[htpb]
	\begin{tikzpicture}[scale=.5]
		\begin{scope}[shift={(0,0)}, scale=.5]
			\draw[ultra thick]plot [smooth,tension=1] coordinates {(3,-4.5) (-90:4) (-3,-4.5)};
			\draw[ultra thick, blue] (90:0.1) \nn (90:8) \nn node[above]{$m$};
			\draw[red] (90:-4) .. controls +(45:5) and +(0:4) .. (0,3.5);
			\draw[red] (90:-4) .. controls +(135:5) and +(180:4) .. (0,3.5);
			\draw[blue, very thick,bend left] (90:8)to(90:0.1);
			\draw[blue, very thick,bend right] (90:8)to(90:0.1);
			\draw(0,4)node{$\sim$}(75:4)node[right,blue]{$\ws^\times$}(105:4)node[left,blue]{$\wt^\times$}(0,1.2)node{$\neq$};
			\draw[ultra thick](3,8)to(-3,8);
			\draw[red] (-90:4)\ww ;
			\draw (0,-6)node{(ii)};
			\draw(8,3)node[thick, black]{$\Longleftrightarrow$};
		\end{scope}	
		\begin{scope}[shift={(8,0)}, scale=.5]
			\draw[ultra thick]plot [smooth,tension=1] coordinates {(3,-4.5) (-90:4) (-3,-4.5)};
			\draw[ultra thick,fill=gray!10] (90:0.5) ellipse (1.5) node {$\Vot_p$};
			\draw[red] (90:-4) .. controls +(45:5) and +(30:4) .. (90:2);
			\draw[red] (90:-4) .. controls +(135:5) and +(150:4) .. (90:2);			
			\draw[very thick,blue] (-90:1)\nn  (0,2)\ww;
			\draw[red] (-90:4)\ww ;
			\draw(45:5.0)node[thick, cyan]{$\ws$}(135:5.0)node[thick, blue]{$\wt$};		
			\draw[blue, ultra thick] (90:8)
			.. controls +(-140:5) and  +(-150:5) .. (-90:1);
			\draw[cyan, ultra thick](90:8)
			.. controls +(-40:5) and  +(-30:5) .. (-90:1);
			\draw[ultra thick](3,8)to(-3,8);
			\draw[ultra thick, blue] (90:8) \nn node[above]{$m$};
			\draw(8,3)node[thick, black]{\Huge{$\rightsquigarrow$}};
		\end{scope}
		\begin{scope}[shift={(16,0)}, scale=.5]
		\draw[ultra thick]plot [smooth,tension=1] coordinates {(3,-4.5) (-90:4) (-3,-4.5)};
		\draw[ultra thick,fill=gray!10] (90:0.5) ellipse (1.5) node {$\Vot_p$};
		\draw[red] (90:-4) .. controls +(45:5) and +(30:4) .. (90:2);
		\draw[red] (90:-4) .. controls +(135:5) and +(150:4) .. (90:2);			
		\draw[very thick,blue] (-90:1)\nn  (0,2)\ww;
		\draw[red] (-90:4)\ww ;
		\draw(45:5.0)node[thick, Green]{$\ws'$}(120:4.8)node[thick, blue,right]{$\wt$};		
		\draw[blue, ultra thick] (90:8)
		.. controls +(-140:5) and  +(-150:5) .. (-90:1);
		\draw[Green, ultra thick] (90:8)
	.. controls +(-160:5) and +(-150:6) .. (0,-2.5)
	.. controls +(30:5) and +(15:4) .. (0,3)
	.. controls +(-165:3.5) and +(-150:2.5) .. (-90:1);
		\draw[ultra thick](3,8)to(-3,8);
		\draw[ultra thick, blue] (90:8) \nn node[above]{$m$};
		\end{scope}
	\end{tikzpicture}
	\caption{Case (ii) in Remark~\ref{rmk:int}}
	\label{fig:case(ii)}
\end{figure}

Next, let $\wsx(t_1)=\wtx(t_2)=p\in\P$ for some $t_1,t_2\in\{0,1\}$ and denote by $(i,j)$ the loop in $\dac$ enclosing the puncture $p$.
Suppose that $(\wsx,\wtx,p)=((\wg_1,\kappa_1),(\wg_2,\kappa_2),q)$ as in one of the types (III)$\sim$(VI) of Definition~\ref{def:int}, where we will show that the oriented intersection at $p$ contributes one in \eqref{eq:intb=intx} in both sides/models.
    \begin{enumerate}
        \item[(III)] In the GMSb model, there is also a clockwise angle at $m_p$ from $\ws$ to $\wt$
        when they are in a minimal position.
        See the second and third pictures in Figure~\ref{fig:case(1)} for the two cases depending on the tagging of $\wsx$ and $\wtx$ at $p$.
        \begin{figure}[htpb]
            \begin{tikzpicture}[scale=.5]
            \begin{scope}[shift={(0,0)}, scale=.5]
			\draw[ultra thick]plot [smooth,tension=1] coordinates {(3,-4.5) (-90:4) (-3,-4.5)};
			\draw[ultra thick] (90:0.1) \nn;
			\draw[red] (90:-4) .. controls +(45:5) and +(0:4) .. (0,3.5);
			\draw[red] (90:-4) .. controls +(135:5) and +(180:4) .. (0,3.5);
			\draw[blue, very thick] (90:0.1)to(60:6)node[right]{$\wt^\times$};
			\draw[cyan, very thick] (90:0.1)to(120:6)node[left]{$\ws^\times$};
			\draw (90:0.1)node[below,black]{$p$}(0,5)node{$\nsim$}(0,1)node{$=$};
			\draw[thick, Green,->-=.99,>=stealth](120:2)to[bend left=30](60:2);
			\draw[red] (-90:4)\ww ;
			\draw (0,-6)node{(III)};
			\draw(0:8)node[thick, black]{$\Longleftrightarrow$};
		    \end{scope}	
		    \begin{scope}[shift={(9,0)}, scale=.5]
			\draw[ultra thick]plot [smooth,tension=1] coordinates {(3,-4.5) (-90:4) (-3,-4.5)};
			\draw[ultra thick,fill=gray!10] (90:2) ellipse (1.5) node {$\Vot_p$};
			\draw[red] (90:-4) .. controls +(45:5) and +(30:4) .. (90:3.5);
			\draw[red] (90:-4) .. controls +(135:5) and +(150:4) .. (90:3.5);
			\draw[blue, ultra thick] (90:8) .. controls +(-45:5) and +(-30:7) .. (90:.5);
			\draw[cyan, ultra thick] (90:5)
			.. controls +(-45:4) and +(-30:3) .. (90:.5);
			\draw[thick, Green,->-=.99,>=stealth](15:2)to[bend left=30](0:3.5);
			\draw[very thick,blue] (90:.5)\nn   (90:.5)node[below]{$m_p$} (0,3.5)\ww;
			\draw(90:8)node[left, blue]{$\wt$}(90:5)node[left, cyan]{$\ws$};
			\draw[red] (-90:4)\ww ;
		    \end{scope}
		    \begin{scope}[shift={(16,0)}, scale=.5]
			\draw[ultra thick]plot [smooth,tension=1] coordinates {(3,-4.5) (-90:4) (-3,-4.5)};
			\draw[ultra thick,fill=gray!10] (90:2) ellipse (1.5) node {$\Vot_p$};
			\draw[red] (90:-4) .. controls +(45:5) and +(30:4) .. (90:3.5);
			\draw[red] (90:-4) .. controls +(135:5) and +(150:4) .. (90:3.5);
			\draw[cyan, ultra thick] (90:8) .. controls +(-135:5) and +(-150:7) .. (90:.5);
			\draw[blue, ultra thick] (90:5)
			.. controls +(-135:4) and +(-150:3) .. (90:.5);
			\draw[thick, Green,->-=.99,>=stealth](180:3.5)to[bend left=30](165:2);
			\draw[very thick,blue] (90:.5)\nn    (90:.5)node[below]{$m_p$}  (0,3.5)\ww;
			\draw(90:8)node[right, cyan]{$\ws$}(90:5)node[right, blue]{$\wt$};
			\draw[red] (-90:4)\ww ;
		    \end{scope}
            \end{tikzpicture}
            \caption{Type (III): GMSp v.s. GMSb}
            \label{fig:case(1)}
        \end{figure}
     \item[(IV)] In the GMSb model,
     there is also a clockwise angle at $m_p$ from $\ws$ to $\wt$ when they are in minimal position.
     See the second and third pictures in Figure~\ref{fig:case(2)} for the two cases depending on the tagging of $\wsx$ and $\wtx$ at $p$.
    \begin{figure}[htpb]\qquad
    	\begin{tikzpicture}[scale=.5]
    		\begin{scope}[shift={(0,0)}, scale=.5]
    			\draw[ultra thick]plot [smooth,tension=1] coordinates {(3,-4.5) (-90:4) (-3,-4.5)};
    			\draw[ultra thick] (90:0.1) \nn;
    			\draw[red] (90:-4) .. controls +(45:5) and +(0:4) .. (0,3.5);
    			\draw[red] (90:-4) .. controls +(135:5) and +(180:4) .. (0,3.5);
    			\draw[cyan, very thick] (90:0.1)to(60:7)node[above]{$\ws^\times$};
    			\draw[blue, very thick] (90:0.1)to(120:7)node[above]{$\wt^\times$};
    			\draw(90:0.1)node[below,black]{$p$}(0,5)node{$\nsim$}(0,1)node{$\neq$};
    			\draw[Green,->-=.99, very thick,>=stealth]plot [smooth,tension=1] coordinates {(60:1.7) (0:1.7) (-90:1.7) (-180:1.7) (-240:1.7)};
    			\draw[red] (-90:4)\ww ;
    			\draw (0,-6)node{(IV)};
    			\draw(0:8)node[thick, black]{$\Longleftrightarrow$};
    		\end{scope}		
    		\begin{scope}[shift={(9,0)}, scale=.5]
    			\draw[ultra thick]plot [smooth,tension=1] coordinates {(3,-4.5) (-90:4) (-3,-4.5)};
    			\draw[ultra thick,fill=gray!10] (90:2) ellipse (1.5) node {$\Vot_p$};
    			\draw[red] (90:-4) .. controls +(45:5) and +(30:4) .. (90:3.5);
    			\draw[red] (90:-4) .. controls +(135:5) and +(150:4) .. (90:3.5);
    			\draw[blue, ultra thick] (100:8) .. controls +(-135:5) and +(-150:7) .. (90:.5);
    			\draw[cyan, ultra thick] (80:8)
    			.. controls +(-45:5) and +(-30:7) .. (90:.5);
    			\draw[thick, Green,->-=.99,>=stealth](-15:1.5)to[bend left=30](-165:1.5);
    			\draw[very thick,blue] (90:.5)\nn  (0,3.5)\ww;
    			\draw(125:6)node[left, thick,blue]{$\wt$}(55:6)node[right, thick, cyan]{$\ws$}(0,5)node[black]{$\nsim$};
    			\draw[red] (-90:4)\ww ;
    		\end{scope}		
		\begin{scope}[shift={(16,0)}, scale=.5]
		\draw[ultra thick]plot [smooth,tension=1] coordinates {(3,-4.5) (-90:4) (-3,-4.5)};
		\draw[ultra thick,fill=gray!10] (90:0.5) ellipse (1.5) node {$\Vot_p$};
		\draw[red] (90:-4) .. controls +(45:5) and +(30:4) .. (90:2);
		\draw[red] (90:-4) .. controls +(135:5) and +(150:4) .. (90:2);			
		\draw[very thick,blue] (-90:1)\nn  (0,2)\ww;
		\draw[red] (-90:4)\ww ;
		\draw(45:5.0)node[thick, blue]{$\wt$}(120:4.8)node[thick, cyan]{$\ws$};
		\draw[blue,dashed, thick] (-2,8)
		.. controls +(-40:5) and  +(-20:7) .. (-90:1);
		\draw[cyan, ultra thick] (2,8)
		.. controls +(-140:5) and  +(-160:7) .. (-90:1);
		\draw[blue, ultra thick] (-2,8)
	.. controls +(-140:5) and +(-150:6) .. (0,-2.5)
	.. controls +(30:5) and +(15:4) .. (0,3)
	.. controls +(-165:3.5) and +(-150:2.5) .. (-90:1);
		\end{scope}
    	\end{tikzpicture}
    	\caption{Type (IV): GMSp v.s. GMSb}
    	\label{fig:case(2)}
    \end{figure}
        \item[(V)] In the GMSb model,
        there is also a clockwise angle at $m_p$ from $\ws$ to $\wt$ when they are in minimal position.
        See the second and third pictures in Figure~\ref{fig:case(3)} for the two cases depending on the tagging of $\wsx$ and $\wtx$ at $p$.
        \begin{figure}[htpb]
            \begin{tikzpicture}[scale=.5]
            \begin{scope}[shift={(0,-10)}, scale=.5]
			\draw[ultra thick]plot [smooth,tension=1] coordinates {(3,-4.5) (-90:4) (-3,-4.5)};
			\draw[ultra thick, blue] (90:0.1) \nn (90:8) \nn;
			\draw[red] (90:-4) .. controls +(45:5) and +(0:4) .. (0,3.5);
			\draw[red] (90:-4) .. controls +(135:5) and +(180:4) .. (0,3.5);
			\draw[blue, very thick,bend left] (90:8)to(90:0.1);
			\draw[cyan, very thick,bend right] (90:8)to(90:0.1);
			\draw (90:0.1)node[below,black]{$p$}(0,4)node{$\sim$}(75:4)node[right,blue]{$\wt^\times$}(105:4)node[left,cyan]{$\ws^\times$}(0,1)node{$=$}(90:8)node[above,black]{$M$} ;
			\draw[thick, Green,->-=.99,>=stealth](115:2)to[bend left=30](65:2);
		 	\draw[ultra thick](2,8)to(-2,8);
			\draw[red] (-90:4)\ww ;
			\draw (0,-6)node{(V.1)};
		\end{scope}
	 \begin{scope}[shift={(5,-10.5)}, scale=.5]
	 	\draw[ultra thick]plot [smooth,tension=1] coordinates {(3,-3.5) (-90:3) (-3,-3.5)};
	 	\draw[ultra thick, blue] (90:0.1) \nn (90:6) \nn;
	 	\draw[red] (90:-3) .. controls +(45:4) and +(0:3) .. (0,2.5);
	 	\draw[red] (90:-3) .. controls +(135:4) and +(180:3) .. (0,2.5);
	 	\draw[blue, very thick,bend left] (90:6)to(90:0.1);
	 	\draw[cyan, very thick,bend right] (90:6)to(90:0.1);
	 	\draw (90:0.1)node[below,black]{$p$}(0,3)node{$\sim$}(75:3)node[right,blue]{$\wt^\times$}(105:3)node[left,cyan]{$\ws^\times$}(0,1)node{$=$}(0,5)node{$=$}(90:6)node[above,black]{$p'$};
	 	\draw[thick, Green,->-=.99,>=stealth](115:2)to[bend left=30](65:2);
	 	\draw[red] (-90:3)\ww ;
	 	\begin{scope}[shift={(0,6)}, rotate=180]
	 		\draw[red] (90:-3) .. controls +(45:4) and +(0:3) .. (0,2.5);
	 		\draw[red] (90:-3) .. controls +(135:4) and +(180:3) .. (0,2.5);
	 		\draw[ultra thick]plot [smooth,tension=1] coordinates {(3,-3.5) (-90:3) (-3,-3.5)};
	 		\draw[red] (-90:3)\ww ;	
	 	\end{scope}
	 	\draw (0,-5)node{(V.2)};
	 	\draw(15:8.5)node[thick, black]{$\Longleftrightarrow$};
	 \end{scope}		
		\begin{scope}[shift={(13,-10)},scale=.5]
			\draw[ultra thick]plot [smooth,tension=1] coordinates {(3,-4.5) (-90:4) (-3,-4.5)};
			\draw[ultra thick,fill=gray!10] (90:2) ellipse (1.5) node {$\Vot_p$};
			\draw[red] (90:-4) .. controls +(45:5) and +(30:4) .. (90:3.5);
			\draw[red] (90:-4) .. controls +(135:5) and +(150:4) .. (90:3.5);
			\draw[blue, ultra thick] (90:8) .. controls +(-25:5) and +(-30:7) .. (90:.5);
			\draw[cyan, ultra thick] (90:8)
			.. controls +(-45:4) and +(-30:3) .. (90:.5);
			\draw[thick, Green,->-=.99,>=stealth](25:1.5)to[bend left=30](0:3.5);
			\draw[very thick,blue] (90:.5)\nn  (0,3.5)\ww;
			\draw(60:5.7)node[right, blue]{$\wt$}(90:5)node[right, cyan]{$\ws$}(60:5.2)node[black]{$\sim$}(90:8)node[above,black]{$M \text{ or } \Vot_{p'}$};
			\draw[thick] (90:8)\nn;
			\draw[red] (-90:4)\ww ;
		\end{scope}
		\begin{scope}[shift={(19,-10)}, scale=.5]
			\draw[ultra thick]plot [smooth,tension=1] coordinates {(3,-4.5) (-90:4) (-3,-4.5)};
			\draw[ultra thick,fill=gray!10] (90:2) ellipse (1.5) node {$\Vot_p$};
			\draw[red] (90:-4) .. controls +(45:5) and +(30:4) .. (90:3.5);
			\draw[red] (90:-4) .. controls +(135:5) and +(150:4) .. (90:3.5);
			\draw[cyan, ultra thick] (90:8) .. controls +(-155:5) and +(-150:7) .. (90:.5);
			\draw[blue, ultra thick] (90:8)
			.. controls +(-135:4) and +(-150:3) .. (90:.5);
			\draw[thick, Green,->-=.99,>=stealth](180:3.5)to[bend left=30](150:2);
			\draw[very thick,blue] (90:.5)\nn  (0,3.5)\ww;
			\draw(120:5.5)node[left, cyan]{$\ws$}(90:4.8)node[left, blue]{$\wt$}(120:5)node[black]{$\sim$}(90:8)node[above,black]{$M \text{ or } \Vot_{p'}$};
			\draw[thick] (90:8)\nn;
			\draw[red] (-90:4)\ww ;
		\end{scope}
            \end{tikzpicture}
            \caption{Type (V): GMSp v.s. GMSb}
            \label{fig:case(3)}
        \end{figure}
   \item[(VI)] In the GMSb model, $\ws$ and $\wt$ are shown in the middle and right pictures of Figure~\ref{fig:case(4)}, for two possible minimal position.
       Note that one needs to check that the intersection indices do match in these two cases, which also follows from the fact that the winding number around a binary is one.
   \begin{figure}[htpb]\centering	
   	\begin{tikzpicture}[scale=.5]
   		\begin{scope}[shift={(0,-0.5)}, scale=.5]
   			\draw[ultra thick]plot [smooth,tension=1] coordinates {(3,-3.5) (-90:3) (-3,-3.5)};
   			\draw[ultra thick, blue] (90:0.1) \nn (90:8) \nn;
   			\draw[red] (90:-3) .. controls +(45:4) and +(0:3) .. (0,2.5);
   			\draw[red] (90:-3) .. controls +(135:4) and +(180:3) .. (0,2.5);
   			\draw[blue, very thick,bend left] (90:8)to(90:0.1);
   			\draw[blue, very thick,bend right] (90:8)to(90:0.1);
   			\draw (90:0.1)node[below,black]{$p$}(0,4)node{$\sim$}(75:4)node[right,blue]{$\ws^\times$}(105:4)node[left,blue]{$\wt^\times$}(0,1)node{$\neq$}(0,6.5)node{$\neq$};
   			\draw[Green,->-=.99, thick,>=stealth]plot [smooth,tension=1] coordinates
   			{(60:1.6) (0:1.6) (-90:1.6) (-180:1.6) (-240:1.6)};
   			\draw[red] (-90:3)\ww ;
   			\begin{scope}[shift={(0,8)}, rotate=180]
   				\draw[red] (90:-3) .. controls +(45:4) and +(0:3) .. (0,2.5);
   				\draw[red] (90:-3) .. controls +(135:4) and +(180:3) .. (0,2.5);
   				\draw[ultra thick]plot [smooth,tension=1] coordinates {(3,-3.5) (-90:3) (-3,-3.5)};
   				\draw[red] (-90:3)\ww ;	
   			\end{scope}
   			\draw (0,-5)node{(VI)};
   			\draw(7,4)node[thick, black]{$\Longleftrightarrow$};
   		\end{scope}		
   		\begin{scope}[shift={(8,0)}, scale=.5]
   			\draw[ultra thick]plot [smooth,tension=1] coordinates {(3,-4.5) (-90:4) (-3,-4.5)};
   			\draw[ultra thick,fill=gray!10] (90:0.5) ellipse (1.5) node {$\Vot_p$};
   			\draw[red] (90:-4) .. controls +(45:5) and +(30:4) .. (90:2);
   			\draw[red] (90:-4) .. controls +(135:5) and +(150:4) .. (90:2);			
   			\draw(45:6.5)node[thick, cyan]{$\ws$}(135:6.5)node[thick, blue]{$\wt$};		
   			\draw[blue, ultra thick] (90:7)
   			.. controls +(150:5) and  +(-150:5) .. (-90:1);
   			\draw[cyan, ultra thick](90:7)
   			.. controls +(30:5) and  +(-30:5) .. (-90:1);
   			\draw[thick, Green,->-=.99,>=stealth](-50:1.9)to[bend left=30](-130:1.9);
   			\draw[very thick,blue] (-90:1)\nn  (0,2)\ww;
   			\draw[red] (-90:4)\ww ;
   			\begin{scope}[shift={(0,6)}, rotate=180]
   				\draw[ultra thick]plot [smooth,tension=1] coordinates {(3,-4.5) (-90:4) (-3,-4.5)};
   				\draw[ultra thick,fill=gray!10] (90:0.5) ellipse (1.5) node {$\Vot_{p'}$};
   				\draw[red] (90:-4) .. controls +(45:5) and +(30:4) .. (90:2);
   				\draw[red] (90:-4) .. controls +(135:5) and +(150:4) .. (90:2);
   				\draw[very thick,blue] (-90:1)\nn  (0,2)\ww;
   				\draw[red] (-90:4)\ww;	
   			\end{scope}
   		\end{scope}
   			\draw(12,1)node[thick, black]{\Huge{$\rightsquigarrow$}};
   		\begin{scope}[shift={(16,0)}, scale=.5]
   			\draw[ultra thick]plot [smooth,tension=1] coordinates {(3,-4.5) (-90:4) (-3,-4.5)};
   			\draw[ultra thick,fill=gray!10] (90:0.5) ellipse (1.5) node {$\Vot_p$};
   			\draw[red] (90:-4) .. controls +(45:5) and +(30:4) .. (90:2);
   			\draw[red] (90:-4) .. controls +(135:5) and +(150:4) .. (90:2);			
   			\draw(135:7.5)node[thick, Green]{$\ws$}(135:5.5)node[thick, blue]{$\wt$};		
   			\draw[blue, ultra thick] (90:7)
   			.. controls +(150:5) and  +(-150:5) .. (-90:1);
\draw[Green,ultra thick](-90:1)arc(270:90:1.8)arc(90:-90:2.3) arc(270:90:5)arc(90:-90:2.3)arc(270:90:1.8);
   			\draw[very thick,blue] (-90:1)\nn  (0,2)\ww;
   			\draw[red] (-90:4)\ww ;
   			\begin{scope}[shift={(0,6)}, rotate=180]
   				\draw[ultra thick]plot [smooth,tension=1] coordinates {(3,-4.5) (-90:4) (-3,-4.5)};
   				\draw[ultra thick,fill=gray!10] (90:0.5) ellipse (1.5) node {$\Vot_{p'}$};
   				\draw[red] (90:-4) .. controls +(45:5) and +(30:4) .. (90:2);
   				\draw[red] (90:-4) .. controls +(135:5) and +(150:4) .. (90:2);
   				\draw[very thick,blue] (-90:1)\nn  (0,2)\ww;
   				\draw[red] (-90:4)\ww;	
   			\end{scope}
   		\end{scope}
   	\end{tikzpicture}
	\caption{Type (VI): GMSp v.s. GMSb}
    \label{fig:case(4)}
   \end{figure}
  \end{enumerate}

Finally, let $\wsx(t_1)=\wtx(t_2)=p\in\P$ as above
and $(\wsx,\wtx,p)=((\wg_1,\kappa_1),(\wg_2,\kappa_2),q)$ as in one of the cases (iii)$\sim$(vii) of Remark~\ref{rmk:int}, where we will show that $p$ does not contribute in \eqref{eq:intb=intx} in both sides/models.
   \begin{enumerate}
   \item[(vii)] In case (vii) of Remark~\ref{rmk:int},
   we want to show that there is no oriented intersection at $m_p$
   between $\ws$ and $\wt$ in the GMSb model.
   This can be done as shown in Figure~\ref{fig:case(5)}, where we only draw one of similar cases
   (as other cases can be obtained by taking a mirror of the upper/lower half of the pictures).
   More precisely, one fixes $\wt$ without loss of generality first and
   \begin{itemize}
   \item by choosing the representative of $\ws$
   in the middle picture of Figure~\ref{fig:case(5)}, one sees that
   there is no oriented intersection at $m_p$ from $\wt$ to $\ws$;
   \item by choosing the representative of $\ws$
   in the right picture of Figure~\ref{fig:case(5)}, one sees that
   there is no oriented intersection at $m_p$ from $\ws$ to $\wt$.
   \end{itemize}

   \begin{figure}[htpb]
\begin{tikzpicture}[scale=.5]
\begin{scope}[shift={(8,-0.5)}, scale=.5]
	\draw[ultra thick]plot [smooth,tension=1] coordinates {(3,-3.5) (-90:3) (-3,-3.5)};
	\draw[ultra thick, blue] (90:0.1) \nn (90:8) \nn;
	\draw[red] (90:-3) .. controls +(45:4) and +(0:3) .. (0,2.5);
	\draw[red] (90:-3) .. controls +(135:4) and +(180:3) .. (0,2.5);
	\draw[blue, very thick,bend left] (90:8)to(90:0.1);
	\draw[blue, very thick,bend right] (90:8)to(90:0.1);
	\draw (90:0.1)node[below,black]{$p$}(0,4)node{$\sim$}(75:4)node[right,blue]{$\wtx$}(105:4)node[left,blue]{$\wsx$}(0,1)node{$=$}(0,6.5)node{$\neq$};
	\draw[red] (-90:3)\ww ;
	\draw(8,4)node{$\Leftrightarrow$}(12,4)node{ };
	\begin{scope}[shift={(0,8)}, rotate=180]
		\draw[red] (90:-3) .. controls +(45:4) and +(0:3) .. (0,2.5);
		\draw[red] (90:-3) .. controls +(135:4) and +(180:3) .. (0,2.5);
		\draw[ultra thick]plot [smooth,tension=1] coordinates {(3,-3.5) (-90:3) (-3,-3.5)};
		\draw[red] (-90:3)\ww ;	
	\end{scope}
	\draw (0,-5)node{(vii)};
\end{scope}	
	\end{tikzpicture}
   	\begin{tikzpicture}[scale=.5]
   		\begin{scope}[shift={(0,0)}, scale=.5]
   			\draw[ultra thick]plot [smooth,tension=1] coordinates {(3,-4.5) (-90:4) (-3,-4.5)};
   			\draw[ultra thick,fill=gray!10] (90:2) ellipse (1.5) node {$\Vot_p$};
   			\draw[red] (90:-4) .. controls +(45:5) and +(30:4) .. (90:3.5);
   			\draw[red] (90:-4) .. controls +(135:5) and +(150:4) .. (90:3.5);			
   			\draw[very thick,blue] (90:.5)\nn  (90:.5)node[below]{$m_p$}   (0,3.5)\ww;
   			\draw[red] (-90:4)\ww ;
   			\begin{scope}[shift={(0,10)}, rotate=180]
   				\draw[ultra thick]plot [smooth,tension=1] coordinates {(3,-4.5) (-90:4) (-3,-4.5)};
   				\draw[ultra thick,fill=gray!10] (90:2) ellipse (1.5) node {$\Vot_{p'}$};
   				\draw[red] (90:-4) .. controls +(45:5) and +(30:4) .. (90:3.5);
   				\draw[red] (90:-4) .. controls +(135:5) and +(150:4) .. (90:3.5);
   				\draw[very thick,blue] (90:.5)\nn (0,3.5)\ww;
   				\draw[red] (-90:4)\ww;	
   			\end{scope}
   			\draw(120:8)node[thick, left, cyan]{$\ws$}(120:5)node[thick, blue]{$\wt$};		
   			\draw[cyan, ultra thick] (90:9.5)
   			.. controls +(150:7) and  +(-150:7) .. (90:.5);
   			\draw[blue, ultra thick] (90:9.5)
   			.. controls +(30:4) and +(15:5) .. (0,5)
   			.. controls +(-165:5) and +(-150:4) .. (0,.5);
   			\draw(8,5.5)node[thick, black]{\Huge{$\rightsquigarrow$}};
   		\end{scope}
   		\begin{scope}[shift={(8,0)}, scale=.5]
   			\draw[ultra thick]plot [smooth,tension=1] coordinates {(3,-4.5) (-90:4) (-3,-4.5)};
   			\draw[ultra thick,fill=gray!10] (90:2) ellipse (1.5) node {$\Vot_p$};
   			\draw[red] (90:-4) .. controls +(45:5) and +(30:4) .. (90:3.5);
   			\draw[red] (90:-4) .. controls +(135:5) and +(150:4) .. (90:3.5);			
   			\draw[very thick,blue] (90:.5)\nn  (0,3.5)\ww  (90:.5)node[below]{$m_p$} ;
   			\draw[red] (-90:4)\ww ;
   			\begin{scope}[shift={(0,10)}, rotate=180]
   				\draw[ultra thick]plot [smooth,tension=1] coordinates {(3,-4.5) (-90:4) (-3,-4.5)};
   				\draw[ultra thick,fill=gray!10] (90:2) ellipse (1.5) node {$\Vot_{p'}$};
   				\draw[red] (90:-4) .. controls +(45:5) and +(30:4) .. (90:3.5);
   				\draw[red] (90:-4) .. controls +(135:5) and +(150:4) .. (90:3.5);
   				\draw[very thick,blue] (90:.5)\nn  (0,3.5)\ww;
   				\draw[red] (-90:4)\ww;	
   			\end{scope}
   			\draw(60:7.5)node[thick, right, Green]{$\ws'$}(120:5)node[thick, blue]{$\wt$};
   			\draw[blue, ultra thick] (90:9.5)
   			.. controls +(30:4) and +(15:5) .. (0,5)
   			.. controls +(-165:5) and +(-150:4) .. (0,.5);
   			\draw[Green, ultra thick] (90:9.5)
   			.. controls +(30:3) and +(10:3) .. (0,6)
   			.. controls +(195:5) and +(205:4) .. (0,10.5)
   			.. controls +(30:5) and +(15:6) .. (0,4)
   			.. controls +(-165:4) and +(-150:3) .. (0,.5);
   		\end{scope}
   	\end{tikzpicture}
   	\caption{Case (vii): GMSp v.s. GMSb}
   	\label{fig:case(5)}
   \end{figure}
\item[O/W] If $(\wsx,\wtx,p)=((\wg_2,\kappa_2),(\wg_1,\kappa_1),q)$ as in one of the cases (iii)$\sim$(vi) in Remark~\ref{rmk:int}, one can switch $\ws$ and $\wt$ in Figures~\ref{fig:case(1)}$\sim$\ref{fig:case(4)}
respectively.
Then the pictures show that there is no oriented intersection at $q$ from $\ws$ to $\wt$ in GMSb model.

\end{enumerate}
\end{proof}

\subsection{Morphisms between direct summands of the graded skew-gentle algebra }

This subsection devotes to giving a basis of the space of morphisms between direct summands of $\sg$, which is compatible with the GMSb model. Note that $\{z\sg\mid z\in\OMEx\}$ is a complete set of non-isomorphic indecomposable direct summands of $\sg$.

Recall from Section~\ref{subsec:clannish2} that we have a basis $\{\lu(x_1,x_2)\mid (x_1,x_2)\in\mathcal{R}\}$ of $\rad\sg$, where $$\mathcal{R}=\{((i,j_1^{\kappa_1}),(i,j_2^{\kappa_2}))\in\OMEpm\times\OMEpm\mid 1\leq i\leq t,1\leq j_1< j_2\leq m_i\}.$$
For any $(x_1,x_2)\in\mathcal{R}$, denote
$$\begin{array}{rccc}
f^{\sg}_{x_1,x_2}:&\ec{x_2}\sg&\to& \ec{x_1}\sg\\
&h&\mapsto&\lu(x_1,x_2)h
\end{array}$$
the morphism in $\huaHom_\sg(\ec{x_2}\sg,\ec{x_1}\sg)$ induced by $\lu(x_1,x_2)$. When there is no confusion arising, we shall drop $\sg$ from the notation $f^{\sg}_{x_1,x_2}$.

For any $z_1,z_2\in\OMEx$, set
$$\mathcal{R}(z_1,z_2)=\{(x_1,x_2)\in\mathcal{R}\mid z_1=\ec{x_1},z_2=\ec{x_2}\}.$$
Then $\huaHom_\sg(z_2\sg,z_1\sg)$ admits a basis $f_{x_1,x_2}$, $(x_1,x_2)\in\mathcal{R}(z_1,z_2)$, together with the identity when $z_1=z_2$.

\begin{lemma}\label{lem:rmkcomp}
For any $(x_1,x_2),(x_3,x_4)\in\mathcal{R}$ with $\overline{x_2}=\overline{x_3}$, we have
$$f_{x_1,x_2}\circ f_{x_3,x_4}=\begin{cases}
f_{x_1,x_4}&\text{if $x_2=x_3$,}\\
0&\text{otherwise.}
\end{cases}$$
\end{lemma}

\begin{proof}
	This follows directly from \eqref{eq:comp} in Remark~\ref{rmk:BR}.
\end{proof}

We consider the sets $\OMEepm:=\OME\cup\OMEpm$ and $\OMEsx:=\OMEs\cup\OMEx$. Note that
$$\OMEepm=\OMEpm\cup\fc{\OME}=\OME\cup\fc{\OMEpm}.$$
Then each element in $\OMEepm$ can be presented as $(i,j^\kappa)$, where $\kappa=\emptyset$ if $(i,j)\notin\fc{\OME}$ and $\kappa\in\{+,-,\emptyset\}$ if $(i,j)\in\fc{\OME}$. There is an induced relation $\simeq$ on $\OMEepm$ and we have $\OMEepm/\simeq=\OMEsx$. We set
$$\mathcal{R}'=\{((i,j_1^{\kappa_1}),(i,j_2^{\kappa_2}))\in\OMEepm\times\OMEepm\mid 1\leq i\leq t,1\leq j_1< j_2\leq m_i\}.$$

For any $z\in\OMEsx$, denote
$$z\sg=\begin{cases}
(i,j^+)\sg\oplus(i,j^-)\sg&\text{if $z=(i,j)\in\fc{\OME}\subset\OMEsx$,}\\
z\sg&\text{if $z\in\OMEx$.}
\end{cases}$$

For any $y\in\OMEepm$, define a subset $\OMEpm(y)\subset\OMEpm$ as
$$\OMEpm(y)=\begin{cases}
\{(i,j^+),(i,j^-)\}&\text{if $y=(i,j)\in\fc{\OME}\subset\OMEepm$,}\\
\{y\}&\text{if $y\in\OMEpm$.}
\end{cases}$$
By definition, for any $y\in\OMEepm$, we have $\ec{y}\sg=\bigoplus_{x\in\OMEpm(y)}\overline{x}\sg$, and for any $(y_1,y_2)\in\mathcal{R}'$, we have $\OMEpm(y_1)\times\OMEpm(y_2)\subset\mathcal{R}$.

\begin{construction}\label{cons:fij}
For any $(y_1,y_2)\in\mathcal{R}'$, define $f_{y_1,y_2}\in\huaHom_{\sg}(\overline{y_2}\sg,\overline{y_1}\sg)$ as the matrix
$$f_{y_1,y_2}=\left(f_{x_1,x_2}\right)_{x_1\in\OMEpm(y_1),x_2\in\OMEpm(y_2)}.$$
\end{construction}
Explicitly, there are the following four cases for Construction~\ref{cons:fij}.
\begin{enumerate}
	\item For $y_1,y_2\in\OMEpm$ (i.e. $(y_1,y_2)\in\mathcal{R}$), we have $$f_{y_1,y_2}=\left(f_{y_1,y_2}\right):\overline{y_2}\sg\to \overline{y_1}\sg.$$
	\item For $y_1=(i,j)\in\fc{\OME},y_2\in\OMEpm$, we have
	$$f_{y_1,y_2}=\begin{pmatrix}f_{(i,j^+),y_2}\\ f_{(i,j^-),y_2}\end{pmatrix}:\overline{y_2}\sg\to\overline{y_1}\sg=\overline{(i,j^+)}\sg\oplus\overline{(i,j^-)}\sg.$$
	\item For $y_1\in\OMEpm,y_2=(i,j)\in\fc{\OME}$, we have
	$$f_{y_1,y_2}=\begin{pmatrix}f_{y_1,(i,j^+)}& f_{y_1,(i,j^-)}\end{pmatrix}:\overline{y_2}\sg=\overline{(i,j^+)}\sg\oplus\overline{(i,j^-)}\sg\to\overline{y_1}\sg.$$
	\item For $y_1=(i,j_1),y_2=(i,j_2)\in\fc{\OME}$, we have
	$$f_{y_1,y_2}=\begin{pmatrix} f_{(i,j_1^+),(i,j_2^+)}& f_{(i,j_1^+),(i,j_2^-)}\\ f_{(i,j_1^-),(i,j_2^+)}& f_{(i,j_1^-),(i,j_2^-)}\end{pmatrix}:\overline{(i,j_2^+)}\sg\oplus\overline{(i,j_2^-)}\sg\to \overline{(i,j_1^+)}\sg\oplus\overline{(i,j_1^-)}\sg.$$
\end{enumerate}

For any $z_1,z_2\in\OMEsx$, let
$$\mathcal{R}'(z_1,z_2)=\{(y_1,y_2)\in\mathcal{R}'\mid \ec{y_1}=z_1,\ec{y_2}=z_2\}.$$
and denote by $\mathcal{R}ad_{\sg}(z_2\sg,z_1\sg)$ the radical of $\huaHom_{\sg}(z_2\sg,z_1\sg)$.

\begin{remark}\label{lem:edge2}
	For any $z_1,z_2\in\OMEsx$, we have the direct sum decomposition:
	$$\mathcal{R}ad_{\sg}(z_2\sg,z_1\sg)=\bigoplus_{(y_1,y_2)\in\mathcal{R}'(z_1,z_2)}\mathcal{M}(y_1,y_2),$$
	where
    $$\mathcal{M}(y_1,y_2)=\{\left(\lambda_{x_1,x_2}f_{x_1,x_2}\right)_{x_1\in\OMEpm(y_1),x_2\in\OMEpm(y_2)}\mid\lambda_{x_1,x_2}\in\k \}.$$
    There is a natural basis of $\mathcal{M}(y_1,y_2)$:
    $$g_{x_1',x_2'}:=\left(\delta_{x_1,x_1'}\delta_{x_2,x_2'}f_{x_1,x_2}\right)_{x_1\in\OMEpm(y_1),x_2\in\OMEpm(y_2)},\ x'_1\in\OMEpm(y_1),x'_2\in\OMEpm(y_2),$$
    such that $$f_{y_1,y_2}=\sum_{x_1'\in\OMEpm(y_1),x_2'\in\OMEpm(y_2)}g_{x_1',x_2'},$$
\end{remark}

To give a basis of $\mathcal{M}(x_1,x_2)$ which is compatible with the GMSb model, we introduce the following notations, which will be used frequently later.

\begin{notations}\label{not:mo}
    For any $z\in\OMEsx$, we define a set
    $$\mdf{z}:=\begin{cases}\{\fen{(i,j^+)}{(i,j^-)},\fen{(i,j^-)}{(i,j^+)}\}&\text{if $z=(i,j)\in\fc{\OME}$,}\\
    \{\fen{z}{\emptyset}\}&\text{otherwise.}
    \end{cases}$$
	For any $z\in\OMEsx$ and any $\xi\in\mdf{z}$, we define $\fout{\xi},\fin{\xi}\in \huaHom_{\sg}(z\sg,z\sg)$ as follows.
	\begin{itemize}
	\item For the case $z=(i,j)\in\fc{\OME}$, define
	$$\fout{\fen{(i,j^+)}{(i,j^-)}}=\begin{pmatrix}
	1&0\\0&-1
	\end{pmatrix}:(i,j^+)\sg\oplus(i,j^-)\sg\to(i,j^+)\sg\oplus(i,j^-)\sg,$$
	$$\fout{\fen{(i,j^-)}{(i,j^+)}}=\begin{pmatrix}
	0&0\\0&1
	\end{pmatrix}:(i,j^+)\sg\oplus(i,j^-)\sg\to(i,j^+)\sg\oplus(i,j^-)\sg,$$
	$$\fin{\fen{(i,j^+)}{(i,j^-)}}=\begin{pmatrix}
	1&0\\0&0
	\end{pmatrix}:(i,j^+)\sg\oplus(i,j^-)\sg\to(i,j^+)\sg\oplus(i,j^-)\sg,$$
	$$\fin{\fen{(i,j^-)}{(i,j^+)}}=\begin{pmatrix}
	1&0\\0&1
	\end{pmatrix}:(i,j^+)\sg\oplus(i,j^-)\sg\to(i,j^+)\sg\oplus(i,j^-)\sg.$$
	\item Otherwise, define
	$$\fout{\fen{z}{\emptyset}}=\fin{\fen{z}{\emptyset}}=\id_{z\sg}.$$
	\end{itemize}
\end{notations}

\begin{lemma}\label{lem:dec}
	For any $z_1,z_2\in\OMEsx$ and any $(y_1,y_2)\in\mathcal{R}'(z_1,z_2)$, the space $\mathcal{M}(y_1,y_2)$ has a basis
	$$\fin{\xi_1}\circ f_{y_1,y_2}\circ \fout{\xi_2},\ \xi_1\in\mdf{z_1},\ \xi_2\in\mdf{z_2}.$$
\end{lemma}

\begin{proof}
	We only deal the most complicated case that $y_l=(i,j_l)\in\fc{\OME},l=1,2$. Then $\mdf{z_l}=\{\fen{(i,j_l^+)}{(i,j_l^-)},\fen{(i,j_l^-)}{(i,j_l^+)}\}$ and we have
	$$\fin{\fen{(i,j_1^+)}{(i,j_1^-)}}\circ f_{y_1,y_2}\circ \fout{\fen{(i,j_2^+)}{(i,j_2^-)}}=\begin{pmatrix} f_{(i,j_1^+),(i,j_2^+)}& -f_{(i,j_1^+),(i,j_2^-)}\\ 0& 0\end{pmatrix},$$
	$$\fin{\fen{(i,j_1^+)}{(i,j_1^-)}}\circ f_{y_1,y_2}\circ \fout{\fen{(i,j_2^-)}{(i,j_2^+)}}=\begin{pmatrix} 0& f_{(i,j_1^+),(i,j_2^-)}\\ 0& 0\end{pmatrix},$$
	$$\fin{\fen{(i,j_1^-)}{(i,j_1^+)}}\circ f_{y_1,y_2}\circ \fout{\fen{(i,j_2^+)}{(i,j_2^-)}}=\begin{pmatrix} f_{(i,j_1^+),(i,j_2^+)}& -f_{(i,j_1^+),(i,j_2^-)}\\ f_{(i,j_1^-),(i,j_2^+)}& -f_{(i,j_1^-),(i,j_2^-)}\end{pmatrix},$$
	$$\fin{\fen{(i,j_1^-)}{(i,j_1^+)}}\circ f_{y_1,y_2}\circ \fout{\fen{(i,j_2^-)}{(i,j_2^+)}}=\begin{pmatrix} 0& f_{(i,j_1^+),(i,j_2^-)}\\ 0& f_{(i,j_1^-),(i,j_2^-)}\end{pmatrix}.$$
	Compared with the natural basis given in Remark~\ref{lem:edge2}, these four form a basis of $\mathcal{M}(y_1,y_2)$.
\end{proof}

This basis has the following nice property.

\begin{lemma}\label{lem:comp1}
    For any $(y_1,y_2),(y_2,y_3)\in\mathcal{R}'$ and for any $\xi,\xi'\in\mdf{\overline{y_2}}$, we have
	$$f_{y_1,y_2}\circ \fout{\xi'}\circ \fin{\xi}\circ f_{y_2,y_3}=\begin{cases}f_{y_1,y_3}&\text{if $\xi=\xi'$,}\\0&\text{if $\xi\neq\xi'$.}\end{cases}$$
\end{lemma}

\begin{proof}
	The case $|\mdf{\overline{y_2}}|=1$ is trivial. So we only consider the case $|\mdf{\overline{y_2}}|=2$. Let $y_2=(i,j)\in\fc{\OME}$. Then we have $\mdf{\overline{y_2}}=\{\fen{(i,j^+)}{(i,j^-)},\fen{(i,j^-)}{(i,j^+)}\}$. By definition, we have
	$$\fout{\fen{(i,j^+)}{(i,j^-)}}\circ \fin{\fen{(i,j^+)}{(i,j^-)}}=\begin{pmatrix}
		1&0\\0&0
	\end{pmatrix},\ \fout{\fen{(i,j^-)}{(i,j^+)}}\circ \fin{\fen{(i,j^+)}{(i,j^-)}}=\begin{pmatrix}
		0&0\\0&0
	\end{pmatrix},\ $$
	$$\fout{\fen{(i,j^-)}{(i,j^+)}}\circ \fin{\fen{(i,j^-)}{(i,j^+)}}=\begin{pmatrix}
		0&0\\0&1
	\end{pmatrix},\ \fout{\fen{(i,j^+)}{(i,j^-)}}\circ \fin{\fen{(i,j^-)}{(i,j^+)}}=\begin{pmatrix}
		1&0\\0&-1
	\end{pmatrix}.$$
	Then one can check the formula case by case, using Lemma~\ref{lem:rmkcomp}.
\end{proof}

We remark that if we replace $f_{y_2,y_3}$ with the identity of $\ec{y_2}\sg$ in Lemma~\ref{lem:comp1}, the formula may fail.

\subsection{Dg modules associated to graded unknotted arcs}\label{subsec:cons}

We fix some notations for a graded unknotted arc.

\begin{notations}\label{not:seg}
	Let $\ws$ be a graded unknotted arc in $\gUC(\gmsx)$ with a fixed orientation. We denote by ${\wks_1},\cdots,{\wks_{p_{\ws}}}$ the arcs in $\OMEx=\dac\x$ which $\ws$ crosses in order, at the points $\Vs_1,\cdots,\Vs_{p_{\ws}}$ with intersection indices $\wsi_1,\cdots,\wsi_{p_{\ws}}$, respectively. Denote by $\Vs_0$ and $\Vs_{p_{\ws}+1}$ respectively the starting point and the ending point of $\ws$. Denote by $\asu{\ws}{l}{l+1}$ the segment of $\ws$ between $\Vs_l$ and $\Vs_{l+1}$, for any $0\leq l\leq p_{\ws}$, and call them the \emph{arc segments of $\ws$}. The orientation of $\asu{\ws}{l}{l+1}$ from $\Vs_l$ to $\Vs_{l+1}$ is denoted by $\aso{\ws}{l}{l+1}$ and the opposite orientation is denoted by $\aso{\ws}{l+1}{l}$. Denote by $\as(\ws)=\{\asu{\ws}{l}{l+1}\mid 0\leq l\leq p_{\ws}\}$ the set of arc segments of $\ws$. Similarly as before, denote by $\upas(\ws)$ (resp. $\cias(\ws)$, $\inter{\as(\wsx)}$, $\tou{\as(\wsx)}$) the set of unbinaried (resp. binaried, interior, non-interior) arc segments of $\ws$. For any unbinaried $\aso{\ws}{l}{l+1}\in \upas(\ws)$, denote it by $\was_l\to \wbs_{l+1}$ (see Notation~\ref{not:as} for the meaning of this notation) with $\was_l, \wbs_{l+1}\in\OMEpm$, i.e.  See Figure~\ref{fig:ws1}.
	
	Passing to the GMSp model, the corresponding graded tagged arc $\ws^\times$ crosses the arcs ${\hks_1},\cdots,{\hks_{p_{\ws}}}\in\OMEs=\dac$ in order, at the (same) points $\Vs_1,\cdots,\Vs_{p_{\ws}}$ with (same) intersection indices $\wsi_1,\cdots,\wsi_{p_{\ws}}$, respectively. Identifying $m_P$ with $P$ for any $P\in\P$, the starting point and the ending point of $\ws$ are also $\Vs_0$ and $\Vs_{p_{\ws}+1}$ respectively. The arc segment of $\ws^\times$ between $\Vs_l$ and $\Vs_{l+1}$ is $\asu{\wsx}{l}{l+1}=\shk(\asu{\ws}{l}{l+1})$. The unpunctured $\aso{\wsx}{l}{l+1}$ is denoted by $\has_l\to \hbs_{l+1}$ with $\has_l, \hbs_{l+1}\in\OME$. See Figure~\ref{fig:ws2}. Note that we have
	$$\widehat{?}^{\ws}_l=\begin{cases}
	(i,j)&\text{if $\widetilde{?}^{\ws}_l=(i,j^+)$ or $(i,j^-)$,}\\
	\widetilde{?}^{\ws}_l&\text{otherwise,}\end{cases}$$
	where $?=k,a,b$.

	Denote by $\as(\wsx)=\upas(\wsx)\cup\cias(\wsx)=\inter{\as(\wsx)}\cup\tou{\as(\wsx)}$ the set of arc segments of $\wsx$, where $\upas(\wsx)$ (resp. $\cias(\wsx)$, $\inter{\as(\wsx)}$, $\tou{\as(\wsx)}$) denotes the set of unpunctured (resp. punctured, interior, non-interior) arc segments of $\wsx$.
\end{notations}

\begin{figure}[htpb]\centering
	\begin{tikzpicture}[scale=1.6]
	\draw[blue,thick,->-=.65,>=stealth](0,0)node[blue,left]{$\ws\colon\quad \Vs_{0}$} to (7,0)node[blue,right]{$\Vs_{p_{\ws}+1}$};
	\draw[blue] (.5,.7)node{$\asu{\ws}{0}{1}$} (2,.7)node{$\asu{\ws}{1}{2}$} (6.5,.7)node{$\asu{\ws}{p_{\ws}}{p_{\ws}+1}$};
	\draw[red, thick](1,1)to(1,-1)node[below]{$\wks_{1}$} (1,0)\nn node[blue!50,above right]{$\Vs_1$};
	\draw[red, thick](3,1)to(3,-1)node[below]{$\wks_{2}$} (3,0)\nn node[blue!50,above right]{$\Vs_2$};
	\draw[red, thick](6,1)to(6,-1)node[below]{$\wks_{p_{\ws}}$} (6,0)\nn node[blue!50,above right]{$\Vs_{p_{\ws}}$} (4.5,.5)node{$\cdots$}(4.5,-.5)node{$\cdots$};
	\draw[Green,thick,-<-=.6,>=stealth](1,-.3)to[bend left=-45] (1+.3,0);
	\draw[Green](1+.4,-.4)node{\small{$\wsi_1$}};
	\draw[Green,thick,-<-=.6,>=stealth](3,-.3)to[bend left=-45] (3+.3,0);
	\draw[Green](3+.4,-.4)node{\small{$\wsi_2$}};
	\draw[Green,thick,-<-=.6,>=stealth](6,-.3)to[bend left=-45] (6+.3,0);
	\draw[Green](6+.4,-.4)node{\small{$\wsi_{p_{\ws}}$}};
	\draw[blue](0,0)\nn(7,0)\nn;
	\draw[blue] (1,-.9)node[left]{\tiny $\wbs_1$} (1,-.9)node[right]{\tiny $\was_1$} (3,-.9)node[left]{\tiny $\wbs_2$} (3,-.9)node[right]{\tiny $\was_2$} (6,-.9)node[left]{\tiny $\wbs_{p_{\ws}}$} (6,-.9)node[right]{\tiny $\was_{p_{\ws}}$};
	\end{tikzpicture}
	\caption{Arc segments of a graded unknotted arc $\ws$, cut out by $\dac\x$}\label{fig:ws1}
\end{figure}

\begin{figure}[htpb]\centering
	\begin{tikzpicture}[scale=1.6]
	\draw[blue,thick,->-=.65,>=stealth](0,0)node[blue,left]{$\ws^\times\colon\quad \Vs_{0}$} to (7,0)node[blue,right]{$\Vs_{p_{\ws}+1}$};
	\draw[blue] (.5,.7)node{$\asu{\wsx}{0}{1}$} (2,.7)node{$\asu{\wsx}{1}{2}$} (6.5,.7)node{$\asu{\wsx}{p_{\ws}}{p_{\ws}+1}$};
	\draw[red, thick](1,1)to(1,-1)node[below]{$\hks_{1}$} (1,0)\nn node[blue!50,above right]{$\Vs_1$};
	\draw[red, thick](3,1)to(3,-1)node[below]{$\hks_{2}$} (3,0)\nn node[blue!50,above right]{$\Vs_2$};
	\draw[red, thick](6,1)to(6,-1)node[below]{$\hks_{p_{\ws}}$} (6,0)\nn node[blue!50,above right]{$\Vs_{p_{\ws}}$} (4.5,.5)node{$\cdots$}(4.5,-.5)node{$\cdots$};
	\draw[Green,thick,-<-=.6,>=stealth](1,-.3)to[bend left=-45] (1+.3,0);
	\draw[Green](1+.4,-.4)node{\small{$\wsi_1$}};
	\draw[Green,thick,-<-=.6,>=stealth](3,-.3)to[bend left=-45] (3+.3,0);
	\draw[Green](3+.4,-.4)node{\small{$\wsi_2$}};
	\draw[Green,thick,-<-=.6,>=stealth](6,-.3)to[bend left=-45] (6+.3,0);
	\draw[Green](6+.4,-.4)node{\small{$\wsi_{p_{\ws}}$}};
	\draw[blue](0,0)\nn(7,0)\nn;
	\draw[blue] (1,-.9)node[left]{\tiny $\hbs_1$} (1,-.9)node[right]{\tiny $\has_1$} (3,-.9)node[left]{\tiny $\hbs_2$} (3,-.9)node[right]{\tiny $\has_2$} (6,-.9)node[left]{\tiny $\hbs_{p_{\ws}}$} (6,-.9)node[right]{\tiny $\has_{p_{\ws}}$};
	\end{tikzpicture}
	\caption{Arc segments of a graded tagged arc $\ws^\times$, cut out by $\dac$}\label{fig:ws2}
\end{figure}

Throughout the rest of the paper, we will use Notations~\ref{not:seg} for any graded unknotted arc in $\gUC(\gmsx)$.

\begin{construction}\label{cons:qui}
Let $\ws\in\gUC(\gmsx)$ be a graded unknotted arc. We define a quiver $Q^{\ws}=(Q^{\ws}_0,Q^{\ws}_1)$ of type $A_{p_{\ws}}$ as follows.
\begin{itemize}
\item The vertex set $Q^{\ws}_0=\{1,2,\cdots,p_{\ws}\}$.
\item The arrow set $Q^{\ws}_1=\inter{\as(\wsx)}$ with $s(\asu{\wsx}{l}{l+1})=l$ and $t(\asu{\wsx}{l}{l+1})=l+1$ if $\aso{\ws}{l}{l+1}$ is positive, or $s(\asu{\wsx}{l}{l+1})=l+1$ and $t(\asu{\wsx}{l}{l+1})=l$ if $\aso{\ws}{l+1}{l}$ is positive (see Definition~\ref{def:seg} for the notion of positive orientation).
\end{itemize}
We define an equivalence relation on $Q^{\ws}_0$ induced by $l\simeq l+1$ if $\asu{\ws}{l}{l+1}$ is binaried, see Figure~\ref{fig:equiv}. Denote by $\ec{Q_0^{\ws}}$ the corresponding quotient set of $Q^{\ws}_0$.
\end{construction}

For any $l\simeq l+1\in Q^{\ws}_0$, we have (cf. the pictures in the second row of Figure~\ref{fig:equiv})
$$\hks_{l}=\hks_{l+1},\ \wsi_{l}=\wsi_{l+1}\ \text{and}\ \hbs_{l}=\has_{l+1}.$$
Hence we can define the following notations.

\begin{notations}\label{not:uni}
For any $\ell\in \ec{Q_0^{\ws}}$, we define $\wsi_{\ell}\in\ZZ$, $\ks_{\ell}\in\OMEsx$ and $\uas_{\ell},\ubs_{\ell}\in\OMEepm$ to be
$$\wsi_{\ell}=\wsi_l,\ ?^{\ws}_{\ell}=\begin{cases}
\widehat{?}^{\ws}_{l}&\text{if $|\ell|=2$,}\\
\widetilde{?}^{\ws}_{l}&\text{if $|\ell|=1$,}
\end{cases}$$
where $?=k,a,b$ and $l\in\ell$.
\end{notations}

Note that although $\has_{l}$ and $\was_{l}$ (resp. $\hbs_{l}$ and $\wbs_{l}$) are not defined for some $l$, $\uas_\ell$ (resp. $\ubs_\ell$) is defined for every $\ell\in \ec{Q_0^{\ws}}$ unless $\ell=\{p_{\ws}\}$ (resp. $\ell=\{1\}$) and $\ws$ ends (resp. starts) at a binary. Note also that we have $\{\uas_\ell,\ubs_{\ell}\}=\ks_{\ell}$, where $\uas_{\ell}=\ubs_{\ell}$ if and only if $|\ell|=2$.

\begin{figure}[htpb]\centering\makebox[\textwidth][c]{
    \begin{tikzpicture}[scale=.45,font=\scriptsize]
        \draw[ultra thick]plot [smooth,tension=1] coordinates {(3,-4.5) (-90:4) (-3,-4.5)};
		\draw[ultra thick,fill=gray!10] (90:2) ellipse (1.5) node{$\Vot_p$};
		\draw[red,thick] (90:-4) .. controls +(45:5) and +(30:4) .. (90:3.5);
		\draw[red,thick] (3,4)node{$\wks_{l+1}$} (-3,4)node{$\wks_l$};
		\draw[red,thick] (90:-4) .. controls +(135:5) and +(150:4) .. (90:3.5);
		\draw[blue, thick,->-=.5,>=stealth] (-4,-2)node[left]{$\ws$}tonode[above]{$\asu{\ws}{l}{l+1}$}(4,-2);
		\draw[blue] (-2,-2)node[below]{$\wbs_l$} (2,-2)node[below]{$\was_{l+1}$};
		\draw[very thick,blue] (90:.5)\nn  (0,3.5)\ww;
		\draw[red] (-90:4)\ww;
    \end{tikzpicture}\quad
    \begin{tikzpicture}[scale=.45,font=\scriptsize]
        \draw[ultra thick]plot [smooth,tension=1] coordinates {(3,-4.5) (-90:4) (-3,-4.5)};
		\draw[ultra thick,fill=gray!10] (90:2) ellipse (1.5) node{$\Vot_p$};
		\draw[red,thick] (90:-4) .. controls +(45:5) and +(30:4) .. (90:3.5);
		\draw[red,thick] (3,4)node{$\wks_{l}$} (-3,4)node{$\wks_{l+1}$};
		\draw[red,thick] (90:-4) .. controls +(135:5) and +(150:4) .. (90:3.5);
		\draw[blue, thick,->-=.5,>=stealth] (4,-2)node[right]{$\ws$}tonode[above]{$\asu{\ws}{l}{l+1}$}(-4,-2);
		\draw[blue] (-2,-2)node[below]{$\uas_{l+1}$} (2,-2)node[below]{$\ubs_l$};
		\draw[very thick,blue] (90:.5)\nn  (0,3.5)\ww;
		\draw[red] (-90:4)\ww;
    \end{tikzpicture}\quad
    \begin{tikzpicture}[scale=.45,font=\scriptsize]
        \draw[ultra thick]plot [smooth,tension=1] coordinates {(3,-4.5) (-90:4) (-3,-4.5)};
		\draw[red,thick] (90:-4) .. controls +(45:5) and +(0:4) .. (90:3.5);
		\draw[red,thick] (0,4)node{$\hks_l=\hks_{l+1}$};
		\draw[red,thick] (90:-4) .. controls +(135:5) and +(180:4) .. (90:3.5);
		\draw[blue, thick,->-=.5,>=stealth] (-4,-2)node[left]{$\wsx$}tonode[above]{$\asu{\wsx}{l}{l+1}$}(4,-2);
		\draw[blue] (-2,-2)node[below]{$\hbs_l$} (2,-2)node[below]{$\has_{l+1}$};
		\draw[very thick,blue] (90:.5)\nn (90:.5)node[above]{$p$};
		\draw[red] (-90:4)\ww;
    \end{tikzpicture}\quad
    \begin{tikzpicture}[scale=.45,font=\scriptsize]
        \draw[ultra thick]plot [smooth,tension=1] coordinates {(3,-4.5) (-90:4) (-3,-4.5)};
		\draw[red,thick] (90:-4) .. controls +(45:5) and +(0:4) .. (90:3.5);
		\draw[red,thick] (0,4)node{$\hks_l=\hks_{l+1}$};
		\draw[red,thick] (90:-4) .. controls +(135:5) and +(180:4) .. (90:3.5);
		\draw[blue, thick,->-=.5,>=stealth] (4,-2)node[right]{$\wsx$}tonode[above]{$\asu{\wsx}{l}{l+1}$}(-4,-2);
		\draw[blue] (-2,-2)node[below]{$\has_{l+1}$} (2,-2)node[below]{$\hbs_l$};
		\draw[very thick,blue] (90:.5)\nn (90:.5)node[above]{$p$};
		\draw[red] (-90:4)\ww;
    \end{tikzpicture}}
    \caption{Equivalence relation on $Q_0^{\ws}$}\label{fig:equiv}
\end{figure}

\begin{notations}\label{not:mdf}
    For any $\ell\in \ec{Q^{\ws}_0}$ and any $l\in\ell$, we define $\xis_l\in\mdf{\ks_{\ell}}$ (see Notations~\ref{not:mo} for the notation $\mdf{\ks_{\ell}}$) to be
    $$\xis_l=\begin{cases}
    \fen{(i,j^+)}{(i,j^-)}&\text{if $|\ell|=2$ and $\wks_l=(i,j^+)$,} \vspace{1mm}\\
    \fen{(i,j^-)}{(i,j^+)}&\text{if $|\ell|=2$ and $\wks_l=(i,j^-)$,} \vspace{1mm}\\
    \fen{\wks_l}{\emptyset}&\text{if $|\ell|=1$.}
    \end{cases}$$
\end{notations}

\begin{construction}\label{cons:dgm}
	Let $\ws\in\gUC(\gmsx)$ be a graded unknotted arc. We construct $\cpy{\ws}=(|\cpy{\ws}|,d_{\ws})$ as follows.
	\begin{enumerate}
	\item The underlying graded $\sg$-module
	$$|\cpy{\ws}|=\bigoplus_{1\leq l\leq p_{\ws}}\wks_l\sg[\wsi_l]=\bigoplus_{\ell\in \ec{Q^{\ws}_0}}\cply{\ws}{\ell},$$
	where $\cply{\ws}{\ell}:=\ks_{\ell}\sg[\wsi_{\ell}]=\bigoplus_{l\in\ell}\wks_l\sg[\wsi_{l}].$
	\item Each unbinaried arc segment $\asu{\ws}{l}{l+1}$ of $\ws$ contributes a component $d_{\asu{\ws}{l}{l+1}}$ of $d_{\ws}$:
	$$d_{\asu{\ws}{l}{l+1}}=\begin{cases}
	f^{in}_{\xis_{l}}\circ f_{\uas_{\overline{l}},\ubs_{\overline{l+1}}}\circ f^{out}_{\xis_{l+1}}:\cply{\ws}{\overline{l+1}}\to \cply{\ws}{\overline{l}}&\text{if $\aso{\ws}{l}{l+1}$ is positive,}\\
	f^{in}_{\xis_{l+1}}\circ f_{\ubs_{\overline{l+1}},\uas_{\overline{l}}}\circ f^{out}_{\xis_{l}}:\cply{\ws}{\overline{l}}\to \cply{\ws}{\overline{l+1}}&\text{if $\aso{\ws}{l+1}{l}$ is positive.}
	\end{cases}$$
	\end{enumerate}
\end{construction}

\begin{lemma}\label{lem:msp}
	For any $\ws\in\gUC(\gmsx)$,  $\cpy{\ws}=(|\cpy{\ws}|,d_{\ws})$ is a minimal strictly perfect dg $\sg$-module.
\end{lemma}

\begin{proof}
To show $\cpy{\ws}$ is a dg module, we need to show that for any $\ell\in \ec{Q_0^{\ws}}$, if there is a component $d_{\asu{\ws}{l_1}{l_1+1}}$ of $d_{\ws}$ to $\cply{\ws}{\ell}$ and a component $d_{\asu{\ws}{l_2}{l_2+1}}$ of $d_{\ws}$ from $\cply{\ws}{\ell}$, we have $d_{\asu{\ws}{l_2}{l_2+1}}\circ d_{\asu{\ws}{l_1}{l_1+1}}=0$. We only prove this for the case $l_1<l_2$, since the case $l_1>l_2$ is similar. There are two subcases.
\begin{enumerate}
\item $\ell=\{l\}$: then $l_1=l-1$ and $l_2=l$. So $\xis_{l}=\fen{\wks_l}{\emptyset}$, which implies $f^{in}_{\xis_{l}}=\id_{\cply{\ws}{l}}=f^{out}_{\xis_{l}}$. Since $\ubs_{\overline{l}}\neq\uas_{\overline{l}}$, by Lemma~\ref{lem:rmkcomp}, we have $f_{\ubs_{\overline{l+1}},\uas_{\overline{l}}}\circ f_{\ubs_{\overline{l}},\uas_{\overline{l-1}}}=0$. Then
$$\begin{array}{rcl}
    d_{\asu{\ws}{l}{l+1}}\circ d_{\asu{\ws}{l-1}{l}} & = & f^{in}_{\xis_{l+1}}\circ f_{\ubs_{\overline{l+1}},\uas_{\overline{l}}}\circ f^{out}_{\xis_{l}}\circ f^{in}_{\xis_{l}}\circ f_{\ubs_{\overline{l}},\uas_{\overline{l-1}}}\circ f^{out}_{\xis_{l-1}}  \\
    & = & f^{in}_{\xis_{l+1}}\circ f_{\ubs_{\overline{l+1}},\uas_{\overline{l}}}\circ f_{\ubs_{\overline{l}},\uas_{\overline{l-1}}}\circ f^{out}_{\xis_{l-1}} \\
    & = & 0.
\end{array}$$
\item $\ell=\{l,l+1\}$: then $l_1=l-1$ and $l_2=l+1$. So we have $\uas_{\overline{l+1}}=\ubs_{\overline{l}}=\ks_{\ell}$. Then by Lemma~\ref{lem:comp1}, we have
$$d_{\asu{\ws}{l+1}{l+2}}\circ d_{\asu{\ws}{l-1}{l}}=f^{in}_{\xis_{l+2}}\circ f_{\ubs_{\overline{l+2}},\uas_{\overline{l+1}}}\circ f^{out}_{\xis_{l+1}}\circ f^{in}_{\xis_{l}}\circ f_{\ubs_{\overline{l}},\uas_{\overline{l-1}}}\circ f^{out}_{\xis_{l-1}}=0.$$
\end{enumerate}
The property of minimal strictly perfectness follows directly from the construction.
\end{proof}

\begin{remark}\label{rmk:shift}
    By the construction, for any $\ws\in\gUC(\gmsx)$ and any $\rho\in\ZZ$, we have
    $$\cpy{\ws[\rho]}\cong\cpy{\ws}[\rho].$$
\end{remark}

\begin{example}\label{ex:dg}
    Let $\ws$ and $\wt$ be the graded unknotted arcs in Example~\ref{ex:unknotted}, whose tagged versions are shown in Example~\ref{ex:tag}. The dg modules $\cpy{\ws}$ and $\cpy{\wt}$ are respectively given by
    $$|\cpy{\ws}|=\ec{(1,1)}\sg\oplus(1,3^+)\sg[2]\oplus(1,3^-)\sg[2]\oplus{(1,4^-)}\sg[3],$$ $$d_{\ws}=\begin{pmatrix}
    0&f_{(1,1),(1,3^+)}&-f_{(1,1),(1,3^-)}&0\\
    0&0&0&f_{(1,3^+),(1,4^-)}\\
    0&0&0&f_{(1,3^-),(1,4^-)}\\
    0&0&0&0
    \end{pmatrix}$$
    and
    $$\begin{array}{rl}
        |\cpy{\wt}|= & \ec{(1,1)}\sg\oplus(1,3^+)\sg[2]\oplus(1,3^-)\sg[2]\oplus(1,4^+)\sg[3]\oplus(1,4^-)\sg[3] \\
         & \oplus(1,3^+)\sg[2]\oplus(1,3^-)\sg[2]\oplus(1,4^+)\sg[3]\oplus(1,4^-)\sg[3]\oplus\ec{(1,7)}\sg[5],
    \end{array}$$
    $$d_{\wt}=\left(\begin{smallmatrix}
    0&f_{1,3^+}&-f_{1,3^-}&0&0&0&0&0&0&0\\
    0&0&0&f_{3^+,4^+}&-f_{3^+,4^-}&0&0&0&0&0\\
    0&0&0&f_{3^-,4^+}&-f_{3^-,4^-}&0&0&0&0&0\\
    0&0&0&0&0&0&0&0&0&0\\
    0&0&0&0&0&0&0&0&0&0\\
    0&0&0&0&f_{3^+,4^-}&0&0&f_{3^+,4^+}&-f_{3^+,4^-}&0\\
    0&0&0&0&0&0&0&f_{3^-,4^+}&-f_{3^-,4^-}&0\\
    0&0&0&0&0&0&0&0&0&f_{4^+,7}\\
    0&0&0&0&0&0&0&0&0&f_{4^-,)}\\
    0&0&0&0&0&0&0&0&0&0
    \end{smallmatrix}\right),$$
    where $f_{j_1^{\kappa_1},j_2^{\kappa_2}}$ stands for $f_{(1,j_1^{\kappa_1}),(1,j_2^{\kappa_2})}$. Note that although $d_{\wt}$ is not a strictly upper triangular matrix, after reordering the direct summands of $|\cpy{\wt}|$, the resulting matrix of $d_{\wt}$ is.
\end{example}

\subsection{Compatibility}

This subsection is devoted to proving the compatibility between Theorem~\ref{thm1} and Construction~\ref{cons:dgm}.

\begin{remark}\label{rmk:Rmx}
    Combining the following maps
    $$\careII\xrightarrow[\eqref{eq:tag to ls}]{}\care\xrightarrow[\eqref{eq:l.s.}]{\bione}\OO\xrightarrow[\eqref{eq:rep}]{R} \repb(S)$$
    and using the notations in Notations~\ref{not:seg}, for any $\wsx\in\careII$, the representation $R(\bione(\wsx))=(X_{\wsx},f_{\wsx})$ is formed by an object
    $$X_{\wsx}=\bigoplus_{1\leq l\leq p_{\ws}}(\wks_l,\wsi_l)\oplus\bigoplus_{\mu\in\upas(\wsx)}\mu\in\add\bush,$$
    and for any $\asu{\wsx}{l}{l+1}\in\upas(\wsx)$ (cf. Figure~\ref{fig:fw}),
    $$f_{\wsx}(\underline{\aso{\wsx}{l}{l+1}})=\begin{cases}\underline{\was_l}+\underline{\wbs_{l-1}}&\text{if $\aso{\wsx}{l-1}{l}$ is positive, interior and punctured,}\\\underline{\was_l}&\text{otherwise,}\end{cases}$$
    $$f_{\wsx}(\underline{\aso{\wsx}{l+1}{l}})=\begin{cases}\underline{\wbs_{l+1}}+\underline{\was_{l+2}}&\text{if $\aso{\wsx}{l+2}{l+1}$ is positive, interior and punctured,}\\\underline{\wbs_{l+1}}&\text{otherwise.}\end{cases}$$
\end{remark}

\begin{theorem}\label{thm2}
We have the following commutative diagram
$$\xymatrix{
\gUC(\gmsx)\ar[rr]^{(-)^\times}_{\eqref{eq:bij}}\ar[rd]_{\cpy{?}}&&\careII\ar[dl]^{\widetilde{X}}\\
&\tw\sg}.$$
In particular, we obtain a bijection
$$\boxed{
\cpy{?}\colon\gUC(\gmsx)\to\goodind.
}$$
\end{theorem}

\begin{proof}
    We first show that it suffices to show $ R\circ\bione(\wsx)\cong \fM\circ\fE(\cpy{\ws})$ in $\repb(S)$ for any $\ws\in\gUC(\gmsx)$, where $\fM$ is the functor in Theorem~\ref{thm:tribu} and $\fE$ is the functor in Theorem~\ref{thm:BD}. This is because, then by definition, we have $\careII\subseteq\cares$, and hence the bijection~\eqref{eq:glue} can restrict to $\careII$ and by the commutative diagram~\eqref{eq:thmcomm2}, we have $\fM\circ\fE(\widetilde{X}(\wsx))\cong  R\circ\bione(\wsx)$. Therefore, by Theorems~\ref{thm:tribu} and \ref{thm:BD}, we have $\widetilde{X}(\wsx)\cong\cpy{\ws}$, i.e., we get the required commutative diagram.

    Now we prove $R\circ\bione(\wsx)\cong \fM\circ\fE(\cpy{\ws})$. Since by Theorem~\ref{thm:tribu}, $\fM$ is dense, there is an object $(Y^\bullet,V^\bullet,\theta)$ in $\Tri\sg$ such that $R\circ\bione(\wsx)=\fM(Y^\bullet,V^\bullet,\theta)$. Then by Theorem~\ref{thm:tribu}~(1), we only need to show $(Y^\bullet,V^\bullet,\theta)\cong \fE(\cpy{\ws})$.  Denote
    $\fE(\cpy{\ws})=({Y'}^\bullet,{V'}^\bullet,\theta').$

    By Remark~\ref{rmk:Rmx} and the construction of $\fM$, we have
\begin{itemize}
\item $Y^\bullet=\bigoplus_{\wmu\in\upas(\wsx)}W^\bullet(\wmu)=\left(\bigoplus_{\ell\in \ec{Q_0^{\ws}}}H_{\uas_{\ell}}[\wsi_{\ell}]\oplus H_{\ubs_{\ell}}[\wsi_{\ell}],d\right)$ (see Construction~\ref{cons:W} for the structure of $W^\bullet(\wmu)$);
\item $V^\bullet=\bigoplus_{\ell\in \ec{Q_0^{\ws}}}\ks_{\ell}\overline{\sg}[\wsi_{\ell}]$;
\item using the identifications $$Y^\bullet\otimes_{H}\overline{H}=\bigoplus_{\ell\in \ec{Q_0^{\ws}}}\overline{H}_{\uas_{\ell}}[\wsi_{\ell}]\oplus\overline{H}_{\ubs_{\ell}}[\wsi_{\ell}]=V^\bullet\otimes_{\overline{\sg}} \overline{H},$$
we have the map
$$\theta=\operatorname{diag}(\theta_\ell:\overline{H}_{\uas_{\ell}}[\wsi_{\ell}]\oplus\overline{H}_{\ubs_{\ell}}[\wsi_{\ell}]\to\overline{H}_{\uas_{\ell}}[\wsi_{\ell}]\oplus\overline{H}_{\ubs_{\ell}}[\wsi_{\ell}],\ell\in \ec{Q_0^{\ws}}),$$
where
$$\theta_\ell=\begin{cases}\begin{pmatrix}1&0\\0&1\end{pmatrix}&\text{if $|\ell|=1$,}\\
\begin{pmatrix}1&1\\0&1\end{pmatrix}&\text{if $|\ell|=2$,}\end{cases}$$
as shown in Figures~\ref{fig:fw2} and \ref{fig:fw3} (compared with Figure~\ref{fig:fw}), respectively. Here, the rows/columns of the matrix for case $|\ell|=2$ are indexed by the order first $\uas_{\ell}$ then $\ubs_{\ell}$ if $\ell=\{l,l+1\}$, $\was_{l+1}=(i,j^+)$ and $\wbs_{l}=(i,j^-)$, or else by the opposite order.
\end{itemize}

\begin{figure}[htpb]
	\begin{tikzpicture}[yscale=1.5,xscale=2]
	\draw[red,very thick] (0,-.7)to(0,1.8)node[above]{$\ks_{\overline{l}}$} (0,.2)node[right]{$\uas_{\overline{l}}$} (0,.2)node[left]{$\ubs_{\overline{l}}$};
	\draw[blue,thick,->-=.5,>=stealth] (-2,0)to(2,0);
	\draw[blue] (-1,0)node[above]{$\asu{\wsx}{l-1}{l}$} (1,0)node[above]{$\asu{\wsx}{l}{l+1}$};
	\draw[blue] (-1,1.5)node[above]{$W^\bullet(\asu{\wsx}{l-1}{l})\otimes_H\overline{H}$}
	(-.95,1.35)node{$||$} (.95,1.35)node{$||$} (1,1.5)node[above]{$W^\bullet(\asu{\wsx}{l}{l+1})\otimes_H\overline{H}$};
	\draw[blue] (-.1,1)node[left]{$\overline{H}_{\ubs_{\overline{l}}}[\wsi_l]$}
	(-2,1)node[right]{$\overline{H}_{\uas_{\overline{l-1}}}[\wsi_{l-1}]$}
	(-.95,1)node{$\oplus$} (.1,1)node[right]{$\overline{H}_{\uas_{\overline{l}}}[\wsi_l]$}
	(2,1)node[left]{$\overline{H}_{\ubs_{\overline{l+1}}}[\wsi_{l+1}]$}
	(.95,1)node{$\oplus$};
	\draw[blue](0,-1.7)node{$\ks_\ell\overline{\sg}[\wsi_l]\otimes_{\overline{\sg}}\overline{H}$} (0,-1.35)node{$||$}
	(0,-1)node{$\oplus$} (-.1,-1)node[left]{$\overline{H}_{\ubs_{\overline{l}}}[\wsi_l]$}
	(.1,-1)node[right]{$\overline{H}_{\uas_{\overline{l}}}[\wsi_l]$};
	\draw[Green,thick,->-=1,>=stealth] (-.35,.8)to(-.35,-.8);
	\draw[Green,thick,->-=1,>=stealth] (.35,.8)to(.35,-.8);
	\end{tikzpicture}
	\caption{Triple from arc, $\ell=\{l\}$ case}\label{fig:fw2}
\end{figure}
	
\begin{figure}[htpb]
	\begin{tikzpicture}[xscale=2,yscale=1.6]
	\draw[thick] (0,-1)node{$\bullet$};
	\draw[red, thick]plot [smooth, tension=1] coordinates {(0,0) (-.6,-1.2) (0,-1.7) (.6,-1.2) (0,0)};
	\draw[blue,thick,->-=.5,>=stealth]plot [smooth, tension=1] coordinates {(-2.2,-1) (0,-.3) (2.2,-1)};
	\draw[blue] (-1.3,-.85)node{$\asu{\wsx}{l-1}{l}$} (1.3,-.85)node{$\asu{\wsx}{l+1}{l+2}$}
	(0,-.6)node{$\asu{\wsx}{l}{l+1}$};
	\draw[blue] (-1.1,.5)node[above]{$W^\bullet(\asu{\wsx}{l-1}{l})\otimes_H\overline{H}$}
	(-1.1,.35)node{$||$} (1.1,.35)node{$||$} (1.1,.5)node[above]{$W^\bullet(\asu{\wsx}{l+1}{l+2})\otimes_H\overline{H}$};
	\draw[blue] (-.2,0)node[left]{$\overline{H}_{\ubs_{\ell}}[\wsi_\ell]$}
	(-2.2,0)node[right]{$\overline{H}_{\uas_{\overline{l-1}}}[\wsi_{\overline{l-1}}]$}
	(-1.1,0)node{$\oplus$} (.2,0)node[right]{$\overline{H}_{\uas_{\ell}}[\wsi_{\ell}]$}
	(2.2,0)node[left]{$\overline{H}_{\ubs_{\overline{l+2}}}[\wsi_{\overline{l+2}}]$}
	(1.1,0)node{$\oplus$};
	\draw[blue](0,-2.7)node{$\ks_{\ell}\overline{\sg}[\wsi_{\ell}]\otimes_{\overline{\sg}}\overline{H}$} (0,-2.35)node{$||$} (0,-2)node{$\oplus$}  (-.55,-2)node{$\overline{H}_{\ubs_{\ell}}[\wsi_\ell]$}
	(.55,-2)node{$\overline{H}_{\uas_{\ell}}[\wsi_\ell]$};
	\draw[Green,thick,->-=.99,>=stealth, bend right] (-.65,-.2)to(-.5,-1.8);
	\draw[Green,thick,->-=.99,>=stealth,bend left] (.65,-.2)to(.55,-1.8);
	\draw[Green,thick,->-=.99,>=stealth,bend right=30] (-.55,-.2)to(.4,-1.8);
	\draw[red] (0,0)\ww ;
	\end{tikzpicture}
	\caption{Triple from arc, $\ell=\{l,l+1\}$ ($\was_{l+1}=(i,j^+)$, $\wbs_{l}=(i,j^-)$) case}\label{fig:fw3}
\end{figure}

On the other hand, since $\cpy{\ws}=(|\cpy{\ws}|,d_{\ws})$ is minimal strictly perfect, by the constructions of $\cpy{\ws}$ and $\fE$, we have
\begin{itemize}
\item ${Y'}^\bullet=\cpy{\ws}\otimes_\sg H=\left(\bigoplus_{\ell\in \ec{Q_0^{\ws}}}H_{\uas_{\ell}}[\wsi_{\ell}]\oplus H_{\ubs_{\ell}}[\wsi_{\ell}],d^H_{\ws}\right)$;
\item ${V'}^\bullet=\cpy{\ws}\otimes_\sg\overline{\sg}=\bigoplus_{\ell\in \ec{Q_0^{\ws}}}\wks_{\ell}\overline{\sg}[\wsi_{\ell}]$;
\item using the identifications $${Y'}^\bullet\otimes_{H}\overline{H}=\bigoplus_{\ell\in \ec{Q_0^{\ws}}}\overline{H}_{\uas_{\ell}}[\wsi_{\ell}]\oplus\overline{H}_{\ubs_{\ell}}[\wsi_{\ell}]={V'}^\bullet\otimes_{\overline{\sg}} \overline{H},$$
we have that the canonical map $\theta'=\theta_{\cpy{\ws}}:(\cpy{\ws}\otimes_\sg H)\otimes_{H}\overline{H}\to \overline{H}\otimes_{\overline{\sg}}(\overline{\sg}\otimes_\sg\cpy{\ws})$ is the identity of $\bigoplus_{\ell\in \ec{Q_0^{\ws}}}\overline{H}_{\uas_{\ell}}[\wsi_{\ell}]\oplus\overline{H}_{\ubs_{\ell}}[\wsi_{\ell}]$.
\end{itemize}
We shall construct an isomorphism
$$(g,h):(Y^\bullet,V^\bullet,\theta)\to ({Y'}^\bullet,{V'}^\bullet,\theta').$$
Since $V^\bullet={V'}^\bullet$, we can take $h$ to be the identity. To construct $g:Y^\bullet\to {Y'}^\bullet$, since $$|Y^\bullet|=|{Y'}^\bullet|=\bigoplus_{\ell\in \ec{Q_0^{\ws}}}H_{\uas_{\ell}}[\wsi_{\ell}]\oplus H_{\ubs_{\ell}}[\wsi_{\ell}],$$
we can take $$g=\operatorname{diag}(g_\ell:H_{\uas_{\ell}}[\wsi_{\ell}]\oplus H_{\ubs_{\ell}}[\wsi_{\ell}]\to H_{\uas_{\ell}}[\wsi_{\ell}]\oplus H_{\ubs_{\ell}}[\wsi_{\ell}],\ \ell\in \ec{Q_0^{\ws}}),$$
where $g_\ell$ is given by the same matrix as $\theta_\ell$. By this construction, if $g$ is a morphism in $\per H$, then $(g,h)$ is a morphism in $\Tri\sg$ and it is moreover an isomorphism.
	
To show $g$ is a morphism, we need to show that it commutes with the differentials $d$ of $Y^\bullet$ and $d^H_{\ws}$ of ${Y'}^\bullet$. By the construction, $d^H_{\ws}=d_{\ws}\otimes_\sg H$ is formed by the components:
$$d^H_{\asu{\ws}{l}{l+1}}=\begin{cases}
f^{in}_{\xis_{l}}\circ f^H_{\uas_{\overline{l}},\ubs_{\overline{l+1}}}\circ f^{out}_{\xis_{l+1}}:\mathcal{H}_{\ws}^{\overline{l+1}}\to \mathcal{H}_{\ws}^{\overline{l}}&\text{if $\aso{\ws}{l}{l+1}$ is positive,}\\
f^{in}_{\xis_{l+1}}\circ f^H_{\ubs_{\overline{l+1}},\uas_{\overline{l}}}\circ f^{out}_{\xis_{l}}:\mathcal{H}_{\ws}^{\overline{l}}\to \mathcal{H}_{\ws}^{\overline{l+1}}&\text{if $\aso{\ws}{l+1}{l}$ is positive,}
\end{cases}$$
for any unpunctured arc segment $\asu{\ws}{l}{l+1}$ of $\ws$, where $\mathcal{H}_{\ws}^{\ell}=H_{\uas_{\ell}}[\wsi_{\ell}]\oplus H_{\ubs_{\ell}}[\wsi_{\ell}]$ for any $\ell\in \ec{Q_0^{\ws}}$. On the other hand, $d$ is the same as $d_{\ws}$, except replacing $\fout{\fen{X}{Y}}$ and $\fin{\fen{X}{Y}}$ by $\mathfrak{f}^{out}_{\fen{X}{Y}}$ and $\mathfrak{f}^{in}_{\fen{X}{Y}}$, respectively, where
$$\mathfrak{f}^{out}_{\fen{z}{\emptyset}}=\mathfrak{f}^{in}_{\fen{z}{\emptyset}}=1,\ \mathfrak{f}^{out}_{\fen{(i,j^+)}{(i,j^-)}}=\mathfrak{f}^{in}_{\fen{(i,j^+)}{(i,j^-)}}=\begin{pmatrix}
1&0\\0&0
\end{pmatrix},\ \mathfrak{f}^{out}_{\fen{(i,j^-)}{(i,j^+)}}=\mathfrak{f}^{in}_{\fen{(i,j^-)}{(i,j^+)}}=\begin{pmatrix}
0&0\\0&1
\end{pmatrix}$$
are idempotent matrices. It is straightforward to show that
$$d\circ g=g\circ d^H_{\ws},$$
which completes the proof.
\end{proof}

\section{Morphisms in the perfect category}\label{sec:mor}

This section is devoted to showing the following result.

\begin{theorem}\label{thm:int=dim}
For any graded unknotted arc $\ws,\wt\in\gUC(\gmsx)$, we have
$$\boxed{
\oIntd(\ws,\wt)=\dim\Hom_{\per\sg}(\cpy{\ws},\cpy{\wt}[\rho])
}$$
for any $\rho\in\ZZ$.
\end{theorem}

By $\cpy{\wt[\rho]}\cong\cpy{\wt}[\rho]$ (Remark~\ref{rmk:shift}),  $\oIntd(\ws,\wt)=\oInt^0(\ws,\wt[\rho])$ (due to \eqref{eq:shift}) and $\oIntd(\ws,\wt)=\oIntd(\ws^\times,\wt^\times)$ \eqref{eq:intb=intx}, it suffices to show
\begin{equation}\label{eq:aim}
    \oInt^0(\wsx,\wtx)=\dim\Hom_{\per\sg}(\cpy{\ws},\cpy{\wt}).
\end{equation}

We refer to Notations~\ref{not:seg}, Construction~\ref{cons:qui}, Notations~\ref{not:uni} and Notations~\ref{not:mdf} for certain notations concerning graded unknotted arcs and the corresponding graded tagged arcs, and refer to Definition~\ref{def:seg} for the notations about arc segments.

By Construction~\ref{cons:dgm} of $\cpy{?}$, we have
\begin{equation}\label{eq:decomp}
    |\huaHom_{\sg}(\cpy{\ws},\cpy{\wt})|=\bigoplus_{\ell_1\in \ec{Q^{\wt}_0},\ell_2\in \ec{Q^{\ws}_0}}|\huaHom_{\sg}(\cply{\ws}{\ell_2},\cply{\wt}{\ell_1})|.
\end{equation}
For any morphism $\varphi\in|\huaHom_{\sg}(\cpy{\ws},\cpy{\wt})|$, denote by $\varphi_{\ell_1,\ell_2}\colon\cply{\ws}{\ell_2}\to \cply{\wt}{\ell_1}$ its component in $|\huaHom_{\sg}(\cply{\ws}{\ell_2},\cply{\wt}{\ell_1})|$ with respect to the decomposition~\eqref{eq:decomp}.

Let $|\huaRad_{\sg}(\cpy{\ws},\cpy{\wt})|$ be the radical of $
|\huaHom_{\sg}(\cpy{\ws},\cpy{\wt})|$. Then for any $\varphi\in|\huaHom_{\sg}(\cpy{\ws},\cpy{\wt})|$, it is in $|\huaRad_{\sg}(\cpy{\ws},\cpy{\wt})|$ if and only if each of its component $\varphi_{\ell_1,\ell_2}\colon\cply{\ws}{\ell_2}\to \cply{\wt}{\ell_1}$ is in the radical, i.e., we have
\begin{equation}\label{eq:rad}
    |\huaRad_{\sg}(\cpy{\ws},\cpy{\wt})|=\bigoplus_{\ell_1\in \ec{Q^{\wt}_0},\ell_2\in \ec{Q^{\ws}_0}}|\huaRad_{\sg}(\cply{\ws}{\ell_2},\cply{\wt}{\ell_1})|.
\end{equation}
Since $\cpy{\ws}$ and $\cpy{\wt}$ are minimal strictly perfect (Lemma~\ref{lem:msp}), we have $$B\huaHom_{\sg}(\cpy{\ws},\cpy{\wt})\subset\huaRad_{\sg}(\cpy{\ws},\cpy{\wt}).$$ So $|\huaRad_{\sg}(\cpy{\ws},\cpy{\wt})|$, together with the restriction of the differential of $
\huaHom_{\sg}(\cpy{\ws},\cpy{\wt})$, becomes a differential graded $\k$-subspace $\huaRad_{\sg}(\cpy{\ws},\cpy{\wt})$ of $
\huaHom_{\sg}(\cpy{\ws},\cpy{\wt})$. Set $$\operatorname{Rad}_{\per\sg}(\cpy{\ws},\cpy{\wt})=\frac{Z^0\huaRad_{\sg}(\cpy{\ws},\cpy{\wt})}{B^0\huaHom_{\sg}(\cpy{\ws},\cpy{\wt})}\subset H^0\huaHom_{\sg}(\cpy{\ws},\cpy{\wt})= \Hom_{\per\sg}(\cpy{\ws},\cpy{\wt}).$$

To show \eqref{eq:aim}, we shall establish a bijection from the set of TOIs from $\wsx$ to $\wtx$ of index 0 to a basis of $\Hom_{\per\sg}(\cpy{\ws},\cpy{\wt})$. By taking suitable representatives in the homotopy classes, we may assume the following.
\begin{enumerate}[label=(T\arabic*),ref=(T\arabic*)]
\item\label{item:T1} The tagged arcs $\wsx$ and $\wtx$ are in a minimal position.
\item\label{item:T2} All intersections near which $\wsx$ and $\wtx$ share at least one arc in $\dac$ are in $\M\cup\P\cup\dac$. That is, except for the intersections which is given by an arc segment $\wmu$ of $\wsx$ and an arc segment $\wnu$ of $\wtx$ such that any two of the four edges that $\wmu$ and $\wnu$ touch are different, any other intersection is in $\M\cup\P\cup\dac$.
\end{enumerate}

\begin{construction}\label{cons:TOI}
    We divide the set of TOIs from $\wsx$ to $\wtx$ of index 0 into the following three subsets.
\begin{itemize}
    \item $I_0$ consists of the TOIs $q$ from $\wsx$ to $\wtx$, such that $\wsx$ and $\wtx$ do not share any common arc in $\dac$ near $q$.
    \item $I_1$ consists of the TOIs $q$ from $\wsx$ to $\wtx$, such that $\wsx$ and $\wtx$ share at least one common arc in $\dac$ near $q$, for each of whose sides, if separating, $\wsx$ turns to the left while $\wtx$ turns to the right in an $\dac$-polygon, and if not separating and ending at a puncture, $\wsx$ and $\wtx$ have the same tagging there, see Figure~\ref{fig:real-h}.
    \item $I_2$ consists of the TOIs $q$ from $\wsx$ to $\wtx$, such that $\wsx$ and $\wtx$ share at least one common arc in $\dac$ near $q$, for each of whose sides, if separating, $\wsx$ turns to the right while $\wtx$ turns to the left in an $\dac$-polygon, and if not separating, $\wsx$ and $\wtx$ end at a puncture with different taggings, see Figure~\ref{fig:real-h} after switching $\wsx$ and $\wtx$ there.
\end{itemize}
\end{construction}

Recall from Definition~\ref{def:int} the types of TOIs (see also Figure~\ref{fig:oriented int}). The TOIs of type (I) are in $I_0\cup I_1\cup I_2$. The TOIs of type (II) are in $I_0\cup I_1$, those of type (III)/(V) are in $I_1$, and those of type (IV)/(VI) are in $I_2$.

In Section~\ref{subsec:mor1}, we give a basis of a completion of $\operatorname{Rad}_{\per\sg}(\cpy{\ws},\cpy{\wt})$ indexed by $I_1$, and in Section~\ref{subsec:mor2}, we give a basis of  $\operatorname{Rad}_{\per\sg}(\cpy{\ws},\cpy{\wt})$ indexed by $I_0\cup I_2$.

\subsection{Bi-quivers and lines}\label{subsec:biquiver}

The technique introduced in this subsection is essentially from \cite{G}.

For any $\wmu\in\inter{\as(\wsx)}$, denote by $s(\wmu)$ and $t(\wmu)$ its start and terminal in $Q^{\ws}$, respectively. For any $\wmu\in\tou{\as(\wsx)}=\{\asu{\wsx}{0}{1}, \asu{\wsx}{p_{\ws}}{p_{\ws}+1}\}$, we define the vertex $e(\wmu)$ of $Q^{\ws}$ to be
$$e(\wmu)=\begin{cases}
1&\text{if $\wmu=\asu{\wsx}{0}{1}$,}\\
p_{\ws}&\text{if $\wmu=\asu{\wsx}{p_{\ws}}{p_{\ws}+1}$.}
\end{cases}$$

\begin{construction}\label{cons:biq}
    The bi-quiver $\qst=(\qst_0,\qst_1,\qst_2=\qst_{2,o}\cup \qst_{2,-})$ is constructed as follows, where arrows in $\qst_1$ are solid while arrows in $\qst_2$ are dashed.
\begin{enumerate}
    \item $\qst_0=\{(u,v)\mid 1\leq u\leq p_{\ws},1\leq v\leq p_{\wt},\hks_u=\hkt_v,\wsi_u=\wti_v\}\subset Q^{\ws}_0\times Q^{\wt}_0$.
	\item Each pair of $\wmu\in\inter{\as(\wsx)}$ and $\wnu\in\inter{\as(\wtx)}$ with $\wmu=\wnu$, gives rise to an arrow $$(\wmu,\wnu):(s(\wmu),s(\wnu))\to(t(\wmu),t(\wnu))$$ in $\qst_1$.
	\item Each pair of $\wmu\in\tou{\cias(\wsx)}$ and $\wnu\in\tou{\cias(\wtx)}$ with $\wmu=\wnu$, gives rise to a loop
	$$(\wmu,\wnu)_o:(e(\wmu),e(\wnu))\to(e(\wmu),e(\wnu))$$
	at $(e(\wmu),e(\wnu))$ in $\qst_1$.
	\item Each pair $\wmu\in\inter{\cias(\wsx)}$ and $\wnu\in\inter{\cias(\wtx)}$ with $\wmu$ and $\wnu$ in the same punctured $\dac$-polygon, gives rise to two arrows in $\qst_{2,-}$:  $$(\wmu,\wnu)_+:(t(\wmu),s(\wnu))\dashrightarrow(s(\wmu),s(\wnu)),\ (\wmu,\wnu)_-:(t(\wmu),s(\wnu))\dashrightarrow(t(\wmu),t(\wnu)),$$
	which are called the \emph{twin} of each other. See the cases in the first row of Figure~\ref{fig:Q2}.
	\item Each pair $\wmu\in\tou{\cias(\wsx)}$ and $\wnu\in\inter{\cias(\wtx)}$, with $\wmu$ and $\wnu$ in the same punctured $\dac$-polygon, gives rise to an arrow $$(\wmu,\wnu)_{\ominus}:(e(\wmu),s(\wnu))\dashrightarrow(e(\wmu),t(\wnu))$$ in $\qst_{2,o}$, see the cases in the second row of Figure~\ref{fig:Q2}.
	\item Each pair $\wmu\in\inter{\cias(\wsx)}$ and $\wnu\in\tou{\cias(\wsx)}$, with $\wmu$ and $\wnu$ in the same punctured $\dac$-polygon, gives rise to an arrow $$(\wmu,\wnu)_{\oplus}:(t(\wmu),e(\wnu))\dashrightarrow(s(\wmu),e(\wnu))$$ in $\qst_{2,o}$, see the cases in the third row of Figure~\ref{fig:Q2}.
\end{enumerate}

\begin{figure}[htpb]
	\begin{tikzpicture}[scale=.45]
	\draw[red,thick] (0,-3)..controls +(130:6) and +(180:1) .. (0,2) .. controls +(0:1) and +(50:6) .. (0,-3);
	\draw[blue,thick] (-3,3)node[above]{$\wtx$}..controls +(-75:4) and +(-180:3) .. (2.5,-2.5);
	\draw[blue,thick] (3,3)node[above]{$\wsx$}..controls +(-105:4) and +(0:3) .. (-2.5,-2.5);
	\draw (0,0)\nn;
	\draw (1.3,-2)node{$v'$} (-1.3,-2)node{$u'$} (1.5,0)node{$u$} (-1.5,0)node{$v$};
	\draw (5,2)node{$(u,v)$} (9,2)node{$(u',v)$} (5,-2)node{$(u,v')$} (9,-2)node{$(u',v')$};
	\draw[<-] (5.5,-1.5)tonode[below,right]{\tiny$(\wmu,\wnu)$}(8.5,1.5);
	\draw[dashed,->] (6,2)tonode[above]{\tiny $(\wmu,\wnu)_+$}(8,2);
	\draw[dashed,->] (5,1.5)tonode[left]{\tiny $(\wmu,\wnu)_-$}(5,-1.5);
	\end{tikzpicture}\quad
	\begin{tikzpicture}[scale=.45]
	\draw[red,thick] (0,-3)..controls +(130:6) and +(180:1) .. (0,2) .. controls +(0:1) and +(50:6) .. (0,-3);
	\draw[blue,thick] (-3,3)node[above]{$\wsx$}..controls +(-75:4) and +(-180:3) .. (2.5,-2.5);
	\draw[blue,thick] (3,3)node[above]{$\wtx$}..controls +(-105:4) and +(0:3) .. (-2.5,-2.5);
	\draw (0,0)\nn;
	\draw (1.3,-2)node{$u$} (-1.3,-2)node{$v$} (1.5,0)node{$v'$} (-1.4,0)node{$u'$};
	\draw (5,2)node{$(u',v')$} (9,2)node{$(u,v')$} (5,-2)node{$(u',v)$} (9,-2)node{$(u,v)$};
	\draw[->] (5.5,-1.5)tonode[above,left]{\tiny$(\wmu,\wnu)$}(8.5,1.5);
	\draw[dashed,<-] (6,-2)tonode[below]{\tiny $(\wmu,\wnu)_+$}(8,-2);
	\draw[dashed,<-] (9,1.5)tonode[right]{\tiny $(\wmu,\wnu)_-$}(9,-1.5);
	\end{tikzpicture}
	
	\begin{tikzpicture}[scale=.45]
	\draw[red,thick] (0,-3)..controls +(130:6) and +(180:1) .. (0,2) .. controls +(0:1) and +(50:6) .. (0,-3);
	\draw[blue,thick] (-3,3)node[above]{$\wtx$}..controls +(-75:4) and +(-180:3) .. (2.5,-2.5);
	\draw[blue,thick] (3,3)node[above]{$\wsx$}..controls +(-105:3) and +(0:1) .. (0,0);
	\draw (0,0)\nn;
	\draw (1.3,-2)node{$v'$} (1.5,0)node{$u$} (-1.5,0)node{$v$};
	\draw (5,2)node{$(u,v)$} (9,2)node[white]{$(u',v)$} (5,-2)node{$(u,v')$} (9,-2)node[white]{$(u',v')$};
	\draw[dashed,->] (5,1.5)tonode[left]{\tiny $(\wmu,\wnu)_{\ominus}$}(5,-1.5);
	\end{tikzpicture}\quad
	\begin{tikzpicture}[scale=.45]
	\draw[red,thick] (0,-3)..controls +(130:6) and +(180:1) .. (0,2) .. controls +(0:1) and +(50:6) .. (0,-3);
	\draw[blue,thick] (-3,3)node[above]{$\wsx$}..controls +(-75:3) and +(-180:1) .. (0,0);
	\draw[blue,thick] (3,3)node[above]{$\wtx$}..controls +(-105:4) and +(0:3) .. (-2.5,-2.5);
	\draw (0,0)\nn;
	\draw (-1.3,-2)node{$v$} (1.5,0)node{$v'$} (-1.5,0)node{$u$};
	\draw (5,2)node[white]{$(u',v')$} (9,2)node{$(u,v')$} (5,-2)node[white]{$(u',v)$} (9,-2)node{$(u,v)$};
	\draw[dashed,<-] (9,1.5)tonode[right]{\tiny $(\wmu,\wnu)_{\ominus}$}(9,-1.5);
	\end{tikzpicture}
	
	\begin{tikzpicture}[scale=.45]
	\draw[red,thick] (0,-3)..controls +(130:6) and +(180:1) .. (0,2) .. controls +(0:1) and +(50:6) .. (0,-3);
	\draw[blue,thick] (-3,3)node[above]{$\wtx$}..controls +(-75:3) and +(-180:1) .. (0,0);
	\draw[blue,thick] (3,3)node[above]{$\wsx$}..controls +(-105:4) and +(0:3) .. (-2.5,-2.5);
	\draw (0,0)\nn;
	\draw (-1.3,-2)node{$u'$} (1.5,0)node{$u$} (-1.5,0)node{$v$};
	\draw (5,2)node{$(u,v)$} (9,2)node{$(u',v)$} (5,-2)node[white]{$(u,v')$} (9,-2)node[white]{$(u',v')$};
	\draw[dashed,->] (6,2)tonode[above]{\tiny $(\wmu,\wnu)_{\oplus}$}(8,2);
	\end{tikzpicture}\quad
	\begin{tikzpicture}[scale=.45]
	\draw[red,thick] (0,-3)..controls +(130:6) and +(180:1) .. (0,2) .. controls +(0:1) and +(50:6) .. (0,-3);
	\draw[blue,thick] (-3,3)node[above]{$\wsx$}..controls +(-75:4) and +(-180:3) .. (2.5,-2.5);
	\draw[blue,thick] (3,3)node[above]{$\wtx$}..controls +(-105:3) and +(0:1) .. (0,0);
	\draw (0,0)\nn;
	\draw (1.3,-2)node{$u$} (1.5,0)node{$v$} (-1.4,0)node{$u'$};
	\draw (5,2)node[white]{$(u',v')$} (9,2)node[white]{$(u,v')$} (5,-2)node{$(u',v)$} (9,-2)node{$(u,v)$};
	\draw[dashed,<-] (6,-2)tonode[above]{\tiny $(\wmu,\wnu)_{\oplus}$}(8,-2);
	\end{tikzpicture}
	\caption{Cases for arrows in $Q_2$}\label{fig:Q2}
\end{figure}
\end{construction}

Components of $(\qst_0,\qst_1)$ are called \emph{lines} in $\qst$. Let $\mathcal{L}(\ws,\wt)$ be the set of lines in $\qst$.

\begin{definition}
    For any vertex $(u,v)$ in $\qst_0$, a pair $(\wmu,\wnu)$ of $\wmu\in\as(\wsx)$ and $\wnu\in\as(\wtx)$ is called a \emph{side} of $(u,v)$ provided that $\Vs_u$ and $\Vt_v$ are endpoints of $\wmu$ and $\wnu$ respectively, $\wmu$ and $\wnu$ are in the same side of the arc $\hks_u=\hkt_v\in\dac$. Moreover, this side is called
    \begin{itemize}
        \item \emph{unpunctured} (resp. \emph{punctured}) provided that $\wmu$ and $\wnu$ are in an unpunctured (resp. punctured) $\dac$-polygon,
        \item \emph{splitting} provided that when taking the orientations starting from $\Vs_u$ and $\Vt_v$ respectively, $\wmu\nsim\wnu$, and
        \item \emph{ending} provided that when taking the orientations starting from $\Vs_u$ and $\Vt_v$ respectively, either $\wmu\nsim\wnu$, or $\wmu\sim\wnu$ and they end at the same point in $\M\cup\P$.
    \end{itemize}
    A vertex admitting a splitting (resp. ending) side is called an \emph{outlet} (\emph{endpoint}).
\end{definition}

By definition, any splitting side is ending, so any outlet is an endpoint. Each vertex admits exactly two sides, at least one of which is unpunctured.

\begin{remark}\label{rmk:types}
Each line in $\mathcal{L}(\ws,\wt)$ is one of the types $A,D,\widetilde{D}$ (cf. \cite[Remark~3.5~(1)]{G}), see Figure~\ref{fig:types}. Each line has exactly two ending sides. A line of type $A$ (resp. $D$, $\widetilde{D}$) has at most two (resp. one, zero) outlets. If an outlet has two different splitting sides, then it is the vertex of a line of type $A$ consisting of a single vertex; if an outlet admits a loop, then it is the vertex of a line of type $D$ consisting of a single vertex with a single loop.
\end{remark}

\begin{figure}[htpb]
	\begin{tikzpicture}[scale=.8]
	\draw[thick] (-1,2)node{Type $A:$} (1,2)\nn (3,2)\nn (5,2)node{$\cdots$} (7,2)\nn;
	\draw (1.3,2)to(2.7,2) (3.3,2)to(4.7,2) (5.3,2)to(6.7,2);
	\draw[thick] (-1,0)node{Type $D:$} (1,0)\nn (3,0)\nn (5,0)node{$\cdots$} (7,0)\nn;
	\draw (1.3,0)to(2.7,0) (3.3,0)to(4.7,0) (5.3,0)to(6.7,0);
	\draw (7.2,-.3) .. controls +(-45:1) and +(45:1)..(7.2,.3);
	\draw[thick] (-1,-2)node{Type $\widetilde{D}:$} (1,-2)\nn (3,-2)\nn (5,-2)node{$\cdots$} (7,-2)\nn;
	\draw (1.3,-2)to(2.7,-2) (3.3,-2)to(4.7,-2) (5.3,-2)to(6.7,-2);
	\draw (7.2,-2.3) .. controls +(-45:1) and +(45:1)..(7.2,-1.7);
	\draw (.8,-2.3) .. controls +(-135:1) and +(135:1)..(.8,-1.7);
	\end{tikzpicture}
	\caption{Types of lines}\label{fig:types}
\end{figure}

\begin{remark}\label{rmk:Q22}
	By the definition of $\qst_2=\qst_{2,o}\cup \qst_{2,-}$, cf. Figure~\ref{fig:Q2}, both the start and the terminal of any arrow in $\qst_{2,o}$ are outlets of lines in $\mathcal{L}(\ws,\wt)$, while the start of any arrow in $\qst_{2,-}$ is an outlet of a line in $\mathcal{L}(\ws,\wt)$ although the terminal may not be.

    Any vertex in $\qst_0$ is incident to at most one arrow in $\qst_2$, except the start of an arrow $\alpha$ in $\qst_{2,-}$, which is also the start of the twin of $\alpha$ in $\qst_{2,-}$, see the cases in the first row of Figure~\ref{fig:Q2}.
\end{remark}

\begin{example}\label{ex:biq}
    Let $\ws$ and $\wt$ be the graded unknotted arcs in Example~\ref{ex:unknotted}, whose tagged versions are shown in Example~\ref{ex:tag}. Denote by $\nu_1,\nu_2$ the interior punctured arc segment in $E_{p_1}$ and $E_{p_2}$, respectively. Then the biquiver $\qst$ is the following.
    $$\xymatrix{
    &&&Q^{\ws}\\
    &&1\ar[r]^{\mu_2}&2\ar[r]^{\nu_1}&3\ar[r]^{\mu_3}&4\\
    &1\ar[d]_{\mu_2}&(1,1)\ar[dr]\\
    &2\ar[d]_{\nu_1}&&(2,2)\ar[dr]&(3,2)\ar@{-->}[l]\ar@{-->}[d]\\
    &3\ar[d]_{\mu_3}&&(2,3)&(3,3)\ar[dr]\\
    &4\ar[d]_{\nu_2}&&&&(4,4)\ar@{-->}[d]\\
    Q^{\wt}&5&&&&(4,5)\\
    &6\ar[u]^{\mu_3}\ar[d]_{\nu_1}&&(2,6)\ar[rd]&(3,6)\ar[ru]\ar@{-->}[l]\ar@{-->}[d]\\
    &7\ar[d]_{\mu_3}&&(2,7)&(3,7)\ar[rd]\\
    &8\ar[d]_{\nu_2}&&&&(4,8)\ar@{-->}[d]\\
    &9\ar[d]_{\mu_4}&&&&(4,9)\\
    &10
    }$$
    There are seven lines in $\mathcal{L}(\ws,\wt)$, which contain the vertices $(1,1)$, $(3,2)$, $(2,3)$, $(4,5)$, $(2,6)$, $(2,7)$ and $(4,9)$, respectively. All of them are of type $A$. The line containing $(1,1)$ has exactly one outlet $(4,4)$ and this outlet has exactly one splitting side $(\asu{\wsx}{4}{5},\asu{\wtx}{4}{5})$.
\end{example}

\begin{definition}\label{def:0-pt}
A splitting side $(\wmu,\wnu)$ is called a \emph{0-side} if neither $\wmu$ nor $\wnu$ is punctured non-interior, and $\wmu$ is to the right of $\wnu$. An outlet is called a \emph{0-point} if it admits a 0-side. That is, a 0-point is a vertex $(u',v')$ arising from the case in Figure~\ref{fig:0} or the cases in the first row of Figure~\ref{fig:Q2}.

\begin{figure}[htpb]
	\begin{tikzpicture}[xscale=1.2,scale=.8]
	\draw[ultra thick] (-2,2)to(2,2);
	\draw[red,dashed,bend left=50,thick] (-1,-2)to(-1,2);
	\draw[red,dashed,bend right=50,thick] (1,-2)to(1,2);
	\draw[blue,thick] (-.5,-2.2)..controls +(90:.7) and +(-60:.7)..(-1,-.5);
	\draw[blue,dashed,thick] (-1,-.5)..controls +(120:.7) and +(-60:.7)..(-1.3,0);
	\draw[blue,thick] (.5,-2.2)..controls +(90:.7) and +(-120:.7)..(1,-.5);
	\draw[blue,dashed,thick] (1,-.5)..controls +(60:.7) and +(-120:.7)..(1.3,0);
	\draw[blue] (-.8,-2.3)node{$\Vt_{v'}$} (.8,-2.3)node{$\Vs_{u'}$} (0,-1)node{$<$};
	\draw[blue] (1,-1)node{$\wsx$} (-1,-1)node{$\wtx$};
		\draw[red,thick] (-1,-2)\ww to(1,-2)\ww;
	\draw[red,thick] (-1,2)\ww (1,2)\ww;
	\end{tikzpicture}
	\caption{A 0-point}\label{fig:0}
\end{figure}

An \emph{$h$-line} is a line in $\mathcal{L}(\ws,\wt)$ whose connected component in $(\qst_0,\qst_1\cup\qst_{2,o})$ does not contain any 0-point. A \emph{real $h$-line} is a line in $\mathcal{L}(\ws,\wt)$ which does not contain any 0-point and the terminal of any arrow in $\qst_{2,o}$. A real $h$-line is called \emph{tagged} provided that for any of its loops $(\wmu,\wnu)$ from $\qst_1$, the taggings of $\wsx$ and $\wtx$ at $\wmu$ and $\wnu$ coincide.
\end{definition}

\begin{remark}\label{rmk:hline}
	A vertex in $\qst_0$ is a 0-point or the terminal of any arrow in $\qst_{2,o}$ if and only if it admits a splitting side $(\wmu,\wnu)$ with $\wmu$ to the right of $\wnu$. Hence there is a natural bijection from the set of tagged real $h$-line $L$ to the set $I_1$ (see Construction~\ref{cons:TOI}). Moreover, a real $h$-line $L$ is of type $D$ (resp. $\widetilde{D}$) if and only if for the corresponding TOI $q$ and the common part of $\wsx$ and $\wtx$ near $q$, exactly on one side (resp. on each side), $\wsx$ and $\wtx$ do not separate and end at a puncture, see the second (resp. third) picture in Figure~\ref{fig:real-h} with each common endpoint a puncture.
\end{remark}

\begin{figure}[htpb]
	\begin{tikzpicture}
	\draw[red,dashed,thick] (-1,1)to(1,1) (-1,-1)to(1,-1);
	\draw[blue,thick] (2,.7)node[above]{$\wsx$}to(-2,-.7);
	\draw[blue,thick] (2,-.7)node[below]{$\wtx$}to(-2,.7);
	\draw[blue] (0,0)node[above]{$q$};
	\draw[orange,bend left,->-=.6,>=stealth] (22:-.5)to(-22:-.5);
	\draw[orange,bend left,->-=.6,>=stealth] (22:.5)to(-22:.5);
	\draw[red,thick] (-1,1)\ww to (-1,-1)\ww (1,1)\ww to (1,-1)\ww;
	\end{tikzpicture}\quad
	\begin{tikzpicture}
	\draw[red,dashed,thick] (-1,1)to(1,1) (-1,-1)to(1,-1);
	\draw[blue,thick] (2,.7)node[above]{$\wsx$}to(-2,0);
	\draw[blue,thick] (2,-.7)node[below]{$\wtx$}to(-2,0)\nn;
	\draw[blue] (-2,0)node[above]{$q$};
	\draw[orange,bend left,->-=.6,>=stealth] (180-8:1.2)to(180+8:1.2);
	\draw[red,thick] (-1,1)\ww to (-1,-1)\ww (1,1)\ww to (1,-1)\ww;
	\end{tikzpicture}\qquad
	\begin{tikzpicture}
	\draw[red,dashed,thick] (-1,1)to(1,1) (-1,-1)to(1,-1);
	\draw[blue,thick,bend right] (2,0)tonode[below]{$\wsx$}(-2,0);
	\draw[blue,thick,bend left] (2,0)\nn tonode[above]{$\wtx$}(-2,0)\nn;
	\draw[blue] (-2,0)node[above]{$q$};
	\draw[orange,bend left,->-=.6,>=stealth] (180-10:1.5)to(180+10:1.5);
	\draw[red,thick] (-1,1)\ww to (-1,-1)\ww (1,1)\ww to (1,-1)\ww;
	\end{tikzpicture}
	\caption{Intersections corresponding to real $h$-lines}
	\label{fig:real-h}
\end{figure}

Recall from \cite{G} the following results.

\begin{proposition}[{\cite[Propositions~2 and 3]{G}}]\label{prop:G}
    There is a partial order $>$ on the set $\mathcal{L}(\ws,\wt)$, satisfying that if there is an arrow in $Q_2$ from a vertex in $L'$ to a vertex in $L$, then $L'>L$. Moreover, the following hold.
    \begin{itemize}
    \item[(1)] Let $L'>L\in\mathcal{L}(\ws,\wt)$. If $L'$ is an $h$-line, then $L$ is an $h$-line of type $A$ or $D$.
    \item[(2)] If $L$ is an $h$-line, then there is at most one arrow in $Q_{2,o}$ whose terminal is a vertex of $L$; if $L$ is additionally of type $D$, there is no such arrow.
\end{itemize}
\end{proposition}

\begin{example}
    In Example~\ref{ex:biq}, the 0-points are $(2,3)$, $(3,2)$ and $(2,7)$, so there are four $h$-lines and three non-$h$-lines. The poset $\mathcal{L}(\ws,\wt)$ has three connected components: one consists of the line $(2,3)$, one consists of the line $(2,7)$, and the third one is a linearly ordered subset consisting of the other five lines. There are two real $h$-lines: the one containing $(1,1)$ and the one containing $(2,6)$.
\end{example}

\subsection{Morphisms from extensions of lines}\label{subsec:mor1}

In this subsection, we find a basis of a complement of  $\operatorname{Rad}_{\per\sg}(\cpy{\ws},\cpy{\wt})$ in $\Hom_{\per\sg}(\cpy{\ws},\cpy{\wt})$, whose images under the functor $\fM\circ\fE$ (see Theorems~\ref{thm:tribu} and \ref{thm:BD} for these two functors) are the basis of morphisms in $\repb(S)$ described in \cite{G}.

Let $(u,v)$ be an outlet of a real $h$-line $L$ with side $(\wmu,\wnu)$. An \emph{extension} $E$ of $L$ from $(u,v)$ in the side $(\wmu,\wnu)$ is a full subquiver of type $A$ of $\qst$
\begin{equation}\label{eq:ex}
    \xymatrix{
    (\uex{E}{0},\vex{E}{0})=(u,v)\ar@{-}[r]^{\qquad\aex{E}{1}}&(u^E_1,v^E_1)\ar@{-}[r]^{\quad\aex{E}{2}}&\cdots\ar@{-}[r]^{\aex{E}{s_E}\qquad\quad}&(u^E_{s_E},v^E_{s_E})=:(u_E,v_E),
    }
\end{equation}
with $(u_E,v_E)$ an outlet, and a splitting unpunctured side $\alpha_{s_E+1}^E=(\wmu_E,\wnu_E)$ of $(u_E,v_E)$, such that the following hold.
\begin{enumerate}[label=(E\arabic*),ref=(E\arabic*)]
    \item\label{itm:E1} The vertex $(u,v)$ is the only common vertex of $E$ and $L$.
    \item\label{itm:E0} The arrow (when $s_E\neq 0$) or side (when $s_E=0$) $\alpha^E_{1}$ is given by $(\wmu,\wnu)$.
    \item\label{itm:E4} For any $1\leq \exi\leq s_E$, if $\aex{E}{\exi}\in \qst_2$, then $\aex{E}{\exi}$ is from $(\uex{E}{\exi-1},\vex{E}{\exi-1})$ to $(\uex{E}{\exi},\vex{E}{\exi})$.
    \item\label{itm:E5} For any $1\leq\exi\leq s_E$, if $\aex{E}{\exi}\in \qst_{2,-}$, then the twin of $\aex{E}{\exi}$ is not in $E$ (or equivalently the terminal of the twin of $\aex{E}{\exi}$ is not a vertex of $E$).
    \item\label{itm:E2}
    The arrow $\aex{E}{s_E}$ is not given by $(\wmu_E,\wnu_E)$ when $s_E\neq 0$.
\end{enumerate}

By Proposition~\ref{prop:G}, an extension of any real $h$-line from any outlet on any side exists although is not necessarily unique.

\begin{notations}\label{Not:c}
    In the form \eqref{eq:ex}, for any $1\leq \exi\leq s_E+1$, we define a number $c^E_\exi\in\ZZ$ (with $c^E_0=1$) as follows.
    \begin{itemize}
        \item If $\aex{}{\exi}\in \qst_1$, take $c^E_{\exi}=c^E_{\exi-1}$.
        \item If $\aex{}{\exi}\in \qst_{2,-}$ (see the cases in the first row of Figure~\ref{fig:Q2}), take
        $$c^E_\exi=\begin{cases}-c^E_{\exi-1}&\text{if $\aex{}{l}=(\asu{\wsx}{\uex{}{\exi-1}}{\uex{}{\exi}},\asu{\wtx}{\vex{}{\exi-1}}{\vex{}{\exi}})_{-}$}\\
        c^E_{\exi-1}&\text{if $\aex{}{l}=(\asu{\wsx}{\uex{}{\exi-1}}{\uex{}{\exi}},\asu{\wtx}{\vex{}{\exi-1}}{\vex{}{\exi}})_{+}$.}\end{cases}$$
        \item If $\aex{}{\exi}=(\asu{\wsx}{\uex{}{\exi-1}}{u'},\asu{\wtx}{\vex{}{\exi-1}}{\vex{}{\exi}})_{\ominus}$ with $\Vs_{u'}\in\P$ (see the cases in the second row of Figure~\ref{fig:Q2}), take
        $$c^E_\exi=\begin{cases}0&\text{if the tagging at $\asu{\wsx}{\uex{}{\exi}}{u}$ is $+$,}\\
        -c^E_{\exi-1}&\text{if the tagging at $\asu{\wsx}{\uex{}{\exi}}{u}$ is $-$.}\end{cases}$$
        \item If $\aex{}{\exi}=(\asu{\wsx}{\uex{}{\exi-1}}{\uex{}{\exi}},\asu{\wtx}{\vex{}{\exi-1}}{v'})_{\oplus}$ with $\Vt_{v'}\in\P$ (see the cases in the third row of Figure~\ref{fig:Q2}), take
        $$c^E_l=\begin{cases}c^E_{\exi-1}&\text{if the tagging at $\asu{\wtx}{\uex{}{\exi}}{u}$ is $+$,}\\
        0&\text{if the tagging at $\asu{\wtx}{\uex{}{\exi}}{u}$ is $-$.}\end{cases}$$
        \item For $l=s_E+1$, take $c^E_{s_E+1}=c^E_{s_E}.$
    \end{itemize}
\end{notations}

\begin{example}\label{ex:line}
    Let $L$ be the real $h$-line in Example~\ref{ex:biq} which contains $(1,1)$. There are two extensions of $L$ from $(4,4)$ in the side $(\asu{\wsx}{4}{5},\asu{\wtx}{4}{5})$:
    $$E'\colon\xymatrix{
    (4,4)\ar@{-->}[r]&(4,5)&(3,6)\ar[l]\ar@{-->}[r]&(2,6),
    }$$
    and
    $$E''\colon\xymatrix{
    (4,4)\ar@{-->}[r]&(4,5)&(3,6)\ar[l]\ar@{-->}[r]&(3,7)\ar[r]&(4,8)\ar@{-->}[r]&(4,9).
    }$$
    For $E'$, we have $c_l^{E'}=-1$ for any $1\leq l\leq 4$. For $E''$, we have $c_1^{E''}=c_2^{E''}=-1$, $c_3^{E''}=c_4^{E''}=1$, $c_5^{E''}=c_6^{E''}=-1$.

    Note that if we replace the tagging of $\wsx$ (at its unique punctured endpoint) from $-$ to $+$ to get another tagged arc $(\widetilde{\sigma}')^\times$, we have the same biquiver $Q^{(\widetilde{\sigma}')^\times,\wtx}=Q^{\wsx,\wtx}$. So we have the same real $h$-line $L$ and the same extension $E'$. However, we have $c_l^{E'}=0$ for any $1\leq l\leq 4$ in this case.
\end{example}

An \emph{extension} $\widetilde{L}$ of a real $h$-line $L$ is a set consisting of an extension of $L$ from each side of each outlet of $L$.

\begin{construction}\label{cons:fL}
    Let $\widetilde{L}$ be an extension of a tagged real $h$-line $L\in\mathcal{L}(\wsx,\wtx)$. We define a morphism $f_{\widetilde{L}}:\cpy{\ws}\to\cpy{\wt}$ to be
    $$f_{\widetilde{L}}=\sum_{(u,v)\in L_0}f^L_{(u,v)}+\sum_{E\in\widetilde{L}}\sum_{1\leq \exi\leq s_E+1}c_\exi^Ef^E_{(\uex{E}{\exi},\vex{E}{\exi})},$$
    where
    $$f^L_{(u,v)}=\fin{\xit_v}\circ\iota_{(u,v)}\circ\fout{\xis_u}:\cply{\ws}{\overline{u}}\to\cply{\wt}{\overline{v}},$$
    $$f^E_{(\uex{E}{i},\vex{E}{i})}=\fin{\xit_{v}}\circ\iota_i\circ\fout{\xis_{u}}:\cply{\ws}{\overline{u}}\to\cply{\wt}{\overline{v}}.$$
    Here, $\iota_{(u,v)}$, $(u,v)\in L_0$, and $\iota_i$, $1\leq i\leq s_E$, are canonical embeddings, projections or identities, and  $\iota_{s_E+1}=f_{(i,j_1),(i,j_2)}$ for $\wsx_{E,s_E+1}=(i,j)\to (i,j_2)$ and $\wtx_{E,s_E+1}=(i,j)\to (i,j_1)$, where if $j_1>j_2$ or $j_1=0$, we take $f_{(i,j_1),(i,j_2)}=0$.
\end{construction}

\begin{example}
    We continue Example~\ref{ex:line}. Then there are two extensions of $L$: $\widetilde{L}'=\{(1,1),E'\}$ and $\widetilde{L}''=\{(1,1),E''\}$. Then the morphisms $f_{\widetilde{L}'},f_{\widetilde{L}''}:\cpy{\ws}\to\cpy{\wt}$ are given respectively by
    $$f_{\widetilde{L}'}=\begin{pmatrix}
    1&0&0&0\\
    0&1&0&0\\
    0&0&1&0\\
    0&0&0&0\\
    0&0&0&-1\\
    0&-1&0&0\\
    0&0&0&0\\
    0&0&0&0\\
    0&0&0&0\\
    0&0&0&0\\
    \end{pmatrix},\ f_{\widetilde{L}''}=\begin{pmatrix}
    1&0&0&0\\
    0&1&0&0\\
    0&0&1&0\\
    0&0&0&0\\
    0&0&0&-1\\
    0&0&0&0\\
    0&0&1&0\\
    0&0&0&0\\
    0&0&0&-1\\
    0&0&0&0\\
    \end{pmatrix}$$
    where $\cpy{\ws}$ and $\cpy{\wt}$ are given in Example~\ref{ex:dg}. Note that  $\iota_{s_{E'}+1}=f_{(1,4),(1,1)}=0$ and $\iota_{s_{E''}+1}=f_{(1,7),(1,4)}=0$. A direct calculation shows that $f_{\widetilde{L}''}-f_{\widetilde{L}'}=f_{\widetilde{L}'''}$, where $\widetilde{L}'''$ is the unique extension of the tagged real $h$-line containing $(2,6)$ as a vertex.
\end{example}

\begin{lemma}\label{lem:fL}
	Let $\widetilde{L}$ be an extension of a tagged real $h$-line $L$. Then $$f_{\widetilde{L}}\in Z^0\huaHom_{\sg}(\cpy{\ws},\cpy{\wt}).$$
\end{lemma}

\begin{proof}
    A direct calculation (or Lemma~\ref{lem:homotopy} in the next subsection, with a little modification of signs due to the degree of the identities) shows this lemma.
\end{proof}

Due to Lemma~\ref{lem:fL}, for any extension $\widetilde{L}$ of a tagged real $h$-line $L$, we obtain a morphism $$f_{\widetilde{L}}\in\Hom_{\per\sg}(\cpy{\ws},\cpy{\wt}).$$

\begin{example}\label{exm:tildeD}
	A line $L\in\mathcal{L}(\ws,\wt)$ is of type $\widetilde{D}$ if and only if $\wsx$ and $\wtx$ have the same underlying arcs, their endpoints are in $\P$, and $L$ is the diagonal of $\qst$. In this case, $L$ is a real $h$-line (cf. the third case in Figure~\ref{fig:real-h}), and since $L$ has no outlet by Remark~\ref{rmk:hline}, $\widetilde{L}=\emptyset$ is the unique extension of $L$. Moreover, $L$ is tagged if and only if $\wsx$ and $\wtx$ have the same taggings at their common start and end, respectively. In this case, the morphism $f_{L}$ is the identity.
\end{example}

A set $\widetilde{\mathcal{L}}$ of extensions of tagged real $h$-lines is called \emph{complete} provided that for any tagged real $h$-line $L$, exactly one of its extensions is contained in $\widetilde{\mathcal{L}}$. Recall from Theorems~\ref{thm:tribu} and \ref{thm:BD} the functors $\fM$ and $\fE$.

\begin{lemma}\label{lem:G}
    Let $\widetilde{\mathcal{L}}$ be a complete set of extensions of tagged real $h$-lines. Then $\{\fM\circ\fE(f_{\widetilde{L}})\mid \widetilde{L}\in\widetilde{\mathcal{L}}\}$ form a basis of $\Hom_{\repb(S)}(\fM\circ\fE(\cpy{\ws}),\fM\circ\fE(\cpy{\wt}))$.
\end{lemma}

\begin{proof}
    See Appendix~\ref{app:pf2}.
\end{proof}

\begin{proposition}\label{prop:rad}
    Let $\widetilde{\mathcal{L}}$ be a complete set of extensions of tagged real $h$-lines. Then the set of morphisms $\{f_{\widetilde{L}}\mid\widetilde{L}\in\widetilde{\mathcal{L}}\}$ is linearly independent in $\Hom_{\per\sg}(\cpy{\ws},\cpy{\wt})$. Moreover, we have
    $$\Hom_{\per\sg}(\cpy{\ws},\cpy{\wt})=\k\{f_{\widetilde{L}}\mid\widetilde{L}\in\widetilde{\mathcal{L}}\}\oplus \operatorname{Rad}_{\per\sg}(\cpy{\ws},\cpy{\wt}),$$
    where $\k\{f_{\widetilde{L}}\mid\widetilde{L}\in\widetilde{\mathcal{L}}\}$ is the subspace of $\Hom_{\per\sg}(\cpy{\ws},\cpy{\wt})$ linearly generated by the set $\{f_{\widetilde{L}}\mid\widetilde{L}\in\widetilde{\mathcal{L}}\}$.
\end{proposition}

\begin{proof}
    See Appendix~\ref{app:pfs1}.
\end{proof}

\subsection{Morphisms in the radical}\label{subsec:mor2}\

In this subsection, we continue to investigate $\huaRad_{\sg}(\cpy{\ws},\cpy{\wt})$. By Remark~\ref{lem:edge2} and Lemma~\ref{lem:dec}, the direct sum \eqref{eq:rad} becomes
\begin{equation}\label{eq:dec1}
|\huaRad_{\sg}(\cpy{\ws},\cpy{\wt})|=\bigoplus_{\begin{smallmatrix}\ell_1\in \ec{Q^{\wt}_0}\\\ell_2\in \ec{Q^{\ws}_0}\end{smallmatrix}
}\bigoplus_{(y_1,y_2)\in\mathcal{R}'(\kt_{\ell_1},\ks_{\ell_2})}\bigoplus_{\begin{smallmatrix}\xi_1\in\mdf{\kt_{\ell_1}}\\ \xi_2\in\mdf{\ks_{\ell_2}}\end{smallmatrix}}\k\ \fin{\xi_1}\circ f_{(y_1,y_2)}\circ\fout{\xi_2}.
\end{equation}

To replace the index of the direct sum decomposition~\eqref{eq:dec1}, we introduce the following notations. For any $\wmu\in\upas(\wsx)$, denote
$$\ec{Q_0^{\ws}}(\wmu)=\{\ec{s(\wmu)},\ec{t(\wmu)}\}\subset\ec{Q_0^{\ws}}.$$
For any $\wmu=\asu{\wsx}{l}{l+1}\in\upas(\wsx)$ and any $\ell\in\ec{Q_0^{\ws}}(\wmu)$, define $\omo{\ell}{\wmu}\in\ks_{\ell}$ and $\pmo{\ell}{\wmu}\in\mdf{\ks_{\ell}}$ as follows:
$$(\omo{\ell}{\wmu},\pmo{\ell}{\wmu})=
\begin{cases}(\uas_{\ell},\xis_l)&\text{if $\ell=\ec{l}$,}\\(\ubs_{\ell},\xis_{l+1})&\text{if $\ell=\ec{l+1}$.}\end{cases}$$
Here, we refer to Notations~\ref{not:uni} and \ref{not:mdf} for the notations used in the above formula.

\begin{lemma}\label{lem:bi}
    For any $\ell\in\ec{Q_0^{\ws}}$, the map $(\omo{\ell}{-},\pmo{\ell}{-})$ gives a bijection from the set $\{\wmu\in\upas(\wsx)\mid \ell\in\ec{Q_0^{\ws}}(\wmu)\}$ to the set $\ks_{\ell}\times\mdf{\ks_{\ell}}$.
\end{lemma}

\begin{proof}
    For the case $|\ell|=1$, say $\ell=\{l\}$, there are the following subcases:
    \begin{itemize}
        \item for $l=1$ and $\Vs_0\in\P$, we have $\{\wmu\in\upas(\wsx)\mid \ell\in\ec{Q_0^{\ws}}(\wmu)\}=\{\asu{\wsx}{1}{2}\}$ and $\ks_{\ell}\times\mdf{\ks_{\ell}}=\{(\was_{1},\xis_1)\}$;
        \item for $l=p_{\ws}$ and $\Vs_{p_{\ws}+1}\in\P$, we have $\{\wmu\in\upas(\wsx)\mid \ell\in\ec{Q_0^{\ws}}(\wmu)\}=\{\asu{\wsx}{p_{\ws}-1}{p_{\ws}}\}$ and $\ks_{\ell}\times\mdf{\ks_{\ell}}=\{(\wbs_{p_{\ws}},\xis_{p_{\ws}})\}$;
        \item otherwise, we have $\{\wmu\in\upas(\wsx)\mid \ell\in\ec{Q_0^{\ws}}(\wmu)\}=\{\asu{\wsx}{l-1}{l},\asu{\wsx}{l}{l+1}\}$ and $\ks_{\ell}\times\mdf{\ks_{\ell}}=\{(\wbs_{l},\xis_l),(\was_{l},\xis_l)\}$.
    \end{itemize}
    For the case $|\ell|=2$, say $\ell=\{l,l+1\}$, we have $\{\wmu\in\upas(\wsx)\mid \ell\in\ec{Q_0^{\ws}}(\wmu)\}=\{\asu{\wsx}{l-1}{l},\asu{\wsx}{l+1}{l+2}\}$ and $\ks_{\ell}\times\mdf{\ks_{\ell}}=\{(\hbs_{l},\xis_l),(\has_{l+1},\xis_{l+1})\}$. In each case, it is straightforward to check that the map is a bijection.
\end{proof}
	
The bijection in Lemma~\ref{lem:bi} induces a bijection
$$(\wmu,\wnu,\ell_1,\ell_2)\mapsto(\ell_1,\ell_2,\omo{\wnu}{\ell_1},\omo{\wmu}{\ell_2},\pmo{\wnu}{\ell_1},\pmo{\wmu}{\ell_2})$$
from the set $$\{(\wmu,\wnu,\ell_1,\ell_2)\mid \wmu\in\upas(\wsx), \wnu\in\upas(\wtx), \ell_1\in \ec{Q_0^{\wt}}(\wnu), \ell_2\in \ec{Q_0^{\ws}}(\wmu), (\omo{\ell_1}{\wnu},\omo{\ell_2}{\wmu})\in\mathcal{R}'\}$$
to the set
$$\{(\ell_1,\ell_2,y_1,y_2,\xi_1,\xi_2\mid \ell_1\in \ec{Q^{\wt}_0},\ell_2\in \ec{Q^{\ws}_0},(y_1,y_2)\in\mathcal{R}'(\kt_{\ell_1},\ks_{\ell_2}),\xi_1\in\mdf{\kt_{\ell_1}}, \xi_2\in\mdf{\ks_{\ell_2}}\}.$$
	
Then the direct sum decomposition~\eqref{eq:dec1} can be rewritten as
\begin{equation}\label{eq:dec2}
    |\huaRad_{\sg}(\cpy{\ws},\cpy{\wt})|=\bigoplus_{\wmu\in\upas(\wsx), \wnu\in\upas(\wtx)}|\huaRad_{\sg}(\cpy{\ws},\cpy{\wt})_{\wmu,\wnu}|,
\end{equation}
where
\begin{equation}\label{eq:dec4}
    |\huaRad_{\sg}(\cpy{\ws},\cpy{\wt})_{\wmu,\wnu}|=\bigoplus_{\begin{smallmatrix}\ell_1\in \ec{Q_0^{\wt}}(\wnu)\\\ell_2\in \ec{Q_0^{\ws}}(\wmu)\\(\omo{\ell_1}{\wnu},\omo{\ell_2}{\wmu})\in\mathcal{R}'\end{smallmatrix}}\k\ \fin{\pmo{\ell_1}{\wnu}}\circ f_{\omo{\ell_1}{\wnu},\omo{\ell_2}{\wmu}}\circ\fout{\pmo{\ell_2}{\wmu}}.
\end{equation}

Denote $\upas(\wsx;\wtx)$ the set of pairs $(\wmu,\wnu)$ of $\wmu\in\upas(\wsx)$ and $\wnu\in\upas(\wtx)$ such that $\wmu$ and $\wnu$ are in the same unpunctured $\dac$-polygon. Note that for any $\wmu\in\upas(\wsx)$ and $\wnu\in\upas(\wtx)$, if $(\omo{\ell_1}{\wnu},\omo{\ell_2}{\wmu})\in\mathcal{R}'$ for some $\ell_1\in \ec{Q_0^{\wt}}(\wnu)$ and $\ell_2\in \ec{Q_0^{\ws}}(\wmu)$, then $(\wmu,\wnu)\in\upas(\wsx;\wtx)$. So the decomposition~\eqref{eq:dec2} can be simplified as
\begin{equation}\label{eq:dec3}
    |\huaRad_{\sg}(\cpy{\ws},\cpy{\wt})|=\bigoplus_{(\wmu,\wnu)\in \upas(\wsx;\wtx)}|\huaRad_{\sg}(\cpy{\ws},\cpy{\wt})_{\wmu,\wnu}|.
\end{equation}

\begin{notations}\label{not:munu}
    For any $\varphi\in|\huaRad_{\sg}(\cpy{\ws},\cpy{\wt})|$, denote by $\varphi_{\wmu,\wnu}$ the component of $\varphi$ in $|\huaRad_{\sg}(\cpy{\ws},\cpy{\wt})_{\wmu,\wnu}|$ with respect to the decomposition~\eqref{eq:dec3}.
\end{notations}

Let $(\wmu,\wnu)\in\upas(\wsx;\wtx)$ and $f\in|\huaRad_{\sg}(\cpy{\ws},\cpy{\wt})_{\wmu,\wnu}|$. By Construction~\ref{cons:dgm}, we have $d_{\ws}=\sum_{\wmu'\in\upas(\wsx)}d_{\wmu'}$. For any $\ell_1\in \ec{Q_0^{\wt}}(\wnu)$ and $\ell_2\in\ec{Q_0^{\ws}}(\wmu)$ such that $(\omo{\ell_1}{\wnu},\omo{\ell_2}{\wmu})\in\mathcal{R}'$ and any $\wmu'\in\upas(\wsx)$ such that $d_{\wmu'}$ is to $\cply{\ws}{\ell_2}$, by Lemma~\ref{lem:comp1}, we have
$$\fin{\pmo{\ell_1}{\wnu}}\circ f_{\omo{\ell_1}{\wnu},\omo{\ell_2}{\wmu}}\circ\fout{\pmo{\ell_2}{\wmu}}\circ d_{\wmu'}=\begin{cases}
\fin{\pmo{\ell_1}{\wnu}}\circ f_{\omo{\ell_1}{\wnu},\omo{\ell_3}{\wmu}}\circ\fout{\pmo{\ell_3}{\wmu}}&\text{if $\wmu'=\wmu$,}\\
0&\text{otherwise,}
\end{cases}$$
where $\ell_3$ is the other element in $\ec{Q_0^{\ws}}(\wmu)$ from $\ell_2$. Hence we have
$$f\circ d_{\ws}=f\circ d_{\wmu}=(f\circ d_{\ws})_{\wmu,\wnu}.$$
Similarly, we have $$d_{\wt}\circ f=d_{\wnu}\circ f=(d_{\wt}\circ f)_{\wmu,\wnu}.$$
So $|\huaRad_{\sg}(\cpy{\ws},\cpy{\wt})_{\wmu,\wnu}|$ is closed under taking differential and hence it, together with the restriction of differential, becomes a dg $\k$-subspace $\huaRad_{\sg}(\cpy{\ws},\cpy{\wt})_{\wmu,\wnu}$ of $\huaRad_{\sg}(\cpy{\ws},\cpy{\wt})$. In particular, we have
\[
Z\huaRad_{\sg}(\cpy{\ws},\cpy{\wt})=\bigoplus_{(\wmu,\wnu)\in\upas(\wsx;\wtx)}Z\huaRad_{\sg}(\cpy{\ws},\cpy{\wt})_{\wmu,\wnu},
\]
and
\[
B\huaRad_{\sg}(\cpy{\ws},\cpy{\wt})=\bigoplus_{(\wmu,\wnu)\in\upas(\wsx;\wtx)}B\huaRad_{\sg}(\cpy{\ws},\cpy{\wt})_{\wmu,\wnu}.
\]

We divide the set $\upas(\wsx;\wtx)$ into the following disjoint subsets.
\begin{itemize}
    \item $\upas(\wsx;\wtx;\times)$ consists of the $(\wmu,\wnu)$ in $\upas(\wsx;\wtx)$ such that $\wmu$ crosses $\wnu$ in their interiors.
    \item $\upas(\wsx;\wtx;\parallel)$ consists of the $(\wmu,\wnu)$ in $\upas(\wsx;\wtx)$ such that $\wmu$ does not cross $\wnu$ and any endpoint of $\wmu$ is not in the same edge as any endpoint of $\wnu$.
    \item $\upas(\wsx;\wtx;\overset{\M}{\wedge})$ consists of the $(\wmu,\wnu)$ in $\upas(\wsx;\wtx)$ such that with certain orientations, $\wmu$ and $\wnu$ start at the same marked point in $\M$ and $\wmu\nsim\wnu$.
    \item $\upas(\wsx;\wtx;\prod)$ consists of the $(\wmu,\wnu)$ in $\upas(\wsx;\wtx)$ such that with certain orientations, $\wmu$ and $\wnu$ start at the same edge which is not a boundary segment.
\end{itemize}

\begin{lemma}\label{lem:rad1}
	Let $(\wmu,\wnu)\in\upas(\wsx;\wtx)$. Assume that one of the following holds.
	\begin{itemize}
	    \item[(1)] $(\wmu,\wnu)\in\upas(\wsx;\wtx;\parallel)$.
	    \item[(2)] $(\wmu,\wnu)\in\upas(\wsx;\wtx;\prod)$ and with the orientations such that $\wmu$ and $\wnu$ start at the same edge which is not a boundary segment, $\wmu\nsim\wnu$ and $\wmu$ is to the left of $\wnu$.
	    \item[(3)] $(\wmu,\wnu)\in\upas(\wsx;\wtx;\overset{\M}{\wedge})$ and with the orientations such that $\wmu$ and $\wnu$ start at the same point in $\M$, $\wmu$ is to the right of $\wnu$.
	\end{itemize}
	Then $Z\huaRad_{\sg}(\cpy{\ws},\cpy{\wt})_{\wmu,\wnu}=B\huaRad_{\sg}(\cpy{\ws},\cpy{\wt})_{\wmu,\wnu}$.
\end{lemma}

\begin{proof}
    See Appendix~\ref{app:pfs2}.
\end{proof}

Let $D^\circ(\ws,\wt)$ be the set of unpunctured sides of vertices in $Q_0^{\ws,\wt}$. Then we have that $\upas(\wsx;\wtx;\prod)$ is the disjoint union of $D^\circ(\ws,\wt[\rho])$, $\rho\in\ZZ$.

We also consider the following subsets.
\begin{itemize}
    \item $\upas(\wsx;\wtx;\times;0)$ is the subset of $\upas(\wsx;\wtx;\times)$ consisting of those $(\wmu,\wnu)$ such that the intersection index from $\wmu$ to $\wnu$ is 0.
    \item $\upas(\wsx;\wtx;\overset{\M}{\wedge};0)$ is the subset of $\upas(\wsx;\wtx;\overset{\M}{\wedge})$ consisting of those $(\wmu,\wnu)$ such that with the orientations such that $\wmu$ and $\wnu$ start at the same point in $\M$, $\wmu$ is to the left of $\wnu$, and the intersection index from $\wmu$ to $\wnu$ is 0.
\end{itemize}
We set
$$\upas(\wsx;\wtx;0):=\upas(\wsx;\wtx;\times;0)\cup\upas(\wsx;\wtx;\overset{\M}{\wedge};0)\cup D^\circ(\ws,\wt[-1]).$$

\begin{construction}\label{cons:rad}
    For any $(\wmu,\wnu)\in\upas(\wsx;\wtx;0)$, we construct a morphism $\psi(\wmu,\wnu)\in\huaRad_{\sg}(\cpy{\ws},\cpy{\wt})_{\wmu,\wnu}$ as follows.
    \begin{enumerate}
        \item In the case that $(\wmu,\wnu)\in\upas(\wsx;\wtx;\times;0)$, write $\wmu=\asu{\wsx}{u}{u'}=(i,j_1)-(i,j_2)$ and $\wnu=\asu{\wtx}{v}{v'}=(i,j_3)-(i,j_4)$ with $j_1<j_2$ and $j_3<j_4$. By Assumption~\ref{item:T2}, any two of $j_1,j_2,j_3,j_4$ are not the same. So there are the following subcases.
        \begin{enumerate}
            \item $j_1<j_3<j_2<j_4$, see the first picture in Figure~\ref{fig:rad2}. Define
            $$\psi(\wmu,\wnu)=\fin{\xit_{v}}\circ f_{\omo{\wnu}{v},\omo{\wmu}{u'}}\circ\fout{\xis_{u'}}.$$
            \item $0=j_3<j_1<j_4<j_2$, see the second picture in Figure~\ref{fig:rad2}. Define
            $$\psi(\wmu,\wnu)=\fin{\xit_{v'}}\circ f_{\omo{\wnu}{v'},\omo{\wmu}{u'}}\circ\fout{\xis_{u'}}.$$
            \item $0<j_3<j_1<j_4<j_2$, see the third picture in Figure~\ref{fig:rad2}. Define
            $$\psi(\wmu,\wnu)=\fin{\xit_{v'}}\circ f_{\omo{\wnu}{v'},\omo{\wmu}{u'}}\circ\fout{\xis_{u'}}+\fin{\xit_{v}}\circ f_{\omo{\wnu}{v},\omo{\wmu}{u}}\circ\fout{\xis_{u}}.$$
        \end{enumerate}
        \item In the case that $(\wmu,\wnu)\in\upas(\wsx;\wtx;\overset{\M}{\wedge};0)$, write $\wmu=\asu{\wsx}{u}{u'}=(i,0)-(i,j_1)$ and $\wnu=\asu{\wtx}{v}{v'}=(i,0)-(i,j_2)$, with $j_2<j_1$. Define
        $$\psi(\wmu,\wnu)=\fin{\xit_{v'}}\circ f_{\omo{\wnu}{v'},\omo{\wmu}{u'}}\circ\fout{\xis_{u'}}.$$
        \item In the case that $(\wmu,\wnu)\in D^\circ(\ws,\wt[-1])$, write $\wmu=\asu{\wsx}{u}{u'}=(i,j)-(i,j_1)$ and $\wnu=\asu{\wtx}{v}{v'}=(i,j)-(i,j_2)$ with $j\neq 0$. Define
        $$\psi(\wmu,\wnu)=\fin{\xit_{v'}}\circ f_{\omo{\wnu}{v'},\omo{\wmu}{u}}\circ\fout{\xis_{u}}+\fin{\xit_{v}}\circ f_{\omo{\wnu}{v},\omo{\wmu}{u'}}\circ\fout{\xis_{u'}},$$
        where $f_{y_1,y_2}$ is taken to be zero if $y_1=(i,j_1^{\kappa_1})$ and $y_2=(i,j_2^{\kappa_2})$ with $j_1\geq j_2$. In particular, $\psi(\wmu,\wnu)$ is possibly zero, e.g. when $j_1=j_2=0$.
    \end{enumerate}
\end{construction}

For any $(\wmu,\wnu)\in\upas(\wsx;\wtx;0)$, we denote
$$\overline{\psi}(\wmu,\wnu)=\psi(\wmu,\wnu)+B^0\huaRad_{\sg}(\cpy{\ws},\cpy{\wt})_{\wmu,\wnu}\in\frac{Z^0\huaRad_{\sg}(\cpy{\ws},\cpy{\wt})_{\wmu,\wnu}}{B^0\huaRad_{\sg}(\cpy{\ws},\cpy{\wt})_{\wmu,\wnu}}.$$

\begin{lemma}\label{lem:rad2}
	For any $(\wmu,\wnu)\in\upas(\wsx;\wtx)$, the space
	$$\frac{Z^0\huaRad_{\sg}(\cpy{\ws},\cpy{\wt})_{\wmu,\wnu}}{B^0\huaRad_{\sg}(\cpy{\ws},\cpy{\wt})_{\wmu,\wnu}}=\begin{cases}
	\k\overline{\psi}(\wmu,\wnu) & \text{if $(\wmu,\wnu)\in\upas(\wsx;\wtx;0)$,}\\
	0 & \text{otherwise.}
	\end{cases}$$
	Moreover, in case $(\wmu,\wnu)\in\upas(\wsx;\wtx;0)$, $\overline{\psi}(\wmu,\wnu)=0$ if and only if $(\wmu,\wnu)\in D^\circ(\ws,\wt[-1])$ such that the condition (2) in Lemma~\ref{lem:rad1} holds.
\end{lemma}

\begin{proof}
    See Appendix~\ref{app:pfs3}.
\end{proof}

There is a natural basis of a complement of  $|\huaRad^{-1}_{\sg}(\cpy{\ws},\cpy{\wt})|$ in $|\huaHom^{-1}_{\sg}(\cpy{\ws},\cpy{\wt})|$:
$$\iota(u,v):=\id:\wks_u\sg[\wsi_u]\to\wkt_u\sg[\wsi_v]$$
for all $(u,v)\in Q_0^{\ws,\wt[-1]}$ satisfying $\wks_u=\wkt_v$. For any such $(u,v)$, we define a number $c(u,v)$ (compared with Notation~\ref{Not:c}) as follows.
\begin{itemize}
    \item If $\hks_u=\hkt_v=(i,j^+)$, let $\wmu$ and $\wnu$ be the punctured arc segments of $\wsx$ and $\wtx$ with $\Vs_{u}$ and $\Vt_{v}$ as endpoints respectively. Then
    $$c(u,v)=\begin{cases}
    0 & \text{if $\wmu$ is non-interior and $\wnu$ is interior,}\\
    1 & \text{otherwise.}
    \end{cases}$$
    \item If $\hks_u=\hkt_v=(i,j^-)$, let $\wmu$ and $\wnu$ be the punctured arc segments of $\wsx$ and $\wtx$ with $\Vs_{u}$ and $\Vt_{v}$ as endpoints respectively. Then
    $$c(u,v)=\begin{cases}
    0 & \text{if $\wmu$ is interior and $\wnu$ is non-interior,}\\
    -1 & \text{otherwise.}
    \end{cases}$$
    \item Otherwise, $c(u,v)=1$.
\end{itemize}

\begin{lemma}\label{lem:homotopy}
    Let $(u,v)\in Q_0^{\ws,\wt[-1]}$ satisfying $\wks_u=\wkt_v$. Then
    $$d(\iota(u,v))=\psi(\wmu,\wnu)+c(u,v)\psi(\wmu',\wnu'),$$
    where $(\wmu,\wnu)$ is the unpunctured side of $(u,v)$, and
    \begin{itemize}
        \item $(\wmu',\wnu')$ is the unpunctured side of the starting of the arrow in $\qst_{2,-}$ whose terminal is $(u,v)$, if $\hks_u=\hkt_v$ is an edge of a once-punctured monogon and any arc segment in the punctured side of $(u,v)$ is interior (which implies that $(u,v)$ is the terminal of an arrow in $\qst_{2,-}$), or
        \item $(\wmu',\wnu')$ is the other unpunctured side (if exists) of a vertex in $\ec{u}\times\ec{v}$ otherwise.
    \end{itemize}
\end{lemma}

\begin{proof}
    See Appendix~\ref{app:pfs4}.
\end{proof}

By Lemma~\ref{lem:homotopy}, for any $(u,v)\in Q_0^{\ws,\wt[-1]}$ such that $\wks_u=\wkt_v$, we have $$d(\iota(u,v))\in\bigoplus_{(\wmu,\wnu)\in D^\circ(\ws,\wt[-1])}\huaRad_{\sg}(\cpy{\ws},\cpy{\wt})_{(\wmu,\wnu)}.$$
Hence we have
$$
   \operatorname{Rad}_{\per\sg}(\cpy{\ws},\cpy{\wt})=\frac{Z^0\huaRad_{\sg}(\cpy{\ws},\cpy{\wt})}{B^0\huaHom_{\sg}(\cpy{\ws},\cpy{\wt})}=W_1\oplus W_2,$$
where
$$W_1=\frac{\bigoplus_{(\wmu,\wnu)\in D^\circ(\ws,\wt[-1])}Z^0\huaRad_{\sg}(\cpy{\ws},\cpy{\wt})_{(\wmu,\wnu)}}{\k\{d(\iota(u,v))\mid (u,v)\in Q_0^{\ws,\wt[-1]},\wks_u=\wkt_v\}+\bigoplus_{(\wmu,\wnu)\in D^\circ(\ws,\wt[-1])}B^0\huaRad_{\sg}(\cpy{\ws},\cpy{\wt})_{(\wmu,\wnu)}}$$
and
$$W_2=\bigoplus_{(\wmu,\wnu)\in\upas(\wsx;\wtx)\setminus D^\circ(\ws,\wt[-1])}\frac{Z^0\huaRad_{\sg}(\cpy{\ws},\cpy{\wt})_{\wmu,\wnu}}{B^0\huaRad_{\sg}(\cpy{\ws},\cpy{\wt})_{\wmu,\wnu}}.$$

\begin{lemma}\label{lem:W2}
    The space $W_2$ admits a basis $$\{\psi(\wmu,\wnu)\mid   (\wmu,\wnu)\in\upas(\wsx;\wtx;\times;0)\cup\upas(\wsx;\wtx;\overset{\M}{\wedge};0)\}.$$
\end{lemma}

\begin{proof}
    This follows directly from Lemma~\ref{lem:rad2}.
\end{proof}

Since there is a bijection between $I_0$ (see Construction~\ref{cons:TOI}) and $\upas(\wsx;\wtx;\times;0)\cup\upas(\wsx;\wtx;\overset{\M}{\wedge};0)$, by Lemma~\ref{lem:W2}, we get a basis of $W_2$ indexed by $I_0$. The left what we need to give is a basis of $W_1$ indexed by $I_2$.

Let $\mathcal{C}(\ws,\wt[-1])$ be the set of connected components of the bi-quiver $Q^{\ws,\wt[-1]}$. For any $C\in\mathcal{C}(\ws,\wt[-1])$, denote by $D^\circ(C)$ the set of unpunctured sides of vertices in $C$. Then $D^\circ(\ws,\wt[-1])$ is the disjoint union of $D^\circ(C)$, $C\in\mathcal{C}(\ws,\wt[-1])$.

For any $C\in\mathcal{C}(\ws,\wt[-1])$, by Lemma~\ref{lem:homotopy}, for any vertex $(u,v)$ of $C$ such that $\wks_u=\wkt_v$, we have $$d(\iota(u,v))\in\bigoplus_{(\wmu,\wnu)\in D^\circ(C)}\huaRad_{\sg}(\cpy{\ws},\cpy{\wt})_{(\wmu,\wnu)}.$$
Hence we have $W_1=\bigoplus_{C\in\mathcal{C}(\ws,\wt[-1])}W_1(C)$, where
$$W_1(C)=\frac{\bigoplus_{(\wmu,\wnu)\in D^\circ(C)}Z^0\huaRad_{\sg}(\cpy{\ws},\cpy{\wt})_{(\wmu,\wnu)}}{\k\{d(\iota(u,v))\mid (u,v)\in C_0,\wks_u=\wkt_v\}+\bigoplus_{(\wmu,\wnu)\in D^\circ(C)}B^0\huaRad_{\sg}(\cpy{\ws},\cpy{\wt})_{(\wmu,\wnu)}}.$$

Recall from the construction (Construction~\ref{cons:biq}) of $Q^{\ws,\wt[-1]}$ that each vertex $(u,v)$ has at most one punctured side, and if it has one punctured splitting side, then it is the start or terminal of an arrow in $Q^{\ws,\wt[-1]}_2$. Hence by Lemma~\ref{lem:homotopy}, we have the following description.

\begin{lemma}\label{lem:description}
    For any $C\in\mathcal{C}(\ws,\wt[-1])$, the set $\{d(\iota(u,v))\mid (u,v)\in C_0,\wks_u=\wkt_v\}$ consists of the following.
    \begin{itemize}
        \item $\psi(\wmu,\wnu)+\psi(\wmu',\wnu')$, for any vertex $(u,v)$ of $C$ which has two unpunctured sides $(\wmu,\wnu)$ and $(\wmu',\wnu')$,
        \item $\psi(\wmu,\wnu)$, for any vertex $(u,v)$ of $C$  which has one unpunctured side $(\wmu,\wnu)$ and one non-splitting punctured side $(\wmu''',\wnu''')$ such that $\wmu'''$ and $\wnu'''$ are non-interior and with the same tagging,
        \item $\psi(\wmu',\wnu')-\psi(\wmu,\wnu)$ and $\psi(\wmu'',\wnu'')+\psi(\wmu,\wnu)$, for any twin arrows $\alpha=(\wmu''',\wnu''')_{-},\beta=(\wmu''',\wnu''')_{+}\in Q^{\ws,\wt[-1]}_{2,-}$, where $(\wmu,\wnu)$ is the unpunctured side of the common start of $\alpha$ and $\beta$, and $(\wmu',\wnu')$ (resp. $(\wmu'',\wnu'')$) is the unpunctured side of the terminal of $\alpha$ (resp. $\beta$),
        \item $\psi(\wmu,\wnu)$, for any arrow $\alpha=(\wmu''',\wnu''')_{\ominus}\in Q^{\ws,\wt[-1]}_{2,o}$ with the tagging at $\wmu'''$ being $+$, where $(\wmu,\wnu)$ is the unpunctured side of the start of $\alpha$,
        \item $\psi(\wmu,\wnu)+\psi(\wmu',\wnu')$, for any arrow $\alpha=(\wmu''',\wnu''')_{\oplus}\in Q^{\ws,\wt[-1]}_{2,o}$ with the tagging at $\wnu'''$ being $+$, where $(\wmu,\wnu)$ and $(\wmu',\wnu')$ are the unpunctured sides of the start and the terminal of $\alpha$, respectively,
        \item $-\psi(\wmu,\wnu)+\psi(\wmu',\wnu')$, for any arrow $\alpha=(\wmu''',\wnu''')_{\ominus}\in Q^{\ws,\wt[-1]}_{2,o}$ with the tagging at $\wmu'''$ being $-$, where $(\wmu,\wnu)$ and $(\wmu',\wnu')$ are the unpunctured sides of the start and the terminal of $\alpha$, respectively,
        \item $\psi(\wmu,\wnu)$, for any arrow $\alpha=(\wmu''',\wnu''')_{\oplus}\in Q^{\ws,\wt[-1]}_{2,o}$ with the tagging at $\wnu'''$ being $-$, where $(\wmu,\wnu)$ is the unpunctured side of the start of $\alpha$.
    \end{itemize}
\end{lemma}

For any line $L\in\mathcal{L}(\ws,\wt[-1])$, we denote by $D^\circ(L)$ the set of unpunctured sides of vertices of $L$. We shall need the following dual notion of real $h$-line.

\begin{definition}
    A line $L$ in $\mathcal{L}(\ws,\wt[-1])$ is called an $r$-line provided that for each splitting side $(\wmu,\wnu)\in D^\circ(L)$, $\wmu$ is to the right of $\wnu$. In other words, each splitting side of an outlet $(u,v)$ is either a 0-side, or gives rise to an arrow in $Q^{\ws,\wt[-1]}_{2,o}$ whose terminal is $(u,v)$ (cf. Remark~\ref{rmk:hline}). An $r$-line is called \emph{tagged} provided that for each ending but not splitting side $(\wmu,\wnu)$, $(\wmu,\wnu)$ is a loop in $Q^{\ws,\wt[-1]}_1$ and the taggings of $\wsx$ and $\wtx$ at $\wmu$ and $\wnu$ respectively are different.
\end{definition}

\begin{remark}\label{rmk:rline}
The quiver $(Q_0^{\ws,\wt[-1]},Q_1^{\ws,\wt[-1]})$ is the same as the quiver $(Q_0^{\wt[-1],\ws},Q_1^{\wt[-1],\ws})$. So $\mathcal{L}(\ws,\wt[-1])=\mathcal{L}(\wt[-1],\ws)$. Then the $r$-lines in $\mathcal{L}(\ws,\wt[-1])$ are exactly the real $h$-lines in $\mathcal{L}(\wt[-1],\ws)$. So by Remark~\ref{rmk:hline}, there is a natural bijection between the set of tagged $r$-lines and the set $I_2$ (see Construction~\ref{cons:TOI}).
\end{remark}

\begin{remark}\label{rmk:noth}
    Let $L$ be a tagged $r$-line in $\mathcal{L}(\ws,\wt[-1])$. Then $L$ is not an $h$-line unless it is of type $\widetilde{D}$. This is because if $L$ is of type $A$ (resp. $D$), then each ending side of $L$ which does not give rise to a loop is splitting. Thus, there are two (resp. one) splitting sides of $L$. By Proposition~\ref{prop:G}~(2), there is at most one (zero) arrow in $Q^{\ws,\wt[-1]}_{2,o}$ whose terminal is an outlet of $L$. So there is a splitting side of $L$ which is a $0$-side. This implies that $L$ is not an $h$-line.

    In particular, for any tagged $r$-line $L$ in $\mathcal{L}(\ws,\wt[-1])$ and any line $L\in\mathcal{L}(\ws,\wt[-1])$ such that $L'>L$, by Proposition~\ref{prop:G}, $L'$ is not an $h$-line.
\end{remark}

In a partially ordered set, an element is called \emph{minimal} provided that any element is not smaller than it, or called \emph{least} if it is smaller than any other element.

By definition, there is no arrow in $Q^{\ws,\wt[-1]}_2$ starting from a vertex in an $r$-line. So any $r$-line is a minimal element in $\mathcal{L}(\ws,\wt[-1])$. Let $\mathcal{L}(C)$ be the set of lines in $C$.

\begin{lemma}\label{lem:r-line}
    For any $C\in\mathcal{C}(\ws,\wt[-1])$, if there is a tagged $r$-line $L$ in $\mathcal{L}(C)$, then $L$ is the least element in $\mathcal{L}(C)$. In particular, there is at most one tagged $r$-line in $\mathcal{L}(C)$.
\end{lemma}

\begin{proof}
    Let $L$ be a tagged $r$-line. Let $T$ be the subset of $\mathcal{L}(C)$ consisting of the lines equal to or bigger than $L$. Suppose conversely that $T\neq \mathcal{L}(C)$. Since $C$ is connected, there are $L'\in\mathcal{L}(C)\setminus T$ and $L''\in T$ such that there is an arrow $\alpha$ in $Q^{\ws,\wt[-1]}_2$ from $L''$ to $L'$. But there is another arrow $\beta$ in $Q^{\ws,\wt[-1]}_2$ from $L''$ to a line bigger than $L$. Since $L'\neq L$, we have that $\alpha$ and $\beta$ have different starts. Thus, $L''$ has two endpoints, each of which is the start of an arrow in $Q^{\ws,\wt[-1]}_2$, and hence it is a real $h$-line (see Remark~\ref{rmk:hline} and cf. Figure~\ref{fig:Q2}), a contradiction with Remark~\ref{rmk:noth}.
\end{proof}

\begin{lemma}\label{lem:rline2}
    Let $C\in\mathcal{C}(\ws,\wt[-1])$ containing a tagged $r$-line $L$ and $L'\in \mathcal{L}(C)$. Then $L'$ contains exactly one endpoint which is the start of an arrow in $Q^{\ws,\wt[-1]}_2$, unless $L'=L$. Moreover, for any endpoint of $L'$ which is neither the start nor the terminal of an arrow in $Q^{\ws,\wt[-1]}_2$, its ending side $(\wmu,\wnu)$ is either an unpunctured splitting side with $\wmu$ to the right of $\wnu$, or a punctured non-splitting side with $\wmu,\wnu$ non-interior and having different taggings.
\end{lemma}

\begin{proof}
    For the case $L'=L$, the lemma follows from Lemma~\ref{lem:r-line} and Remark~\ref{rmk:rline}.

    For the case $L'>L$, by Remark~\ref{rmk:noth}, $L'$ is not an $h$-line. By the definition of order, $L'$ has an endpoint which is the start of an arrow in $Q^{\ws,\wt[-1]}_2$. If $L'$ contains two endpoints which are the start of arrows in $Q^{\ws,\wt[-1]}_2$, by Remark~\ref{rmk:hline} (cf. also Figure~\ref{fig:Q2}), $L'$ is a real $h$-line, a contradiction. Let $(u,v)$ be an endpoint of $L'$ different from the one which is the start of an arrow in $Q^{\ws,\wt[-1]}_2$. Assume $(u,v)$ is not the terminal of an arrow in $Q^{\ws,\wt[-1]}_2$. Let $(\wmu,\wnu)$ be the ending side of $(u,v)$. If $(\wmu,\wnu)$ is non-splitting, since $L'$ is a tagged $r$-line, $\wmu$ and $\wnu$ are non-interior punctured and have different taggings. If $(\wmu,\wnu)$ is splitting, since $L'$ is not a real $h$-line, $\wmu$ is to the right of $\wnu$. Then $(\wmu,\wnu)$ is unpunctured, because otherwise, it gives rise to an arrow in $Q^{\ws,\wt[-1]}_2$, which makes $(u,v)$ be the start or the terminal of that arrow, a contradiction.
\end{proof}

\begin{proposition}
    Let $C\in\mathcal{C}(\ws,\wt[-1])$ and $(\wmu_0,\wnu_0)\in D^\circ(C)$. If $C$ contains a tagged $r$-line, then $\psi(\wmu_0,\wnu_0)$ is a basis of $W_1(C)$; otherwise, $\dim W_1(C)=0$.
\end{proposition}

\begin{proof}
    By Lemma~\ref{lem:rad2}, $W_1(C)$ is linearly generated by $\overline{\psi}(\wmu,\wnu)$, $(\wmu,\wnu)\in D^\circ(C)$. By Lemma~\ref{lem:description}, for any lines $L>L'\in\mathcal{L}(\ws,\wt[-1])$ and for any $(\wmu,\wnu)\in D^\circ(L)$ and $(\wmu',\wnu')\in D^\circ(L')$, we have $\overline{\psi}(\wmu',\wnu')=c\overline{\psi}(\wmu,\wnu)$ in $W_1$ for some $c\in\k$, where $c\neq0$ if $L=L'$. So we only need to consider the minimal elements in $\mathcal{L}(C)$.

    If $C$ does not contain any tagged $r$-line, then for any minimal element $L$ in $\mathcal{L}(C)$, at least one of the following occurs.
    \begin{enumerate}
        \item $L$ has an endpoint $(u,v)$ with an ending side whose arc segments are non-interior punctured with the same tagging. In this case, by Lemma~\ref{lem:homotopy}, we have $\psi(\wmu,\wnu)=d(\iota(u,v))$, where $(\wmu,\wnu)$ is the unpunctured side of $(u,v)$.
        \item $L$ has an endpoint $(u,v)$ with an ending side $(\wmu,\wnu)$ with $\wmu,\wnu$ non-interior unpunctured. In this case, by Construction~\ref{cons:rad}, we have $\psi(\wmu,\wnu)=0$.
        \item If $L$ is not an $r$-line, then it has a splitting side $(\wmu,\wnu)$ with $\wmu$ left of $\wnu$. In this case, by Lemma~\ref{lem:rad1}~(2), $\psi(\wmu,\wnu)\in B\huaRad_{\sg}(\cpy{\ws},\cpy{\wt})_{\wmu,\wnu}$.
    \end{enumerate}
    Hence, $\dim W_1(C)=0$.

    If $C$ contains a tagged $r$-line, say $L$, then by Lemma~\ref{lem:r-line}, any other line in $\mathcal{L}(C)$ is bigger than $L$. Let $B$ be the subset of $Q^{\ws,\wt[-1]}_{2,o}$ consisting of arrows $(\wmu',\wnu')$ with either $\wmu'$ is non-interior with tagging $+$ or $\wnu'$ is non-interior with tagging $-$. Let $C'$ be the subset of $\mathcal{L}(C)$ consisting of the lines $L'$ such that there is no sequence in $\mathcal{L}(C)$:
    $$L'>L_1>L_2>\cdots>L_s>L,$$
    satisfying that there exists $i$ such that $L_i>L_{i+1}$ is given by an arrow in $B$. By Lemma~\ref{lem:rline2}, for any $L'\in C'$ and any sequence $L'>L_1>L_2>\cdots>L_s>L$ with $L_i>L_{i+1}$ given by an arrow $\alpha_i$ in  $Q^{\ws,\wt[-1]}_{2}$ for any $i$, we have $\alpha_i\notin B$. Thus, we have
    $$M=\left(M\cap\k\{\psi(\wmu,\wnu)\mid (\wmu,\wnu)\in D^\circ(C')\}\right)\cup\left(M\cap\k\{\psi(\wmu,\wnu)\mid (\wmu,\wnu)\in D^\circ(C\setminus C')\}\right)$$
    where $M=\{d(\iota(u,v))\mid (u,v)\in C_0,\wks_u=\wkt_v\}$. By Lemma~\ref{lem:rline2}, for each non-splitting punctured side $(\wmu,\wnu)$ of $C'$ with $\wmu,\wnu$ non-interior, $\wmu$ and $\wnu$ have different taggings. So by Lemma~\ref{lem:description}, each term in $M\cap\k\{\psi(\wmu,\wnu)\mid (\wmu,\wnu)\in D^\circ(C')\}$ is of the form $\psi(\wmu,\wnu)\pm\psi(\wmu',\wnu')$. It follows that the number of elements in $M\cap\{\psi(\wmu,\wnu)\mid (\wmu,\wnu)\in D^\circ(C')\}$ is less than that of $\{(\wmu,\wnu)\in D^\circ(C')\}$. By Lemma~\ref{lem:rline2} again, each unpunctured ending side of $C'$ is splitting and does not satisfy the condition (2) in Lemma~\ref{lem:rad1}. So by Lemma~\ref{lem:rad2}, each $\psi(\wmu,\wnu)$ with $(\wmu,\wnu)\in D^\circ(C\setminus C')$ is not zero in $\frac{Z^0\huaRad_{\sg}(\cpy{\ws},\cpy{\wt})_{\wmu,\wnu}}{B^0\huaRad_{\sg}(\cpy{\ws},\cpy{\wt})_{\wmu,\wnu}}$. Hence $\psi(\wmu_0,\wnu_0)$ is a basis of $W_1(C)$.
\end{proof}

\subsection{Cones of morphisms}

In this subsection, we describe the cones of certain morphisms.

\begin{definition}\label{def:ext}
Let $\theta$ be a clockwise angle at a TOI $\ints$ from $\wsx$ to $\wtx$ of type I or II (see Definition~\ref{def:int} for the definition of types of TOI) and of index 1. Define the \emph{extension} $\wtx\wedge_\theta\wsx$ of $\wsx$ and $\wtx$ \emph{along} $\theta$ as in Figure~\ref{fig:ex}, whose grading and tagging inherit from those of $\wsx$ and $\wtx$. When $\ints$ is of type II (i.e. $\ints\in\M$), since there is exactly one oriented angle at $\ints$, we also denote $\wtx\wedge_\ints\wsx=\wtx\wedge_{\theta}\wsx$.
\begin{figure}[htpb]
	\begin{tikzpicture}[scale=1.5]
	\draw[blue,thick,dashed] (-.5,.5)to (0,0) to (.5,.5);
	\draw[blue,thick] (0,0)node[above]{$\ints$}..controls +(-145:1) and +(0:1) ..(-2,-1);
	\draw[blue,thick] (0,0)..controls +(-35:1) and +(180:1) ..(2,-1);
	\draw[blue] (-1,-.5)node{$\wtx$} (1,-.5)node{$\wsx$};
	\draw[orange,bend right,-<-=.5,>=stealth] (-.4,-.3)tonode[above]{$\theta$}(.4,-.3);
	\draw[blue,thick] (2,-1.1)..controls +(180:1) and +(0:.5)..(0,-.5)..controls +(180:.5) and +(0:1)..(-2,-1.1);
	\draw(0,-0.6)node[below,blue, thick]{$\wtx\wedge_\theta\wsx$};
	\end{tikzpicture}
	\caption{Extension of two tagged arcs along a clockwise angle at a TOI}\label{fig:ex}
\end{figure}
\end{definition}

\begin{example}\label{exm:typeII2}
    A TOI $\ints\in\M$ from $\wsx$ to $\wtx$ (i.e. of type II, cf. Definition~\ref{def:int}) of index $\rho$ can be realized, by choosing certain orientations, as a common starting point $\sigma(0)=\tau(0)=\ints$ such that
	\begin{itemize}
	    \item $\aso{\wsx}{0}{1}$ is to the left of $\aso{\wtx}{0}{1}$, and
	    \item if $\wsx\sim\wtx$ and end at a puncture, then $\wsx$ and $\wtx$ has the same tagging there.
	\end{itemize}
    Assume $\aso{\wsx}{0}{1}\nsim\aso{\wtx}{0}{1}$. We construct $f_{q}\in\huaRad_{\sg}(\cpy{\ws},\cpy{\wt})$ of degree $\rho$ as
	$$(f_q)_{\ell_1,\ell_2}=\begin{cases}
	\fin{\xit_1}\circ f_{\ubt_{\ec{1}},\ubs_{\ec{1}}}\circ \fout{\xis_1}&\text{if $\ell_1=\ec{1}$ and $\ell_2=\ec{1}$,}\\
	0&\text{otherwise.}
	\end{cases}.$$
    If $\rho=1$ and the extension $\wtx\wedge_\ints\wsx$ is admissible, then by the construction $\cpy{-}$ (Construction~\ref{cons:dgm}), there is a triangle
	\begin{equation}\label{eq:tri1}
	    \cpy{\wt}\xrightarrow{g}\cpy{\wtx\wedge_\ints\wsx}\xrightarrow{h}\cpy{\ws}\xrightarrow{f_{q}}\cpy{\wt}[1],
	\end{equation}
	where $g$ and $h$ are given by the canonical embedding and projection on the underlying graded modules, respectively. In particular, we have $f_q\neq 0$ in $\Hom_{\per\sg}(\cpy{\ws},\cpy{\wt})$.
\end{example}

\begin{example}\label{exm:typeII}
	We continue Example~\ref{exm:typeII2} that $\wsx$ and $\wtx$ share a common starting point $q$ of intersection index $\rho$, but consider another case $\aso{\wsx}{0}{1}\sim\aso{\wtx}{0}{1}$. By Remark~\ref{rmk:hline}, there is an associated real $h$-line $L$ in $\mathcal{L}(\ws,\wt)$ containing $(1,1)$ as a vertex. Assume moreover that after a common part from $q$, if separate, $\wsx$ and $\wtx$ separate in an unpunctured $\dac$-polygon. Then $L$ is the unique maximal extension of itself (cf. also Figure~\ref{fig:Q2}). In this case, we denote $f_q=f_L$. If $\wsx\sim\wtx$ and $\rho=0$, we have that $f_q=\id$ and its cone is zero. If $\wsx\nsim\wtx$, $\rho=1$, and $\wtx\wedge_\ints\wsx$ is admissible, we have a triangle of the same form as \eqref{eq:tri1}. We also note that the morphisms $g$ and $h$ in the triangle~\eqref{eq:tri1} in Example~\ref{exm:typeII2} are of this form.
\end{example}

\begin{example}\label{ex:2}
    We consider a simple case that the $L$ in Example~\ref{exm:typeII} is not a maximal extension of itself: $\wsx$ and $\wtx$ separate as in the first picture in Figure~\ref{fig:Q2}. Moreover, we assume that $(u,v')$ and $(u',v)$ are the endpoints of the line where they live. Then there are exactly two maximal extensions $\widetilde{L}_1$ and $\widetilde{L}_2$ of $L$ such that one contains $(u,v')$ while the other contains $(u',v)$. The cones of $f_{\widetilde{L}_1}$ and $f_{\widetilde{L}_2}$ are not isomorphic to each other, since one is the direct sum $\cpy{\wax_+}\oplus\cpy{\wbx_-}$ and the other is $\cpy{\wax_-}\oplus\cpy{\wbx_+}$, where $\wax_+$ denotes the tagged arc whose underlying arc is $\wa$ in the first picture in Figure~\ref{fig:sp} with tagging at the middle puncture being $+$.
	
    In this case, the cones of linear combinations of $f_{\widetilde{L}_1}$ and $f_{\widetilde{L}_2}$ are also of interesting. The cone of $f_{\widetilde{L}_1}+f_{\widetilde{L}_2}$ is the direct sum $\cpy{\wt\wedge_{\theta_1}\ws}\oplus\cpy{\wt\wedge_{\theta_2}\ws}$, see the second picture in Figure~\ref{fig:sp}.
	
    The cone of $\lambda_1f_{\widetilde{L}_1}+\lambda_2f_{\widetilde{L}_2}$ with $\lambda_1,\lambda_2\neq0$ and $\lambda_1-\lambda_2\neq 0$ is $\cpy{\we_{ad}}$, where $\wex_{ad}$ is the ``admissible version" of the (non-admissible) extension $\wex=\wtx\wedge_\ints\wsx$, see the third picture in Figure~\ref{fig:sp}.

	\begin{figure}[htpb]
		\begin{tikzpicture}[scale=.6]
			\draw (0,-1) coordinate (v);
			\draw($(v)+(90:3)$) coordinate (A) ($(v)+(210:3)$) coordinate (B)($(v)+(-30:3)$) coordinate (C)($(v)+(-90:.5)$) coordinate (O);
			\draw (v)\nn;
			\draw[ultra thick] (-1,2)to(1,2);
			\draw[orange,thick] (B) to (v) to (C);
			\draw[orange] (-1.8,-1.45)node{$\wa$} (1.8,-1.45)node{$\wb$};
			\draw[blue,thick] (A)node[above]{$M$}.. controls +(-50:4) and +(-10:4) .. (B);
			\draw[blue,thick] (C).. controls +(190:4) and +(-130:4) .. (A);
			\draw[blue] (-1,1.2)node{$\wtx$} (1.2,1.2)node{$\wsx$};
			\draw[blue] (A)\nn(B)\nn(C)\nn;
		\end{tikzpicture}\quad
		\begin{tikzpicture}[scale=.6]
			\draw (0,-1) coordinate (v);
			\draw($(v)+(90:3)$) coordinate (A) ($(v)+(210:3)$) coordinate (B)($(v)+(-30:3)$) coordinate (C)($(v)+(-90:.5)$) coordinate (O);
			\draw (v)\nn;
			\draw[ultra thick] (-1,2)to(1,2);
			\draw[orange,thick,bend left=50] (B) tonode[left]{$\wt\wedge_{\theta_1}\ws$} (A);
			\draw[orange,thick,bend right=50] (C) tonode[right]{$\wt\wedge_{\theta_2}\ws$} (A);
			\draw[orange,bend left,->] (-1,-2.6)tonode[left]{$\theta_1$}(-1,-1.9);
			\draw[orange,bend right,<-] (1,-2.6)tonode[right]{$\theta_2$}(1,-1.9);
			\draw[blue,thick] (A)node[above]{$M$}.. controls +(-50:4) and +(-10:4) .. (B);
			\draw[blue,thick] (C).. controls +(190:4) and +(-130:4) .. (A);
			\draw[blue] (-1.1,1.0)node{$\wtx$} (1.3,1.0)node{$\wsx$};
			\draw[blue] (A)\nn(B)\nn(C)\nn;
		\end{tikzpicture}\quad
		\begin{tikzpicture}[scale=.6]
			\draw (0,-1) coordinate (v);
			\draw($(v)+(90:3)$) coordinate (A) ($(v)+(210:3)$) coordinate (B)($(v)+(-30:3)$) coordinate (C)($(v)+(-90:.5)$) coordinate (O);
			\draw (v)\nn;
			\draw[ultra thick] (-1,2)to(1,2);
			\draw[green,thick] (C) .. controls +(170:2) and +(-90:1) .. (-1,-1) .. controls +(90:.5) and +(180:.5) .. (0,0) .. controls +(0:.5) and +(90:.5) .. (1,-1) .. controls +(-90:1) and +(10:2) .. (B);
			\draw[blue,thick] (A)node[above]{$M$}.. controls +(-50:4) and +(-10:4) .. (B);
			\draw[orange,thick] (B).. controls +(70:4) and +(110:4) .. (C);
			\draw[blue,thick] (C).. controls +(190:4) and +(-130:4) .. (A);
			\draw[blue] (-1,1.2)node{$\wtx$} (1.2,1.2)node{$\wsx$} (-2.55,-1)node[orange]{$\we^\times_{ad}$}(1.6,-1.9)node[green]{$\we^\times$};
			\draw[blue] (A)\nn(B)\nn(C)\nn;
		\end{tikzpicture}
		\caption{Cones of (combinations of) morphisms associated to different maximal extensions of a tagged real $h$-line}\label{fig:sp}
	\end{figure}
\end{example}

\begin{example}\label{ex:d4}
    See Figure~\ref{fig:AR-D4} for the AR-quiver of type $D_4$ using the binary model, which contains the phenomenon in Example~\ref{ex:2}.
\end{example}

\begin{figure}
\centering\makebox[\textwidth][c]{
  \includegraphics[width=17cm]{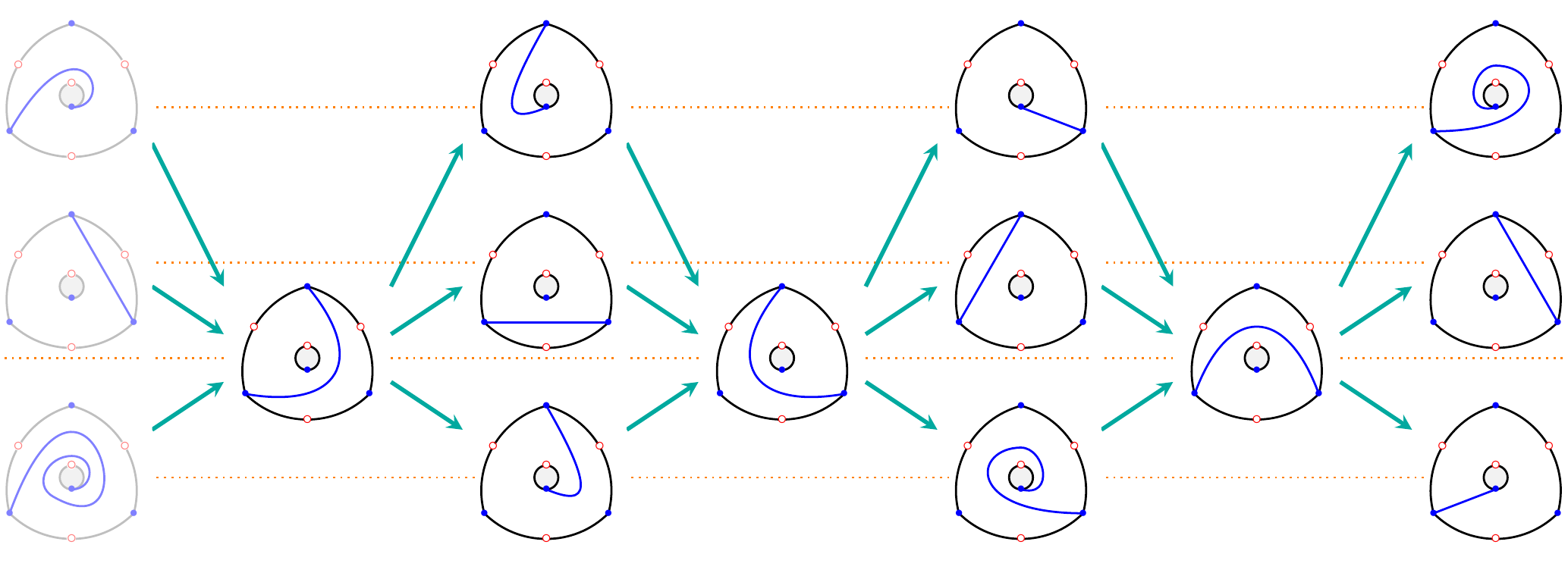}
}\caption{The Auslander-Reiten quiver of type $D_4$}
	\label{fig:AR-D4}
\end{figure}

\appendix

\section{Representations associated to period words}\label{app:period}

In this appendix, we recall from \cite[Sections~2.5 and 2.6]{Deng} the explicit construction of the representation $R(w,N)$ associated to a pair $(w,N)$ of a period word $w$ and an indecomposable $A_w$-module $N$.

\subsection{Asymmetric period words}\label{subapp:apw}

By the theory of algebraic representation, the isoclasses of indecomposable $\k[x,x^{-1}]$-modules $N$ are parameterized by the matrices
$$P(N)=\begin{pmatrix}
0&0&\cdots&0&a_{d}\\
1&0&\cdots&0&a_{d-1}\\
\vdots&\ddots&\ddots&\vdots&\vdots\\
0&0&\ddots&0&a_2\\
0&0&\cdots&1&a_1
\end{pmatrix}$$
with $x^d-a_1x^{d-1}-a_2x^{d-2}-\cdots-a_{d}\in\{p(x)^h\mid p(x)\in\k[x] \text{ irreducible unitary},p(x)\neq x,h\in\NN\}$, where $d=\dim_\k N$.

The representation $R(w,N)=(X_{w,N};f_{w,N})$ of $S$ as follows. Set
$$X_{w,N}=\widehat{w_1}^d\oplus\cdots \oplus\widehat{w_\period}^d.$$
Then each term $x$ of the sequence $w_1w_1^\ast w_2w_2^\ast\cdots w_\period w_\period^\ast$, which is in $S_{(i,j),r}^\varepsilon$, contributes a $d$-dimensional direct summand $\k\underline{x}\otimes_{\k}\k^d$ to the space $\bvs_{(i,j),r}^\varepsilon (X_{w,N})$. Up to rotation, we may assume $w_1\in S^+$.

For any $\underline{x}\in\{\underline{w_1},\underline{w_1^\ast},\underline{w_2},\underline{w_2^\ast},\cdots,\underline{w_\period},\underline{w_\period^\ast}\}$ with $x\in S^-$ and for any $v\in\k^d$, define $f_{w,N}(\underline{x}\otimes v)$ as follows.
\begin{enumerate}
\item Case $x=w_{k+1},1\leq k< \period$:
\begin{enumerate}
\item if $w_{k}^\ast\in S_{(i,j),r}^+$ for some $(i,j)\notin\fc{\OME}$ and $r\in\ZZ$,
$$f_{w,N}(\underline{x}\otimes v)=\underline{w_{k}^\ast}\otimes v,$$
\item if $w_{k}^\ast\in S_{(i,j),r}^+$ for some $(i,j)\in\fc{\OME}$ and $r\in\ZZ$,
$$f_{w,N}(\underline{x}\otimes v)=\begin{cases}
\underline{w_{k}^\ast}\otimes v&\text{if $w_k^\ast=((i,j^+),r)$,}\\
(\underline{w_{k}^\ast}+\underline{w_{k}})\otimes v&\text{if $w_k^\ast=((i,j^-),r)$.}
\end{cases}$$
\end{enumerate}
\item Case $x=w_{k-1}^\ast,1< k\leq \period$:
\begin{enumerate}
\item if $w_{k}=\in S_{(i,j),r}^+$ for some $(i,j)\notin\fc{\OME}$ and $r\in\ZZ$,
$$f_{w,N}(\underline{x}\otimes v)=\underline{w_{k}}\otimes v;$$
\item if $w_{k}\in S_{(i,j),r}^+$ for some $(i,j)\in\fc{\OME}$ and $r\in\ZZ$, $$f_{w,N}(\underline{x}\otimes v)=\begin{cases}
\underline{w_{k}}\otimes v&\text{if $w_k=((i,j^+),r)$,}\\
(\underline{w_{k}}+\underline{w_{k}^\ast})\otimes v&\text{if $w_k=((i,j^-),r)$.}
\end{cases}$$
\end{enumerate}
\item Case $x=w_{\period}^\ast$:
\begin{enumerate}
\item if $w_{1}\in S_{(i,j),r}^+$ for some $(i,j)\notin\fc{\OME}$ and $r\in\ZZ$,
$$f_{w,N}(\underline{x}\otimes v)=\underline{w_{1}}\otimes P(N)v;$$
\item if $w_{1}\in S_{(i,j),r}^+$ for some $(i,j)\in\fc{\OME}$ and $r\in\ZZ$,
$$f_{w,N}(\underline{x}\otimes v)=\begin{cases}
\underline{w_{1}}\otimes P(N)v&\text{if $w_1=((i,j^+),r)$,}\\
(\underline{w_{1}}+\underline{w_{1}^\ast})\otimes P(N)v&\text{if $w_1=((i,j^-),r)$.}
\end{cases}$$
\end{enumerate}
\end{enumerate}

\subsection{Symmetric period words}\label{subapp:spw}

By the theory of algebraic representation, the isoclasses of indecomposable $\k\langle x,y\rangle/(x^2-x,y^2-y)$-modules $N$ are parameterized by the following $\dim_\k N\times\dim_\k N$ block matrices $P(N)=\left(\begin{array}{@{}c|c@{}} A & B \\ \hline C & D \end{array}\right)$:
\begin{enumerate}
	\item $\left(\begin{array}{@{}c|c@{}}
	  \begin{matrix}
	  0 \\
	 I_q
	  \end{matrix}
	  & I_{q+1} \\
	\hline
	  I_q &
	  \begin{matrix}
	  I_q & 0
	  \end{matrix}
	\end{array}\right)$, $\left(\begin{array}{@{}c|c@{}}
	  \begin{matrix}
	I_q&0
	\end{matrix}
	  & I_{q} \\
	\hline
	  I_{q+1} &
	  \begin{matrix}
	  0\\I_q
	  \end{matrix}
	\end{array}\right)$, $\left(\begin{array}{@{}c|c@{}}
	I_{q+1} & I_{q+1}\\ \hline I_{q+1} & J_{q+1}
	\end{array}\right)$, $\left(\begin{array}{@{}c|c@{}}
	J_{q+1} & I_{q+1}\\ \hline I_{q+1} & I_{q+1}
	\end{array}\right)$,
	\item $\left(\begin{array}{@{}c|c@{}}
	I_{q+1}&\begin{matrix}
	0\\I_q
	\end{matrix}\\\hline\begin{smallmatrix}
	I_q&0
	\end{smallmatrix}&I_q
	\end{array}\right)$,
	$\left(\begin{array}{@{}c|c@{}}
	I_{q}&\begin{matrix}
	I_q&0
	\end{matrix}\\\hline\begin{smallmatrix}
	0\\I_q
	\end{smallmatrix}&I_{q+1}
	\end{array}\right)$,
	$\left(\begin{array}{@{}c|c@{}}
	I_{q+1}&I_{q+1}\\\hline J_{q+1}&I_{q+1}
	\end{array}\right)$,
	$\left(\begin{array}{@{}c|c@{}}
	I_{q+1}&J_{q+1}\\\hline I_{q+1}&I_{q+1}
	\end{array}\right)$,
	\item $\left(\begin{array}{@{}c|c@{}}
		K&I_{q+1}\\\hline I_{q+1}&I_{q+1}
	\end{array}\right)$,
\end{enumerate}
where $q\geq 0$,
$$J_{q+1}=\begin{pmatrix}0&1&0&\cdot&0&0\\
0&0&1&\cdot&0&0\\0&0&0&\cdot&0&0\\
\cdot&\cdot&\cdot&\cdot&\cdot&\cdot\\0&0&0&\cdot&0&1\\0&0&0&\cdot&0&0\end{pmatrix},\ K=\begin{pmatrix}
0&0&\cdots&0&a_{q+1}\\
1&0&\cdots&0&a_{q}\\
\vdots&\ddots&\ddots&\vdots&\vdots\\
0&0&\ddots&0&a_2\\
0&0&\cdots&1&a_1
\end{pmatrix}$$
with $x^{q+1}-a_1x^{q}-a_2x^{q-1}-\cdots-a_{q+1}\in\{p(x)^h\mid p(x)\in\k[x] \text{ irreducible unitary},p(x)\neq x,x-1,h\in\NN\}$.

The representation $R(w,N)=(X_{w,N};f_{w,N})$ of $S$ as follows. Up to rotation, we may assume
\begin{itemize}
    \item $w_1^\ast\cdots w_{\frac{\period}{2}-1}^\ast=w_{\frac{\period}{2}+1}\cdots w_{\period-1}$,
	\item $w_{\frac{\period}{2}}=((i,j^+),r)$ for some $(i,j)\in\fc{\OME}$ and $r\in\ZZ$, and
	\item $w_{\period}=((i',{j'}^+),r')$ for some $(i',j')\in\fc{\OME}$ and $r'\in\ZZ$.
\end{itemize}

Set
$$X_{w,K}=\widehat{w_1}^d\oplus\cdots\oplus \widehat{w_{\frac{\period}{2}-1}}^d\oplus \{w_{\frac{\period}{2}}\}^d\oplus\{w_{\frac{\period}{2}}^\ast\}^{d'} \oplus \widehat{w_{\frac{\period}{2}+1}}^{d'}\oplus\cdots\oplus \widehat{w_{\period-1}}^{d'}\oplus \{w_{\period}\}^{c'}\oplus\{w_{\period}^\ast\}^{c},$$
where $c$ (resp. $c'$) is the number of rows of $A$ (resp. $C$), and $d$ (resp. $d'$) is the number of columns of $A$ (resp. $B$). Then each term $x$ of the sequence $w_1w_1^\ast w_2w_2^\ast\cdots w_\period w_\period^\ast$ which is in $S^\varepsilon_{(i,j),r}$ contributes a $d(\underline{x})$-dimensional direct summand $\k\underline{x}\otimes_{\k}\k^{d(\underline{x})}$ to the space $\bvs_{(i,j),r}^\varepsilon (X_{w,N})$, where
$$d(\underline{x})=\begin{cases}
d&\text{if $\underline{x}\in\{\underline{w_1},\underline{w_1^\ast},\cdots, \underline{w_{\frac{\period}{2}-1}},\underline{w_{\frac{\period}{2}-1}^\ast}, \underline{w_{\frac{\period}{2}}\}}$,}\\
d'&\text{if $\underline{x}\in\{\underline{w_{\frac{\period}{2}}^\ast}, \underline{w_{\frac{\period}{2}+1}},\underline{w_{\frac{\period}{2}+1}^\ast},\cdots, \underline{w_{\period-1}},\underline{w_{\period-1}^\ast}\}$,}\\
c'&\text{if $\underline{x}=\underline{w_\period}$,}\\
c&\text{if $\underline{x}=\underline{w_\period^\ast}$.}
\end{cases}$$

For any $\underline{x}\in\{\underline{w_1},\underline{w_1^\ast},\underline{w_2},\underline{w_2^\ast},\cdots,\underline{w_\period},\underline{w_\period^\ast}\}$ with $x\in S^-$ and for any $v\in\k^{d(\underline{x})}$, define $f_{w,K}(\underline{x}\otimes v)$ as follows.
\begin{enumerate}
\item Case $x=w_{k+1},k\neq 0,\frac{\period}{2}$:
\begin{enumerate}
\item if $w_{k}^\ast\in S_{(i,j),r}^+$ for some $(i,j)\notin\fc{\OME}$ and $r\in\ZZ$,
$$f_{w,K}(\underline{x}\otimes v)=\underline{w_{k}^\ast}\otimes v,$$
\item if $w_{k}^\ast\in S_{(i,j),r}^+$ for some $(i,j)\in\fc{\OME}$ and $r\in\ZZ$,
$$f_{w,K}(\underline{x}\otimes v)=\begin{cases}
\underline{w_{k}^\ast}\otimes v, &\text{if $w_k^\ast=((i,j^+),r)$,}\\
(\underline{w_{k}^\ast}+\underline{w_{k}})\otimes v, &\text{if $w_k^\ast=((i,j^-),r)$.}
\end{cases}$$
\end{enumerate}
\item Case $x=w_{\frac{\period}{2}+1}$:
$$f_{w,K}(\underline{x}\otimes v)=\underline{w_{\frac{\period}{2}}^\ast}\otimes v.$$
\item Case $x=w_1$:
$$f_{w,K}(\underline{x}\otimes v)=\underline{w_{\period}^\ast}\otimes Av+\underline{w_{\period}}\otimes Cv.$$
\item Case $x=w_{k-1}^\ast,k\neq \frac{\period}{2},\period$:
\begin{enumerate}
\item if if $w_{k}\in S_{(i,j),r}^+$ for some $(i,j)\notin\fc{\OME}$ and $r\in\ZZ$, $$f_{w,K}(\underline{x}\otimes v)=\underline{w_{k}}\otimes v;$$
\item if if $w_{k}\in S_{(i,j),r}^+$ for some $(i,j)\in\fc{\OME}$ and $r\in\ZZ$, $$f_{w,K}(\underline{x}\otimes v)=\begin{cases}
\underline{w_{k}}\otimes v, &\text{if $w_k=((i,j^+),r)$,}\\
(\underline{w_{k}}+\underline{w_{k}^\ast})\otimes v, &\text{if $w_k=((i,j^-),r)$.}
\end{cases}$$
\end{enumerate}
\item Case $x=w_{\frac{\period}{2}-1}^\ast$:
$$f_{w,K}(\underline{x}\otimes v)=\underline{w_{\frac{\period}{2}}}\otimes v.$$
\item Case $x=w_{\period-1}^\ast$:
$$f_{w,K}(\underline{x}\otimes v)=\underline{w_{\period}^\ast}\otimes Bv+\underline{w_{\period}}\otimes Dv.$$
\end{enumerate}

\section{Proof of Theorem~\ref{thm:BD}}\label{app:pfBD}


We first show that $\fK$ can be defined on the class of isoclasses of objects. Let
$$(g,h):(Y_1^\bullet,V_1^\bullet,\theta_1)\to(Y_2^\bullet,V_2^\bullet,\theta_2)$$ be an isomorphism in $\glue\sg$. Since the differentials of $V_i^\bullet$, $i=1,2$, are zero, $h$ is indeed an isomorphism of graded $\overline{\sg}$-modules. Let $\widetilde{g}: Y_1^\bullet \to Y_2^\bullet$ be a morphism in $\C(H)$ whose image in $\per H$ is $g$. Since the differentials of $Y_i^\bullet$, $i=1,2$, are in the radical, $\widetilde{g}$ is an isomorphism in $\C(H)$. Since $Y_i^\bullet$ and $V_i^\bullet$, $i=1,2$, are minimal strictly perfect, we have $Y^\bullet_i\otimes^{\mathbf{L}}_H\overline{H}=Y^\bullet_i\otimes_H\overline{H}$, $V^\bullet_i\otimes^{\mathbf{L}}_{\overline{\sg}}\overline{H}=V^\bullet_i\otimes_{\overline{\sg}}\overline{H}$, $i=1,2$, $g\otimes^{\mathbf{L}}_H {\overline{H}}=\widetilde{g}\otimes_{H}\id_{\overline{H}}$ and $h\otimes^{\mathbf{L}}_{\overline{\sg}} {\overline{H}}=h\otimes_{\overline{\sg}}\id_{\overline{H}}$. Hence, the commutative diagram \eqref{eq:comm} becomes the following commutative diagram in $\C(\overline{H})$ (and hence in $\C(\sg)$)
$$\xymatrix{
	Y^\bullet_1\otimes_H\overline{H}\ar[r]^{\theta_1}\ar[d]_{\widetilde{g}\otimes_H {\overline{H}}}&V^\bullet_1\otimes_{\overline{\sg}}\overline{H}\ar[d]^{h\otimes_{\overline{\sg}}{\overline{H}}}\\
	Y_2^\bullet\otimes_H \overline{H}\ar[r]_{\theta_2}&V^\bullet_2\otimes_{\overline{\sg}}\overline{H}
    }$$
This, together with the following two commutative diagrams in $\C(\sg)$
$$
\xymatrix{
	Y_1^\bullet\ar[r]^{\iota_{H}^{\overline{H}}}\ar[d]_{\widetilde{g}}&Y_1^\bullet\otimes_{H} \overline{H}\ar[d]^{\widetilde{g}\otimes_H {\overline{H}}}\\
		Y_2^\bullet\ar[r]^{\iota_{H}^{\overline{H}}}&Y_2^\bullet\otimes_H \overline{H}
}\quad
\xymatrix{
	V_1^\bullet\ar[r]^{\iota_{\overline{\sg}}^{\overline{H}}}\ar[d]_{h}&V_1^\bullet\otimes_{\overline{\sg}} \overline{H}\ar[d]^{h\otimes_{\overline{\sg}} {\overline{H}}}\\
		V_2^\bullet\ar[r]^{\iota_{\overline{\sg}}^{\overline{H}}}&V_2^\bullet\otimes_{\overline{\sg}} \overline{H}
}$$
gives rise to the following commutative diagram
$$
    \xymatrix{
    Y_1^\bullet\oplus V_1^\bullet \ar[d]_{\left(\begin{smallmatrix}
        \widetilde{g}\\&h
    \end{smallmatrix}\right)}\ar[rr]^{\left(\begin{smallmatrix}
\iota_H^{\overline{H}}&-\widetilde{\theta_1}
\end{smallmatrix}\right)}&& Y_1^\bullet\otimes_H \overline{H}\ar[d]^{\widetilde{g}\otimes_H {\overline{H}}}\\
Y_2^\bullet\oplus V_2^\bullet\ar[rr]^{\left(\begin{smallmatrix}
\iota_H^{\overline{H}}&-\widetilde{\theta_2}
\end{smallmatrix}\right)}&& Y_2^\bullet\otimes_H \overline{H}
    }
$$
Using the universal property of kernel, this commutative diagram induces an isomorphism from $\fK(Y_1^\bullet,V_1^\bullet,\theta_1)$ to $\fK(Y_2^\bullet,V_2^\bullet,\theta_2)$.

Next, we show that for any object $(Y^\bullet,V^\bullet,\theta)$ in $\glue\sg$, we have
\begin{equation}\label{eq:fEfK2} \fE\circ\fK(Y^\bullet,V^\bullet,\theta)\cong(Y^\bullet,V^\bullet,\theta).
\end{equation}
By Remark~\ref{rmk:id}, the set $\OMEx$ is a complete set of pairwise orthogonal primitive idempotents of $\sg$ (and hence of $\overline{\sg}$). So, up to isomorphism, there are $z_1,\cdots,z_s\in\OMEx$ and $d_1\leq\cdots\leq d_s\in\ZZ$ such that $|V^\bullet|=\bigoplus_{1\leq l\leq s}z_l\overline{\sg}[d_l]$. Then $V^\bullet\otimes_{\overline{\sg}}\overline{H}=\bigoplus_{1\leq l\leq s}\bigoplus_{x\in z_l}x\overline{H}[d_l]$. Since $\overline{H}$ is semisimple, up to isomorphism, we have $Y^\bullet\otimes_{H}\overline{H}=V^\bullet\otimes_{\overline{\sg}}\overline{H}$ and $\theta$ is the identity. Thus, we have $|Y^\bullet|=\bigoplus_{1\leq l\leq s}\bigoplus_{x\in z_l}xH[d_l]$ and $\widetilde{\theta}$ is the canonical embedding induced by the inclusion $\overline{\sg}\subseteq\overline{H}$. It is straightforward to check that the following sequence of graded $\sg$-modules is exact.
\begin{equation}\label{eq:sesfenjie}
    0\to \bigoplus_{1\leq l\leq s}z_l\sg[d_l]\xrightarrow{\left(\begin{smallmatrix}
    \iota_{\sg}^{H}\\ \iota_{\sg}^{\overline{\sg}}
    \end{smallmatrix}\right)} \bigoplus_{1\leq l\leq s}\bigoplus_{x\in z_l}xH[d_l] \oplus \bigoplus_{1\leq l\leq s}z_l\overline{\sg}[d_l]\xrightarrow{\left(\begin{smallmatrix}
   \iota_H^{\overline{H}}&-\widetilde{\theta}
    \end{smallmatrix}\right)}\bigoplus_{1\leq l\leq s}\bigoplus_{x\in z_l}x\overline{H}[d_l]\to 0.
\end{equation}
Thus, the underlying graded module of $\fK(Y^\bullet,V^\bullet,\theta)$
\begin{equation}\label{eq:sp}
    |\fK(Y^\bullet,V^\bullet,\theta)|=\bigoplus_{1\leq l\leq s}z_l\sg[d_l],
\end{equation}
and we have
$$|V^\bullet|=|\fK(Y^\bullet,V^\bullet,\theta)|\otimes_{\sg}\overline{\sg},\  |Y^\bullet|=|\fK(Y^\bullet,V^\bullet,\theta)|\otimes_{\sg}H.$$
To complete the proof of \eqref{eq:fEfK2}, we only need to show these two equalities also hold on the dg module level.
Write the differentials
$$d_{\fK(Y^\bullet,V^\bullet,\theta)}=(f_{l,l'}:z_{l'}\sg[d_{l'}]\to z_l\sg[d_l])_{1\leq l',l\leq s},$$
and
$$d_{Y^\bullet}=(f'_{l,l'}:\bigoplus_{x'\in z_{l'}}x'H[d_{l'}]\to \bigoplus_{x\in z_{l}}xH[d_l])_{1\leq l',l\leq s}.$$
Since the differential of $V^\bullet$ is zero, each $f_{l',l}$ is in the radical. So $\fK(Y^\bullet,V^\bullet,\theta)$ is minimal strictly perfect. It follows that $V^\bullet= \fK(Y^\bullet,V^\bullet,\theta)\otimes_{\sg}\overline{\sg}=\fK(Y^\bullet,V^\bullet,\theta)\otimes_{\sg}^{\mathbf{L}}\overline{\sg}$. For any $1\leq l, l'\leq s$ and any $x\in z_l, x'\in z_{l'}$, we have
$$\begin{array}{rcl}
    f'_{l,l'}(x')|_{xH[d_l]} & = & d_{Y^\bullet}(x')|_{xH[d_l]} \\
    & = & x\cdot d_{Y^\bullet}(x')|_{\bigoplus_{y\in z_l}yH[d_l]} \\
     & = & x \cdot d_{Y^\bullet}(\sum_{y'\in z_{l'}}y')|_{\bigoplus_{y\in z_l}yH[d_l]}\cdot x'\\
     & = & x\cdot d_{X^\bullet}(z_{l'})|_{z_l\sg[d_l]}\cdot x'\\
     & = & x\cdot f_{l,l'}(z_{l'})\cdot x'
\end{array}$$
Thus, $d_{Y^\bullet}$ is determined by $d_{\fK(Y^\bullet,V^\bullet,\theta)}$ as the differential of $\fK(Y^\bullet,V^\bullet,\theta)\otimes_{\sg}H$. So $Y^\bullet=\fK(Y^\bullet,V^\bullet,\theta)\otimes_{\sg}H=\fK(Y^\bullet,V^\bullet,\theta)\otimes^{\mathbf{L}}_{\sg}H$ as required.

Thirdly, let $X^\bullet$ be a minimal strictly perfect $\sg$-module. Then $X^\bullet\otimes_\sg H$ and $X^\bullet\otimes_\sg \overline{\sg}$ are also minimal strictly perfect. Consider the following sequence in $\C(\sg)$
\begin{equation}\label{eq:ese2}
    0\to X^\bullet\xrightarrow{\left(\begin{smallmatrix}
    \iota_{\sg}^{H}\\ \iota_{\sg}^{\overline{\sg}}
    \end{smallmatrix}\right)} (X^\bullet\otimes_\sg H)\oplus(X^\bullet\otimes_\sg \overline{\sg})\xrightarrow{\left(\begin{smallmatrix}
   \iota_H^{\overline{H}}&-\iota_{\overline{\sg}}^{\overline{H}}
    \end{smallmatrix}\right)}X^\bullet\otimes_{\sg}\overline{H}\to 0.
\end{equation}
By Remark~\ref{rmk:id}, the set $\OMEx$ is a complete set of pairwise orthogonal primitive idempotents of $\sg$. So, up to isomorphism, there are $z_1,\cdots,z_s\in\OMEx$ and $d_1\leq\cdots\leq d_s\in\ZZ$ such that $|X^\bullet|=\bigoplus_{1\leq l\leq s}z_l\sg[d_l]$. Then, by taking the underlying graded module of each term in \eqref{eq:ese2}, we get the sequence \eqref{eq:sesfenjie}, which is exact. Hence the sequence \eqref{eq:ese2} is also exact. Thus, we have
\begin{equation}\label{eq:fKfE2}
\fK\circ\fE(X^\bullet)\cong X^\bullet.
\end{equation}
Now, the properties (1), (2) and (3) follow directly from \eqref{eq:fEfK2} and \eqref{eq:fKfE2}.

In the case that $\sg$ is non-positive, on the one hand, by Lemma~\ref{lem:KY}, any object in $\per\sg$ is isomorphic to a minimal strictly perfect $\sg$-module, and on the other hand, any object in $\Tri\sg$ is isomorphic to an object in $\glue\sg$, because \eqref{eq:sp} implies that $\fK(Y^\bullet,V^\bullet,\theta)$is strictly perfect due to that $\sg$ is non-positive. Thus, in this case, the functor $\fE$ is from $\per\sg$ to $\Tri\sg$.


\section{Technical proofs for Section~\ref{sec:mor}}\label{app:pfs}
\subsection{Proof of Lemma~\ref{lem:G}}\label{app:pf2}

By Theorem~\ref{thm2}, there is an isomorphism $$(g,h):\fE(\widetilde{X}(\wsx))\to \fE(\cpy{\ws}),$$
and hence we have an isomorphism $$\fM((g,h)):\fM\circ\fE(\widetilde{X}(\wsx))\to \fM\circ\fE(\cpy{\ws}).$$
By Theorem~\ref{thm1}, $\fM\circ\fE(\cpy{\ws})=\widetilde{X}(\wsx)=R(\bione(\wsx))$, where the representation $R(\bione(\wsx))$ is given in Remark~\ref{rmk:Rmx}. Similarly, we have an isomorphism $$\fM((g',h')):\fM\circ\fE(\widetilde{X}(\wtx))\to R(\bione(\wtx)).$$
By Remark~\ref{rmk:rep and mor}, a morphism $\alpha:R(\bione(\wsx))\to R(\bione(\wtx))$ is given by a pair $(\alpha_-,\alpha_+)$ which determine each other, where $\alpha_+$ is from $\bigoplus_{1\leq l\leq p_{\ws}}(\wks_l,\wsi_l)$ to $\bigoplus_{1\leq l\leq p_{\wt}}(\wkt_l,\wti_l)$. For any $1\leq l\leq p_{\ws}$, we define
$$Y^{\ws}_l=\begin{cases}
    (\wks_l,\wsi_l)+(\wks_{l-1},\wsi_{l-1}) & \text{if $\aso{\wsx}{l-1}{l}$ is positive, interior and punctured,}\\
    (\wks_l,\wsi_l)+(\wks_{l+1},\wsi_{l+1})&\text{if $\aso{\wsx}{l+1}{l}$ is positive, interior and punctured,}\\
    (\wks_l,\wsi_l)&\text{otherwise.}
\end{cases}$$
Similarly, we define $Y^{\wt}_l$ for any $1\leq l\leq p_{\wt}$.

For any $\widetilde{L}\in\widetilde{\mathcal{L}}$, let
$$\alpha_{\widetilde{L}}:=\fM((g',h'))\circ\fM\circ\fE(f_{\widetilde{L}})\circ\fM((g,h))^{-1}:R(\bione(\wsx))\to R(\bione(\wtx)).$$
By the construction of $f_{\widetilde{L}}$ (see Construction~\ref{cons:fL}), the positive part of $\alpha_{\widetilde{L}}$ is
$$\sum_{(u,v)\in L_0}g_{(u,v)}+\sum_{E\in\widetilde{L}}\sum_{1\leq \exi\leq s_E}c_\exi^Eg_{(\uex{E}{\exi},\vex{E}{\exi})},$$
where for any $(l,l')=(u,v)\in L_0$ or $(l,l')=(\uex{E}{\exi},\vex{E}{\exi})$ for $E\in\widetilde{L}$ and $1\leq \exi\leq s_E$, we have
$g_{(l,l')}(Y_l^{\ws})=(Y_{l'}^{\wt})$. Since the coefficients $c_l^E$ are the same as $U_{\bione(\wsx),\bione(\wtx)}(\aex{E}{l})$ given in \cite[Section~3.3]{G}, the morphism $\alpha_{\widetilde{L}}$ coincides with the one constructed in \cite[Section~4.1]{G} associated to $\widetilde{L}$. Hence, by the theorem in \cite[Section~4]{G}, $\{\alpha_{\widetilde{L}} \mid\widetilde{L}\in\widetilde{\mathcal{L}}\}$ forms a basis of $\Hom(R(\bione(\wsx)),R(\bione(\wtx)))$. Then $\{\fM\circ\fE(f_{\widetilde{L}})\mid \widetilde{L}\in\widetilde{\mathcal{L}}\}$ form a basis of $\Hom_{\repb(S)}(\fM\circ\fE(\cpy{\ws}),\fM\circ\fE(\cpy{\wt}))$.

\subsection{Proof of Proposition~\ref{prop:rad}}\label{app:pfs1}

By Lemma~\ref{lem:G}, the set of morphisms $\{f_{\widetilde{L}}\mid\widetilde{L}\in\widetilde{\mathcal{L}}\}$ is linearly independent, and for any $f\in\Hom_{\per\sg}(\cpy{\ws},\cpy{\wt})$, there is $f'\in\k\{f_{\widetilde{L}}\mid\widetilde{L}\in\widetilde{\mathcal{L}}\}$ such that $\fM\circ\fE(f-f')=0$. Take  a representative $\varphi$ of $f-f'$ in $Z^0\huaHom_\sg(\cpy{\ws},\cpy{\wt})$. By the constructions of $\fM$ and $\fE$ (see Theorems~\ref{thm:tribu} and~\ref{thm:BD}), if any component of $\varphi$ is an isomorphism then so is the corresponding component of $\fM\circ\fE(\varphi)$, a contradiction. Hence we have $\varphi\in\huaRad_\sg(\cpy{\ws},\cpy{\wt})$. So we have $\Hom_{\per\sg}(\cpy{\ws},\cpy{\wt})=\k\{f_{\widetilde{L}}\mid\widetilde{L}\in\widetilde{\mathcal{L}}\}+\operatorname{Rad}_{\per\sg}(\cpy{\ws},\cpy{\wt}).$

    To show it is a direct sum, it suffices to show that for any nonzero $g\in\k\{f_{\widetilde{L}}\mid\widetilde{L}\in\widetilde{\mathcal{L}}\}$, we have $g\notin \operatorname{Rad}_{\per\sg}(\cpy{\ws},\cpy{\wt})$. There exist tagged real $h$-lines $L_1,\cdots,L_s\in\mathcal{L}(\ws,\wt)$ and nonzero $\lambda_1,\cdots,\lambda_n\in\k$ such that $g=\sum_{i=1}^s\lambda_i f_{\widetilde{L}_i}$. Without loss of generality, assume that $L_1$ is maximal in the set $\{L_1,\cdots,L_s\}$ under the order in $\mathcal{L}(\ws,\wt)$ given in Proposition~\ref{prop:G}. Let $(u,v)$ be a vertex of $L_1$ and $g_{(u,v)}$ (resp. $(f_{\widetilde{L}_i})_{(u,v)}$) the component of $g$ (resp. $f_{\widetilde{L}_i}$) from $\cply{\ws}{\ec{u}}$ to $\cply{\wt}{\ec{v}}$. We claim that $g_{(u,v)}\neq 0$. We divide the proof of the claim into the following cases.
    \begin{enumerate}
        \item For the case that $(u,v)$ is not the start of an arrow in $\qst_2$, by the construction, $(f_{\widetilde{L}_1})_{(u,v)}=\id$ and $(f_{\widetilde{L}_i})_{(u,v)}=0$ for $i\neq1$. So we have $g_{(u,v)}=\lambda_1\id\neq 0$.
        \item For the case that $(u,v)$ is the start of an arrow in $\qst_2$, by Remark~\ref{rmk:Q22}, $(u,v)$ is an endpoint of $L_1$ and one extension $E$ in $\widetilde{L}_1$ from $(u,v)$ contains an arrow $\alpha\in\qst_2$ whose start is $(u,v)$. Note that there is at most one $L_i$, $i\neq 1$, which contains the terminal of $\alpha$.
        \begin{enumerate}
            \item If $\alpha\in\qst_{2,-}$, for $i\neq 1$, we have $(f_{\widetilde{L}_i})_{(u,v)}=\id$ if $L_i$ contains the terminal of $\alpha$, or zero otherwise.
            \begin{enumerate}
                \item In the subcase $\alpha=(\wmu,\wnu)_{+}$, we have $(f_{\widetilde{L}_1})_{(u,v)}=\begin{pmatrix}
                1&0\\0&0
                \end{pmatrix}$. So $g_{(u,v)}=\begin{pmatrix}
                \lambda_1&0\\0&0
                \end{pmatrix}$ or $\begin{pmatrix}
                \lambda_1+\lambda_i&0\\0&\lambda_i
                \end{pmatrix}$, each of which is not zero.
                \item In the subcase $\alpha=(\wmu,\wnu)_{-}$, we have
                $(f_{\widetilde{L}_1})_{(u,v)}=\begin{pmatrix}
                0&0\\0&-1
                \end{pmatrix}$. So $g_{(u,v)}=\begin{pmatrix}
                0&0\\0&-\lambda_1
                \end{pmatrix}$ or $\begin{pmatrix}
                \lambda_i&0\\0&\lambda_i-\lambda_1
            \end{pmatrix}$, each of which is zero.
            \end{enumerate}
            \item If $\alpha\in\qst_{2,o}$, since $L_i$ is a real $h$-line, we have that $L_i$ does not contain the terminal of $\alpha$. So $(f_{\widetilde{L}_i})_{(u,v)}=0$ for any $i\neq 1$.
            \begin{enumerate}
                \item In the subcase $\alpha=(\wmu,\wnu)_{\oplus}$ and the corresponding tagging of $\wtx$ is $+$, we have $(f_{\widetilde{L}_1})_{(u,v)}=\begin{pmatrix}
                1&0
                \end{pmatrix}$.
                \item In the subcase $\alpha=(\wmu,\wnu)_{\oplus}$ and the corresponding tagging of $\wtx$ is $-$, we have $(f_{\widetilde{L}_1})_{(u,v)}=\begin{pmatrix}
                0&1
                \end{pmatrix}$.
        \item In the subcase $\alpha=(\wmu,\wnu)_{\ominus}$ and the corresponding tagging of $\wsx$ is $+$,
        $(f_{\widetilde{L}_1})_{(u,v)}=\begin{pmatrix}
        1\\0
        \end{pmatrix}$.
        \item In the subcase $\alpha=(\wmu,\wnu)_{\ominus}$ and the corresponding tagging of $\wsx$ is $-$,
        $(f_{\widetilde{L}_1})_{(u,v)}=\begin{pmatrix}
        0\\-1
        \end{pmatrix}$.
            \end{enumerate}
            In each case,  $g_{(u,v)}=\lambda_1(f_{\widetilde{L}_1})_{(u,v)}\neq0$.
        \end{enumerate}
    \end{enumerate}
    Since $(u,v)$ is a vertex of $\qst$, there are no nonzero radical morphisms from $\cply{\wsx}{\ec{u}}$ to $\cply{\wtx}{\ec{v}}$. So $g_{(u,v)}\neq 0$ implies that $g_{(u,v)}$ is not in the radical. Hence we have $g\notin\operatorname{Rad}_{\per\sg}(\cpy{\ws},\cpy{\wt})$.

\subsection{Proof of Lemma~\ref{lem:rad1}}\label{app:pfs2}

	For case (1), write $\wmu=\asu{\wsx}{u}{u'}=(i,j_1)-(i,j_2)$ and $\wnu=\asu{\wtx}{v}{v'}=(i,j_3)-(i,j_4)$ with $j_1<j_2$ and $j_3<j_4$. There are the following subcases.
	\begin{itemize}
	    \item $j_1<j_2<j_3<j_4$, or $0=j_3<j_1<j_2<j_4$, see the first picture in Figure~\ref{fig:rad101}. Then $\huaRad_{\sg}(\cpy{\ws},\cpy{\wt})_{\wmu,\wnu}=0$.
	    \item $j_1<j_3<j_4<j_2$, see the second picture in Figure~\ref{fig:rad101}. Then
	    $$|\huaRad_{\sg}(\cpy{\ws},\cpy{\wt})_{\wmu,\wnu}|=\k\ \fin{\xit_{v}}\circ f_{\omo{\wnu}{\ec{v}},\omo{\wmu}{\ec{u'}}}\circ\fout{\xis_{u'}}\oplus \k\ \fin{\xit_{v'}}\circ f_{\omo{\wnu}{\ec{v'}},\omo{\wmu}{\ec{u'}}}\circ\fout{\xis_{u'}}.$$
	    So we have
	    $$Z\huaRad_{\sg}(\cpy{\ws},\cpy{\wt})_{\wmu,\wnu}=\k\ \fin{\xit_{v}}\circ f_{\omo{\wnu}{v},\omo{\wmu}{u'}}\circ\fout{\xis_{u'}}=B\huaRad_{\sg}(\cpy{\ws},\cpy{\wt})_{\wmu,\wnu}.$$
	    \item $j_3<j_1<j_2<j_4$. This is the dual of the subcase $j_1<j_3<j_4<j_2$.
	    \item $j_3<j_4<j_1<j_2$, see the third picture in Figure~\ref{fig:rad101}. If $j_3=0$, this is the dual case of $0=j_1<j_3<j_4<j_2$. So we only consider the case $j_3>0$. Then we have
	    $$\begin{array}{rl}
	        & Z\huaRad_{\sg}(\cpy{\ws},\cpy{\wt})_{\wmu,\wnu}  \\
	       = & \k\ \fin{\xit_{v}}\circ f_{\omo{\wnu}{v},\omo{\wmu}{u'}}\circ\fout{\xis_{u'}}\oplus \k\ ((-1)^{\wsi_1-\wti_3}\fin{\xit_{v}}\circ f_{\omo{\wnu}{v},\omo{\wmu}{u}}\circ\fout{\xis_{u}}+\fin{\xit_{v'}}\circ f_{\omo{\wnu}{v'},\omo{\wmu}{u'}}\circ\fout{\xis_{u'}})\\
	       = & B\huaRad_{\sg}(\cpy{\ws},\cpy{\wt})_{\wmu,\wnu}.
	    \end{array}$$
	\end{itemize}
	
	\begin{figure}[htpb]
	    	    \begin{tikzpicture}[xscale=1.6,yscale=1.4]
	    	\begin{scope}[shift={(-3,0)}]
	    		\draw[ultra thick](-1,1)to(1,1);
	    		\draw[red,thick,dotted](-.4,1)to(-.8,.6) (-1,0)to(-.8,-.6) (-.4,-1)to(.4,-1) (.8,-.6)to(1,0) (.8,.6)to(.4,1);
	    		\draw[red,thick](-.8,.6)to(-1,0) (-.8,-.6)to(-.4,-1) (.4,-1)to(.8,-.6) (1,0)to(.8,.6);
	    		\draw[blue](0,1)\nn;
	    		\draw[red](-1.2,.3)node{$(i,j_1)$}(-.9,-1)node{$(i,j_2)$}(1.2,.3)node{$(i,j_4)$}(.9,-1)node{$(i,j_3)$}(0,1.2)node{$(i,j_3)$};
	    		
	    		\draw[blue,thick,bend right=20](0,1)to(.85,.4) (.92,.2)to(.6,-.8) (-.6,-.8)to(-.92,.2) (.3,.5)node{$\tilde{\nu}$}(.5,-.4)node{$\tilde{\nu}$}(-.5,-.4)node{$\tilde{\mu}$};
	    	\end{scope}
	    	\draw[ultra thick](-1,1)to(1,1);
	    	\draw[red,thick,dotted](-.4,1)to(-.8,.6) (-1,0)to(-.8,-.6) (-.4,-1)to(.4,-1) (.8,-.6)to(1,0) (.8,.6)to(.4,1);
	    	\draw[red,thick](-.8,.6)to(-1,0) (-.8,-.6)to(-.4,-1) (.4,-1)to(.8,-.6) (1,0)to(.8,.6);
	    	\draw[blue](0,1)\nn;
	    	\draw[red](-1.2,.3)node{$(i,j_1)$}(-.9,-1)node{$(i,j_3)$}(1.2,.3)node{$(i,j_2)$}(.9,-1)node{$(i,j_4)$};
	    	
	    	\draw[blue,thick,bend left=20](-.9,.3)to(.9,.3) (-.6,-.8)to(.6,-.8) (0,.6)node{$\tilde{\mu}$}(0,-.5)node{$\tilde{\nu}$};
	    	\begin{scope}[shift={(3,0)}]
	    		\draw[ultra thick](-1,1)to(1,1);
	    		\draw[red,thick,dotted](-.4,1)to(-.8,.6) (-1,0)to(-.8,-.6) (-.4,-1)to(.4,-1) (.8,-.6)to(1,0) (.8,.6)to(.4,1);
	    		\draw[red,thick](-.8,.6)to(-1,0) (-.8,-.6)to(-.4,-1) (.4,-1)to(.8,-.6) (1,0)to(.8,.6);
	    		\draw[blue](0,1)\nn;
	    		\draw[red](-1.2,.3)node{$(i,j_3)$}(-.9,-1)node{$(i,j_4)$}(1.2,.3)node{$(i,j_2)$}(.9,-1)node{$(i,j_1)$};
	    		
	    		\draw[blue,thick,bend right=20](.9,.3)to(.6,-.8) (-.6,-.8)to(-.9,.3) (.45,-.3)node{$\tilde{\mu}$}(-.45,-.3)node{$\tilde{\nu}$};
	    	\end{scope}
	    \end{tikzpicture}
	    \caption{Cases for disjoint $\wmu$ and $\wnu$}
	    \label{fig:rad101}
	\end{figure}
	
	For case (2), write $\wmu=\asu{\wsx}{u}{u'}=(i,j)-(i,j_1)$ and $\wnu=\asu{\wtx}{v}{v'}=(i,j)-(i,j_2)$ with $j\neq 0$. There are the following subcases.
	\begin{itemize}
	    \item $j_2<j_1<j$, see the first picture in Figure~\ref{fig:rad102}. Then
	    $$|\huaRad_{\sg}(\cpy{\ws},\cpy{\wt})_{\wmu,\wnu}|=\k\ \fin{\xit_{v'}}\circ f_{\omo{\wnu}{v'},\omo{\wmu}{u}}\circ\fout{\xis_{u}}\oplus \k\ \fin{\xit_{v'}}\circ f_{\omo{\wnu}{v'},\omo{\wmu}{u'}}\circ\fout{\xis_{u'}}.$$
	    So we have
	    $$Z\huaRad_{\sg}(\cpy{\ws},\cpy{\wt})_{\wmu,\wnu}=\k\ \fin{\xit_{v'}}\circ f_{\omo{\wnu}{v'},\omo{\wmu}{u}}\circ\fout{\xis_{u}}=B\huaRad_{\sg}(\cpy{\ws},\cpy{\wt})_{\wmu,\wnu}.$$
	    \item $j<j_2<j_1$. This is the dual of the subcase $j_2<j_1<j$.
	    \item $j_1<j<j_2$, see the second picture in Figure~\ref{fig:rad102}. Then we have $\huaRad_{\sg}(\cpy{\ws},\cpy{\wt})_{\wmu,\wnu}=0$.
	\end{itemize}
	
	\begin{figure}[htpb]
	    	    \begin{tikzpicture}[xscale=1.6,yscale=1.4]
	    	\draw[ultra thick](-1,1)to(1,1);
	    	\draw[red,thick,dotted](-.4,1)to(-.8,.6) (-1,0)to(-.8,-.6) (-.4,-1)to(.4,-1) (.8,-.6)to(.4,1);
	    	\draw[red,thick](-.8,.6)to(-1,0) (-.8,-.6)to(-.4,-1) (.4,-1)to(.8,-.6);
	    	\draw[blue](0,1)\nn;
	    	\draw[red](-1.2,.3)node{$(i,j_2)$}(-.9,-1)node{$(i,j_1)$}(.9,-1)node{$(i,j)$};
	    	
	    	\draw[blue,thick,bend left=20](-.9,.3)to(.65,-.75) (-.55,-.85)to(.55,-.85) (0,.2)node{$\tilde{\nu}$}(0,-.6)node{$\tilde{\mu}$};
	        \begin{scope}[shift={(3,0)}]
	        	\draw[ultra thick](-1,1)to(1,1);
	        	\draw[red,thick,dotted](-.4,1)to(-.8,.6) (-1,0)to(-.8,-.6) (.8,-.6)to(1,0) (.8,.6)to(.4,1);
	        	
	        	\draw[red,thick](-.8,.6)to(-1,0) (-.8,-.6)to(.8,-.6) (1,0)to(.8,.6);
	        	
	        	\draw[blue](0,1)\nn;
	        	\draw[red](-1.2,.3)node{$(i,j_1)$}(1.2,.3)node{$(i,j_2)$};
	        	
	        	\draw[blue,thick,bend right=20](.9,.3)to(.4,-.6) (-.4,-.6)to(-.9,.3) (.4,0)node{$\tilde{\nu}$}(-.4,0)node{$\tilde{\mu}$};
	        \end{scope}
        \end{tikzpicture}
	    \caption{Cases for $\wmu$ and $\wnu$ sharing exactly one edge}
	    \label{fig:rad102}
	\end{figure}
	
	For case (3), we have $\huaRad_{\sg}(\cpy{\ws},\cpy{\wt})_{\wmu,\wnu}=0$.

\subsection{Proof of Lemma~\ref{lem:rad2}}\label{app:pfs3}

    Since $\upas(\wsx;\wtx)$ is the union $$\upas(\wsx;\wtx;\times)\cup\upas(\wsx;\wtx;\parallel)\cup\upas(\wsx;\wtx;\overset{\M}{\wedge})\cup(\cup_{\rho\in\ZZ}D^\circ(\ws,\wt[\rho])),$$ by Lemma~\ref{lem:rad1}, we only need to consider the following cases.

    \textbf{(a).} $(\wmu,\wnu)\in\upas(\wsx;\wtx;\times)$. Let $\rho$ be the intersection index. There are the following subcases.
        \begin{enumerate}
            \item $j_1<j_3<j_2<j_4$ or $0=j_3<j_1<j_4<j_2$, see the first and second pictures in Figure~\ref{fig:rad2}. If $\rho=0$, then $\huaRad_{\sg}(\cpy{\ws},\cpy{\wt})_{\wmu,\wnu}=\huaRad^0_{\sg}(\cpy{\ws},\cpy{\wt})_{\wmu,\wnu}=\k\ \psi(\wmu,\wnu)$. So we have  $Z^0\huaRad_{\sg}(\cpy{\ws},\cpy{\wt})_{\wmu,\wnu}=\k\ \psi(\wmu,\wnu)$ and $B^0\huaRad_{\sg}(\cpy{\ws},\cpy{\wt})_{\wmu,\wnu}=0$. If $\rho\neq 0$, we have $\huaRad^0_{\sg}(\cpy{\ws},\cpy{\wt})_{\wmu,\wnu}=0$.
            \item $0<j_3<j_1<j_4<j_2$, see the third picture in Figure~\ref{fig:rad2}. If $\rho=0$, we have $$\huaRad_{\sg}(\cpy{\ws},\cpy{\wt})_{\wmu,\wnu}=\huaRad^0_{\sg}(\cpy{\ws},\cpy{\wt})_{\wmu,\wnu}\oplus\huaRad^1_{\sg}(\cpy{\ws},\cpy{\wt})_{\wmu,\wnu},$$ where $$\huaRad^0_{\sg}(\cpy{\ws},\cpy{\wt})_{\wmu,\wnu}=\k\ \fin{\xit_{v'}}\circ f_{(i,j_4),(i,j_2)}\circ\fout{\xis_{u'}}\oplus \k\ \fin{\xit_{v}}\circ f_{(i,j_3),(i,j_1)}\circ\fout{\xis_{u}}$$
            and $$\huaRad^1_{\sg}(\cpy{\ws},\cpy{\wt})_{\wmu,\wnu}=\k\ \fin{\xit_{v}}\circ f_{(i,j_3),(i,j_2)}\circ\fout{\xis_{u'}}.$$
            So we have $Z^0\huaRad_{\sg}(\cpy{\ws},\cpy{\wt})_{\wmu,\wnu}=\k\ \psi(\wmu,\wnu)$ and $B^0\huaRad^1_{\sg}(\cpy{\ws},\cpy{\wt})_{\wmu,\wnu}=0$. If $\rho=-1$, we have $$Z^0\huaRad_{\sg}(\cpy{\ws},\cpy{\wt})_{\wmu,\wnu}=B^0\huaRad^1_{\sg}(\cpy{\ws},\cpy{\wt})_{\wmu,\wnu}=\k\ \fin{\xit_{v}}\circ f_{(i,j_3),(i,j_2)}\circ\fout{\xis_{u'}};$$
            if $\rho\neq 0,-1$, we have $\huaRad^1_{\sg}(\cpy{\ws},\cpy{\wt})_{\wmu,\wnu}=0$.
        \end{enumerate}

    \textbf{(b).} $(\wmu,\wnu)\in\upas(\wsx;\wtx;\overset{\M}{\wedge})$ and with the orientations such that $\wmu$ and $\wnu$ start at the same point in $\M$, $\wmu$ is to the left of $\wnu$. Let $\rho$ be the intersection index. Write $\wmu=(i,0)-(i,j_1)$ and $\wnu=(i,0)-(i,j_2)$. If $\rho=0$, then
        $Z^0\huaRad_{\sg}(\cpy{\ws},\cpy{\wt})_{\wmu,\wnu}=\huaRad^0_{\sg}(\cpy{\ws},\cpy{\wt})_{\wmu,\wnu}=\k\ \psi(\wmu,\wnu)$ and $B^0\huaRad_{\sg}(\cpy{\ws},\cpy{\wt})_{\wmu,\wnu}=0$. If $\rho\neq 0$, then $\huaRad_{\sg}(\cpy{\ws},\cpy{\wt})_{\wmu,\wnu}=0$.

    \textbf{(c).} $(\wmu,\wnu)\in D^\circ(\ws,\wt[\rho])$ and either $\wmu\sim\wnu$, or $\wmu\nsim\wnu$ and with the orientations such that $\wmu$ and $\wnu$ start at the same edge which is not a boundary segment, $\wmu$ is to the right of $\wnu$. There are the following two subcases.
        \begin{enumerate}
            \item $\wmu\sim\wnu$. If $\rho\neq -1$, we have $$\huaRad^0_{\sg}(\cpy{\ws},\cpy{\wt})_{\wmu,\wnu}=0.$$
            If $\rho=-1$, we have $$\huaRad_{\sg}(\cpy{\ws},\cpy{\wt})_{\wmu,\wnu}=\huaRad^0_{\sg}(\cpy{\ws},\cpy{\wt})_{\wmu,\wnu}=\k\psi(\wmu,\wnu).$$
            So we have $Z^0\huaRad_{\sg}(\cpy{\ws},\cpy{\wt})_{\wmu,\wnu}=\k\ \psi(\wmu,\wnu)$ and $B^0\huaRad^1_{\sg}(\cpy{\ws},\cpy{\wt})_{\wmu,\wnu}=0$.
            \item $\wmu\nsim\wnu$. Write $\wmu=\asu{\wsx}{u}{u'}=(i,j)-(i,j_1)$ and $\wmu=\asu{\wtx}{v}{v'}=(i,j)-(i,j_2)$. If $\rho\neq -1$, then $\huaRad^0_{\sg}(\cpy{\ws},\cpy{\wt})_{\wmu,\wnu}=0$ unless $0<j_2<j<j_1$ and $\rho=-2$, where $$Z^0\huaRad_{\sg}(\cpy{\ws},\cpy{\wt})_{\wmu,\wnu}=\k\ \fin{\xit_{v'}}\circ f_{(i,j_2),(i,j_1)}\circ\fout{\xis_{u'}}=B^0\huaRad_{\sg}(\cpy{\ws},\cpy{\wt})_{\wmu,\wnu}.$$
            If $\rho=-1$, we have
            $$\huaRad_{\sg}(\cpy{\ws},\cpy{\wt})_{\wmu,\wnu}=\huaRad^0_{\sg}(\cpy{\ws},\cpy{\wt})_{\wmu,\wnu}=\k\{\fin{\xit_{v'}}\circ f_{\omo{\wnu}{v'},\omo{\wmu}{u}}\circ\fout{\xis_{u}},\fin{\xit_{v}}\circ f_{\omo{\wnu}{v},\omo{\wmu}{u'}}\circ\fout{\xis_{u'}}\}.$$
            So we have $Z^0\huaRad_{\sg}(\cpy{\ws},\cpy{\wt})_{\wmu,\wnu}=\k\ \psi(\wmu,\wnu)$ and $B^0\huaRad^1_{\sg}(\cpy{\ws},\cpy{\wt})_{\wmu,\wnu}=0$.
        \end{enumerate}

\subsection{Proof of Lemma~\ref{lem:homotopy}}\label{app:pfs4}

    We divide the proof into the following cases.

    \textbf{(a).} If $\hks_u=\hkt_v$ is an edge of a once-punctured monogon and any arc segment in the punctured side of $(u,v)$ is interior, we only need to consider the case $\wks_u=\wkt_v=(i,j^+)$, because we have $\hks_{u'}=\hkt_{v'}=(i,j^-)$ for the vertex $(u',v')$ of $Q_0^{\ws,\wt[-1]}$ such that $\asu{\wsx}{u}{u'}$ and $\asu{\wtx}{v}{v'}$ are punctured. Then there are two arrows in $Q_{2,-}^{\ws,\wt[-1]}$, one from $(u',v)$ to $(u,v)$ and the other from $(u',v)$ to $(u',v')$. Write the unpunctured arc segments $\asu{\wsx}{u''}{u}$, $\asu{\wsx}{v''}{v}$, $\asu{\wsx}{u'}{u'''}$ and $\asu{\wsx}{v'}{v'''}$. So the unpunctured side of $(u,v)$ (resp. $(u',v)$ and $(u',v')$) is $(\asu{\wsx}{u''}{u},\asu{\wtx}{v''}{v})$ (resp. $(\asu{\wsx}{u'}{u'''}, \asu{\wsx}{v''}{v})$ and $(\asu{\wsx}{u'}{u'''},\asu{\wsx}{v'}{v'''})$). Note that $\omo{\wmu}{u}=\omo{\wnu}{v}=\omo{\wmu}{u'}=\omo{\wnu}{v'}=(i,j)$. Then we have the following as required.
        $$\begin{array}{rcl}
            d(\iota(u,v)) & = & \begin{pmatrix}
            1&0\\0&0
        \end{pmatrix}\circ\fin{\xis_{u}}\circ f_{\omo{\wmu}{u},\omo{\wmu}{u''}}\circ\fout{\xis_{u''}}+\begin{pmatrix}
        1&0\\0&0
        \end{pmatrix}\circ\fin{\xis_{u'}}\circ f_{\omo{\wmu}{u'},\omo{\wmu}{u'''}}\circ\fout{\xis_{u'''}} \\
        & & +\fin{\xis_{v''}}\circ f_{\omo{\wnu}{v''},\omo{\wnu}{v}}\circ\fout{\xis_{v}}\circ\begin{pmatrix}
        1&0\\0&0
        \end{pmatrix}+\fin{\xis_{v'''}}\circ f_{\omo{\wnu}{v'''},\omo{\wnu}{v'}}\circ\fout{\xis_{v'}}\circ\begin{pmatrix}
            1&0\\0&0
        \end{pmatrix}\\
        & = & \begin{pmatrix}
            1&0\\0&0
        \end{pmatrix}\circ f_{\omo{\wmu}{u},\omo{\wmu}{u''}}\circ\fout{\xis_{u''}}+\begin{pmatrix}
        1&0\\0&0
        \end{pmatrix}\circ f_{\omo{\wmu}{u'},\omo{\wmu}{u'''}}\circ\fout{\xis_{u'''}} \\
         & & +\fin{\xis_{v''}}\circ f_{\omo{\wnu}{v''},\omo{\wnu}{v}}\circ\begin{pmatrix}
        1&0\\0&0
    \end{pmatrix}\\
    & =   & \fin{\xis_{v}}\circ f_{\omo{\wnu}{v},\omo{\wmu}{u''}}\circ\fout{\xis_{u''}}+\fin{\xis_{v}}\circ f_{\omo{\wnu}{v},\omo{\wmu}{u'''}}\circ\fout{\xis_{u'''}} \\
         & & +\fin{\xis_{v''}}\circ f_{\omo{\wnu}{v''},\omo{\wmu}{u}}\circ\fout{\xis_{u}}+\fin{\xis_{v''}}\circ f_{\omo{\wnu}{v''},\omo{\wmu}{u'}}\circ\fout{\xis_{u'}}\\
    & = & \psi(\asu{\wsx}{u''}{u},\asu{\wtx}{v''}{v})+\psi(\asu{\wsx}{u'}{u'''}, \asu{\wsx}{v''}{v}).
    \end{array}$$
    $$\begin{array}{rcl}
        d(\iota(u',v')) & = & \begin{pmatrix}
        0&0\\0&1
    \end{pmatrix}\circ\fin{\xis_{u}}\circ f_{\omo{\wmu}{u},\omo{\wmu}{u''}}\circ\fout{\xis_{u''}}+\begin{pmatrix}
        0&0\\0&1
    \end{pmatrix}\circ\fin{\xis_{u'}}\circ f_{\omo{\wmu}{u},\omo{\wmu}{u'''}}\circ\fout{\xis_{u'''}} \\
         & & +\fin{\xis_{v''}}\circ f_{\omo{\wnu}{v''},\omo{\wnu}{v}}\circ\fout{\xis_{v}}\circ\begin{pmatrix}
        0&0\\0&1
    \end{pmatrix}+\fin{\xis_{v'''}}\circ f_{\omo{\wnu}{v'''},\omo{\wnu}{v'}}\circ\fout{\xis_{v'}}\circ\begin{pmatrix}
        0&0\\0&1
    \end{pmatrix}\\
    & = &\begin{pmatrix}
        0&0\\0&1
    \end{pmatrix}\circ f_{\omo{\wmu}{u'},\omo{\wmu}{u'''}}\circ\fout{\xis_{u'''}} \\
         & & +\fin{\xis_{v''}}\circ f_{\omo{\wnu}{v''},\omo{\wnu}{v}}\circ\begin{pmatrix}
        0&0\\0&-1
    \end{pmatrix}+\fin{\xis_{v'''}}\circ f_{\omo{\wnu}{v'''},\omo{\wnu}{v'}}\circ\begin{pmatrix}
        0&0\\0&1
    \end{pmatrix}\\
    & =   & -\fin{\xis_{v}}\circ f_{\omo{\wnu}{v},\omo{\wmu}{u''}}\circ\fout{\xis_{u''}}+\fin{\xis_{v'}}\circ f_{\omo{\wnu}{v'},\omo{\wmu}{u'''}}\circ\fout{\xis_{u'''}} \\
         & & -\fin{\xis_{v''}}\circ f_{\omo{\wnu}{v''},\omo{\wmu}{u'}}\circ\fout{\xis_{u'}}+\fin{\xis_{v'''}}\circ f_{\omo{\wnu}{v'''},\omo{\wmu}{u'}}\circ\fout{\xis_{u'}}\\
    & = & -\psi(\asu{\wsx}{u}{u''},\asu{\wtx}{v''}{v})+\psi(\asu{\wsx}{u'}{u'''}, \asu{\wsx}{v'}{v'''}).
    \end{array}$$

    \textbf{(b).} If $\hks_u=\hkt_v$ is an edge of a once-punctured monogon and exactly one arc segment in the punctured side of $(u,v)$ is interior, there are the following subcases.
    \begin{enumerate}
        \item The non-interior arc segment in the punctured side of $(u,v)$ is of $\wsx$ and has tagging $+$. Then $\ec{u}=\{u\}$, $\ec{v}=\{v,v'\}$ and $\wks_u=\wkt_v=(i,j)^+$. Write the unpunctured arc segments $\asu{\wsx}{u''}{u}$, $\asu{\wsx}{v''}{v}$ and $\asu{\wsx}{v'}{v'''}$. So the unpunctured side of $(u,v)$ (resp. $(u,v')$) is $(\asu{\wsx}{u''}{u},\asu{\wtx}{v''}{v})$ (resp. $(\asu{\wsx}{u''}{u}, \asu{\wsx}{v'}{v'''})$). Note that $\omo{\wmu}{u}=(i,j^+)$ and $\omo{\wnu}{v}=\omo{\wnu}{v'}=(i,j)$. Then we have
        $$\begin{array}{rcl}
        d(\iota(u,v)) & = & \begin{pmatrix}
        1\\0
        \end{pmatrix}\circ\fin{\xis_{u}}\circ f_{(i,j^+),\omo{\wmu}{u''}}\circ\fout{\xis_{u''}}\\
         & & +\fin{\xis_{v''}}\circ f_{\omo{\wnu}{v''},(i,j)}\circ\fout{\xis_{v}}\circ\begin{pmatrix}
        1\\0
        \end{pmatrix}+\fin{\xis_{v'''}}\circ f_{\omo{\wnu}{v'''},(i,j)}\circ\fout{\xis_{v'}}\circ\begin{pmatrix}
        1\\0
        \end{pmatrix}\\
        & = & \begin{pmatrix}
        f_{(i,j^+),\omo{\wmu}{u''}}\\0
        \end{pmatrix}\circ\fout{\xis_{u''}}+\fin{\xis_{v''}}\circ f_{\omo{\wnu}{v''},(i,j^+)}\\
        & =  & \fin{\xis_{v}}\circ f_{\omo{\wnu}{v},\omo{\wmu}{u''}}\circ\fout{\xis_{u''}}+\fin{\xis_{v''}}\circ f_{\omo{\wnu}{v''},\omo{\wmu}{u}}\circ\fout{\xis_{u}}\\
        & = & \psi(\asu{\wsx}{u''}{u},\asu{\wtx}{v''}{v}).
        \end{array}$$
        In this case, there is an arrow in $Q^{\ws,\wt[-1]}_{2,o}$ from $(u,v)$ to $(u,v')$.
        \item The non-interior arc segment in the punctured side of $(u,v)$ is of $\wsx$ and has tagging $-$. Then $\ec{u}=\{u\}$, $\ec{v}=\{v,v'\}$ and $\wks_u=\wkt_v=(i,j)^-$. Write the unpunctured arc segments $\asu{\wsx}{u''}{u}$, $\asu{\wsx}{v''}{v}$ and $\asu{\wsx}{v'}{v'''}$. So the unpunctured side of $(u,v)$ (resp. $(u,v')$) is $(\asu{\wsx}{u''}{u},\asu{\wtx}{v''}{v})$ (resp. $(\asu{\wsx}{u''}{u}, \asu{\wsx}{v'}{v'''})$). Note that $\omo{\wmu}{u}=(i,j^-)$ and $\omo{\wnu}{v}=\omo{\wnu}{v'}=(i,j)$. Then we have
        $$\begin{array}{rcl}
        d(\iota(u,v)) & = & \begin{pmatrix}
        0\\1
        \end{pmatrix}\circ\fin{\xis_{u}}\circ f_{(i,j^-),\omo{\wmu}{u''}}\circ\fout{\xis_{u''}}\\
         & & +\fin{\xis_{v''}}\circ f_{\omo{\wnu}{v''},(i,j)}\circ\fout{\xis_{v}}\circ\begin{pmatrix}
        0\\1
        \end{pmatrix}+\fin{\xis_{v'''}}\circ f_{\omo{\wnu}{v'''},(i,j)}\circ\fout{\xis_{v'}}\circ\begin{pmatrix}
        0\\1
        \end{pmatrix}\\
        & = & \begin{pmatrix}
        0\\f_{(i,j^-),\omo{\wmu}{u''}}
        \end{pmatrix}\circ\fout{\xis_{u''}}\\
        & & +\fin{\xis_{v''}}\circ f_{\omo{\wnu}{v''},(i,j^-)}-\fin{\xis_{v'''}}\circ f_{\omo{\wnu}{v'''},(i,j^-)}\\
        & =   & \fin{\xis_{v}}\circ f_{\omo{\wnu}{v},\omo{\wmu}{u''}}\circ\fout{\xis_{u''}}-f_{\omo{\wnu}{v'},\omo{\wmu}{u''}}\circ\fout{\xis_{u''}}\\
        & & +\fin{\xis_{v''}}\circ f_{\omo{\wnu}{v''},\omo{\wmu}{u}}\circ\fout{\xis_{u}}-\fin{\xis_{v'''}}\circ f_{\omo{\wnu}{v'''},\omo{\wmu}{u}}\circ\fout{\xis_{u}}\\
        & = & \psi(\asu{\wsx}{u''}{u},\asu{\wtx}{v''}{v})-\psi(\asu{\wsx}{u''}{u},\asu{\wtx}{v'}{v'''}).
        \end{array}$$
        In this case, there is an arrow in $Q^{\ws,\wt[-1]}_{2,o}$ from $(u,v')$ to $(u,v)$.
        \item The non-interior arc segment in the punctured side of $(u,v)$ is of $\wtx$ and has tagging $+$. Then $\ec{u}=\{u,u'\}$, $\ec{v}=\{v\}$ and $\wks_u=\wkt_v=(i,j)^+$. Write the unpunctured arc segments $\asu{\wsx}{u''}{u}$, $\asu{\wsx}{u'}{u'''}$ and $\asu{\wsx}{v''}{v}$. So the unpunctured side of $(u,v)$ (resp. $(u',v)$) is $(\asu{\wsx}{u''}{u},\asu{\wtx}{v''}{v})$ (resp. $(\asu{\wsx}{u'}{u'''}, \asu{\wsx}{v''}{v})$). Note that $\omo{\wmu}{u}=\omo{\wmu}{u'}=(i,j)$ and $\omo{\wnu}{v}=(i,j^+)$. Then we have
        $$\begin{array}{rcl}
        d(\iota(u,v)) & = & \begin{pmatrix}
        1&0
        \end{pmatrix}\circ\fin{\xis_{u}}\circ f_{(i,j),\omo{\wmu}{u''}}\circ\fout{\xis_{u''}}+\begin{pmatrix}
        1&0
        \end{pmatrix}\circ\fin{\xis_{u'}}\circ f_{(i,j),\omo{\wmu}{u'''}}\circ\fout{\xis_{u'''}}\\
         & & +\fin{\xis_{v''}}\circ f_{\omo{\wnu}{v''},(i,j^+)}\circ\fout{\xis_{v}}\circ\begin{pmatrix}
        1&0
        \end{pmatrix}\\
        & = & f_{(i,j^+),\omo{\wmu}{u''}}\circ\fout{\xis_{u''}}+ f_{(i,j^+),\omo{\wmu}{u'''}}\circ\fout{\xis_{u'''}}\\
         & & +\fin{\xis_{v''}}\circ \begin{pmatrix}
        f_{\omo{\wnu}{v''},(i,j^+)}&0
        \end{pmatrix}\\
        & =   & \fin{\xis_{v}}\circ f_{\omo{\wnu}{v},\omo{\wmu}{u''}}\circ\fout{\xis_{u''}}+ \fin{\xis_{v}}\circ f_{\omo{\wnu}{v},\omo{\wmu}{u'''}}\circ\fout{\xis_{u'''}}\\
        & &+\fin{\xis_{v''}}\circ f_{\omo{\wnu}{v''},\omo{\wmu}{u}}\circ\fout{\xis_{u}}+\fin{\xis_{v''}}\circ f_{\omo{\wnu}{v''},\omo{\wmu}{u'}}\circ\fout{\xis_{u'}}\\
        & = & \psi(\asu{\wsx}{u''}{u},\asu{\wtx}{v''}{v})+\psi(\asu{\wsx}{u'}{u'''},\asu{\wtx}{v''}{v}).
        \end{array}$$
        In this case, there is an arrow in $Q^{\ws,\wt[-1]}_{2,o}$ from $(u',v)$ to $(u,v)$.
        \item The non-interior arc segment in the punctured side of $(u,v)$ is of $\wtx$ and has tagging $-$. Then $\ec{u}=\{u,u'\}$, $\ec{v}=\{v\}$ and $\wks_u=\wkt_v=(i,j)^-$. Write the unpunctured arc segments $\asu{\wsx}{u''}{u}$, $\asu{\wsx}{u'}{u'''}$ and $\asu{\wsx}{v''}{v}$. So the unpunctured side of $(u,v)$ (resp. $(u',v)$) is $(\asu{\wsx}{u''}{u},\asu{\wtx}{v''}{v})$ (resp. $(\asu{\wsx}{u'}{u'''}, \asu{\wsx}{v''}{v})$). Note that $\omo{\wmu}{u}=\omo{\wmu}{u'}=(i,j)$ and $\omo{\wnu}{v}=(i,j^-)$. Then we have
        $$\begin{array}{rcl}
        d(\iota(u,v)) & = & \begin{pmatrix}
        0&1
        \end{pmatrix}\circ\fin{\xis_{u}}\circ f_{(i,j),\omo{\wmu}{u''}}\circ\fout{\xis_{u''}}+\begin{pmatrix}
        0&1
        \end{pmatrix}\circ\fin{\xis_{u'}}\circ f_{(i,j),\omo{\wmu}{u'''}}\circ\fout{\xis_{u'''}}\\
         & & +\fin{\xis_{v''}}\circ f_{\omo{\wnu}{v''},(i,j^-)}\circ\fout{\xis_{v}}\circ\begin{pmatrix}
        0&1
        \end{pmatrix}\\
        & = & f_{(i,j^-),\omo{\wmu}{u''}}\circ\fout{\xis_{u''}}+\fin{\xis_{v''}}\circ \begin{pmatrix}
        0&f_{\omo{\wnu}{v''},(i,j^-)}
        \end{pmatrix}\\
        & =   & \fin{\xis_{v}}\circ f_{\omo{\wnu}{v},\omo{\wmu}{u''}}\circ\fout{\xis_{u''}}+\fin{\xis_{v''}}\circ f_{\omo{\wnu}{v''},\omo{\wmu}{u}}\circ\fout{\xis_{u}}\\
        & = & \psi(\asu{\wsx}{u''}{u},\asu{\wtx}{v''}{v}).
        \end{array}$$
        In this case, there is an arrow in $Q^{\ws,\wt[-1]}_{2,o}$ from $(u,v)$ to $(u',v)$.
    \end{enumerate}

    If $\hks_u=\hkt_v$ is an edge of a once-punctured monogon and none of arc segments in the punctured side of $(u,v)$ is interior, since $\wks_u=\wkt_v$, the taggings of the arc segments in the punctured side of $(u,v)$ are the same. So we have the required formula.

    If $\hks_u=\hkt_v$ is not an edge of a once-punctured monogon, this case is trivial.


\end{document}